\documentclass[11pt,noamsfonts,a4paper]{amsart}
\usepackage[left=1in, right=1in, top=1in, bottom=1in]{geometry}
\usepackage{mathtools}
\usepackage{braket}
\usepackage[charter]{mathdesign}
\usepackage[tracking]{microtype}
\usepackage{tabularx}
\usepackage{etoolbox}
\usepackage{graphicx}
\usepackage{stmaryrd} % \mapsfrom
\usepackage{comment}

\usepackage{amsmath}
\usepackage{amsthm}
\newtheoremstyle{pineapple}%
  {1em}{1em}%
  {\itshape}{}%
  {\bfseries}{. ---}
  {0.5em}{}

\newtheoremstyle{durian}%
  {1em}{1em}%
  {}{}%
  {\bfseries}{. ---}
  {0.5em}{}

\makeatletter
\def\swappedhead#1#2#3{%
  \thmnumber{\@upn{\the\thm@headfont#2\@ifnotempty{#1}{.~}}}%
  \thmname{#1}%
  \thmnote{ {\the\thm@notefont(#3)}}}
\makeatother

\theoremstyle{pineapple}
\newtheorem{IntroTheorem}{Theorem}

\swapnumbers
\newtheorem{Theorem}[subsection]{Theorem}
\newtheorem{Lemma}[subsection]{Lemma}
\newtheorem{Proposition}[subsection]{Proposition}
\newtheorem{Corollary}[subsection]{Corollary}

\theoremstyle{durian}

\newtheorem{Example}[subsection]{Example}

\usepackage[dvipsnames]{xcolor}
\usepackage{hyperref} %always add last
\hypersetup{
  colorlinks = true,
  linkcolor = Fuchsia,
  urlcolor = Fuchsia,
  citecolor = ForestGreen,
  linkbordercolor = {white}}

\usepackage{tikz-cd}
\usetikzlibrary{arrows}

\tikzset{
  symbol/.style={
    draw=none,
    every to/.append style={
      edge node={node [sloped, allow upside down, auto=false]{$#1$}}}
  }
}

\tikzset{>={Straight Barb[length=2pt,width=4pt]}, commutative diagrams/arrow style=tikz}

\usepackage{enumitem}
\setlist[1]{labelindent=\parindent}
\setlist[1]{labelsep=0.5em}
\setlist[enumerate,1]{label={\upshape (\roman*)}, ref={\upshape (\roman*)}}
\setlist[enumerate]{leftmargin=*, itemsep={2pt}}
\setlist[itemize,1]{leftmargin={2em},itemsep={2pt},label={--}}

\makeatletter
\def\@sect#1#2#3#4#5#6[#7]#8{%
  \edef\@toclevel{\ifnum#2=\@m 0\else\number#2\fi}%
  \ifnum #2>\c@secnumdepth \let\@secnumber\@empty
  \else \@xp\let\@xp\@secnumber\csname the#1\endcsname\fi
  \@tempskipa #5\relax
  \ifnum #2>\c@secnumdepth
    \let\@svsec\@empty
  \else
    \refstepcounter{#1}%
    \edef\@secnumpunct{%
      \ifdim\@tempskipa>\z@ % not a run-in section heading
        \@ifnotempty{#8}{.~}%
      \else
        \@ifempty{#8}{.}{.~}%
      \fi
    }%
    \@ifempty{#8}{%
      \ifnum #2=\tw@ \def\@secnumfont{\bfseries}\fi}{}%
    \protected@edef\@svsec{%
      \ifnum#2<\@m
        \@ifundefined{#1name}{}{%
          \ignorespaces\csname #1name\endcsname\space
        }%
      \fi
      \@seccntformat{#1}%
    }%
  \fi
  \ifdim \@tempskipa>\z@ % then this is not a run-in section heading
    \begingroup #6\relax
    \@hangfrom{\hskip #3\relax\@svsec}{\interlinepenalty\@M #8\par}%
    \endgroup
    \ifnum#2>\@m \else \@tocwrite{#1}{#8}\fi
  \else
  \def\@svsechd{#6\hskip #3\@svsec
    \@ifnotempty{#8}{\ignorespaces#8\unskip
       \@addpunct.}%
    \ifnum#2>\@m \else \@tocwrite{#1}{#8}\fi
  }%
  \fi
  \global\@nobreaktrue
  \@xsect{#5}}
\makeatother

\makeatletter
\def\@seccntformat#1{%
  \protect\textup{\protect\@secnumfont
    \ifnum\pdfstrcmp{subsection}{#1}=0 \bfseries\fi% subsection # in \bfseries
    \csname the#1\endcsname
    \protect\@secnumpunct
  }%
}
\makeatother

\makeatletter
\newcommand*{\coloneqq}{\mathrel{\rlap{%
           \raisebox{0.3ex}{$\m@th\cdot$}}%
           \raisebox{-0.3ex}{$\m@th\cdot$}}%
           =}
\newcommand{\eqqcolon}{=%
           \mathrel{\rlap{%
           \raisebox{0.3ex}{$\m@th\cdot$}}%
           \raisebox{-0.3ex}{$\m@th\cdot$}}}
\makeatother

\makeatletter
\newcommand*\rel@kern[1]{\kern#1\dimexpr\macc@kerna}
\newcommand*\widebar[1]{%
  \begingroup
  \def\mathaccent##1##2{%
    \rel@kern{0.8}%
    \overline{\rel@kern{-0.8}\macc@nucleus\rel@kern{0.2}}%
    \rel@kern{-0.2}%
  }%
  \macc@depth\@ne
  \let\math@bgroup\@empty \let\math@egroup\macc@set@skewchar
  \mathsurround\z@ \frozen@everymath{\mathgroup\macc@group\relax}%
  \macc@set@skewchar\relax
  \let\mathaccentV\macc@nested@a
  \macc@nested@a\relax111{#1}%
  \endgroup
}
\makeatother

\DeclareMathOperator{\Spec}{Spec}
\DeclareMathOperator{\coker}{coker}
\DeclareMathOperator{\Sym}{Sym}
\DeclareMathOperator{\Proj}{Proj}

\DeclareMathOperator{\Mor}{Mor}
\DeclareMathOperator{\Hom}{Hom}
\DeclareMathOperator{\Ext}{Ext}

\DeclareMathOperator{\image}{im}
\DeclareMathOperator{\Sing}{Sing}
\DeclareMathOperator{\rad}{rad}
\DeclareMathOperator{\corank}{corank}
\DeclareMathOperator{\rank}{rank}
\DeclareMathOperator{\Disc}{\mathbf{Disc}}
\DeclareMathOperator{\sDisc}{\mathcal{D}\mathit{isc}}
\DeclareMathOperator{\Br}{Br}
\DeclareMathOperator{\QCoh}{QCoh}
\DeclareMathOperator{\id}{id}

\AtBeginDocument{%
   \def\MR#1{}
}

\newcommand{\punct}[1]{\makebox[0pt][l]{\,#1}} %so that punctuation does not mess with tikz layout
\newcommand{\parref}[1]{{\bf\ref{#1}}}

\newcommand{\Dqc}{\mathrm{D}_{\mathrm{qc}}}
\newcommand{\Dbcoh}{\mathrm{D}^{\mathrm{b}}_{\mathrm{coh}}}
\newcommand{\PP}{\mathbf{P}}

\newcommand{\kk}{\mathbf{k}}
\newcommand{\sO}{\mathcal{O}}
\newcommand{\sHom}{\mathcal{H}\!\mathit{om}}
\newcommand{\sAut}{\mathcal{A}\!\mathit{ut}}
\newcommand{\Ku}{\mathrm{Ku}}
\newcommand{\pr}{\mathrm{pr}}
\newcommand{\Cl}{\mathit{C}\ell}
\newcommand{\hrefSP}[1]{\href{https://stacks.math.columbia.edu/tag/#1}{#1}}
\newcommand{\citeSP}[1]{\cite[\hrefSP{#1}]{stacks-project}}
\newcommand{\sCoh}{\mathcal{C}\!\mathit{oh}}
\newcommand{\ssCoh}{\mathit{s}\mathcal{C}\!\mathit{oh}}
\newcommand{\subsectiondash}[1]{\subsection{#1}\textbf{---}\;}

\newcommand{\liset}{\mathrm{lis}\text{-}\mathrm{\acute{e}t}}

\newcommand{\ComplexesStack}{\mathcal{C}\mathit{omplexes}}
\newcommand{\sComplexesStack}{\mathit{s}\mathcal{C}\mathit{omplexes}}

\address{SB MATH CAG \\
  EPFL \\
  Station 8 \\
  1015 Lausanne \\
  Switzerland
}
\email{raymond.cheng@epfl.ch}

\address{
  Department of Mathematics \\
  UC Berkeley \\
  Berkeley, CA 94720-3840 \\
  United States of America
}
\email{nolander@berkeley.edu}

\title{Derived categories of quadric bundles and moduli stacks of spinor sheaves}
\author{Raymond Cheng}
\author{Noah Olander}
\begin{document}
\thispagestyle{empty}
\begin{abstract}
We prove that the Kuznetsov component of a flat family of even-dimensional
quadrics of corank at most \(2\) is equivalent to the twisted derived category
of an algebraic space whenever: (i) the open subset of the base over which the
quadrics has corank at most \(1\) is scheme-theoretically dense; and (ii) a
certain \'etale double cover of the closed complement admits a section. This
provides the first general geometricity result for Kuznetsov components of
higher dimensional quadrics, thereby generalizing works of Kapranov, Bondal,
Orlov, Kuznetsov, Moschetti, Xie, and others. Our main tool is the moduli stack
of spinor sheaves on a family of quadrics, which we define and study in detail.
In the situation of our main result, we produce an open substack which is a
$\mathbf{G}_m$-gerbe, and show that the associated twisted derived category is
equivalent to the Kuznetsov component of the family of quadrics, thereby
providing a geometric interpretation of the Brauer classes appearing in
previous works.
\end{abstract}
\setcounter{tocdepth}{1}
\maketitle
\tableofcontents
\thispagestyle{empty}

\section*{Introduction}
Quadric bundles, being geometrically rich yet tractable, abound in algebraic
geometry. Often, varieties of interest may be fibred in, or else otherwise
related to quadrics, at which point invariants may be determined by studying a
family of quadrics. Such a method has been particularly successful in the study
of derived categories of varieties, especially in regards to rationality
problems: see, for example, \cite{Kuznetsov:Survey, AB:Survey} for surveys.
Results and techniques for quadric bundles are therefore foundational for
applications. Our aim here is to generalize some of the basic results regarding
derived categories of quadric bundles, making the theory applicable for higher
dimensional quadrics and for more general base schemes; in particular, we make
an effort to include points of even characteristic throughout. Our general
treatment furthermore allows us to use moduli-theoretic techniques in our
constructions.

Let \(\rho \colon Q \to S\) be a quadric bundle, by which we mean a flat and
proper family of degree \(2\) hypersurfaces in a \(\PP^n\)-bundle
\(\pi \colon \PP\mathcal{E} \to S\) over a base scheme \(S\). Each
fibre of \(\rho\) is a Fano variety of index \(n-1\), so the
derived category of the total space \(Q\) admits an \(S\)-linear
semiorthogonal decomposition of the form
\[
\Dqc(Q) =
\langle \Ku(Q),\,
\rho^*\Dqc(S) \otimes \sO_\rho,\,
\rho^*\Dqc(S) \otimes \sO_\rho(1),\,
\ldots,\,
\rho^*\Dqc(S) \otimes \sO_\rho(n-2)
\rangle.
\]
The \emph{Kuznetsov component} \(\Ku(Q)\) is the ``interesting component'' of
\(\Dqc(Q)\) and it is of interest to identify it more explicitly. In general,
\(\Ku(Q)\) may be described algebraically as the derived category of the sheaf
of even Clifford algebras of $\rho \colon Q \to S$, a certain finite locally
free non-commutative $\mathcal{O}_S$-algebra, see \cite[Theorem 4.2]{Kuznetsov:Quadrics}.

A more geometric description of \(\Ku(Q)\) is often desirable and potentially
more useful. Kapranov's decomposition from \cite[\S4]{Kapranov:Homogeneous} for
a quadric hypersurface \(Q \subset \PP^{2\ell+1}\) might be considered the
first example: here, \(S\) is a point and \(\Ku(Q)\) is equivalent to two
copies of \(\Dqc(S)\). Bondal and Orlov, in their seminal work \cite{BO:SOD}
on semiorthogonal decompositions of algebraic varieties, then give perhaps
the first nontrivial example in the case \(\rho \colon Q \to \PP^1\) is a
general pencil of smooth even dimensional quadrics, wherein they show that
\(\Ku(Q)\) is equivalent to the derived category of the classically associated
hyperelliptic curve. General recent interest around these questions perhaps
stems from Kuznetsov's work in \cite[Theorem 4.3]{Kuznetsov:Cubic}, where he
proves that
%over an algebraically closed field of characteristic \(0\),
for the quadric surface fibration \(\rho \colon Q \to \PP^2\) associated with a
smooth cubic fourfold containing a plane, if all fibres of \(\rho\) have at
worst isolated singularities, then \(\Ku(Q)\) is equivalent to the derived
category of the K3 double cover of \(\PP^2\) branched along the discriminant
divisor of \(\rho \colon Q \to \PP^2\), though perhaps twisted by a Brauer
class; the hypothesis on the fibres were later removed by \cite{Moschetti}.
His
extensive study of this case led Kuznetsov to propose that a cubic fourfold
should be rational if and only if its Kuznetsov component is equivalent to the
untwisted derived category of a K3 surface. Spectacular recent progress was
made on this conjecture by the recent work of Katzarkov, Kontsevich, Pantev,
and Yu on Hodge atoms, see especially \cite[Theorem 6.8 and Examples 6.17--6.21]{KKPY}.
% the category $\operatorname{Ku}(Q)$ 
% can be described as the twisted (with respect to some Brauer class) derived 
% category of the double cover of $S$ branched along the degeneracy locus of 
% $Q \to S$. This is an important part of Kuznetsov's work on rationality of 
% cubic fourfolds. In \cite{Moschetti}, the assumption that the corank of each fiber is $\leq 1$
% is removed.
% These results are further generalized in \cite{Xie:Quadrics}
% to the case where $Q \to S$ is a flat family of quadric surfaces where $S$ is a smooth projective
% surface over an algebraically closed field of characteristic $0$ and $Q$ is a smooth fourfold. 

Returning to general quadric bundles, Xie shows in \cite[Theorem 5.10]{Xie:Quadrics}
that over an algebraically closed field of characteristic \(0\), for a flat
family \(\rho \colon Q \to S\) of quadric surfaces over a smooth projective
surface, \(\Ku(Q)\) is equivalent to the twisted derived category of a
resolution of the discriminant double cover of \(S\). She further conjectures
that whenever a quadric surface bundle $Q\to S$ over an integral Noetherian
scheme has simple degeneration generically and each fibre has corank $\leq 2$,
the Kuznetsov component of $Q$ is equivalent to the twisted derived category of
a scheme. Theorems 4.2 and 4.4 in \emph{loc. cit.} show that such schemes,
Brauer classes, and equivalences exist \'etale locally on $S$, suggesting
an approach via glueing. One view of our work here is that we carry out this
program.

Our main result identifies the Kuznetsov component of a quadric bundle
with a twisted derived category of an algebraic space \(M\), generically a double cover over \(S\), whenever the
degeneracy loci
\[
S_c
\coloneqq \{s \in S : \operatorname{corank} Q_s \geq c\}
= \{s \in S : \dim \Sing Q_s \geq c-1\}
\]
stratifying \(S\) with respect to the singularities in the fibres of \(\rho\)
satisfy three conditions. For example, the hypotheses below are satisfied
whenever \(\rho \colon Q \to S\) is a quadric bundle of even relative dimension
over an algebraically closed field with \(S_3 = \varnothing\) and \(S_2\) a
finite set of closed points. %In general, we prove the following generalization
%of the results of Fei Xie from \cite{Xie:Quadrics}:

\begin{IntroTheorem}\label{intro-general-result}
Let \(\rho \colon Q \to S\) be a quadric bundle of relative dimension
\(2\ell\). Assume that:
\begin{enumerate}
\item\label{intro-general-result.S3}
\(S_3 = \varnothing\);
\item\label{intro-general-result.associated-points}
\(S_2\) contains no weakly associated points of \(S\); and
% \item\label{intro-general-result.depth}
% \(\operatorname{depth}\sO_{S,s} \geq 2\) for every \(s \in S_2\); and
\item\label{intro-general-result.covering}
the \'etale double cover \(\tilde{S}_2 \to S_2\) parameterizing families of
\((\ell+1)\)-planes in \(\rho \colon Q \to S\) splits.
\end{enumerate}
Then there exists an algebraic space \(M\), a proper morphism \(M \to S\) which
is finite of degree \(2\) away from \(S_2\), a Brauer class
\(\beta \in \operatorname{Br}(M)\), and an \(S\)-linear
equivalence
\[
\Phi \colon
\Dqc(M,\beta) \to
\Ku(Q).
\]
\end{IntroTheorem}

A slightly more precise version is given in \parref{spinor-main-theorem}.
Although this and other results here are phrased in terms of the full
quasi-coherent derived category, statements for, say, the bounded derived
category in suitable situations may be deduced by adapting the arguments from
 \cite[Theorem 6.2]{BS:Descent}.  The hypotheses are discussed in more detail
in \parref{spinors-hypotheses}, though briefly: When \(S\) is integral,
\ref{intro-general-result.associated-points} simply means that \(S_2 \neq S\).
The covering \(\tilde{S}_2 \to S_2\) in \ref{intro-general-result.covering} is
discussed in \parref{quadric-bundles-fano-schemes-corank-2}. The algebraic
space \(M\) is canonically an open subspace---depending on a choice of section
of \(\tilde{S}_2 \to S_2\)---of the coarse moduli space of spinor sheaves for
the quadric bundle \(\rho \colon Q \to S\), \(\beta\) is the Brauer obstruction
for the existence of a universal sheaf on \(Q \times_S M\), and the
equivalence is induced by the Fourier--Mukai transform with respect to the
universal sheaf on the associated moduli stack. In this way, our result loosely
says that the Kuznetsov component of an even-dimensional quadric bundle is
generated by its spinor sheaves, making it apparent that this is a vast
generalization of Kapranov's semiorthogonal decomposition from
\cite[\S4]{Kapranov:Homogeneous} of a smooth quadric \(Q\) of dimension
\(2\ell\) over an algebraically closed field:
\[
\Dqc(Q) =
\langle
\mathcal{S}_+^\vee,
\mathcal{S}_-^\vee,
\sO_Q,
\sO_Q(1),
\ldots,
\sO_Q(2\ell-1)
\rangle,
\]
where \(\mathcal{S}_+^\vee\) and \(\mathcal{S}_-^\vee\) are the even and odd
spinor bundles on \(Q\). For example, when \(\ell = 1\) so that
\(Q \cong \PP^1 \times \PP^1\), the spinors are
\(\mathcal{S}_+^\vee = \sO_Q(-1,0)\) and
\(\mathcal{S}_-^\vee = \sO_Q(0,-1)\).

Our approach to Theorem \parref{intro-general-result} begins with a careful
study of the local situation on \(S\), wherein \(\rho \colon Q \to S\) admits a
\emph{regular section}; geometrically, this is a section
\(\sigma \colon S \to Q\) contained in the smooth locus of \(\rho\). In this
case, the scheme
\[
Q' \coloneqq \{[\ell] \in \mathbf{F}_1(Q/S) : \sigma \in \ell\}
\]
parameterizing lines in the fibres of \(\rho \colon Q \to S\) passing through
the section \(\sigma\) is another family of quadrics \(\rho' \colon Q' \to S\),
possibly not flat and of \(2\) dimensions less, called the \emph{hyperbolic
reduction of \(\rho \colon Q \to S\) along \(\sigma\)}. This construction is
well-known, perhaps popularized in the setting of derived categories by the
work \cite[\S\S1.3--1.4]{ABB} of Auel--Bernardara--Bolognesi, and has since
featured in many other works, such as \cite{Xie:Quadrics, KS:Quadrics, JS};
perhaps most notably for us, Kuznetsov shows in \cite[Proposition
1.1(3)]{Kuznetsov:Reduction} that, often, the even Clifford algebras associated
with \(Q'\) and \(Q\) are Morita equivalent. In particular, in view of
Kuznetsov's description of the derived category of quadric bundles from
\cite[Theorem 4.2]{Kuznetsov:Quadrics}, this implies that \(\Ku(Q')\) and
\(\Ku(Q)\) are equivalent.

The next result provides a generalization of this fact. Our result is applicable
for more general base schemes \(S\)---in particular, for \(S\) on which \(2\)
may not be invertible---as well as quadric bundles that may not be generically
smooth. Moreover, we can explicitly identify the kernel underlying the
equivalence between \(\Ku(Q')\) and \(\Ku(Q)\).

\begin{IntroTheorem}\label{intro-hyperbolic-reduction}
Let \(\rho \colon Q \to S\) be a quadric bundle of relative dimension
\(n - 1 \geq 2\), and let \(\rho' \colon Q' \to S\) be its hyperbolic reduction
along a regular section. Assume that \(S_{n-1}\) does not contain any weakly
associated points of \(S\). Then there exists an \(S\)-linear equivalence
\(
\Phi \colon \Ku(Q') \to \Ku(Q)
\).
\end{IntroTheorem}

This is the content of \parref{hypred-equivalence}. As above, when \(S\) is
integral, the hypothesis regarding weakly associated points means that
\(S_{n-1} \neq S\); see \parref{lemma-cartier-divisor-criterion} for its import
here. The equivalence \(\Phi\) is of Fourier--Mukai type, and
\parref{proposition-identify-the-kernel} identifies its kernel as a twist of
the ideal sheaf of the family over \(Q'\) of lines in \(Q\) through the chosen
section. An important feature of \(\Phi\) is that it preserves dual spinor
sheaves: see \parref{cliff-hyperbolic-reduction-spinors}.

Iterating this equivalence gives: If \(\rho \colon Q \to S\) is a quadric
bundle of relative dimension \(2\ell\) which carries a complete flag of linear
spaces ending in a family of \((\ell-1)\)-planes \(\PP\mathcal{W}\) contained
in the smooth locus of \(\rho\)---so that fibres have corank at most
\(2\)!---and \(S_2\) does not contain any weakly associated points of \(S\),
then there is an \(S\)-linear equivalence
\[
\Psi \colon \Dqc(M) \to \Ku(Q)
\]
where \(\mu \colon M \to S\) is the (maximal) hyperbolic reduction of
\(\rho \colon Q \to S\) along \(\PP\mathcal{W}\). The morphism \(\mu\) is
generically finite of degree \(2\) and has \(\PP^1\) as its geometric fibres
over points of \(S_2\). Crucially, \(\Psi\) is identified in
\parref{reduction-fourier-mukai-equivalence} as the Fourier--Mukai transform
induced by a family of dual spinor sheaves on \(Q \times_S M\), essentially
verifying the folklore fact that \(M\) is a moduli space of spinors for \(Q\).

This suggests a natural approach to Theorem \parref{intro-general-result} in
general: Consider a moduli stack \(\widebar{\mathcal{M}} \to S\) of dual spinor
sheaves for \(\rho \colon Q \to S\) and study the Fourier--Mukai transform
associated with the universal sheaf on \(Q \times_S \widebar{\mathcal{M}}\). In
pursuing this idea, however, one is confronted with the conundrum that the
coarse moduli space of \(\widebar{\mathcal{M}}\) is not separated over
\(S\) because its geometric fibres over \(S_2\) consist of two copies of
\(\PP^1\) and an extra point. This is most readily seen in case of a quadric
surface \(Q\): Spinor sheaves here are essentially ideal sheaves
\(\mathcal{I}_{\ell/Q}\) of lines \(\ell \subset Q\), see
\parref{reduction-quadric-surfaces-example}. When
\(Q = \PP^2 \cup_{\ell_0} \PP^2\) has corank \(2\) with singular line
\(\ell_0\), the isomorphism class of \(\mathcal{I}_{\ell/Q}\) for
\(\ell \neq \ell_0\) is determined by the irreducible component it lies on and
the intersection point \(\ell \cap \ell_0\): see \cite[Example
5.2]{Hartshorne:Divisors}. To proceed, one must be able to consistently choose
a family of spinor sheaves on corank \(2\) fibres, and this is encoded by
condition \ref{intro-general-result.covering} in Theorem
\parref{intro-general-result}. Fixing a choice of family determines an open
substack \(\mathcal{M} \subset \widebar{\mathcal{M}}\), on which the strategy
may be carried through.
\smallskip

\noindent\textbf{Technical tools. --- }
Implementing this argument required developing several general technical
results which may be of independent interest. First is a result which allows us
to verify whether or not an \(S\)-linear Fourier--Mukai functor is an equivalence
fppf locally on \(S\).  When analyzing the functor
\(\Dqc(\mathcal{M}) \to \Dqc(S)\) in the setting of Theorem
\parref{intro-general-result}, this allows us to replace the base \(S\)
by a cover on which \(\rho \colon Q \to S\) admits a regular section, at which
point Theorem \parref{intro-hyperbolic-reduction} applies. Our statement, set
up and stated in \parref{thm-locallytwisted}, is a variant of Bergh's and
Schn\"urer's conservative descent from \cite{BS:Descent}, but modified to apply
in a setting where the source may be a \(\mathbf{G}_m\)-gerbe over a proper and
perfect algebraic stack and where the Fourier--Mukai kernel is only relatively
perfect over the two factors.

Second is a precise formulation of the general philosophy, originating from
ideas of Mukai, Orlov and Lieblich and Olsson, that Fourier--Mukai transforms correspond to
morphisms of moduli spaces of complexes. Combined with the descent technique
above, this allows us to relate derived category computations with geometric
properties of the stack of spinors. Our statements here are phrased in terms of
Lieblich's stack of complexes on a flat, proper, and finitely-presented
morphism \(X \to S\) as introduced in \cite{Lieblich:Complexes}. We formulate a
Yoneda-type lemma in \parref{complexes-yoneda} which canonically
relates \(T\)-valued points of the stack of complexes with certain complexes
in \(\Dqc(X \times_S T)\). We then show that complexes underlying fully faithful
Fourier--Mukai transforms induce open immersions on stacks of complexes:

\begin{IntroTheorem}
Let \(X\) and \(Y\) be flat, proper, and finitely-presented schemes over
a scheme \(S\), and let
\[
\Phi_K \colon \Dqc(X) \to \Dqc(Y)
\]
be the Fourier--Mukai transform associated with an object
\(K \in \Dqc(X \times_S Y)\) which is perfect relative to both \(X\) and \(Y\). If $\Phi_K$ is fully faithful, %need this condition for FM_K to be defined 
the assignment \( (T,E) \mapsto (T,\Phi_{K_T}(E))\) induces a morphism of stacks
\[
\mathrm{FM}_K \colon
\ComplexesStack_{X/S} \to
\ComplexesStack_{Y/S}
\]
which is an open immersion
immersion.
\end{IntroTheorem}

This is \parref{theorem-fully-faithful-open-immersion}, and a variant for
objects in a semiorthogonal component is formulated in
\parref{theorem-fully-faithful-open-immersion-2}.  Related results
appear in \cite[Lemma 5.2]{LO:FM} and \cite[Proposition 3.2.3]{HLT}, for
example, but have been restricted to embedding spaces of skyscraper sheaves of
one into the stack of complexes of the other. We expect that this statement
is more generally applicable.
\smallskip

\noindent\textbf{Intersections of two quadrics. ---}
Two applications of Theorem \parref{intro-general-result} are sketched in
\S\parref{section-applications}. The first concerns the derived category of
an intersection \(X\) of \(m \leq 4\) even-dimensional quadrics. For the
Introduction, we focus on the most classical case, \(m = 2\);
results for \(m = 3\) and \(m = 4\) are given in
\parref{applications-intersections-result}\ref{applications-intersections-result.surface}
and
\parref{applications-intersections-result}\ref{applications-intersections-result.threefold}
and further discussed in \parref{applications-intersections-remarks}. In one of
the first works regarding semiorthogonal decompositions in algebraic geometry,
Bondal and Orlov show in \cite[\S2]{BO:SOD} that, over an algebraically closed
field of characteristic different from \(2\), the bounded derived category
\(\Dbcoh(X)\) of coherent sheaves contains that of of the hyperelliptic curve
\(C\), as already considered by Weil and Reid in \cite{Weil:Footnote,
Reid:Thesis}, arising from the associated pencil of quadrics. This was
interpreted as a ``categorical explanation'' of classical results of \cite{DS}
relating rank \(2\) vector bundles on \(C\) to a Fano scheme of linear spaces
in \(X\). Our results imply a generalization of the theorem of Bondal and Orlov
to an arbitrary base field:

\begin{IntroTheorem}\label{intro-theorem-intersection-quadrics}
Let \(X \subset \PP^{2\ell+1}\) be a smooth complete intersection of two general
quadrics over an arbitrary field \(\kk\). There is a semiorthogonal
decomposition
\[
\Dbcoh(X) = \langle
\Dbcoh(C,\alpha),
\sO_X(1),
\sO_X(2),
\ldots,
\sO_X(2\ell-2)
\rangle
\]
where \(C\) is a smooth projective curve and \(\Br(C)\).
\end{IntroTheorem}

This is
\parref{applications-intersections-result}\ref{applications-intersections-result.curve}.
For \(\kk \neq \overline{\kk}\), this appears to be new even in characteristic
\(0\); see, however, \cite[\S5]{ABB} for related results. It would be
interesting to relate the Brauer class appearing here to \(\kk\)-rationality of
\(X\) as studied in \cite{HT:Two-Quadrics, BW:Two-Quadrics}. For
\(\operatorname{char}\kk = 2\), this is completely new and may be seen as a
``categorical explanation'' of the work \cite{Bhosle} of Bhosle.
\smallskip

\noindent\textbf{Cubic fourfolds. ---}
The second application concerns smooth cubic fourfolds. Their derived
categories possess perhaps the most well-known semiorthogonal decomposition in
algebraic geometry: see \cite{Kuznetsov:Cubic, Huybrechts:K3,
Huybrechts:K3-update} for example. Specifically, since such \(X \subset \PP^5\)
is a Fano variety of index \(3\), there is a semiorthogonal decomposition
\[
\Dbcoh(X) =
\langle
\mathcal{A}_X, \sO_X, \sO_X(1), \sO_X(2)
\rangle.
\]
The category \(\mathcal{A}_X\) looks like the derived category of a K3 surface:
Its Serre functor is the shift-by-two functor by \cite[Corollary
4.1]{Kuznetsov:CY} and the dimensions of its Hochschild homology coincide with
the Betti numbers of a K3. This latter statement may be deduced from the weak
form of the Hochschild--Kostant--Rosenberg theorem from \cite[Example
1.7]{AV:HKR}.

Whether or not \(\mathcal{A}_X\) is equivalent to the (twisted) derived
category of an actual K3 surface depends on the geometry of \(X\)---see
\cite[Theorem 1.1]{AT:Cubic}, \cite[Theorem 1.4]{Huybrechts:K3}, and
\cite[Theorem 5.1]{Huybrechts:K3-update}---and is conjectured in
\cite[Conjecture 1.1]{Kuznetsov:Cubic} to be related to the rationality of
\(X\). The most well-studied example for when \(\mathcal{A}_X\) is twisted
geometric is the case \(X\) contains a plane. Kuznetsov showed in
\cite[\S4]{Kuznetsov:Cubic} that for a general such \(X\) over a field of
characteristic different from \(2\),
 \(\mathcal{A}_X \simeq \Dbcoh(S,\alpha)\) where \(S\) is a double sextic K3
and \(\alpha \in \mathrm{Br}(S)\). This was extended to all smooth cubic
fourfolds containing a plane over an algebraically closed field of
characteristic zero by Moschetti in \cite[Theorem 1.2]{Moschetti}; Xie later
gives in \cite[Example 6.2]{Xie:Quadrics} a much more direct proof in this
setting. Our results apply to prove this in any characteristic: see
\parref{applications-cubic-fourfolds-decomposition}.

Rather than quote the general result, we highlight a particularly striking
example: Consider the Fermat cubic fourfold \(X \subset \PP^5\) over an
algebraically closed field \(\kk\). Of course, \(X\) is extremely special
for a variety of reasons. For example, when \(\kk = \mathbf{C}\), it has
the largest automorphism group by \cite[Corollary 6.14]{LZ:Automorphisms}, it
contains the most planes by \cite{DIO:Planes}, and it is contained in every
Hassett divisor by \cite[Theorem 1.2]{YY:Fermat}. When
\(\operatorname{char}\kk = 2\), the geometry of \(X\) becomes in many ways
even more special. Most notably in this context, every fibre of any quadric
surface bundle obtained via projection from a plane in \(X\) is singular!
Nevertheless, \parref{applications-cubic-fourfolds-decomposition} associates
with \(X\) a K3 surface \(S\), and \parref{applications-cubic-fourfolds-fermat}
furthermore determines \(S\) as the most special K3 over \(\kk\):

\begin{IntroTheorem}\label{intro-theorem-fermat-cubic}
Let \(X \subset \PP^5_\kk\) be the Fermat cubic fourfold over an algebraically
closed field \(\kk\) of characteristic \(2\). Then there is a semiorthgonal
decomposition
\[
\Dbcoh(X) =
\langle \Dbcoh(S), \sO_X, \sO_X(1), \sO_X(2) \rangle
\]
where \(S\) is the supersingular K3 surface of Artin invariant \(1\).
\end{IntroTheorem}

In fact, the computations in \parref{applications-cubic-fourfolds-fermat}
are valid over any field \(\kk\) and relate \(X\) to an explicit \(S\) given as
a complete intersection in \(\PP^2 \times \PP^2\). A precise relationship
between \(X\) and \(S\) taken over a number field \(\kk\) appears to
have first appeared in \cite[\S5]{HK:L-series}, where the global Hasse--Weil
zeta function of \(X\) is related to that of \(S\). This may be viewed as a
motivic shadow of Theorem \parref{intro-theorem-fermat-cubic}, in line with
Orlov's conjecture from \cite{Orlov:Motives} that semiorthogonal summands
correspond to direct summands on the level of rational Chow motives.

\medskip\noindent\textbf{Outline. --- }%
The first three sections develop the general technical tools used in the
proof of the main results: \S\parref{section-derived} discusses derived
categories of algebraic stacks, the main result being the descent statement
\parref{thm-locallytwisted} for Fourier--Mukai transforms. Moduli stacks of
complexes are discussed in \S\parref{section-complexes} with the open immersion
induced by a fully faithful Fourier--Mukai transform in
\parref{theorem-fully-faithful-open-immersion}. Relative exceptional
collections and their residual categories are then discussed in
\S\parref{section-decompositions} where, notably, a generalization of
\parref{theorem-fully-faithful-open-immersion} to the stack of complexes
parameterizing objects in a residual category is formulated in
\parref{theorem-fully-faithful-open-immersion-2}.

Families of quadrics are discussed starting from
\S\parref{section-quadric-bundles}; since our base scheme is rather general,
extra care is required in discussing the corank stratification: see
\parref{quadric-bundles-corank}--\parref{quadric-bundles-even-rank}. Spinor
sheaves for families of quadrics are introduced in
\S\parref{section-clifford-and-spinor}, where we give new statements about
how spinors are related along non-generic hyperplane sections and cones:
see \parref{cliff-subbundle-situation} and \parref{cliff-cone-situation}.
In \S\parref{section-hyperbolic-reduction}, we study how the Kuznetsov
component of a quadric bundle behaves under hyperbolic reduction,
cumulating in the proof of Theorem \parref{intro-hyperbolic-reduction}.
In \S\parref{section-spinor-moduli}, we construct the stack of spinor sheaves
associated with a quadric bundle as the image of a smooth morphism from the
Fano scheme of \(\ell\)-planes: see \parref{prop-fanotocoh}.  After discussing
some properties of this stack, we prove Theorem \parref{intro-general-result},
see \parref{spinor-main-theorem}.  Finally, we end off with
\S\parref{section-applications} where we sketch applications of the main
results to intersections of quadrics in \parref{applications-intersections} and
cubic fourfolds in \parref{applications-cubic-fourfolds}.

\medskip\noindent\textbf{Conventions. --- }%
All stacks are algebraic stacks and we follow the conventions of the Stacks
Project: see \citeSP{026O}.
% A quasi-coherent sheaf on a stack $\mathcal{X}$ means a quasi-coherent module either on the ringed site $(\mathcal{X}_{fppf}, \mathcal{O}_{\mathcal{X}})$ or the ringed site $(\mathcal{X}_{lis-\acute{e}t}, \mathcal{O}_{\mathcal{X}})$, see \cite[\href{https://stacks.math.columbia.edu/tag/07B1}{Tag 07B1}]{stacks-project}.
We use the standard abbreviations \emph{fppf}, \emph{fpqc}, and \emph{qcqs} for
faithfully flat and finitely presented, faithfully flat and quasi-compact, and
quasi-compact and quasi-separated, respectively. Given an algebraic space
\(X\) and a finite locally free \(\sO_X\)-module \(\mathcal{E}\), we write
\(\pi \colon \PP\mathcal{E} \to X\) for the projective bundle of lines
in \(\mathcal{E}\), so that \(\pi_*\sO_{\pi}(1) \cong \mathcal{E}^\vee\).
Unless otherwise stated, all tensor products are taken over the structure sheaf
of the ambient space.

\medskip\noindent\textbf{Acknowledgements. ---}
A portion of this project was carried out during the Junior Trimester Program
in Algebraic Geometry at the Hausdorff Research Institute for Mathematics
during the autumn of 2023, funded by the Deutsche Forschungsgemeinschaft under
Germany's Excellence Strategy - EXC-2047/1 - 390685813. During the completion
of this work, RC was partially supported by a Humboldt Postdoctoral Research
Fellowship and NO was partially supported by the National Science Foundation
under grant DMS-2402087.

\section{Derived categories of algebraic stacks}
\label{section-derived}
In this section, we collect and develop some facts regarding the derived
categories of stacks. One of the primary aims is to formulate and prove a
descent result regarding Fourier--Mukai functors from a single weight
component of the derived category of a \(\mathbf{G}_m\)-gerbe to a scheme:
see \parref{thm-locallytwisted}.

\subsectiondash{}\label{derived-generalities}
We follow the conventions of \cite{hallrydh} regarding derived categories
of stacks. Namely, for a stack \(\mathcal{X}\), we write \(\Dqc(\mathcal{X})\)
for the full subcategory of
\(\mathrm{D}(\mathcal{X}_{\liset}, \sO_{\mathcal{X}})\) consisting of objects
with quasi-coherent cohomology sheaves. When \(\mathcal{X} = X\) is a scheme,
this coincides with the usual quasi-coherent derived category of \(X\).
In any case, \(\Dqc(\mathcal{X})\) admit the following operations:
\begin{enumerate}
\item\label{derived-generalities.tensor}
If \(E,F \in \Dqc(\mathcal{X})\), then
\(E \otimes^L F \in \Dqc(\mathcal{X})\).
\item\label{derived-generalities.hom}
If \(E,F \in \Dqc(\mathcal{X})\) with \(E\) perfect, then
\(R\sHom_{\sO_{\mathcal{X}}}(E,F) \in \Dqc(\mathcal{X})\).
\item\label{derived-generalities.adjoint}
If \(f \colon \mathcal{X} \to \mathcal{Y}\) is a morphism of stacks, there is
a pair of adjoint exact functors
\[
Lf^* \colon \Dqc(\mathcal{Y}) \to \Dqc(\mathcal{X})
\;\;\text{and}\;\;
Rf_* \colon \Dqc(\mathcal{X}) \to \Dqc(\mathcal{Y})
\]
We warn the reader that these are the functors denoted $Lf_{qc}^*$ and $Rf_{qc,*}$ in \cite{hallrydh}. 
\end{enumerate}
Pullback and pushforward behave well when
\(f \colon \mathcal{X} \to \mathcal{Y}\) is a \emph{concentrated morphism} of
stacks, meaning that it is qcqs and, for every morphism \(Y \to \mathcal{Y}\)
from an affine scheme, the stack \(\mathcal{X} \times_{\mathcal{Y}} Y\) has
finite quasi-coherent cohomological dimension: see \cite[Definition
2.4]{hallrydh}. The latter condition is superfluous when \(\mathcal{X}\) and
\(\mathcal{Y}\) are algebraic spaces by \citeSP{073G}. Assuming that
\(f \colon \mathcal{X} \to \mathcal{Y}\) is a concentrated morphism of
stacks:
\begin{enumerate}
\setcounter{enumi}{3}
\item\label{derived-generalities.compatibilities}
\(Lf^*\) and \(Rf_*\) satisfy the projection formula and tor-independent base
change;
\item\label{derived-generalities.coproducts}
\(Rf_*\) commutes with coproducts; and
\item\label{derived-generalities.right-adjoint}
\(Rf_*\) admits an exact right adjoint
\(f^\times \colon \Dqc(\mathcal{Y}) \to \Dqc(\mathcal{X})\).
\end{enumerate}
See \cite[Theorem 2.6, Corollaries 4.12 and 4.13, and Theorem 4.14]{hallrydh}.
Finally, a stack is \emph{concentrated} if its structure morphism to
\(\Spec\mathbf{Z}\) is. With this,
\cite[Lemma 4.4]{hallrydh} shows that
\begin{enumerate}
\setcounter{enumi}{6}
\item\label{derived-generalities.perfect}
if \(\mathcal{X}\) is concentrated, perfect objects of
\(\mathrm{D}(\mathcal{X}_{\liset}, \sO_{\mathcal{X}})\)
are precisely compact objects of \(\Dqc(\mathcal{X})\).
\end{enumerate}

\subsectiondash{Relative generators}\label{derived-relative-generators}
Recall that an object \(G\) of a triangulated category \(\mathcal{D}\) is
called a \emph{generator} if an object \(E \in \mathcal{D}\) is the zero object
if and only if \(\Hom_{\mathcal{D}}(G, E[i]) = 0\) for all
\(i \in \mathbf{Z}\): see \citeSP{09SJ}. A relative version of this for our
setting is as follows:

Given a \(f \colon \mathcal{X} \to \mathcal{S}\) concentrated morphism of
stacks, a perfect object \(G \in \Dqc(\mathcal{X})\) is said to be an
\emph{\(\mathcal{S}\)-linear generator} of \(\Dqc(\mathcal{X})\) if for every
morphism \(T \to \mathcal{S}\) from an affine scheme, the object \(G_T\) is a
compact generator of \(\Dqc(\mathcal{X}_T)\). Basic properties of
relative generators are:
\begin{enumerate}
\item\label{derived-relative-generators.affine}
If \(\mathcal{S}\) is an affine scheme, then \(G\) is an \(\mathcal{S}\)-linear
generator if and only if it is a compact generator.
\item\label{derived-relative-generators.base-change}
If \(G\) is an \(\mathcal{S}\)-linear generator of \(\Dqc(\mathcal{X})\)
and \(\mathcal{T} \to \mathcal{S}\) is a morphism of stacks, then
\(G_{\mathcal{T}}\) is a \(\mathcal{T}\)-linear generator of
\(\Dqc(\mathcal{X}_{\mathcal{T}})\).
\item\label{derived-relative-generators.fpqc-descent}
If \(\{\mathcal{S}_i \to \mathcal{S}\}_i\) is an fpqc covering by stacks and
each \(G_{\mathcal{S}_i}\) is an \(\mathcal{S}_i\)-linear generator of
\(\Dqc(\mathcal{X}_{\mathcal{S}_i})\), then \(G\) is an \(\mathcal{S}\)-linear
generator of \(\Dqc(\mathcal{X})\).
\end{enumerate}

\begin{proof}
Item \ref{derived-relative-generators.affine} follows from the fact that
perfect generators are preserved under pullback by affine morphisms: see
\citeSP{0BQT}. Item \ref{derived-relative-generators.base-change} follows
easily from definitions. For \ref{derived-relative-generators.fpqc-descent},
we may replace the base \(\mathcal{S} = \Spec A\) by an affine scheme,
whereupon to suffices to show that if \(\Spec B \to \Spec A\) is a faithfully
flat morphism such that \(G_B \in \Dqc(\mathcal{X}_B)\) is a compact generator,
then so is \(G \in \Dqc(\mathcal{X})\). So consider
an object \(E \in \Dqc(\mathcal{X})\) with \(\Ext^i_{\mathcal{X}}(G,E) = 0\)
for all \(i \in \mathbf{Z}\). Flat base change gives
\[
\Ext^i_{\mathcal{X}_B}(G_B, E_B) \cong
\Ext^i_{\mathcal{X}}(G,E) \otimes_A B = 0
\]
so, since \(G_B\) is a generator, \(E_B = 0\). Faithful flatness thus implies
\(E = 0\), and so \(G\) is a generator.
\end{proof}

We may now give a more familiar characterization of relative generators in
terms of relative \(\Hom\):

\begin{Lemma}
\label{lemma-relativelynohoms}
Let \(f \colon \mathcal{X} \to \mathcal{S}\) be a concentrated morphism of
stacks. A perfect object \(G \in \Dqc(\mathcal{X})\) is an \(\mathcal{S}\)-linear
generator if and only if for every \(0 \neq E \in \Dqc(\mathcal{X})\),
\[
0 \neq Rf_*R\sHom_{\sO_{\mathcal{X}}}(G,E) \in \Dqc(\mathcal{S}).
\]
\end{Lemma}

\begin{proof}
Flat base change for \(Rf_*R\sHom_{\sO_{\mathcal{X}}}(G,E)\) combined with
\parref{derived-relative-generators}\ref{derived-relative-generators.fpqc-descent}
reduces this to the case \(\mathcal{S}\) is affine, in which case the result
follows from
\parref{derived-relative-generators}\ref{derived-relative-generators.affine},
the definition of a generator, and the fact that the pushforward may be
identified with the object \(R\Hom_{\mathcal{X}}(G,E)\).
\end{proof}

\subsectiondash{Relative generators of subcategories}
\label{derived-relative-generators-subcategory}
The characterization of \parref{lemma-relativelynohoms} allows us to generalize
the definition of a relative generator to linear triangulated subcategories. We
will only need the situation when \(f \colon \mathcal{X} \to S\) is a
concentrated morphism from a stack to a qcqs algebraic space. Let
\(\mathcal{A} \subseteq \Dqc(\mathcal{X})\) be a full triangulated subcategory
which is \emph{\(\mathcal{S}\)-linear}, meaning that for \(E \in \mathcal{A}\)
and \(F \in \Dqc(\mathcal{S})\), \(E \otimes^L Lf^*F \in \mathcal{A}\). A
perfect object \(G \in \mathcal{A}\) is called an \emph{\(\mathcal{S}\)-linear
generator} of \(\mathcal{A}\) if
\[
E \neq 0 \in \mathcal{A} \iff
Rf_*R\sHom_{\sO_{\mathcal{X}}}(G,E) \neq 0 \in \Dqc(\mathcal{S}).
\]
Note that a generator \(G\) of \(\mathcal{A}\) is also an
\(\mathcal{S}\)-linear generator. More interestingly, given a perfect object
\(G \in \mathcal{A}\) which is an \(S\)-linear generator, we have the
following two statements:
\begin{enumerate}
\item\label{derived-relative-generators-subcategory.make-absolute}
If \(F \in \Dqc(S)\) is a perfect generator, then \(G \otimes^L Lf^*F\) is a
compact generator of \(\mathcal{A}\).
\item\label{derived-relative-generators-subcategory.smallest}
The smallest strictly full triangulated \(S\)-linear subcategory of
\(\Dqc(\mathcal{X})\) containing \(G\) is \(\mathcal{A}\).
\end{enumerate}

\begin{proof}
For \ref{derived-relative-generators-subcategory.make-absolute}, the stack
\(\mathcal{X}\) is concentrated, so \(G\) is compact in \(\Dqc(\mathcal{X})\)
and hence also in \(\mathcal{A}\)---this makes sense as \(\mathcal{A}\) is
closed under direct sums by virtue of being \(S\)-linear. Now, for any
\(E \in \mathcal{A}\),
\[
R\Hom_{\mathcal{X}}(G \otimes^L Lf^*F, E) \cong
R\Hom_{\mathcal{X}}(Lf^*F, R\sHom_{\sO_{\mathcal{X}}}(G,E)) \cong
R\Hom_S(F, Rf_*R\sHom_{\sO_{\mathcal{X}}}(G,E)).
\]
The object of \(\Dqc(S)\) on the right is nonzero whenever \(E\) is nonzero
since \(G\) is an \(\mathcal{S}\)-linear perfect generator of \(\mathcal{A}\)
and \(F\) is a perfect generator of \(\Dqc(S)\). Looking at the
object on the left then shows that \(G \otimes^L Lf^*F\) is a compact
generator of \(\mathcal{A}\). This implies that the category in
\ref{derived-relative-generators-subcategory.smallest} contains a compact
generator. Since it is also strictly full, triangulated, and closed under
direct sums, it follows from \citeSP{09SN} that it must be equal to
\(\mathcal{A}\).
\end{proof}

\subsectiondash{Fourier--Mukai transforms}\label{derived-fm}
Given morphisms \(f \colon \mathcal{X} \to \mathcal{S}\) and
\(g \colon \mathcal{Y} \to \mathcal{S}\) of stacks, and an object
\(K \in \Dqc(\mathcal{X} \times_{\mathcal{S}} \mathcal{Y})\),
there is an associated \emph{Fourier--Mukai transform} with
\emph{kernel} \(K\):
\[
\Phi_K \colon \Dqc(\mathcal{X}) \to \Dqc(\mathcal{Y})
\quad\quad
E \mapsto R\pr_{2,*}(L\pr_1^*E \otimes^L K).
\]
When \(f\) and \(g\) are concentrated morphisms of stacks, the functor
\(\Phi_K\) satisfies many familiar properties:
\begin{enumerate}
\item\label{derived-fm.linear}
\(\Phi_K\) is \(\mathcal{S}\)-linear: For \(E \in \Dqc(\mathcal{X})\)
and \(F \in \Dqc(\mathcal{S})\), there is a natural isomorphism
\[
\Phi_K(E \otimes^L_{\sO_{\mathcal{X}}} Lf^*F) \cong
\Phi_K(E) \otimes^L_{\sO_{\mathcal{Y}}} Lg^*F \in \Dqc(\mathcal{Y}).
\]
\item\label{derived-fm.enriched-hom}
\(\Phi_K\) is an enriched functor: For \(E, P \in \Dqc(\mathcal{X})\),
if \(P\) and \(\Phi_K(P)\) are perfect, then there is a natural morphism in
\(\Dqc(\mathcal{S})\)
\[
Rf_*R\sHom_{\sO_{\mathcal{X}}}(P,E) \to
Rg_*R\sHom_{\sO_{\mathcal{Y}}}(\Phi_K(P), \Phi_K(E))
\]
which, after applying the functor \(\mathrm{H}^0(\mathcal{S},-)\), becomes the
homomorphism
\[
\Phi_K \colon
\Hom_{\mathcal{X}}(P,E) \to
\Hom_{\mathcal{Y}}(\Phi_K(P),\Phi_K(E)).
\]
\item\label{derived-fm.base-change}
\(\Phi_K\) is compatible with flat base change: For a morphism
\(h \colon \mathcal{T} \to \mathcal{S}\) of stacks, if either \(h\) is flat or
both \(f\) and \(g\) are flat, then for any \(E \in \Dqc(\mathcal{X})\), there
is a canonical isomorphism
\[
\Phi_K(E)_{\mathcal{T}} \cong
\Phi_{K_{\mathcal{T}}}(E_{\mathcal{T}})
\in
\Dqc(\mathcal{Y} \times_{\mathcal{S}} \mathcal{T}).
\]
\end{enumerate}

\begin{proof}
Part \ref{derived-fm.linear} is a calculation using the projection formula.
Construct the morphism in \ref{derived-fm.enriched-hom} via the Yoneda lemma:
for every \(F \in \Dqc(\mathcal{S})\), there is a map
\begin{multline*}
\Hom_{\mathcal{S}}(F, Rf_*R\sHom_{\sO_{\mathcal{X}}}(P,E))
\cong \Hom_{\mathcal{X}}(P \otimes_{\sO_{\mathcal{X}}}^L Lf^*F, E) \\
\stackrel{\Phi_K}{\longrightarrow}
\Hom_{\mathcal{Y}}(\Phi_K(P \otimes_{\sO_{\mathcal{X}}}^L Lf^*F), \Phi_K(E)) \cong
\Hom_{\mathcal{S}}(F, Rg_*R\sHom_{\sO_{\mathcal{Y}}}(\Phi_K(P), \Phi_K(E)))
\end{multline*}
where the two identifications use perfectness of \(P\) and \(\Phi_K(P)\), and
the bottom identification additionally uses \ref{derived-fm.linear}. Finally,
\ref{derived-fm.base-change} follows from tor-independent base change
applied to the square
\[
\begin{tikzcd}[/tikz/baseline=(tikz@f@1-2-1.base)]
  \mathcal{X}_{\mathcal{T}} \times _{\mathcal{T}} \mathcal{Y}_{\mathcal{T}} \ar[r] \ar[d] &\mathcal{Y}_{\mathcal{T}} \ar[d] \\
  \mathcal{X} \times _{\mathcal{S}} \mathcal{Y} \ar[r] &\mathcal{Y}.
\end{tikzcd}
\qedhere
\]
\end{proof}

The Fourier--Mukai transform often has a right adjoint:

\begin{Lemma}
\label{lemma-rightadjointofFM}
Let \(f \colon \mathcal{X} \to \mathcal{S}\) and
\(g \colon \mathcal{Y} \to \mathcal{S}\) be morphisms of stacks and
\(K \in \Dqc(\mathcal{X} \times_{\mathcal{S}} \mathcal{Y})\). If \(f\) is
concentrated, then the functor
\(\Phi_K \colon \Dqc(\mathcal{X}) \to \Dqc(\mathcal{Y})\) has an exact right
adjoint.
\end{Lemma}

\begin{proof}
$\Phi_K$ commutes with coproducts since it is a composition of three functors
which do, with $R\pr_{2,*}$ doing so by the assumption that $f$ is
concentrated and
\parref{derived-generalities}\ref{derived-generalities.right-adjoint}. Since
the category $\Dqc(\mathcal{X})$ is well-generated by
\cite[Theorem B.1]{Hall_Neeman_Rydh_2019}, the result therefore follows from
Neeman's version of Brown Representability Theorem for well-generated
categories: see \cite[Theorem 8.4.4]{Neeman+2001}.
\end{proof}

We now give a criterion for when \(\Phi_K\) is fully faithful in terms of a
perfect generator. First, the following is a well-known criterion for when an
exact functor preserves compact objects.

\begin{Lemma}
\label{lemma-preservescompact}
Let \(\Psi \colon \mathcal{D} \to \mathcal{D}'\) be an exact functor between
triangulated categories with arbitrary direct sums. Assume that \(\Psi\)
has a right adjoint \(R\) and that \(\mathcal{D}\) is compactly generated.
Then \(\Psi\) takes compact objects to compact objects if and only if
\(R\) commutes with direct sums. \qed
\end{Lemma}

% \begin{proof}
% Suppose \(F\) takes compact objects to compact objects. Then, for any
% compact object \(P \in \mathcal{D}\) and subset
% \(\{M_i\}_i \subset \mathcal{D}'\) of objects,
% \begin{align*}
% \operatorname{Hom}_{\mathcal{D}}(P, R(\oplus_i M_i)) &= \operatorname{Hom}_{\mathcal{D}'}(F(P), \oplus _i M_i) 
% = \oplus _i \operatorname{Hom}_{\mathcal{D}'}(F(P), M_i)=  \oplus_i \operatorname{Hom}_{\mathcal{D}}(P, R(M_i))\\
% &= \operatorname{Hom}_{\mathcal{D}}(P, \oplus_i R(M_i)).
% \end{align*}
% Since \(\mathcal{D}\) is compactly generated, this implies that the canonical
% map $R(\oplus _i M_i) \to \oplus _i R(M_i)$ is an isomorphism, meaning \(R\)
% commutes with direct sums. Conversely, if \(R\) commutes with direct sums,
% then for a compact object \(P \in \mathcal{D}\),
%    \begin{align*}
%    \operatorname{Hom}_{\mathcal{D}'}(F(P), \oplus _i M_i) &= \operatorname{Hom}_{\mathcal{D}}(P, R(\oplus _i M_i)) = \operatorname{Hom}_{\mathcal{D}}(P, \oplus _i R(M_i)) = \oplus _i \operatorname{Hom}_{\mathcal{D}}(P, R(M_i)) 
%    \\
%    &= \oplus _i \operatorname{Hom}_{\mathcal{D}'}(F(P), M_i)
%    \end{align*}
% for any set \(\{M_i\}_i \subset \mathcal{D}'\) of objects. This means
% that \(F(P)\) is compact.
% \end{proof}

Next, the following is a general criterion for fully faithfulness in terms of a
compact generator:

\begin{Lemma}
\label{lemma-ffbygen}
Let \(\Psi \colon \mathcal{D} \to \mathcal{D}'\) be an exact functor between
triangulated categories with arbitrary direct sums. Assume that
\begin{itemize}
\item \(\mathcal{D}\) has a compact generator \(G\);
\item \(\Psi\) has a right adjoint \(R\); and
\item \(\Psi\) takes compact objects to compact objects.
\end{itemize}
Then \(\Psi\) is fully faithful if and only if the following map is an isomorphism
for all \(i \in \mathbf{Z}\):
\[
\Psi \colon \Ext^i_{\mathcal{D}}(G,G) \to \Ext^i_{\mathcal{D}'}(\Psi(G), \Psi(G)).
\]
\end{Lemma}

\begin{proof}
The ``if'' direction is clear. Conversely, suppose that the map between
self extensions of \(G\) and \(\Psi(G)\) are all isomorphisms. Consider the set
of objects \(M \in \mathcal{D}\) such that the unit
\(\eta_M \colon M \to R\Psi(M)\) of adjunction is an isomorphism. This set is
closed under cones since \(R\) and \(\Psi\) are exact. It is closed under
direct sums since both \(R\) and \(\Psi\) preserves direct sums by
\parref{lemma-preservescompact} and left adjointness, respectively. Finally,
this set contains the generator \(G\) since, for every \(i \in \mathbf{Z}\),
\[
\Ext^i_{\mathcal{D}}(G,G) \cong
\Ext^i_{\mathcal{D}'}(\Psi(G),\Psi(G)) \cong
\Ext^i_{\mathcal{D}}(G, R\Psi(G))
\]
and \(G\) is a generator. Since the generator \(G\) is compact, it now
follows from \citeSP{09SR} that the set in question contains every object of
\(\mathcal{D}\). Therefore \(\Psi\) is fully faithful.
\end{proof}

In our setting of Fourier--Mukai transforms, the condition on self extensions
of a generator may be verified on the level of enriched Homs, leading to the
following full faithfulness criterion:

\begin{Lemma}\label{lemma-ffconditions}
Let \(f \colon \mathcal{X} \to \mathcal{S}\) and
\(g \colon \mathcal{Y} \to \mathcal{S}\) be morphisms of concentrated stacks,
and \(K \in \Dqc(\mathcal{X} \times_{\mathcal{S}} \mathcal{Y})\). Assume that
\begin{itemize}
\item \(\Dqc(\mathcal{S})\) has a perfect generator \(F\);
\item \(\Dqc(\mathcal{X})\) has an \(\mathcal{S}\)-linear perfect generator \(G\); and
\item \(\Phi_K \colon \Dqc(\mathcal{X}) \to \Dqc(\mathcal{Y})\) takes perfect
complexes to perfect complexes.
\end{itemize}
Then \(\Phi_K\) is fully faithful if and only if the morphism in \(\Dqc(\mathcal{S})\)
\[
\phi \colon
Rf_*R\sHom_{\sO_{\mathcal{X}}}(G,G) \to
Rg_*R\sHom_{\sO_{\mathcal{Y}}}(\Phi_K(G), \Phi_K(G))
\]
of \parref{derived-fm}\ref{derived-fm.enriched-hom} is an isomorphism.
\end{Lemma}

\begin{proof}
Since \(F \in \Dqc(\mathcal{S})\) is a perfect generator, \(\phi\) is an
isomorphism if and only if it becomes an isomorphism after applying the functor
\(\Ext^i_{\mathcal{S}}(F, -)\) for every \(i \in \mathbf{Z}\). By its
construction in \parref{derived-fm}\ref{derived-fm.enriched-hom}, this gives
the map
\[
\Phi_K \colon
\Ext^i_{\mathcal{X}}(Lf^*F \otimes^L G, G) \to
\Ext^i_{\mathcal{Y}}(\Phi_K(Lg^*F \otimes^L G), \Phi_K(G)).
\]
If \(\Phi_K\) is fully faithful, then this is, of course, an isomorphism for
all \(i \in \mathbf{Z}\), whence \(\phi\) is an isomorphism. Conversely, if
\(\phi\) is an isomorphism, then so is the map
\[
\phi' \colon
Rf_*R\sHom_{\sO_{\mathcal{X}}}(G, G \otimes^L Lf^*F) \to
Rg_*R\sHom_{\sO_{\mathcal{Y}}}(\Phi_K(G), \Phi_K(G \otimes^L Lf^*F))
\]
obtained from \(\phi \otimes^L F\) via the projection formula and
\(\mathcal{S}\)-linearity from \parref{derived-fm}\ref{derived-fm.linear}.
Applying \(\Ext^i_{\mathcal{S}}(F,-)\) to \(\phi'\) then shows that the map
\[
\Phi_K \colon
\Ext^i_{\mathcal{X}}(G \otimes^L Lf^*F, G \otimes^L Lf^*F) \to
\Ext^i_{\mathcal{Y}}(\Phi_K(G \otimes^L Lf^*F), \Phi_K(G \otimes^L Lf^*F))
\]
is an isomorphism for every \(i \in \mathbf{Z}\). Since \(\Phi_K\)
has a right adjoint by \parref{lemma-rightadjointofFM} and
\(G \otimes^L Lf^*F\) is a compact generator of \(\Dqc(\mathcal{X})\) by
\parref{derived-relative-generators-subcategory}\ref{derived-relative-generators-subcategory.make-absolute},
full faithfulness of \(\Phi_K\) now follows from the criterion
\parref{lemma-ffbygen}.
\end{proof}

Combined with tor-independent base change and
\parref{derived-fm}\ref{derived-fm.base-change}, the criterion
\parref{lemma-ffconditions} implies that full faithfulness of \(\Phi_K\) is
often preserved under base change, and also that it may be checked fpqc
locally:

\begin{Lemma}\label{lemma-basechangeofff}
In the setting of \parref{lemma-ffconditions}, furthermore suppose given a
morphism \(h \colon \mathcal{T} \to \mathcal{S}\) of concentrated stacks. If
either \(h\) is flat or both \(f\) and \(g\) are flat, then:
\begin{enumerate}
\item\label{lemma-ffconditions.base-change-ff}
If \(\Phi_K \colon \Dqc(\mathcal{X}) \to \Dqc(\mathcal{Y})\) is fully faithful,
then so is \(\Phi_{K_{\mathcal{T}}} \colon \Dqc(\mathcal{X}_{\mathcal{T}}) \to \Dqc(\mathcal{Y}_{\mathcal{T}})\).
\item\label{lemma-ffconditions.descend-ff}
If \(h \colon \mathcal{T} \to \mathcal{S}\) is faithfully flat and
quasi-compact, then the converse of \ref{lemma-ffconditions.base-change-ff}
also holds. \qed
\end{enumerate}
\end{Lemma}

% \begin{proof}
% We use the criterion of \parref{lemma-ffconditions}. By tor-independent
% base change and \parref{derived-fm}\ref{derived-fm.base-change}, the base
% change of \parref{derived-fm}\ref{derived-fm.enriched-hom} to $\mathcal{T}$ is
% $$
% Rf_{\mathcal{T}, *}(R \mathcal{H}om_{\mathcal{O}_{\mathcal{X}_{\mathcal{T}}}}(G_{\mathcal{T}}, G_{\mathcal{T}}) \to Rg_{\mathcal{T}, *}(R \mathcal{H}om_{\mathcal{O}_{\mathcal{Y}_{\mathcal{T}}}}(\Phi_{K_{\mathcal{T}}}(G_{\mathcal{T}}), \Phi_{K_{\mathcal{T}}}(G_{\mathcal{T}})))
% $$
% and the result follows.
% \end{proof}

\subsectiondash{Gerbes}\label{subsection-gmgerbe}
In this work, a \emph{\(\mathbf{G}_m\)-gerbe} refers to a gerbe banded by
\(\mathbf{G}_m\); that is, the data of
\begin{itemize}
\item a morphism \(\pi \colon \mathcal{X} \to X\) from a stack to an algebraic
space; and
\item an isomorphism
\(\xi \colon \mathbf{G}_{m, \mathcal{X}} \to \mathcal{I}_{\mathcal{X}}\) of sheaves on
\(\mathcal{X}_{\mathrm{fppf}}\) from \(\mathbf{G}_m\) to the inertia stack
of \(\mathcal{X}\).
\end{itemize}
The identification \(\xi\) provides each quasi-coherent
\(\sO_{\mathcal{X}}\)-module \(\mathcal{F}\) with a right action by
\(\mathbf{G}_m\), whence a weight decomposition
\(\mathcal{F} = \bigoplus\nolimits_{i \in \mathbf{Z}} \mathcal{F}_i\)
where the weight \(i\) component is characterized by
\[
\mathcal{F}_i(x) =
\{
s \in \mathcal{F}(x) :
s \cdot \gamma = \gamma^is\;\text{for all}\; \gamma \in \sAut_T(x) = \mathbf{G}_m(T)
\}
\]
for each morphism \(T \to X\) and \(x \in \mathcal{X}_T\).
If \(\mathcal{F} = \mathcal{F}_i\), then the sheaf \(\mathcal{F}\) is said to
\emph{have weight \(i\)}. Sheaves of different weights do not map to one another, so we have a product decomposition $\operatorname{QCoh}(\mathcal{X}) = \prod_{i \in \mathbf{Z}}\operatorname{QCoh}_i(\mathcal{X})$ where $\operatorname{QCoh}_i(\mathcal{X}) \subset \operatorname{QCoh}(\mathcal{X})$ denotes the full subcategory of quasi-coherent sheaves of weight $i$.
Pullback along \(\pi\) induces an equivalence
\[
\pi^* \colon \QCoh(\sO_X) \to \QCoh_0(\sO_{\mathcal{X}})
\]
between the categories of quasi-coherent \(\sO_X\)-modules and quasi-coherent
\(\sO_{\mathcal{X}}\)-modules of weight \(0\). If \(\pi \colon \mathcal{X} \to X\)
admits a section \(\sigma \colon X \to \mathcal{X}\), then the restriction of
pullback along \(\sigma\) induces, for any \(i \in \mathbf{Z}\), an equivalence
\[
\sigma^* \colon \QCoh_i(\sO_{\mathcal{X}}) \to \QCoh(\sO_X)
\]
between the categories of quasi-coherent \(\sO_{\mathcal{X}}\)-modules of
weight \(i\) and quasi-coherent \(\sO_X\)-modules.

\subsectiondash{}\label{gmgerbes-derived-categories}
For each \(i \in \mathbf{Z}\), let
\(\mathrm{D}_{\mathrm{qc}, i}(\mathcal{X}) \subset \Dqc(\mathcal{X})\) be the
full subcategory consisting of objects whose cohomology sheaves all have
weight \(i\). 
%Each \(\mathrm{D}_{\mathrm{qc},i}(\mathcal{X})\) may be
%obtained as the derived category of the abelian category \(\QCoh_i(\sO_{\mathcal{X}})\)
%and is a triangulated subcategory of \(\Dqc(\mathcal{X})\). -- This is false
Furthermore,
\cite[Theorem 5.4]{BS:BS} shows that there is an orthogonal decomposition
\[
\Dqc(\mathcal{X}) \cong
\prod\nolimits_{i \in \mathbf{Z}} \mathrm{D}_{\mathrm{qc},i}(\mathcal{X}).
\]
Namely, the \(\mathrm{D}_{\mathrm{qc},i}(\mathcal{X})\) are completely
orthogonal to one another for different \(i \in \mathbf{Z}\), and every object
\(E \in \Dqc(\mathcal{X})\) may be uniquely expressed as a direct sum
\(E = \bigoplus\nolimits_{i \in \mathbf{Z}} E_i\) with
\(E_i \in \mathrm{D}_{\mathrm{qc},i}(\mathcal{X})\). It is shown
in \cite[\S5]{BS:BS} that these decompositions enjoy the following
compatibilities:
\begin{enumerate}
\item\label{gmgerbes-derived-categories.morphisms}
If \(f \colon \mathcal{X} \to \mathcal{Y}\) is a morphism of
\(\mathbf{G}_m\)-gerbes, then \(Lf^*\) and \(Rf_*\) preserve the decomposition.
\item\label{gmgerbes-derived-categories.tensor}
If \(E \in \mathrm{D}_{\mathrm{qc},i}(\mathcal{X})\) and
\(F \in \mathrm{D}_{\mathrm{qc},j}(\mathcal{X})\), then
\(E \otimes^L F \in \mathrm{D}_{\mathrm{qc},i+j}(\mathcal{X})\).
\item\label{gmgerbes-derived-categories.pullback}
Pullback along \(\pi\) induces an equivalence
\(L\pi^* \colon \Dqc(X) \to \mathrm{D}_{\mathrm{qc},0}(\mathcal{X})\).
\item\label{gmgerbes-derived-categories.pushforward}
The functor \(R\pi_* \colon \mathrm{D}_{\mathrm{qc}}(\mathcal{X}) \to \Dqc(X)\)
may be identified as \(E \mapsto E_0\).
\item\label{gmgerbes-derived-categories.section}
If \(\sigma \colon X \to \mathcal{X}\) is a section, then
\(L\sigma^* \colon \mathrm{D}_{\mathrm{qc},i}(\mathcal{X}) \to \Dqc(X)\) is
an equivalence for any \(i \in \mathbf{Z}\).
\end{enumerate}

Note that \(\pi \colon \mathcal{X} \to X\) is a concentrated morphism of
stacks: this is because, in addition to \(\pi\) being qcqs, the functor
\(\pi_* \colon \QCoh(\sO_{\mathcal{X}}) \to \QCoh(\sO_X)\) is exact. When
the algebraic space \(X\) is qcqs, each of the categories
\(\mathrm{D}_{\mathrm{qc},i}(\mathcal{X})\) admits a perfect generator
by \cite[Example 9.3]{hallrydh}.

\subsectiondash{Trivial gerbes}
\label{example-trivialityofgerbe}
The \emph{trivial \(\mathbf{G}_m\)-gerbe} over \(X\) is the relative
classifying stack \(B\mathbf{G}_{m,X} \to X\). The tautological line bundle on
\(B\mathbf{G}_{m,X}\) is of weight \(1\) and provides a section
\(X \to B\mathbf{G}_{m,X}\). Conversely, suppose \(\pi \colon \mathcal{X} \to X\)
is a \(\mathbf{G}_m\)-gerbe and that \(\mathcal{X}\) carries a line bundle
\(\mathcal{L}\) of weight \(1\). Then
\begin{enumerate}
\item\label{trivialityofgerbe.trivial}
\(\pi \colon \mathcal{X} \to X\) is isomorphic to the trivial \(\mathbf{G}_m\)-gerbe; and
\item\label{trivialityofgerbe.coarse}
there is a canonical isomorphism
\(X \cong \mathbf{A}(\mathcal{L}) \setminus \{0\}\) between \(X\) and the
complement of the zero section in the affine bundle
\(\mathbf{A}(\mathcal{L}) \coloneqq \Spec\Sym(\mathcal{L}^\vee)\) on
\(\mathcal{L}\) over \(\mathcal{X}\).
\end{enumerate}
Item \ref{trivialityofgerbe.trivial} is because \(\mathcal{L}\) is classified
by a morphism of \(\mathbf{G}_m\)-gerbes \(\mathcal{X} \to B\mathbf{G}_{m,X}\)
over \(X\), and this map is necessarily an isomorphism. As for \ref{trivialityofgerbe.coarse},
this comes from identifying the universal \(\mathbf{G}_m\)-torsor over
\(B\mathbf{G}_{m,X}\) as, on the one hand, the canonical morphism
\(X \to B\mathbf{G}_{m,X}\) and, on the other hand, as the affine bundle
\(\underline{\mathit{Isom}}(\sO_{B\mathbf{G}_{m,X}}, \mathcal{L})\)
of local trivializations for the universal line bundle \(\mathcal{L}\).
Finally, the latter is easily seen to be isomorphic to
\(\mathbf{A}(\mathcal{L}) \setminus \{0\}\).

More generally, if \(\mathcal{X}\) carries a rank \(r\) vector bundle of weight
\(1\), then its associated cohomology class
\([\mathcal{X}] \in \mathrm{H}^2_{\mathrm{\acute{e}t}}(X,\mathbf{G}_m)\) is
\(r\)-torsion. This is because the determinant line bundle has weight \(r\),
and this gives a line bundle of weight \(1\) on the \(\mathbf{G}_m\)-gerbe
on \(X\) with class \(r[\mathcal{X}]\).

\subsectiondash{Relatively perfect objects}\label{derived-relatively-perfect}
Given a morphism \(f \colon X \to S\) of schemes which is locally of finite type
an object \(E \in \Dqc(X)\) is said to be \emph{perfect relative to \(S\)} or,
briefly, \emph{\(S\)-perfect} if it is
\begin{itemize}
\item pseudo-coherent relative to \(S\), in the sense of \citeSP{09UI}, and
\item locally has finite tor-dimension as an object of
\(\mathrm{D}(f^{-1}\sO_S)\), see \citeSP{08CG}.
\end{itemize}
We adopt here the generalized definition suggested in \citeSP{0DI9}. In
contrast, the Stacks Project defines relatively perfect objects only for
morphisms which are flat and locally of finite presentation, in which case an
\(S\)-perfect object in \(\Dqc(X)\) is the same as one that can be represented,
locally on \(X\), by a bounded complex of \(S\)-flat finitely presented
quasi-coherent sheaves on \(X\): see
\cite[\href{https://stacks.math.columbia.edu/tag/0DI0}{0DI0} and
\href{https://stacks.math.columbia.edu/tag/0DI2}{0DI2}]{stacks-project}.
An analogous characterization of relatively perfect objects in this setting
can be obtained by combining \parref{lemma-relperfonaffine} and
\parref{lemma-relperfislocal} below. When \(f \colon X \to S\) is smooth,
relatively perfect is the same as perfect:

\begin{Lemma}
\label{lemma-relperfonsmooth}
If \(f \colon X \to S\) is a smooth morphism of schemes, then an object
\(E \in \Dqc(X)\) is perfect relative to \(S\) if and only if it is perfect.
\end{Lemma}

\begin{proof}
That an \(S\)-perfect object is perfect follows from \citeSP{068X}. Conversely,
if \(E\) is perfect, then locally on \(X\), it can be represented by a bounded
complex of vector bundles. Since vector bundles are \(S\)-flat and finitely
presented quasi-coherent sheaves, it follows that \(E\) is \(S\)-perfect.
\end{proof}

When \(X\) and \(S\) are both affine, relatively perfect may be reduced to
perfect in an affine space:

\begin{Lemma}
\label{lemma-relperfonaffine}
Let \(f \colon X \to S\) be a finite type morphism of affine schemes. If
\(f = \pr \circ i \colon X \to \mathbf{A}^n_S \to S\) is any factorization through a closed
immersion \(i \colon X \to \mathbf{A}^n_S\), then an object
\(E \in \Dqc(X)\) is perfect relative to \(S\) if and only if \(Ri_*E\) is a
perfect object of \(\Dqc(\mathbf{A}^n_S)\).
\end{Lemma}

\begin{proof}
First, essentially by definition, \(E\) is pseudo-coherent relative to \(S\) if
and only if \(Ri_*E\) is: compare
\cite[\href{https://stacks.math.columbia.edu/tag/09UI}{09UI} and
\href{https://stacks.math.columbia.edu/tag/09VC}{09VC}]{stacks-project}.
Next, writing \(S = \Spec R\), \(E\) has finite tor-dimension as an object
of \(\mathrm{D}(f^{-1}\sO_S)\) if and only if
\(R\Gamma(X,E) \cong R\Gamma(\mathbf{A}^n_S, Ri_*E)\) has finite
tor-dimension as an object of \(\mathrm{D}(R)\), and so these conditions
are equivalent to \(Ri_*E\) having finite tor-dimension as an object of
\(\mathrm{D}(\pr^{-1}\sO_S)\). Put together, this means that \(E \in \Dqc(X)\)
is \(S\)-perfect if and only if \(Ri_*E \in \Dqc(\mathbf{A}^n_S)\) is
\(S\)-perfect, and now the result follows from applying \parref{lemma-relperfonsmooth}
to \(\pr \colon \mathbf{A}^n_S \to S\).
\end{proof}

In general, being relatively perfect is fppf local on both the source and
target:

\begin{Lemma}
\label{lemma-relperfislocal}
Let \(f \colon X \to S\) be a morphism of schemes which is locally of finite
type and \(E \in \Dqc(X)\).
\begin{enumerate}
\item\label{lemma-relperfislocal.fppf-descent}
If \(\{g_i \colon U_i \to X\}_i\) is an fppf covering, then \(E\) is
perfect relative to \(S\) if and only if each \(Lg_i^*E \in \Dqc(U_i)\)
is perfect relative to \(S\).
\item\label{lemma-relperfislocal.factor}
Given a factorization \(f \colon X \to V \to S\) with \(V \to S\) flat and
finitely presented, \(E\) is perfect relative to \(S\) if and only if it is
perfect relative to \(V\).
\end{enumerate}
\end{Lemma}

\begin{proof}
The analogues of the statements \ref{lemma-relperfislocal.fppf-descent} and
\ref{lemma-relperfislocal.factor} for relative pseudo-coherence hold by
\cite[ \href{https://stacks.math.columbia.edu/tag/0CSN}{0CSN} and
\href{https://stacks.math.columbia.edu/tag/0CSP}{0CSP}]{stacks-project}.
Similarly, for the statements about tor-dimension, note that \(E\) has
tor-amplitude in \([a,b]\) as an object of \(\mathrm{D}(f^{-1}\sO_X)\) if and
only if \(E_x\) has tor-amplitude in \([a,b]\) as an object of
\(\mathrm{D}(\sO_{S,f(x)})\) for every \(x \in X\) by \citeSP{09U9}. With this,
the analogue of \ref{lemma-relperfislocal.fppf-descent} follows from
\citeSP{0DJF} whereas \ref{lemma-relperfislocal.factor} is
\cite[\href{https://stacks.math.columbia.edu/tag/066L}{066L} and
\href{https://stacks.math.columbia.edu/tag/068S}{068S}]{stacks-project}.
Together, this gives the result.
\end{proof}

\subsectiondash{}\label{derived-relatively-perfect-stacks}
This allows us to extend the definition of relative perfectness to stacks:
Let \(f \colon \mathcal{X} \to \mathcal{S}\) be a morphism of stacks
which is locally finite type, and choose a commutative square
\[
\begin{tikzcd}
U \ar[r] \ar[d, "\varphi"'] & V \ar[d]\\
\mathcal{X} \ar[r] &\mathcal{S}
\end{tikzcd}
\]
where \(U\) and \(V\) are schemes, \(V \to \mathcal{S}\) and
\(U \to \mathcal{X} \times_{\mathcal{S}} V\) are smooth surjections, and
\(U \to V\) is locally of finite type. An object \(E \in \Dqc(\mathcal{X})\)
is then said to be \emph{perfect relative to \(\mathcal{S}\)} or
\emph{\(\mathcal{S}\)-perfect} if \(\varphi^*E \in \Dqc(U)\) is perfect relative
to \(V\) in the sense of \parref{derived-relatively-perfect}. By
\parref{lemma-relperfislocal} and a standard argument, the definition is
independent of the choice of the square.

When \(f \colon \mathcal{X} \to \mathcal{S}\) is flat and locally of finite
presentation, relative perfectness is preserved under arbitrary base change:
see \citeSP{0DI5}. In our generality, however, that is not necessarily the
case. Nonetheless, relative perfectness is preserved under flat base change:

\begin{Lemma}
\label{lemma-basechangeofrelperf}
Let \(f \colon \mathcal{X} \to \mathcal{S}\) and
\(g \colon \mathcal{T} \to \mathcal{S}\) be morphisms of stacks, with
\(f\) locally of finite type. Assume either \(g\) is flat or \(f\) is flat
and locally of finite presentation. If \(E \in \Dqc(\mathcal{X})\) is
perfect relative to \(\mathcal{S}\), then
\(E_{\mathcal{T}} \in \Dqc(\mathcal{X}_{\mathcal{T}})\)
is perfect relative to \(\mathcal{T}\).
\end{Lemma}

\begin{proof}
Since relative perfectness is fppf local on both the source and target
by \parref{lemma-relperfislocal}, we may reduce to the case that
\(f \colon X \to S\) and \(g \colon T \to S\) are morphisms of affine
schemes. It remains to consider the case \(g\) is flat. Choose a factorization
\(f = \pr \circ i \colon X \to \mathbf{A}^n_S \to S\) where \(i\) is a
closed immersion. Since \(E \in \Dqc(X)\) is \(S\)-perfect,
\(Ri_*E \in \Dqc(\mathbf{A}^n_S)\) by \parref{lemma-relperfonaffine}.
Flat base change applied to the square
\[
\begin{tikzcd}
  X_T \ar[r, "i_T"'] \ar[d] & \mathbf{A}^n_T \ar[d] \\
  X \ar[r, "i"] \ar[r] & \mathbf{A}^n_S
\end{tikzcd}
\]
implies that \(Ri_{T,*}E_T \in \Dqc(\mathbf{A}^n_T)\) is perfect: see
\cite[\href{https://stacks.math.columbia.edu/tag/08IB}{08IB} and
\href{https://stacks.math.columbia.edu/tag/066X}{066X}]{stacks-project}.
Applying \parref{lemma-relperfonaffine} once again then shows that \(E_T\) is
\(T\)-perfect.
\end{proof}

A key feature of relatively perfect objects is that they push forward to
perfect objects under reasonable proper morphisms. When \(f \colon X \to S\)
is flat and locally of finite presentation, a general version is given
in \citeSP{0DJT}. In our more general setting, we prove a weaker version
that suffices for our purposes. By way of terminology, call a morphism
\(f \colon X \to S\) of algebraic spaces \emph{fpqc locally \(H\)-projective}
if there exists a faithfully flat and quasi-compact cover \(T \to S\) such that
the base change \(f_T \colon X_T \to T\) is \(H\)-projective, in the sense that
it factors as \(f_T = \pr \circ i \colon X_T \to \PP^n_T \to T\) where \(i\) is
a closed immersion. In this setting, \(Rf_*\) takes relatively perfect complexes to perfect ones:

\begin{Lemma}
\label{lemma-perfectpush}
Let \(f \colon X \to S\) be a morphism of algebraic spaces that is fpqc
locally \(H\)-projective. If \(E \in \Dqc(X)\) is perfect relative to \(S\),
then \(Rf_*E \in \Dqc(S)\) is perfect.
\end{Lemma}

\begin{proof}
Since relative perfectness is preserved under flat base change by
\parref{lemma-basechangeofrelperf} and perfectness may be checked fpqc locally
by \citeSP{09UG}, we may reduce to the case \(f \colon X \to S\) itself
admits a factorization \(f = \pr \circ i \colon X \to \PP^n_S \to S\) with
\(i \colon X \to \PP^n_S\) a closed immersion. Since relative perfectness
is Zariski local on the source, the argument of \parref{lemma-relperfonaffine}
may be used to show that \(Ri_*E \in \Dqc(\PP^n_S)\) is perfect. Since the
projection \(\pr \colon \PP^n_S \to S\) is smooth and proper, \(R\pr_*\)
preserves perfect complexes, yielding the result.
\end{proof}

A final useful fact is that, at least for a morphism
\(f \colon \mathcal{X} \to \mathcal{S}\) of stacks which is flat and locally of
finite presentation, formation of internal \(\Hom\) with relatively perfect
objects often preserves \(\Dqc(\mathcal{X})\) and commutes with arbitrary base
change:

\begin{Lemma}
\label{lemma-internalhomandbasechange}
Let \(f \colon \mathcal{X} \to \mathcal{S}\) be a morphism of stacks which is
flat and locally of finite presentation. If \(E, F \in \Dqc(\mathcal{X})\)
are such that \(E\) is pseudo-coherent and \(F\) is perfect relative to
\(\mathcal{S}\), then
\[
R\sHom_{\sO_{\mathcal{X}}}(E,F) \in \Dqc(\mathcal{X})
\]
and its formation commutes with arbitrary base change in that, for
\(\mathcal{T} \to \mathcal{S}\) a morphism of stacks,
\[
R\sHom_{\sO_{\mathcal{X}}}(E,F)_{\mathcal{T}} \cong
R\sHom_{\sO_{\mathcal{X}_{\mathcal{T}}}}(E_{\mathcal{T}}, F_{\mathcal{T}}) \in
\Dqc(\mathcal{X}_{\mathcal{T}}).
\]
\end{Lemma}

%The assumption on $F$ guarantees that $F$ remains in $D+_{QC}$ after arbitrary base change.

\begin{proof}
The first statement holds more generally for \(E\) pseudo-coherent and
\(F\) locally in \(\Dqc^+(\mathcal{X})\): this may be checked smooth locally
on \(\mathcal{X}\) and so it reduces to the affine case, where \citeSP{0A6H}
applies. The second statement may now be reduced to the case where
\(f \colon X \to \Spec A\) is a morphism of affine schemes, and the base
change is of the form \(\Spec B \to \Spec A\). In this setting, \(E\) may be
represented by a bounded above complex \(\mathcal{E}^\bullet\) of finite free
\(\sO_X\)-modules, and \(F\) by a bounded complex \(\mathcal{F}^\bullet\) of
\(A\)-flat finitely presented \(\sO_X\)-modules.

On the one hand, as in the proof of \emph{ibid.}, the internal \(\Hom\) in
question may be represented as
\[
R\sHom_{\sO_X}(E,F) \simeq
\sHom_{\sO_X}^\bullet(\mathcal{E}^\bullet, \mathcal{F}^\bullet).
\]
The boundedness properties of \(\mathcal{E}^\bullet\) and
\(\mathcal{F}^\bullet\) imply that every term of this
complex is a finite direct sum of \(\sO_X\)-modules of the form
\(\sHom_{\sO_X}(\mathcal{E}^i, \mathcal{F}^j)\). On the other hand, since
the terms of \(\mathcal{E}^\bullet\) are free, and \(\mathcal{F}^\bullet\)
are \(A\)-flat, the internal \(\Hom\) on the base change may be
represented as
\[
R\sHom_{\sO_{X_B}}(E_B,F_B) \simeq
\sHom_{\sO_{X_B}}^\bullet(\mathcal{E}^\bullet \otimes_A B, \mathcal{F}^\bullet \otimes_A B)
\]
and the complex satisfies the same finiteness condition. This implies
the result since
\[
\sHom_{\sO_X}(\mathcal{E}^i, \mathcal{F}^j) \otimes_A B \cong
\sHom_{\sO_{X_B}}(\mathcal{E}^i \otimes_A B, \mathcal{F}^j \otimes_A B).
\qedhere
\]
\end{proof}

\subsectiondash{Relative orthogonal categories}\label{derived-orthogonal}
Let \(S\) be a qcqs scheme and \(g \colon Y \to S\) a morphism of schemes which
is flat, proper, and of finite presentation. Given a perfect object
\(P \in \Dqc(Y)\), its \emph{\(S\)-linear right orthogonal category} is
the full subcategory \(\langle P \rangle^\perp \subset \Dqc(Y)\) with objects
\[
\langle P \rangle^\perp \coloneqq
\{E \in \Dqc(Y) : Rg_*R\sHom_{\sO_Y}(P,E) = 0\}.
\]
Two basic properties of this construction are that:
\begin{enumerate}
\item\label{derived-orthogonal.linear}
\(\langle P \rangle^\perp\) is triangulated and \(S\)-linear, and so by our convenctions
closed under direct sums; and
\item\label{derived-orthogonal.local}
membership in \(\langle P \rangle^\perp\) is flat local: for a faithfully flat
morphism \(U \to S\) from a qcqs scheme, if \(E \in \Dqc(Y)\) is such that
\(E_U \in \langle P_U \rangle^\perp\), then
\(E \in \langle P \rangle^\perp\).
\end{enumerate}

\begin{proof}
For \ref{derived-orthogonal.linear}, that the subcategory
\(\langle P \rangle^\perp\) is triangulated follows from the fact that
the functor \(Rg_*R\sHom_{\sO_Y}(P,-)\) is exact. For \(S\)-linearity, let
\(E \in \langle P \rangle^\perp\) and \(F \in \Dqc(S)\), then
\[
Rg_*R\sHom_{\sO_Y}(P, E \otimes^L Lg^*F)
\cong Rg_*R\sHom_{\sO_Y}(P,E) \otimes^L F
= 0
\]
using the projection formula, so \(E \otimes^L Lg^*F \in \langle P \rangle^\perp\).
Item \ref{derived-orthogonal.local} follows directly from flat base change and
faithfully flat descent.
% \[ (Rg_*R\sHom_{\sO_Y}(E,M))_U \cong Rg_{U,*}R\sHom_{\sO_{Y_U}}(E_U,M_U). \]
\end{proof}

We now arrive at the main result of this section, which essentially says that
we may check whether or not a Fourier--Mukai transform from a single
homogeneous component of the derived category of a \(\mathbf{G}_m\)-gerbe to
the derived category of a scheme is an equivalence after passing to an
fppf cover. This may be seen as a variant of the descent result
\cite[Theorem B]{BS:Descent} of Bergh and Schn\"urer; unfortunately, our
statement does not directly follow from theirs because our kernel is not
perfect, only relatively so, and the source of our functor is not the
derived category of a proper algebraic stack. Note also that this result,
although phrased in terms of a \(\mathbf{G}_m\)-gerbe, also applies to schemes
by taking a weight \(0\) kernel on any gerbe and using
\parref{gmgerbes-derived-categories}\ref{gmgerbes-derived-categories.pullback}.
The result is:

\begin{Theorem}
\label{thm-locallytwisted}
Let \(g \colon Y \to S\) be a flat, proper, and finitely presented morphism
of qcqs schemes, \(f \colon X \to S\) a fpqc locally H-projective
morphism from an algebraic space, and \(\pi \colon \mathcal{X} \to X\) a
\(\mathbf{G}_m\)-gerbe. Let
\[
\Phi \colon
\mathrm{D}_{\mathrm{qc},-i}(\mathcal{X}) \subset
\Dqc(\mathcal{X}) \stackrel{\Phi_K}{\longrightarrow}
\Dqc(Y)
\]
be the functor induced by the Fourier--Mukai transform with a
\(Y\)-perfect kernel
\(K \in \mathrm{D}_{\mathrm{qc},i}(\mathcal{X} \times_S Y)\), and let \(U \to S\)
be a faithfully flat morphism from a qcqs scheme.
\begin{enumerate}
\item\label{thm-locallytwisted.fully-faithful}
If the functor
\(\Phi_U \colon \mathrm{D}_{\mathrm{qc},-i}(\mathcal{X}_U) \to \Dqc(Y_U)\)
induced by the Fourier--Mukai transform with the \(Y_U\)-perfect kernel
\(K_U \in \mathrm{D}_{\mathrm{qc},i}(\mathcal{X}_U \times_U Y_U)\)
is fully faithful, then \(\Phi\) is fully faithful.
\end{enumerate}
Let \(P \in \Dqc(Y)\) be a perfect complex and let
\(\mathcal{A} \coloneqq \langle P \rangle^\perp \subset \Dqc(Y)\) be its \(S\)-linear right orthogonal.
\begin{enumerate}
\setcounter{enumi}{1}
\item\label{thm-locallytwisted.essential-image}
If the essential image of \(\Phi_U\) lies in \(\mathcal{A}_U\), then
the essential image of \(\Phi\) is lies in \(\mathcal{A}\).
\item\label{thm-locallytwisted.equivalence}
If \(\Phi'\) is fully faithful with essential image
\(\mathcal{A}_U\), then \(\Phi\) induces an equivalence
\(\mathrm{D}_{\mathrm{qc}, -i}(\mathcal{X}) \cong \mathcal{A}\).
\end{enumerate}
\end{Theorem}

Toward the proof, observe that \(\Phi\) is a composition of an inclusion of the
weight \(-i\) component of the derived category of \(\mathcal{X}\) and the
Fourier--Mukai transform \(\Phi_K\). The hypotheses on the situation guarantee
that \(\Phi\) has some good properties:

\begin{Lemma}
\label{derived-Phi-properties}
The functor
\(\Phi \colon \mathrm{D}_{\mathrm{qc},-i}(\mathcal{X}) \to \Dqc(Y)\)
\begin{enumerate}
\item\label{derived-Phi-properties.adjoint}
has an exact right adjoint; and
\item\label{derived-Phi-properties.perfect}
takes perfect complexes to perfect complexes.
\end{enumerate}
\end{Lemma}

\begin{proof}
For \ref{derived-Phi-properties.adjoint}, observe that \(\Phi\) has an
exact right adjoint since it is a composition of two functors which do: the
Fourier--Mukai transform \(\Phi_K\) does by \parref{lemma-rightadjointofFM},
and the right adjoint to the inclusion
\(\mathrm{D}_{\mathrm{qc},-i}(\mathcal{X}) \subset \Dqc(\mathcal{X})\) is the
projection onto the weight \(i\) component, see
\parref{gmgerbes-derived-categories}.

For \ref{derived-Phi-properties.perfect}, let
\(P \in \mathrm{D}_{\mathrm{qc},-i}(\mathcal{X})\) be a perfect complex. Then
\(L\pr_1^*P \in \mathrm{D}_{\mathrm{qc},-i}(\mathcal{X} \times_S Y)\) is
perfect, so that
\[
L\pr_1^*P \otimes^L K \in
\mathrm{D}_{\mathrm{qc}, 0}(\mathcal{X} \times_S Y) \cong
\Dqc(X \times_S Y)
\]
is \(Y\)-perfect, where the equivalence is induced by the surjective and smooth structure map
\(\pi \times \id \colon \mathcal{X} \times_S Y \to X \times_S Y\) as in
\parref{gmgerbes-derived-categories}\ref{gmgerbes-derived-categories.pullback}.
Under this identification, \(\Phi(P)\) is now obtained via pushforward along
\(f_Y \colon X \times_S Y \to Y\). The hypotheses on
\(g \colon Y \to S\) ensure that \(f_Y\) is also fpqc locally \(H\)-projective,
and so \parref{lemma-perfectpush} implies that \(Rf_{Y,*}\) takes
\(Y\)-perfect complexes to perfect complexes, yielding the result.
\end{proof}

We now phrase full faithfulness and essential surjectivity of the functor
\(\Phi\) in terms of a compact generator of
\(\mathrm{D}_{\mathrm{qc},-i}(\mathcal{X})\)---note that such an object exists,
as discussed in \parref{gmgerbes-derived-categories}.

\begin{Lemma}
\label{derived-descent-generator}
Let \(G \in \mathrm{D}_{\mathrm{qc}, -i}(\mathcal{X})\) be a perfect
\(S\)-linear generator.
\begin{enumerate}
\item\label{derived-descent-generator.fully-faithful}
\(\Phi\) is fully faithful if and only if the morphism in \(\Dqc(S)\)
\[
Rf_*R\sHom_{\sO_{\mathcal{X}}}(G,G) \to
Rg_*R\sHom_{\sO_Y}(\Phi(G), \Phi(G))
\]
from \parref{derived-fm}\ref{derived-fm.enriched-hom} is an isomorphism.
\item\label{derived-descent-generator.essential-image}
If \(\Phi\) is fully faithful, then the essential image of \(\Phi\) is
\(\mathcal{A} \coloneqq \langle P \rangle^\perp\) if and only if the essential
image of \(\Phi\) is contained in \(\mathcal{A}\) and \(\Phi(G) \oplus P\) is
an \(S\)-linear perfect generator of \(\Dqc(Y)\).
\end{enumerate}
\end{Lemma}

\begin{proof}
Item \ref{derived-descent-generator.fully-faithful} may be proven in the same
way as \parref{lemma-ffconditions} since \(\Dqc(S)\) has a perfect generator,
\(\mathrm{D}_{\mathrm{qc},-i}(\mathcal{X})\) a perfect \(S\)-linear generator,
and \(\Phi\) an exact right adjoint by
\parref{derived-Phi-properties}\ref{derived-Phi-properties.adjoint}, and
\(\Phi\) takes perfect complexes to perfect complexes by
\parref{derived-Phi-properties}\ref{derived-Phi-properties.perfect}.

For \ref{derived-descent-generator.essential-image}, suppose that the essential
image of \(\Phi\) is \(\mathcal{A}\). Since \(\Phi(G)\) is perfect by
\parref{derived-Phi-properties}\ref{derived-Phi-properties.perfect}, it remains
to see that \(\Phi(G) \oplus P\) is an \(S\)-linear generator of \(\Dqc(Y)\).
Let \(E \in \Dqc(Y)\) be a nonzero object. If
\(E \notin \mathcal{A}\), then by definition of the orthogonal category,
\[
Rg_*R\sHom_{\sO_Y}(P,E) \neq 0.
\]
If \(E \in \mathcal{A}\), choose
\(E' \in \mathrm{D}_{\mathrm{qc},-i}(\mathcal{X})\) with \(\Phi(E') \cong E\).
Then, using \ref{derived-descent-generator.fully-faithful} and the
hypothesis that \(G\) is an \(S\)-linear generator of
\(\mathrm{D}_{\mathrm{qc},-i}(\mathcal{X})\),
\[
Rg_*R\sHom_{\sO_Y}(\Phi(G),E) \cong
Rf_*R\sHom_{\sO_{\mathcal{X}}}(G,E') \neq 0.
\]
Namely, the collection of $F$ for which 
$$
Rf_*R\sHom_{\sO_{\mathcal{X}}}(G,F) \to Rg_*R\sHom_{\sO_Y}(\Phi(G),\Phi(F))
$$
is an isomorphism is an $S$-linear full triangulated subcategory of $D_{qc, -i}(\mathcal{X})$ by the projection formula. Put together, this implies that \(\Phi(G) \oplus P\) is an \(S\)-linear
generator of \(\Dqc(Y)\).

Conversely, assume that the essential image of \(\Phi\) is contained in
\(\mathcal{A}\) and \(\Phi(G) \oplus P\) is an \(S\)-linear generator of
\(\Dqc(Y)\). Given a nonzero object \(E \in \mathcal{A}\), this gives
\[
Rg_*R\sHom_{\sO_Y}(\Phi(G), E) =
Rg_*R\sHom_{\sO_Y}(\Phi(G) \oplus P, E) \neq 0.
\]
This means that \(\Phi(G)\) is an \(S\)-linear generator of
\(\mathcal{A}\). Since \(\mathcal{A}\) is triangulated by
\parref{derived-orthogonal}\ref{derived-orthogonal.linear} and \(\Phi\) is
\(S\)-linear by \parref{derived-fm}\ref{derived-fm.linear} and fully faithful, hence the essential image of $\Phi$ is triangulated and $S$-linear,
\parref{derived-relative-generators-subcategory}\ref{derived-relative-generators-subcategory.smallest}
now implies that \(\Phi\) is essentially surjective onto \(\mathcal{A}\).
\end{proof}

\begin{proof}[Proof of \parref{thm-locallytwisted}]
Let \(G \in \mathrm{D}_{\mathrm{qc},-i}(\mathcal{X})\) be a perfect
\(S\)-linear generator. Descent implies that
\(G_U \in \mathrm{D}_{\mathrm{qc},-i}(\mathcal{X}_U)\) is a perfect \(U\)-linear
generator. Statement \ref{thm-locallytwisted.fully-faithful} now
follows from the full faithfulness criterion of
\parref{derived-descent-generator}\ref{derived-descent-generator.fully-faithful}
since the condition there is compatible with fppf base change. Part
\ref{thm-locallytwisted.essential-image} follows from
\parref{derived-orthogonal}\ref{derived-orthogonal.local} since, for any flat qcqs
\(S\)-scheme \(T\),
\[
\Phi(\mathrm{D}_{\mathrm{qc},-i}(\mathcal{X}_T)) \subseteq \mathcal{A}_T
\iff
\Phi(G)_T \cong \Phi_T(G_T) \in \mathcal{A}_T
\iff
Rg_*R\sHom_{\sO_{Y_T}}(P_T,G_T) = 0
\]
since \(G_T \in \mathrm{D}_{\mathrm{qc},-i}(\mathcal{X}_T)\) is a \(T\)-linear
generator and the image of \(\Phi\) is a \(T\)-linear triangulated category.
Finally, \ref{thm-locallytwisted.equivalence} follows from the
essential surjectivity criterion
\parref{derived-descent-generator}\ref{derived-descent-generator.essential-image}
since the condition there, once again, is compatible with flat base change.
\end{proof}

\section{Moduli of complexes}\label{section-complexes}
In this section, we discuss Lieblich's stack of complexes on a flat,
proper, and finitely presented morphism \(f \colon X \to S\). The main
result of this section is \parref{theorem-fully-faithful-open-immersion},
which says that fully faithful Fourier--Mukai transforms whose kernels
are perfect relative to both source and target induce open immersions
between stacks of complexes.

\subsectiondash{Glueability and simplicity}\label{complexes-glueable}
Let \(f \colon \mathcal{X} \to \mathcal{S}\) be a flat, proper, finitely
presented, and concentrated morphism of stacks. An \(\mathcal{S}\)-perfect
object \(E \in \Dqc(\mathcal{X})\) is said to be
\begin{itemize}
\item \emph{glueable} if \(R^if_*R\sHom_{\sO_{\mathcal{X}}}(E,E) = 0\) for all
integers \(i < 0\);
\item \emph{simplistic} if
\(\sO_{\mathcal{S}} \to Rf_*R\sHom_{\sO_{\mathcal{X}}}(E,E)\) is an isomorphism
on \(0\)-th cohomology sheaves;
\item \emph{universally glueable} if for every morphism \(\mathcal{T} \to \mathcal{S}\)
of stacks, \(E_{\mathcal{T}} \in \Dqc(\mathcal{X}_{\mathcal{T}})\) is glueable; and
\item \emph{simple} if for every morphism \(\mathcal{T} \to \mathcal{S}\)
of stacks, \(E_{\mathcal{T}} \in \Dqc(\mathcal{X}_{\mathcal{T}})\) is simplistic.
\end{itemize}
These notions depend on the morphism \(f\), so we may sometimes emphasize that
the properties are \emph{relative to \(\mathcal{S}\)}. In particular, the
condition on the base changed object \(E_{\mathcal{T}}\) in universal
glueability and simplicity are relative to \(\mathcal{T}\).

Being glueable or simplistic behave well under base change. Namely, given a
flat and quasi-compact morphism \(\mathcal{T} \to \mathcal{S}\) of stacks, flat
base change together with \parref{lemma-internalhomandbasechange} imply that:
\begin{enumerate}
\item\label{complexes-glueable.base-change}
If \(E \in \Dqc(\mathcal{X})\) is either glueable or simplistic, the same is
true for \(E_{\mathcal{T}} \in \Dqc(\mathcal{X}_{\mathcal{T}})\).
\item\label{complexes-glueable.descend}
If \(\mathcal{T} \to \mathcal{S}\) is surjective and \(E_{\mathcal{T}}\) is
glueable or simplistic, the same is true for \(E\).
\end{enumerate}
In other words, the properties of being glueable or simplistic satisfy fppf
descent. Since stacks have fppf coverings by a disjoint union of affine
schemes, being universally glueable or simple may be characterized more
concretely as follows:
\begin{enumerate}
\setcounter{enumi}{2}
\item\label{complexes-glueable.universally-glueable-affines}
\(E\) is universally glueable if and only if for every morphism \(T \to \mathcal{S}\)
from an affine scheme, \(E_T \in \Dqc(\mathcal{X}_T)\) satisfies \(\Ext^i_{\mathcal{X}_T}(E_T, E_T) = 0\) for all
integers \(i < 0\).
\item\label{complexes-glueable.simple-affines}
\(E\) is simple if and only if for every morphism \(T \to \mathcal{S}\) from
an affine scheme, the canonical map \(\sO_{\mathcal{X}}(T) \to \Hom_{\mathcal{X}_T}(E_T,E_T)\) is
an isomorphism.
\end{enumerate}
In the case that the morphism \(f\) is between locally Noetherian algebraic
spaces, \cite[Proposition 2.1.9]{Lieblich:Complexes} gives a fibrewise
criterion for universal glueability:
\begin{enumerate}
\setcounter{enumi}{4}
\item\label{complexes-glueable.fibrewise}
Let \(f \colon X \to S\) be a flat, proper, and finitely presented morphism between
locally Noetherian algebraic spaces. An \(S\)-perfect object \(E \in \Dqc(X)\)
is universally glueable if and only if \(E_{\bar{s}} \in \Dqc(X_{\bar{s}})\)
is glueable for each geometric point \(\bar{s} \to S\).
\end{enumerate}

Universally glueable complexes have the pleasant property that their presheaves
of endomorphisms are actually sheaves. More generally:

\begin{Lemma}
\label{complexes-homs-glue}
Let \(f \colon \mathcal{X} \to \mathcal{S}\) be a flat, proper, finitely
presented, and concentrated morphism of stacks. Let
\(E, F \in \Dqc(\mathcal{X})\) be \(\mathcal{S}\)-perfect objects that
satisfy, for every object \(T \to \mathcal{S}\), the vanishing
\[
R^if_{T,*} R\sHom_{\sO_{\mathcal{X}_T}}(E_T, F_T) = 0
\;\;\text{for all integers \(i < 0\)}.
\]
Then the assignment \((T \to \mathcal{S}) \mapsto \Hom_{\mathcal{X}_T}(E_T,F_T)\)
defines an fpqc sheaf on \(\mathcal{S}\).
\end{Lemma}

\begin{proof}
This follows from flat base change together with
\parref{lemma-internalhomandbasechange}, which implies that the assignment
\[
(T \to \mathcal{S}) \mapsto
\Gamma(T, R^0f_{T,*}R\sHom_{\sO_{\mathcal{X}_T}}(E_T, F_T))
\]
defines an fpqc sheaf on \(\mathcal{S}\), and the fact that the vanishing of
negative pushforwards implies that the group on the right is
\(\Hom_{\mathcal{X}_T}(E_T,F_T)\).
% $$
% H^0((T, R^0g_{T,*}R \mathcal{H}om_{\mathcal{Y}_T}(K_T, L_T)) = H^0(T, Rg_{T,*}R \mathcal{H}om_{\mathcal{Y}_T}(K_T, L_T))= \operatorname{Hom}_{\mathcal{Y}_T}(K_T, L_T)
% $$
\end{proof}

\subsectiondash{Stack of complexes}\label{complexes-stack}
Lieblich constructs in \cite{Lieblich:Complexes}---though see also
\citeSP{0DLB}---a stack of complexes on a flat, proper, finitely presented
morphism \(f \colon X \to S\) of schemes of the form
\[
\ComplexesStack_{X/S} \coloneqq
\big\{
(T,E) :
T \in \mathrm{Sch}_S
\;\text{and}\;
\text{\(E \in \Dqc(X_T)\) is \(T\)-perfect and universally glueable}
\big\}.
\]
This is algebraic and locally of finite presentation
over \(S\). Its objects are pairs \((T,E)\) consist of an \(S\)-scheme \(T\)
and an object \(E \in \Dqc(X_T)\) which is universally glueable and perfect
relative to \(T\), and a morphism \((T, E) \to (T', E')\) between objects in
consist of a morphism \(h \colon T \to T'\) of schemes over \(S\), and an
isomorphism \(Lh^*E' \to E\). It is shown in \cite[Lemma 4.3.2 and Corollary
4.3.3]{Lieblich:Complexes} that there is an open substack
\[
\sComplexesStack_{X/S} \subset
\ComplexesStack_{X/S}
\]
consisting of those objects \((T,E)\) where \(E\) is simple relative to \(T\),
and that it naturally has the structure of a \(\mathbf{G}_m\)-gerbe over
an algebraic space \(\mathrm{sComplexes}_{X/S}\).
% It is naturally a $\mathbf{G}_m$-gerbe: For any object $(T,K),$ the morphism of sheaves
% $$
% \mathbf{G}_{m, T} \to Aut_T (K, K),
% $$
% which takes an invertible function $f$ to multiplication by $f$, is an isomorphism. We denote the associated algebraic space  $sComplexes_{Y/S}$. See \cite[Corollary 4.3.3]{Lieblich:Complexes}.

Universal objects on these stacks may be constructed, as usual, with a suitable
version of the Yoneda lemma. We have been unable to locate a reference, so
we give a proof of the version we need below in \parref{complexes-yoneda}.
Applying the result to the identity morphisms thus provides objects
\[
E_{\mathrm{univ}}
\in \Dqc(X \times_S \ComplexesStack_{X/S})
\;\;\text{and}\;\;
sE_{\mathrm{univ}}
\in \mathrm{D}_{\mathrm{qc},1}(X \times_S \sComplexesStack_{X/S})
\]
which is perfect and universally glueable relative to
\(\ComplexesStack_{X/S}\) and \(\sComplexesStack_{X/S}\), respectively. Moreover,
\(sE_{\mathrm{univ}}\) is also simple and may be identified as a weight \(1\)
object with respect to the \(\mathbf{G}_m\)-gerbe structure, see
\parref{gmgerbes-derived-categories}.

In the following, the objects on the right are the groupoids with the displayed
set of objects:

\begin{Lemma}\label{complexes-yoneda}
Let \(f \colon X \to S\) be a flat, proper, and finitely presented morphism
of schemes, and let \(\mathcal{T}\) be a stack over \(S\). There are canonical
equivalences of groupoids between
\begin{align*}
\Mor_S(\mathcal{T}, \ComplexesStack_{X/S})
& \simeq \{E \in \Dqc(X \times_S \mathcal{T}) : \text{\(E\) is \(\mathcal{T}\)-perfect and universally glueable}\}, \;\text{and} \\
\Mor_S(\mathcal{T}, \sComplexesStack_{X/S})
& \simeq \{E \in \Dqc(X \times_S \mathcal{T}) : \text{\(E\) is \(\mathcal{T}\)-perfect, universally glueable, and simple}\}.
\end{align*}
\end{Lemma}

\begin{proof}
We explain the first equivalence, the second being analogous. Given
\(E \in \Dqc(X \times_S \mathcal{T})\) which is \(\mathcal{T}\)-perfect and
universally glueable, define a morphism
\(\varphi_E \colon \mathcal{T} \to \ComplexesStack_{X/S}\)
by
\[
\varphi_E(T \to \mathcal{T}) \coloneqq E_T \in \ComplexesStack_{X/S}(T).
\]
To see that the functor \(E \mapsto \varphi_E\) is fully faithful, let
\(E, F \in \Dqc(X \times_S \mathcal{T})\) be \(\mathcal{T}\)-perfect
and universally glueable objects. Then a \(2\)-morphism
\(\varphi_E \to \varphi_F\) is the data of, for every object
\(T \to \mathcal{T}\), an isomorphism \(E_T \to F_T\) in \(\Dqc(X_T)\) such
that, for every morphism in \(\mathcal{T}\), the obvious square commutes. This
amounts to a section of the sheaf of isomorphisms
\[
(T \to \mathcal{T}) \mapsto \mathrm{Isom}_{X_T}(E_T, F_T).
\]
Note that this is a sheaf because it is a subsheaf of \(\Hom_{X_T}(E_T,F_T)\)
from \parref{complexes-homs-glue}, and the vanishing hypothesis there
holds beacuse the isomorphisms \(E_T \to F_T\) provides an isomorphism
\[
R^if_{T,*}R\sHom_{\sO_{X_T}}(E_T,F_T) \cong
R^if_{T,*}R\sHom_{\sO_{X_T}}(E_T,E_T)
\]
and the latter vanishes for all \(i < 0\) since \(E_T\) is universally glueable.
This, in turn, is equivalent to an isomorphism \(E \to F\)
in \(\Dqc(X \times_S \mathcal{T})\).

Essential surjectivity of \(E \mapsto \varphi_E\) is a consequence of the
Beilinson--Bernstein--Deligne glueing lemma \cite[Theorem 3.2.4]{BBD} in
the form stated in \citeSP{0DCB}. Namely, consider the functor
\[
u \colon 
\mathcal{D} \coloneqq \mathcal{T}_{\liset} \to
\mathcal{C} \coloneqq (X \times_S \mathcal{T})_{\liset}
\quad\quad
T \mapsto X_T.
\]
A morphism \(\varphi \colon \mathcal{T} \to \ComplexesStack_{X/S}\) of stacks
gives, for each \(T \in \mathcal{D}\) an object \(E_T \in \Dqc(X_T)\). For
\(T\) in the full subcategory \(\mathcal{B} \subset \mathcal{D}\) consisting
of affine schemes, the objects \(E_T\) are universally bounded and have
vanishing negative self-Exts, and so \emph{ibid.} applies to give a unique
object \(E \in \Dqc(X \times_S \mathcal{T})\) whose restriction to
\(\Dqc(X_T)\) is \(E_T\) for each \(T \in \mathcal{T}\). In other words,
\(\varphi_E \simeq \varphi\).
\end{proof}

\subsectiondash{Autoequivalences}\label{complexes-autoequivalences}
The stack of complexes on a flat, proper, and finitely presented morphism
\(f \colon X \to S\) of schemes carries several natural autoequivalences.
Let \(\mathcal{L}\) be an invertible \(\sO_X\)-module and let
\(\omega_{X/S}^\bullet\) be a relative dualizing complex for
\(f \colon X \to S\); this exists in this generality by \citeSP{0E2X}. Then
the functors \(- \otimes^L_{\sO_X} \mathcal{L}\) and
\(R\sHom_{\sO_X}(-,\omega_{X/S}^\bullet)\) induce equivalences of stacks
\[
\tau_{\mathcal{L}} \colon \ComplexesStack_{X/S} \to \ComplexesStack_{X/S}
\;\;\text{and}\;\;
D \colon \ComplexesStack_{X/S} \to \ComplexesStack_{X/S}
\]
satisfying
\(\tau_{\mathcal{L}^\vee} \circ \tau_{\mathcal{L}} \simeq \id \simeq D \circ D \),
and which preserve the substack of simple complexes.

That \(\tau_{\mathcal{L}}\) exists and has quasi-inverse
\(\tau_{\mathcal{L}^\vee}\) is straightforward. That \(D\) makes sense and is
an equivalence on the stack of complexes, however, does not appear to be
well-documented, so we include a proof here. The main task is to verify that
\(R\sHom_{\sO_X}(-,\omega_{X/S}^\bullet)\) preserves relative perfectness,
universal glueability, and simplicity:

\begin{Lemma}\label{complexes-duality}
Let \(f \colon X \to S\) be a flat, proper, and finitely presented morphism
of schemes. Let \(E \in \Dqc(X)\) and
\(D(E) \coloneqq R\sHom_{\sO_X}(E,\omega_{X/S}^\bullet)\).
\begin{enumerate}
\item\label{complexes-duality.perfect}
If \(E\) is \(S\)-perfect, then so is \(D(E)\).
\item\label{complexes-duality.double-dual}
The canonical morphism \(E \to D(D(E))\) is an isomorphism.
\item\label{complexes-duality.glueability}
Assuming \(E\) is \(S\)-perfect, if \(E\) is universally glueable or
simple, then the same is true for \(D(E)\).
\end{enumerate}
\end{Lemma}

\begin{proof}
The dualizing complex \(\omega_{X/S}^\bullet\) is \(S\)-perfect by its
definition in \citeSP{0E2T}, so \parref{lemma-internalhomandbasechange} applies
to show that \(D(E) = R\sHom_{\sO_X}(E,\omega_{X/S}^\bullet)\) lies in
\(\Dqc(X)\) and that its formation commutes with arbitrary base change. Since
each of the statements are local on the base by \parref{lemma-relperfislocal}
and
\parref{complexes-glueable}\ref{complexes-glueable.universally-glueable-affines}--\ref{complexes-glueable.simple-affines},
we may assume for the remainder that \(S\) is affine.

To see \ref{complexes-duality.perfect}, reduce to the case $S$ is affine and use the criterion \citeSP{0GET},
wherein we must show that \(Rf_*(D(E) \otimes^L F) \in \Dqc(S)\) is perfect for
every perfect object \(F \in \Dqc(X)\). But now properties of \(\omega_{X/S}^\bullet\),
as in \citeSP{0A9Q}, give
\[
Rf_*(D(E) \otimes^L F) \cong
Rf_*R\sHom_{\sO_X}(E \otimes^L F, \omega_{X/S}^\bullet) \cong
R\sHom_{\sO_S}(Rf_*(E \otimes^L F), \sO_S)
\]
and this is perfect since \(Rf_*(E \otimes^L F)\). Item
\ref{complexes-duality.double-dual} now follows from \citeSP{0A89}. So far,
this means that \(D\) defines an anti-equivalence of categories on
\(\Dqc(X)\), and this implies \ref{complexes-duality.glueability} since
the criteria from
\parref{complexes-glueable}\ref{complexes-glueable.universally-glueable-affines}--\ref{complexes-glueable.simple-affines},
are therefore preserved by \(D\).
% it suffices to check it's an
% isomorphism after applying the functor $Rg_*(E \otimes ^{\mathbf{L}} - )$ for
% any perfect complex $E$ on $Y$ (apply this to $E = G^\vee$ with $G$ a perfect
% generator of $D_{qc}(Y)$). So, let $E \in D_{qc}(Y)$ be perfect. Then
% \begin{align*}
% Rg_*(E \otimes^{\mathbf{L}} R \mathcal{H}om_{Y}(R \mathcal{H}om_{Y}(K, \omega_{Y/S}^\bullet), \omega_{Y/S}^\bullet)) &=
% Rg_*(R \mathcal{H}om_{Y}(E \otimes^{\mathbf{L}} R \mathcal{H}om_{Y}(K, \omega_{Y/S}^\bullet), \omega_{Y/S}^\bullet)) \\
% &=
% Rg_*(R \mathcal{H}om_{Y}(R \mathcal{H}om_{Y}(E \otimes^{\mathbf{L}} K, \omega_{Y/S}^\bullet), \omega_{Y/S}^\bullet)) \\
% &= R \mathcal{H}om_{S} (Rg_*(R \mathcal{H}om_{Y}(E \otimes^{\mathbf{L}} K, \omega_{Y/S}^\bullet)), \mathcal{O}_S) \\
% &= R \mathcal{H}om_{S} (R\mathcal{H}om_{S}(Rg_*(E \otimes^{\mathbf{L}} K),\mathcal{O}_S), \mathcal{O}_S) \\
% &= Rg_*(E \otimes^{\mathbf{L}}K)
% \end{align*}
% since $Rg_*(E \otimes^{\mathbf{L}}K)$ is perfect, as needed.
\end{proof}

% \begin{Corollary}\label{complexes-automorphisms-gorenstein-dual}
% If, additionally, \(g \colon Y \to S\) has Gorenstein fibres, then for
% any invertible \(\sO_Y\)-module \(\mathcal{L}\),
% \(R\sHom_{\sO_Y}(-,\mathcal{L})\) induces self equivalences of the stacks
% \(\ComplexesStack_{Y/S}\) and \(\sComplexesStack_{Y/S}\).
% \end{Corollary}

\subsectiondash{Fourier--Mukai open immersion}\label{complexes-FM}
Let \(f \colon X \to S\) and \(g \colon Y \to S\) be flat, proper, and finitely
presented morphisms of schemes and let \(\Phi_K \colon \Dqc(X) \to \Dqc(Y)\)
be the Fourier--Mukai transform associated with an object
\(K \in \Dqc(X \times_S Y)\) that is perfect relative to both \(X\) and \(Y\).
Compatibility of \(\Phi_K\) with base change, as in
\parref{derived-fm}\ref{derived-fm.base-change}, allows one to define
a functor \(\mathrm{FM}_K\) from the stack of complexes on \(f \colon X \to S\) to certain
complexes on \(g \colon Y \to S\) as follows:

On objects, set \(\mathrm{FM}_K(T,E) \coloneqq (T,\Phi_{K_T}(E_T))\), and on a
morphism \((h, \alpha) \colon (T,E) \to (T', E')\), meaning as in
\parref{complexes-stack} a \(S\)-morphism \(h \colon T \to T'\) and an
isomorphism \(\alpha \colon E'_T \to E\) in \(\Dqc(X_T)\), set
\[
\mathrm{FM}_K(h,\alpha) \coloneqq (h, \mathrm{FM}_K(\alpha)) \colon (T,\Phi_{K_T}(E)) \to (T',\Phi_{K_{T'}}(E'))
\]
where \(\mathrm{FM}_K(\alpha)\) is the isomorphism obtained by composing the
compatibility of the Fourier--Mukai transforms with base change together
with \(\Phi_{K_T}(\alpha)\).
% \[
% \mathrm{FM}_K(\alpha) \colon
% \Phi_{K_{T'}}(E')_T \stackrel{\cong}{\longrightarrow}
% \Phi_{K_T}(E'_T) \stackrel{\Phi_K(\alpha)}{\longrightarrow}
% \Phi_{K_T}(E).
% \]
Lemmas \parref{complexes-preserve-perfect} and
\parref{complexes-preserve-glueability} below imply that when \(\Phi_K\)
is fully faithful, \(\mathrm{FM}_K\) takes values in the stack of complexes on
\(g \colon Y \to S\) and that it preserves simple complexes. More notably,
the functor is even an open immersion of stacks:

\begin{Theorem}\label{theorem-fully-faithful-open-immersion}
In the setting of \parref{complexes-FM}, if
\(\Phi_K \colon \Dqc(X) \to \Dqc(Y)\) is fully faithful, then the assignment
\((T,E) \mapsto (T,\Phi_{K_T}(E_T))\) induces an open immersion of \(S\)-stacks
\[
\mathrm{FM}_K \colon \ComplexesStack_{X/S} \to \ComplexesStack_{Y/S}
\]
which preserves the open substack of simple complexes.
\end{Theorem}

We first verify that the Fourier--Mukai transform preserves \(S\)-perfectness:

\begin{Lemma}
\label{complexes-preserve-perfect}
In the setting of \parref{complexes-FM}, \(\Phi_K \colon \Dqc(X) \to \Dqc(Y)\)
takes
\begin{enumerate}
\item\label{complexes-preserve-perfect.absolute}
perfect complexes to perfect complexes;
\item\label{complexes-preserve-perfect.pseudo-coherent}
pseudo-coherent complexes to pseudo-coherent complexes; and
\item\label{complexes-preserve-perfect.relative}
\(S\)-perfect complexes to \(S\)-perfect complexes.
\end{enumerate}
\end{Lemma}

\begin{proof}
For \ref{complexes-preserve-perfect.absolute}, if \(E \in \Dqc(X)\)
perfect, then \(L\pr_1^*E \otimes^L K\) is \(Y\)-perfect and so its direct
image along \(R\pr_{2,*}\), which is \(\Phi_K(E)\), is perfect:
see \cite[\href{https://stacks.math.columbia.edu/tag/0DI4}{0DI4} and
\href{https://stacks.math.columbia.edu/tag/0DJT}{0DJT}]{stacks-project}.
Item \ref{complexes-preserve-perfect.pseudo-coherent} follows upon observing
that \(\Phi_K\) is the composition of three functors which individually
preserve pseudo-coherence, the main point being that \(R\pr_{2,*}\) does by
Kiehl's theorem \cite[Theorem 2.9]{Kiehl}; see also \citeSP{0CSD}. For
\ref{complexes-preserve-perfect.relative}, let \(E \in \Dqc(X)\) be
\(S\)-perfect. By \ref{complexes-preserve-perfect.pseudo-coherent},
\(\Phi_K(E) \in \Dqc(Y)\) is pseudo-coherent, so since \(g \colon Y \to S\) is
flat and finitely presented, \citeSP{0GEH} shows that \(\Phi_K(E)\) is
\(S\)-perfect if and only if its fibre over every \(s \in S\) is bounded below.
Since \(\Phi_K(E)_s \cong \Phi_{K_s}(E_s)\) by
\parref{derived-fm}\ref{derived-fm.base-change}, we may reduce to the case
where \(S = \Spec \kk\) is the spectrum of a field, in which case the result
follows from \citeSP{0FYU}.
\end{proof}

We now observe that when the Fourier--Mukai transform is fully faithful,
\(\Phi_K\) also preserves universal glueability and simplicity, proving that
\(\mathrm{FM}_K\) takes values in the stack of complexes of \(g \colon Y \to S\):

\begin{Lemma}
\label{complexes-preserve-glueability}
In the setting of \parref{complexes-FM}, assume
\(\Phi_K \colon \Dqc(X) \to \Dqc(Y)\) is fully faithful. If an
\(S\)-perfect object \(E \in \Dqc(X)\) is universally glueable or simple, then
the same is true for \(\Phi_K(E) \in \Dqc(Y)\).
\end{Lemma}

\begin{proof}
Let \(E \in \Dqc(X)\) be an \(S\)-perfect object which is universally
glueable. By \parref{complexes-glueable}\ref{complexes-glueable.universally-glueable-affines},
this means that for any affine \(S\)-scheme \(T\), we have
\(\Ext^i_{X_T}(E_T,E_T) = 0\) for all \(i < 0\). Since
\(\Phi_K(E)_T \cong \Phi_{K_T}(E_T)\) by \parref{derived-fm}\ref{derived-fm.base-change},
and the base changed Fourier--Mukai functor
\(\Phi_{K_T} \colon \Dqc(X_T) \to \Dqc(Y_T)\) is also fully faithful by
\parref{lemma-basechangeofff}, this implies
\[
\Ext^i_{Y_T}(\Phi_K(E)_T, \Phi_K(E)_T) =
\Ext^i_{Y_T}(\Phi_{K_T}(E_T), \Phi_{K_T}(E_T)) = 0
\;\text{for all \(i < 0\).}
\]
Applying
\parref{complexes-glueable}\ref{complexes-glueable.universally-glueable-affines}
once again shows that \(\Phi_K(E)\) is universally glueable. The same argument
works for \(E\) simple, upon using the criterion from
\parref{complexes-glueable}\ref{complexes-glueable.simple-affines}.
\end{proof}

The next statement shows that the right adjoint to the Fourier--Mukai
transform, which exists in this setting by \parref{lemma-rightadjointofFM}, is
compatible with base change and preserves relative perfectness:

\begin{Lemma}
\label{complexes-right-adjoint}
In the setting of \parref{complexes-FM}, let \(R \colon \Dqc(Y) \to \Dqc(X)\) be
the right adjoint of \(\Phi_K\).
\begin{enumerate}
\item\label{complexes-right-adjoint.base-change}
\(R\) is compatible with base change in that if \(T \to S\) is a morphism of
schemes and \(R_T \colon \Dqc(Y_T) \to \Dqc(X_T)\) is the right adjoint to
\(\Phi_{K_T}\), then there are canonical isomorphisms
\[
R(E)_T \cong R_T(E_T)
\;\;\text{for all \(E \in \Dqc(Y)\)}.
\]
\item\label{complexes-right-adjoint.relative-perfectness}
\(R\) takes relatively perfect objects to relatively perfect objects.
\end{enumerate}
\end{Lemma}

\begin{proof}
For \ref{complexes-right-adjoint.base-change}, since
\(\Phi_K = R\pr_{2,*} \circ (K \otimes^L -) \circ  L\pr_1^* \colon \Dqc(X) \to \Dqc(Y)\),
its right adjoint \(R\) may be written as the composition of three functors
\[
R = R\pr_{1,*} \circ R\sHom_{\sO_{X \times_S Y}}(K,-) \circ a
\]
where \(a\) is the right adjoint to \(R\pr_{2,*} \colon \Dqc(X \times_S Y) \to \Dqc(Y)\).
Since each of \(a\), \(R\sHom_{\sO_{X \times_S Y}}(K, -)\), and \(R\pr_{1,*}\)
commute with base change by \citeSP{0AA8}, \parref{lemma-internalhomandbasechange},
and tor-independent base change, respectively, \(R\) also commutes with
base change.

For \ref{complexes-right-adjoint.relative-perfectness}, compatibility with base
change in \ref{complexes-right-adjoint.base-change} together the local nature
of relative perfectness from \parref{lemma-relperfislocal} reduces the
statement to the case \(S\) is affine, wherein it suffices to verify the
criterion \citeSP{0GET}: Given \(E \in \Dqc(X)\) perfect and \(F \in \Dqc(Y)\)
\(S\)-perfect, we must show that
\[
R\Gamma(X, E \otimes^L R(F))
\cong R\Hom_X(E^\vee, R(F))
\cong R\Hom_Y(\Phi_K(E^\vee),F)
\cong R\Gamma(Y,\Phi_K(E^\vee)^\vee \otimes^L F) 
\]
is perfect; note that the second identification uses the fact that \(\Phi_K\)
is an enriched functor, as in \parref{derived-fm}\ref{derived-fm.enriched-hom}.
Since \(\Phi_K\) preserves perfect objects by
\parref{complexes-preserve-perfect}\ref{complexes-preserve-perfect.absolute}
and \(F\) itself is \(S\)-perfect, so is the displayed complex.
\end{proof}

\begin{proof}[Proof of \parref{theorem-fully-faithful-open-immersion}]
It remains to show that \(\mathrm{FM}_K\) is an open immersion. Full
faithfulness of \(\Phi_K\) implies that it is a monomorphism of stacks: indeed,
given \(E_1, E_2 \in \ComplexesStack_{X/S}(T)\), \parref{lemma-basechangeofff}
implies that the base changed Fourier--Mukai functor \(\Phi_{K_T}\) still
provides an isomorphism
\[
\Phi_{K_T} \colon
\Hom_{X_T}(E_1,E_2) \cong
\Hom_{Y_T}(\Phi_{K_T}(E_1), \Phi_{K_T}(E_2)).
\]
We may therefore conclude by showing that \(\mathrm{FM}_K\) is a smooth
morphism of stacks. Since the stacks of complexes are locally of finite
presentation over \(S\), we may do so via the infinitesimal criterion for
smoothness: Given a square zero thickening \(T' \to T\) and a solid commutative
square
\[
\begin{tikzcd}
    T' \ar[r] \ar[d] & \ComplexesStack_{X/S} \ar[d, "\mathrm{FM}_K"] \\
    T \ar[ur, dashed] \ar[r] & \ComplexesStack_{Y/S}
\end{tikzcd}
\]
construct a dashed arrow making the diagram commute. In other words, given
\begin{itemize}
\item
objects \((T',E') \in \ComplexesStack_{X/S}\) and
\((T,F) \in \ComplexesStack_{Y/S}\), and
\item an isomorphism \(\alpha \colon F_{T'} \cong \Phi_{K_{T'}}(E')\),
\end{itemize}
the task is to construct an object \((T,E) \in \ComplexesStack_{X/S}\) and
isomorphisms \(E_{T'} \cong E'\) and \(\Phi_{K_T}(E) \cong F\).

Let \(R \colon \Dqc(Y) \to \Dqc(X)\) be the right adjoint to
\(\Phi_K \colon \Dqc(X) \to \Dqc(Y)\). Then \parref{complexes-right-adjoint}
shows that \(R\) is compatible with base change and preserves
\(S\)-perfectness. Writing \(R_T\) for its base change to \(T\), the object
\(E \coloneqq R_T(F) \in \Dqc(X)\) is thus \(S\)-perfect. Base change,
\(\alpha\), and full faithfulness of \(\Phi_{K_{T'}}\), now give an isomorphism
\[
E_{T'} =
R_T(F)_{T'} \cong
R_{T'}(F_{T'}) \cong
R_{T'}(\Phi_{K_{T'}}(E')) \cong
E'.
\]
Similarly, the counit \(\varepsilon \colon \Phi_{K_T}(R(F)) \to F\) becomes an
isomorphism after restriction over \(T'\):
\[
(\Phi_{K_{T'}} \circ R_{T'})(F_{T'}) \cong
(\Phi_{K_{T'}} \circ R_{T'} \circ \Phi_{K_{T'}})(E') \cong
\Phi_{K_{T'}}(E') \cong
F_{T'}.
\]
Since the inclusion \(X_{T'} \to X_T\) is a homeomorphism, applying
\parref{lemma-nakforpseudo} then implies that the cone of \(\varepsilon\)
vanishes, and so \(\varepsilon\) itself is an isomorphism; in other words,
\(\Phi_{K_T}(E) \cong F\). Universal glueability of \(F\) together with
full faithfulness of \(\Phi_{K_T}\) now implies that \(E\) is also
universally glueable. It is now straightforward to see that
\((T,E)\) provides the sought-after dashed arrow.
\end{proof}

\begin{Lemma}
\label{lemma-nakforpseudo}
Let \((R,\mathfrak{m})\) be a local ring and \(E \in \mathrm{D}(R)\)
pseudo-coherent. Then
\[
E \otimes^L_R R/\mathfrak{m} = 0
\iff
E = 0.
\]
\end{Lemma}

\begin{proof}
Suppose that \(E \neq 0\) and let \(i \in \mathbf{Z}\) be maximal such that
\(\mathrm{H}^i(E) \neq 0\). Then \(\mathrm{H}^i(E)\) is a finite \(R\)-module
by \citeSP{0645}, so
\(\mathrm{H}^i(E \otimes^L_R R/\mathfrak{m}) \cong \mathrm{H}^i(E) \otimes_R R/\mathfrak{m} \neq 0\)
by Nakayama's lemma.
\end{proof}

\section{Residual category to an exceptional collection}\label{section-decompositions}
The main object of concern in this paper is the \emph{Kuznetsov component} of a
flat family \(\rho \colon Q \to S\) of quadric hypersurfaces in a
\(\PP^n\)-bundle: this is the full subcategory of \(\Dqc(Q)\) with objects
\[
\Ku(Q) \coloneqq
\{F \in \Dqc(Q) :
R\rho_*(F \otimes^L \sO_\rho(-i)) = 0
\;\text{for}\;i = 0,\ldots,n-2
\}
\]
that are \(S\)-linearly right orthogonal to the \(S\)-exceptional collection
\(\sO_\rho, \sO_\rho(1), \ldots, \sO_\rho(n-2)\): see
\parref{decompositions-exceptional} and \parref{decompositions-collections}.
Thus there is a semiorthogonal decomposition
\[
\Dqc(Q) =
\langle
\Ku(Q),\;
L\rho^*\Dqc(S),\;
L\rho^*\Dqc(S) \otimes^L \sO_\rho(1),
\ldots,\;
L\rho^*\Dqc(S) \otimes^L \sO_\rho(n-2)
\rangle.
\]
In this section, we generalize \parref{theorem-fully-faithful-open-immersion}
to allow the source category to be an admissible subcategory of the
quasi-coherent derived category of a scheme: see
\parref{theorem-fully-faithful-open-immersion-2}; here, \emph{admissible} is
always taken to mean two-sided admissible. Since we work with the entire
quasi-coherent derived category of a scheme---rather than the bounded
derived category of coherent sheaves---we also use this section to document
some facts regarding mutations and semiorthogonal decompositions in this
generality.

\subsectiondash{Relatively exceptional objects}\label{decompositions-exceptional}
Let \(f \colon X \to S\) be a flat, proper, and finitely presented morphism
of schemes. Assume that \(f\) has Gorenstein fibres, so that its relative
dualizing complex \(\omega_{X/S} \coloneqq \omega_{X/S}^\bullet\), moreover, is
a line bundle. A perfect complex \(E \in \Dqc(X)\) is called \emph{exceptional
relative to \(S\)} or, briefly, \emph{\(S\)-exceptional} if the natural
morphism
\[
\sO_S \to Rf_*R\sHom_{\sO_X}(E,E) \in \Dqc(X)
\]
is an isomorphism. Basic properties are that relative exceptionality is
preserved under arbitrary base change and that such objects provide admissible
subcategories of \(\Dqc(X)\). Precisely:
\begin{enumerate}
\item\label{decompositions-exceptional.base-change}
If \(T \to S\) is any morphism of schemes, then \(E_T \in \Dqc(X_T)\) is
\(T\)-exceptional.
\item\label{decompositions-exceptional.admissible}
The functor \(Lf^*(-) \otimes^L E \colon \Dqc(S) \to \Dqc(X)\) is fully
faithful and has both adjoints.
\end{enumerate}

\begin{proof}
Item \ref{decompositions-exceptional.base-change} follows from
\parref{lemma-internalhomandbasechange} together with tor-independent base change.
For \ref{decompositions-exceptional.admissible}, given \(F,G \in \Dqc(S)\),
the projection formula and exceptionality of \(E\) shows
\begin{align*}
\Hom_X(Lf^*F \otimes^L E, Lf^*G \otimes^L E)
& \cong \Hom_X(Lf^*F, Lf^*G \otimes^L R\sHom_{\sO_X}(E,E)) \\
& \cong \Hom_S(F, G \otimes^L Rf_*R\sHom_{\sO_X}(E,E))
\cong \Hom_S(F,G)
\end{align*}
implying full faithfulness. The right adjoint is given by
\(Rf_*R\sHom_{\sO_X}(E,-) \colon \Dqc(X) \to \Dqc(S)\). For the left adjoint,
compute:
\begin{align*}
\Hom_X(F, Lf^*G \otimes^L E)
& \cong \Hom_X(R\sHom_{\sO_X}(E,F), Lf^*G) \\
& \cong \Hom_X(R\sHom_{\sO_X}(E,F) \otimes^L \omega_{X/S}, Lf^*G \otimes^L \omega_{X/S}) \\
& \cong \Hom_X(Rf_*R\sHom_{\sO_X}(E,F \otimes^L \omega_{X/S}), G)
\end{align*}
where the second isomorphism is because \(\omega_{X/S}\) is a line bundle,
and the third is because \(Rf_*\) is left adjoint to
\(Lf^*(-) \otimes^L \omega_{X/S}\). Thus
\(Rf_*R\sHom_{\sO_X}(E,- \otimes^L \omega_{X/S})\) is the left adjoint.
\end{proof}

Let \(\mathcal{A} \coloneqq Lf^*\Dqc(S) \otimes^L E \subset \Dqc(X)\) be the
essential image of the functor in
\parref{decompositions-exceptional}\ref{decompositions-exceptional.admissible}.
This is an admissible subcategory, so it induces two semiorthogonal decompositions
\[
\Dqc(X)
= \langle \mathcal{A}, \prescript{\perp}{}{\mathcal{A}} \rangle
= \langle \mathcal{A}^{\perp}, \mathcal{A}\rangle.
\]
By definition, the left orthogonal \(\prescript{\perp}{}{}\mathcal{A}\) and
right orthogonal \(\mathcal{A}^\perp\) categories to \(\mathcal{A}\) are the
full subcategories of \(\Dqc(X)\) whose objects have no global maps to and from
objects in \(\mathcal{A}\), respectively. Since \(E\) is exceptional relative
to \(S\), these categories also admit a relative description, as follows:

\begin{Lemma}
\label{decompositions-relative-perp}
In the setting of \parref{decompositions-exceptional}, the
orthogonals to \(\mathcal{A} \coloneqq Lf^*\Dqc(S) \otimes^L E\) are
given by
\begin{align*}
\mathcal{A}^\perp
& = \{F \in \Dqc(X) : Rf_*R\sHom_{\sO_X}(E,F) = 0\}, \;\text{and} \\
\prescript{\perp}{}{\mathcal{A}}
& = \{F \in \Dqc(X) : Rf_*R\sHom_{\sO_X}(E, F \otimes^L \omega_{X/S}) = 0\}.
\end{align*}
In particular, \(\prescript{\perp}{}{\mathcal{A}}\) and \(\mathcal{A}^\perp\) are
\(S\)-linear subcategories of \(\Dqc(X)\) and are equivalent via
\(- \otimes^L \omega_{X/S}\)
\end{Lemma}

\begin{proof}
By definition, \(F \in \mathcal{A}^\perp\) if and only if for every
\(G \in \Dqc(S)\),
\[
0
= \Hom_X(Lf^*G \otimes^L E, F)
\cong \Hom_S(G, Rf_*R\sHom_X(E,F)),
\]
and this is equivalent to \(Rf_*R\sHom_{\sO_X}(E,F) = 0\) by Yoneda. Similarly,
\(F \in \prescript{\perp}{}{\mathcal{A}}\) if and only if
\[
0 = \Hom_X(F, Lf^*G \otimes^L E) \cong
\Hom_S(Rf_*R\sHom_{\sO_X}(E,F\otimes^L\omega_{X/S}), G),
\]
and the conclusion follows again by the Yoneda lemma. These descriptions
now imply \(S\)-linearity, as was explained in
\parref{derived-orthogonal}\ref{derived-orthogonal.linear}, and the equivalence
via tensor by \(\omega_{X/S}\).
\end{proof}

\subsectiondash{Relatively exceptional collections}\label{decompositions-collections}
Continuing with the setting of \parref{decompositions-exceptional}, a sequence
of perfect complexes \(E_1,\ldots,E_n \in \Dqc(X)\) is called an
\emph{exceptional collection relative to \(S\)}, or briefly, an
\emph{\(S\)-exceptional collection} if each \(E_i\) is \(S\)-exceptional and
\(Rf_*R\sHom_{\sO_X}(E_i,E_j) = 0\) for \(i > j\). By
\parref{decompositions-exceptional}\ref{decompositions-exceptional.admissible},
each \(\mathcal{A}_i \coloneqq Lf^*\Dqc(S) \otimes^L E_i\) is an admissible
subcategory of \(\Dqc(X)\), and repeatedly applying
\parref{decompositions-relative-perp} and argueing as \cite[Lemma
4.1(2)]{Xie:Quadrics} produces a semiorthogonal decomposition
\[
\Dqc(X) = \langle \mathcal{A}_0, \mathcal{A}_1, \ldots, \mathcal{A}_n \rangle
\]
where each \(\mathcal{A}_i\) is \(S\)-linear and admissible, and the objects
of \(\mathcal{A}_0\) may be described as
\[
\mathcal{A}_0 = \{F \in \Dqc(X) : Rf_*R\sHom_{\sO_X}(E_i, F) = 0\;\text{for \(i = 1,\ldots,n\)}\}.
\]
We refer to \(\mathcal{A}_0\) as the \emph{residual component} to the \(S\)-exceptional
collection \(E_1,\ldots,E_n\).

Since relative exceptionality is compatible with arbitrary base change as in
\parref{decompositions-exceptional}\ref{decompositions-exceptional.base-change},
such a semiorthogonal decomposition is also compatible with arbitrary base
change. Namely, if \(T \to S\) be any morphism of schemes, then there is a
semiorthogonal decomposition
\[
\Dqc(X_T) = \langle \mathcal{A}_{0, T}, \mathcal{A}_{1,T},\ldots, \mathcal{A}_{n,T} \rangle
\;\text{with}\;
\mathcal{A}_{i,T} \coloneqq Lf_T^*\Dqc(T) \otimes^L E_{i,T}
\;\text{for \(i = 1,\ldots,n\)}
\]
and the \(T\)-exceptional collection \(E_{1,T},\ldots,E_{n,T}\), and
\(\mathcal{A}_{0,T}\) is their \(T\)-linear right orthgonal.

\subsectiondash{Cyclic shifts and mutations}\label{decompositions-mutations}
Starting from any semiorthogonal decomposition
\[
\Dqc(X) = \langle \mathcal{B}_0, \mathcal{B}_1, \ldots, \mathcal{B}_n \rangle,
\]
induced from an \(S\)-exceptional collection---meaning that all but
possibly one component is equivalent to a subcategory of the form
\(Lf^*\Dqc(S) \otimes^L E_j\) for a perfect \(S\)-exceptional
\(E_j \in \Dqc(X)\)---there are two standard ways to construct a new one:

First, the descriptions of the two orthogonals of a subcategory generated by a
relatively exceptional object in \parref{decompositions-relative-perp} shows that
it is possible to cyclically shift the semiorthogonal decomposition, at least
after twisting the shifted component by the relative dualizing line bundle:
\[
\Dqc(X)
= \langle \mathcal{B}_n \otimes^L \omega_{X/S}, \mathcal{B}_0, \ldots, \mathcal{B}_{n-1} \rangle
= \langle \mathcal{B}_1, \ldots, \mathcal{B}_n, \mathcal{B}_0 \otimes^L \omega_{X/S}^{\vee} \rangle.
\]

Second, is left or right mutation across a semiorthogonal component
\(\mathcal{B}_i\). When \(\mathcal{B}_i\) is equivalent to a subcategory of the
form \(Lf^*\Dqc(S) \otimes^L E\) for an \(S\)-exceptional object \(E \in
\Dqc(X)\), the mutation functors
\(\mathbf{L}_E, \mathbf{R}_E \colon \Dqc(X) \to \Dqc(X)\) were already
introduced in \cite{Bondal:Reps} and are characterized as follows: They vanish
on \(\mathcal{B}_i\), induce equivalences
\(\mathbf{L}_E \colon \prescript{\perp}{}{\mathcal{B}_i} \to \mathcal{B}_i^\perp\)
and
\(\mathbf{R}_E \colon \mathcal{B}_i^\perp \to \prescript{\perp}{}{\mathcal{B}}_i\),
and are described explicitly on \(F \in \Dqc(X)\) by
\begin{align*}
\mathbf{L}_E(F) &
\coloneqq \operatorname{cone}(\iota R(F) \to F)
\cong \operatorname{cone}(Lf^*Rf_*R\sHom_{\sO_X}(E,F) \otimes^L E \to F),\;\text{and} \\
\mathbf{R}_E(F) &
\coloneqq \operatorname{cone}(F \to \iota L(F))[-1]
\cong \operatorname{cone}(F \to Lf^*Rf_*R\sHom_{\sO_X}(E,F \otimes^L \omega_{X/S}) \otimes^L E)[-1]
\end{align*}
where \(\iota \colon \mathcal{B}_i \to \Dqc(X)\) is the inclusion, and
\(L, R \colon \Dqc(X) \to \mathcal{B}_i\) are its left and right adjoints;
note that the cones are functorial by semiorthogonality, and the second
isomorphism comes from the identification of the adjoints from the proof of
\parref{decompositions-exceptional}\ref{decompositions-exceptional.admissible}.

\begin{Lemma}\label{decompositions-residual-adjoints}
In the setting of \parref{decompositions-collections}, let
\(L \colon \Dqc(X) \to \mathcal{A}_0\) and
\(R \colon \Dqc(X) \to \mathcal{A}_0\) be the left and right adjoints to the
inclusion \(\mathcal{A}_0 \subset \Dqc(X)\) of the residual component.
\begin{enumerate}
\item\label{decompositions-residual-adjoints.base-change}
If \(T\) is an \(S\)-scheme and \(L_T\) and \(R_T\) are the adjoints to the
inclusion \(\mathcal{A}_{0,T} \subset \Dqc(X_T)\), then
\[
L_T(F_T) \cong L(F)_T
\;\;\text{and}\;\;
R_T(F_T) \cong R(F)_T
\;\;\text{for any \(F \in \Dqc(X)\).}
\]
\item\label{decompositions-residual-adjoints.perfect}
The right adjoint \(R\) sends \(S\)-perfect objects to \(S\)-perfect objects.
\end{enumerate}
\end{Lemma}

\begin{proof}
For \ref{decompositions-residual-adjoints.base-change}, note that the left and
right adjoints to the inclusion \(\mathcal{A}_0 \subset \Dqc(X)\) may be
described as a composition of the mutation functors:
\[
L = \mathbf{L}_{E_1} \circ \cdots \circ \mathbf{L}_{E_n}
\;\;\text{and}\;\;
R = ((-) \otimes^L \omega_{X/S}) \circ \mathbf{R}_{E_n} \circ \cdots \circ \mathbf{R}_{E_1}.
\]
Combining the description of the mutation functors from
\parref{decompositions-mutations} together with flat base change and
\parref{lemma-internalhomandbasechange} shows that each of the mutation
functors commute with base change, in that
\[
\mathbf{L}_{E_i}(F)_T \cong \mathbf{L}_{E_{i,T}}(F_T)
\;\;\text{and}\;\;
\mathbf{R}_{E_i}(F)_T \cong \mathbf{R}_{E_{i,T}}(F_T).
\]
Since formation of the relative dualizing bundle \(\omega_{X/S}\) also
commutes with base change, this implies that \(L\) and \(R\) commute
with base change.

For \ref{decompositions-residual-adjoints.perfect}, it suffices to show that
the right mutation functors \(\mathbf{R}_{E_i}\) preserve \(S\)-perfect
objects. Given an \(S\)-perfect object \(F \in \Dqc(X)\), consider the
distinguished triangle
\[
\mathbf{R}_{E_i}(F) \longrightarrow
F \longrightarrow
Lf^*(Rf_*R\sHom_{\sO_X}(E_i,F) \otimes^L \omega_{X/S}) \otimes^L E_i \stackrel{+1}{\longrightarrow}
\]
from the definition of \(\mathbf{R}_{E_i}\). Since \(f \colon X \to S\)
is proper, \(Rf_*\) takes \(S\)-perfect complexes to perfect complexes, so the
term on the right is perfect. Since perfect complexes on \(X\) are also
\(X\)-perfect, and since the category of \(S\)-perfect complexes is
triangulated, this shows that \(\mathbf{R}_{E_i}(F)\) is
\(S\)-perfect.
\end{proof}

\subsectiondash{Generator for residual component}\label{decompositions-generator}
In order to adapt the proofs of
\S\S\parref{section-derived}--\parref{section-complexes} to prove analogous
results for the residual category in \parref{decompositions-collections}, we
prove that, at least when the base scheme \(S\) is additionally qcqs,
\(\mathcal{A}_0\) also has a compact generator that is a perfect complex on
\(X\): see \parref{decompositions-compactness}.

Given a triangulated category \(\mathcal{D}\) and two full triangulated
subcategories \(\mathcal{A}, \mathcal{B} \subseteq \mathcal{D}\), we write
\(\langle \mathcal{A}, \mathcal{B} \rangle \subseteq \mathcal{D}\) for the
smallest strictly full triangulated subcategory containing both \(\mathcal{A}\)
and \(\mathcal{B}\). When \(\mathcal{A}\) and \(\mathcal{B}\) satisfy
semiorthogonality conditions, objects of this category can be explicitly
described:

\begin{Lemma}
\label{lemma-smallestsubcat}
Let \(\mathcal{D}\) be a triangulated category and
\(\mathcal{A}, \mathcal{B} \subseteq \mathcal{D}\) full triangulated
subcategories. If \(\mathcal{A} \subseteq \mathcal{B}^\perp\), then an object
\(C \in \mathcal{D}\) lies in \(\langle \mathcal{A}, \mathcal{B} \rangle\) if
and only if there is a distinguished triangle
\[
B \longrightarrow C \longrightarrow A \stackrel{+1}{\longrightarrow}
\;\;\text{with \(A \in \mathcal{A}\) and \(B \in \mathcal{B}\).}
\]
In particular, if \(\mathcal{D}\) has arbitrary direct sums and \(\mathcal{A}\)
and \(\mathcal{B}\) are closed under them, then so is \(\langle \mathcal{A}, \mathcal{B} \rangle\).
\end{Lemma}

\begin{proof}
Every such \(C\) must lie in \(\langle \mathcal{A}, \mathcal{B} \rangle\)
because it is strictly full and triangulated. Conversely, it suffices to show
that the full subcategory \(\mathcal{C} \subseteq \mathcal{D}\)
consisting of such \(C\) is itself triangulated. It is clear that
\(\mathcal{C}\) closed under taking shifts, so it remains to show it is closed
under cones: Given a morphism \(C_1 \to C_2\) between objects fitting in
distinguished triangles \(B_i \to C_i \to A_i \to B_i[1]\) with
\(A_i \in \mathcal{A}\) and \(B_i \in \mathcal{B}\) for \(i = 1,2\),
semiorthogonality of \(\mathcal{A}\) and \(\mathcal{B}\) implies that there is
a unique morphism \(B_1 \to B_2\) making the diagram
\[
\begin{tikzcd}
B_1 \ar[r] \ar[d] & C_1 \ar[d] \\
B_2 \ar[r] & C_2
\end{tikzcd}
\]
commute. By \cite[Proposition 1.1.11]{BBD}, this square can be completed to a
$4 \times 4$ commutative-up-to-sign diagram whose rows and columns are
distinguished triangles. This provides a distinguished triangle
\(B_3 \to C_3 \to A_3 \to B_3[1]\) where $B_3$ is a cone of $B_1 \to B_2$,
$C_3$ is a cone of $C_1 \to C_2$, and there is a morphism $A_1 \to A_2$ whose
cone is $A_3$. But then $A_3 \in \mathcal{A}$ and $B_3 \in \mathcal{B}$, so the
cone $C_3$ of $C_1 \to C_2$ is itself an extension of an object of
$\mathcal{A}$ by an object of $\mathcal{B}$, showing that \(\mathcal{C}\) is
triangulated.

Regarding direct sums, let
\(\{C_i\}_{i \in I} \subset \langle \mathcal{A}, \mathcal{B} \rangle\) be a set of
objects where, for each \(i \in I\), there is a distinguished triangle
\(B_i \to C_i \to A_i \to B_i[1]\). Then there is a distinguished triangle
\[
\bigoplus\nolimits_{i \in I} B_i \longrightarrow
\bigoplus\nolimits_{i \in I} C_i \longrightarrow
\bigoplus\nolimits_{i \in I} A_i \stackrel{+1}{\longrightarrow}
\]
see \citeSP{0CRG}. Since \(\mathcal{A}\) and \(\mathcal{B}\) are closed under
sums, this implies that \(\bigoplus\nolimits_{i \in I} C_i \in \langle \mathcal{A}, \mathcal{B} \rangle\).
\end{proof}

\begin{Lemma}
\label{lemma-projectingagenerator}
Let \(\mathcal{D}\) be a triangulated category with arbitrary direct sums
and \(\iota \colon \mathcal{A} \to \mathcal{D}\) an admissible subcategory.
If \(\mathcal{A}\) is closed under sums and \(\mathcal{D}\) has a compact
generator \(G\), then the left adjoint \(L \colon \mathcal{D} \to \mathcal{A}\)
to \(\iota\) preserves compact objects and \(L(G) \in \mathcal{A}\) is a
compact generator.
\end{Lemma}

\begin{proof}
Since \(\iota\) also has a right adjoint, it commutes with arbitrary direct
sums. Thus \(L\) has a right adjoint that commutes with arbitrary direct sums,
so \(L\) preserves compact objects by \parref{lemma-preservescompact}.
To see that \(L(G) \in \mathcal{A}\) is a generator, let
\(A \in \mathcal{A}\) be nonzero and choose a nonzero morphism
\(G \to \iota(A)[i] \in \mathcal{D}\) for some \(i \in \mathbf{Z}\).
Adjunction then provides a nonzero morphism \(L(G) \to A[i] \in \mathcal{A}\).
\end{proof}

\begin{Lemma}
\label{lemma-rightadjpreservesplus}
Let \(\mathcal{D}\) be a triangulated category with arbitrary direct sums and
let \(\mathcal{D} = \langle \mathcal{A}, \mathcal{B} \rangle\) be a
semiorthogonal decomposition into components which are full triangulated
subcategories closed under direct sums. Then the right adjoint
\(R \colon \mathcal{D} \to \mathcal{B}\) to the inclusion
\(j \colon \mathcal{B} \to \mathcal{D}\) commutes with direct sums.
\end{Lemma}

\begin{proof}
Write \(\iota \colon \mathcal{A} \to \mathcal{D}\) be the other inclusion
and \(L \colon \mathcal{D} \to \mathcal{A}\) for its left adjoint. It is
standard \(R\) and \(L\) exist, and that every object \(D \in \mathcal{D}\)
fits into a functorial distinguished triangle
\[
jRD \longrightarrow
D \longrightarrow
\iota LD \stackrel{+1}{\longrightarrow}
\]
Since \(\mathcal{A}\) and \(\mathcal{B}\) are closed under direct sums, the
inclusion functors \(\iota\) and \(j\) commute with direct sums, and \(L\)
commutes with direct sums being a left adjoint. Thus, for any set of objects
\(\{D_i\}_{i \in I} \subset \mathcal{D}\), there is a morphism of distinguished
triangles
\[
\begin{tikzcd}
\bigoplus\nolimits_{i \in I} jRD_i \ar[r] \ar[d]
& \bigoplus_{i \in I} D_i \ar[r] \ar[d]
& \bigoplus_{i \in I} \iota LD_i \ar[r, "+1"] \ar[d] & \phantom{1}\\
jR(\bigoplus_{i \in I} D_i) \ar[r]
& \bigoplus _{i \in I} D_i \ar[r]
& \iota L(\bigoplus_{i \in I} D_i) \ar[r, "+1"] & \phantom{1}
\end{tikzcd}
\]
The middle two vertical arrows are isomorphisms by the discussion above, hence
so is the first vertical arrow. Therefore
\(j(\bigoplus\nolimits_{i \in I} RD_i) \cong jR(\bigoplus\nolimits_{i \in I} D_i)\) so,
since \(j\) is fully faithful,
\(\bigoplus\nolimits_{i \in I} RD_i \cong R(\bigoplus\nolimits_{i \in I} D_i)\).
\end{proof}

\begin{Proposition}\label{decompositions-compactness}
In the setting of \parref{decompositions-exceptional}, the residual category
\(\mathcal{A}_0\) satisfies:
\begin{enumerate}
\item\label{decompositions-compactness.generator}
writing \(L \colon \Dqc(X) \to \mathcal{A}_0\) for the left adjoint to
inclusion, if \(G\) is a compact generator of \(\Dqc(X)\), then \(L(G)\) is a
compact generator of \(\mathcal{A}_0\); and
\item\label{decompositions-compactness.compact}
\(E \in \mathcal{A}_0\) is compact as an object of \(\mathcal{A}_0\)
if and only if \(E\) is compact as an object of \(\Dqc(X)\).
\end{enumerate}
\end{Proposition}

\begin{proof}
Item \ref{decompositions-compactness.generator} follows direct from
\parref{lemma-projectingagenerator}. For \ref{decompositions-compactness.compact},
let \(E \in \mathcal{A}_0\) be compact when viewed as an object in \(\Dqc(X)\).
Since \(E = L(E)\) and \(L\) preserves compact objects by 
\parref{lemma-projectingagenerator}, \(E\) is also compact as an object
in \(\mathcal{A}_0\). Conversely, we must show that the inclusion
\(\mathcal{A}_0 \to \Dqc(X)\) preserves compact objects. Since \(\mathcal{A}_0\)
is compactly generated by \ref{decompositions-compactness.generator}, it
suffices by \parref{lemma-preservescompact} to show that the right adjoint to
inclusion commutes with direct sums. So consider the semiorthogonal
decomposition
\[
\Dqc(X) = \langle \mathcal{A}_0^\perp, \mathcal{A}_0 \rangle.
\]
Since \(\mathcal{A}_0\) is an \(S\)-linear triangulated subcategory as explained
in \parref{decompositions-collections}, it is closed under direct sums by our conventions and so
by \parref{lemma-rightadjpreservesplus}, it suffices to see that
\(\mathcal{A}_0^\perp\) is also closed under sums. Applying cyclic shifts as
in \parref{decompositions-mutations} to the given semiorthogonal decomposition
shows that
\[
\mathcal{A}_0^\perp = \langle \mathcal{A}_1 \otimes^L \omega_{X/S}, \ldots, \mathcal{A}_n \otimes^L \omega_{X/S} \rangle
\]
is the smallest strictly full triangulated subcategory of \(\Dqc(X)\) containing
each of the \(\mathcal{A}_i \otimes^L \omega_{X/S}\) for \(i = 1,\ldots,n\).
Since each of these subcategories are also \(S\)-linear, whence closed
under direct sums, it follows from \parref{lemma-smallestsubcat} that
\(\mathcal{A}_0^\perp\) is also closed under sums, as required.
\end{proof}

\subsectiondash{Fourier--Mukai on residual component}\label{decompositions-FM}
Let \(f \colon X \to S\) and \(g \colon X \to S\) be flat, proper, and
finitely presented morphisms of qcqs schemes. Assume additionally that the
fibres of \(f\) are Gorenstein and that there is a semiorthogonal decomposition
\[
\Dqc(X) = \langle \mathcal{A}_0, \mathcal{A}_1, \ldots, \mathcal{A}_n \rangle
\;\text{with}\;
\mathcal{A}_i \coloneqq Lf^*\Dqc(S) \otimes^L E_i
\;\text{for \(i = 1,\ldots,n\)}
\]
and an \(S\)-exceptional collection \(E_1,\ldots,E_n\). Let
\(K \in \Dqc(X \times_S Y)\) be a object that is perfect relative to both \(X\)
and \(Y\), \(\Phi_K \colon \Dqc(X) \to \Dqc(Y)\) the associated Fourier--Mukai
transform, and
\[
\Phi_{K,0} \coloneqq \Phi_K \circ \iota \colon
\mathcal{A}_0 \subset \Dqc(X) \to \Dqc(Y)
\]
its restriction to the residual category. Generalizing
\parref{complexes-preserve-perfect},
\parref{lemma-basechangeofff}, and
\parref{complexes-preserve-glueability}, this satisfies:
\begin{enumerate}
\item\label{decompositions-FM.preserve-perfect}
the functor \(\Phi_{K,0} \colon \mathcal{A}_0 \to \Dqc(Y)\) takes perfect
complexes to perfect complexes;
\item\label{decompositions-FM.base-change-ff}
if \(\Phi_{K,0} \colon \mathcal{A}_0 \to \Dqc(Y)\) is fully faithful, then
for any morphism \(T \to S\) from a qcqs scheme, 
the base changed functor
\(\Phi_{K_T,0} \colon \mathcal{A}_{0,T} \to \Dqc(Y_T)\) is also fully faithful;
\item\label{decompositions-FM.base-change-descend-ff}
if the morphism \(T \to S\) in \ref{decompositions-FM.base-change-ff} is
moreover faithfully flat, then the converse holds; and
\item\label{decompositions-FM.universal-glueability}
if \(\Phi_{K,0} \colon \mathcal{A}_0 \to \Dqc(Y)\) is fully faithful, then it
preserves universal glueability and simplicity.
\end{enumerate}

\begin{proof}
For \ref{decompositions-FM.preserve-perfect}, write
\(\Phi_{K,0} = \Phi_K \circ \iota\) and note that both \(\Phi_K\) and the
inclusion \(\iota\) preserve compact objects by
\parref{complexes-preserve-perfect} and the proof of
\parref{decompositions-compactness}\ref{decompositions-compactness.compact},
respectively. For \ref{decompositions-FM.base-change-ff} and
\ref{decompositions-FM.base-change-descend-ff}, let \(G \in \mathcal{A}_0\) be
a compact generator obtained by applying the left adjoint of inclusion to a
compact generator of \(\Dqc(X)\) as in
\parref{decompositions-compactness}\ref{decompositions-compactness.generator},
and argue as in \parref{lemma-ffconditions} to show that \(\Phi_{K,0}\) is
fully faithful if and only if the morphism
\[
\phi \colon
Rf_*R\sHom_{\sO_X}(G,G) \to
Rg_*R\sHom_{\sO_Y}(\Phi_K(G), \Phi_K(G))
\]
from \parref{derived-fm}\ref{derived-fm.enriched-hom} is an isomorphism;
here, the third hypothesis of the criterion \parref{lemma-ffbygen} is
satisfied by \ref{decompositions-FM.preserve-perfect}.
Since the base change of \(G\) to an \(S\)-scheme \(T\)
remains a compact generator of \(\mathcal{A}_{0,T}\) by
\parref{decompositions-residual-adjoints}\ref{decompositions-residual-adjoints.base-change}
and \parref{lemma-projectingagenerator}, flat base change together with
\parref{derived-fm}\ref{derived-fm.base-change} shows that this
full faithfulness criterion is invariant under base change, implying the
first two items. This also gives \ref{decompositions-FM.universal-glueability}
upon argueing exactly as in \parref{complexes-preserve-glueability}
using this base change property in place of \parref{lemma-basechangeofff}.
\end{proof}

Consider the full subcategory in the stack of complexes for the morphism \(f
\colon X \to S\) from \parref{complexes-stack} parameterizing objects in the
residual component \(\mathcal{A}_0\):
\[
\ComplexesStack_{\mathcal{A}_0/S} \coloneqq
\big\{(T,F) :
T \in \mathrm{Sch}_S
\;\text{and}\;
\text{\(F \in \mathcal{A}_{0,T}\) is \(T\)-perfect and universally glueable}
\big\}.
\]
This is an open substack since membership in this subcategory is defined by
\[
Rf_{T,*}R\sHom_{\sO_{X_T}}(E_{i,T}, F) = 0
\;\;\text{for}\; i = 1,\ldots,n,
\]
and this condition is closed under pullbacks by
\parref{lemma-internalhomandbasechange} and flat base change, and is an open
condition on \(T\) by \ref{lemma-nakforpseudo}. Further write
\[
\sComplexesStack_{\mathcal{A}_0/S} \subset
\ComplexesStack_{\mathcal{A}_0/S}
\]
for the open substack parameterizing simple objects
\(F \in \mathcal{A}_{0,T}\). The analogue of
\parref{theorem-fully-faithful-open-immersion} in this setting is:

\begin{Theorem}\label{theorem-fully-faithful-open-immersion-2}
In the setting of \parref{decompositions-FM}, if \(\Phi_{K,0} \colon \mathcal{A}_0 \to \Dqc(Y)\)
is fully faithful, then the assignment \((T,F) \mapsto (T,\Phi_{K_T}(F_T))\)
induces an open immersion of \(S\)-stacks
\[
\mathrm{FM}_{K,0} \colon
\ComplexesStack_{\mathcal{A}_0/S} \to
\ComplexesStack_{Y/S}
\]
which preserves the open substack of simple complexes.
\end{Theorem}

\begin{proof}
The argument is the same as that for
\parref{theorem-fully-faithful-open-immersion}, except that one uses
\parref{decompositions-FM}\ref{decompositions-FM.base-change-ff} to see that
\(\mathrm{FM}_{K,0}\) is a monomorphism; and for the smoothness argument,
additionally use \parref{decompositions-residual-adjoints}\ref{decompositions-residual-adjoints.base-change}
to see that the right adjoint \(R\) of \(\Phi_{K,0}\) is compatible with base
change; and 
\parref{decompositions-residual-adjoints}\ref{decompositions-residual-adjoints.perfect}
to see that \(R\) preserves \(S\)-perfect complexes.
\end{proof}

\section{Geometry of quadric bundles}\label{section-quadric-bundles}
In this section, we set our notation and conventions regarding families
of quadrics \(\rho \colon Q \to S\) over an arbitrary base scheme \(S\). Much
of the material is well-known, but perhaps not well-documented in this
generality. In particular, since \(2\) is not necessarily invertible on \(S\)
nor is \(S\) assumed to be reduced, some care needs to be taken when discussing
matters like the corank stratification. Throughout, \(S\) is a scheme,
\(\mathcal{E}\) is a locally free \(\sO_S\)-module of rank \(n+1\), and
\(\mathcal{L}\) is an invertible \(\sO_S\)-module.

\subsectiondash{Quadric bundles}\label{quadric-bundles}
A \emph{quadratic form} on \(\mathcal{E}\) with values in \(\mathcal{L}\) is a
morphism of sheaves of sets \(q \colon \mathcal{E} \to \mathcal{L}\) such that \(q(fv) = f^2
q(v)\) for local sections \(f\) of \(\sO_S\) and \(v\) of \(\mathcal{E}\), and
such that the associated \emph{polar form}
\(b_q \colon \mathcal{E} \times \mathcal{E} \to \mathcal{L}\), defined on local
sections \(v\) and \(w\) of \(\mathcal{E}\) by
\[
b_q(v,w) \coloneqq q(v+w) - q(v) - q(w),
\]
is a symmetric bilinear form. Let \(\pi \colon \PP\mathcal{E} \to S\) be the
projective bundle of lines associated with \(\mathcal{E}\), so that a nonzero
quadratic form \(q \colon \mathcal{E} \to \mathcal{L}\) corresponds to a
nonzero section
\[
s_q
\in \Gamma(\PP\mathcal{E}, \sO_\pi(2) \otimes \pi^*\mathcal{L})
\cong \Gamma(S, \Sym^2(\mathcal{E}^\vee) \otimes \mathcal{L}).
\]
Its vanishing locus \(\iota \colon Q \hookrightarrow \PP\mathcal{E}\) is a
family of quadrics over \(S\). The morphism
\(\rho \coloneqq \pi \circ \iota \colon Q \to S\) is flat of relative dimension $n-1$ if and only if the
form \(q\) is \emph{primitive} in that it is nonzero over every residue field
of \(S\). In this case, we refer to the family
\(\rho \colon Q \to S\) as the \emph{quadric \((n-1)\)-fold bundle} associated
with \(q\).

\subsectiondash{Orthogonals}\label{quadric-bundles-orthogonals}
Given a submodule \(\mathcal{F} \subseteq \mathcal{E}\), its \emph{orthogonal}
is
\[
\mathcal{F}^\perp \coloneqq
\ker\big(
b_q(-,\phantom{-})\rvert_{\mathcal{F}} \colon
\mathcal{E} \to
\mathcal{E}^\vee \otimes \mathcal{L} \to
\mathcal{F}^\vee \otimes \mathcal{L}
\big)
\]
where the first map is the adjoint
\(b_q \colon \mathcal{E} \to \mathcal{E}^\vee \otimes \mathcal{L}\) to the
polar form, and the second map is the restriction. A special case is when
\(\mathcal{F} = \mathcal{E}\), wherein the orthogonal
\[
\rad b_q \coloneqq
\mathcal{E}^\perp \coloneqq
\ker(b_q \colon \mathcal{E} \to \mathcal{E}^\vee \otimes \mathcal{L})
\]
is referred to as the \emph{bilinear radical} of \(q\).

\subsectiondash{Singular locus}\label{quadric-bundles-singular-locus}
The \emph{singular locus} \(\Sing\rho\) of a family of quadrics
\(\rho \colon Q \to S\) refers to the locus in \(Q\) where the morphism
\(\rho\) is not smooth of relative dimension \(n-1\). This carries a scheme
structure given by the \((n-1)\)-st Fitting ideal of the sheaf
\(\Omega^1_{Q/S}\) of relative differentials: see \citeSP{0C3K}. 
By the conormal sequence
\[
\mathcal{C}_{Q/S} \to
\Omega^1_{\PP\mathcal{E}/S}\rvert_Q \to
\Omega^1_{Q/S} \to
0,
\]
this is simply the zero scheme of the map of finite locally free $\mathcal{O}_Q$-modules $\mathcal{C}_{Q/S} \to
\Omega^1_{\PP\mathcal{E}/S}$ on $Q$. The vector bundle $\Omega^1_{\PP\mathcal{E}/S}|_Q$ is a subbundle of $\rho^*\mathcal{E}(-1)$ via
the relative Euler sequence for \(\pi \colon \PP\mathcal{E} \to S\),
and we have\(\mathcal{C}_{Q/S} \cong \sO_\rho(-2) \otimes \rho^*\mathcal{L}\)
via the section \(s_q\). Thus twisting up by \(\sO_\rho(1)\) identifies
\(\Sing\rho\) with the zero scheme in \(Q\) of a map of finite locally free $\mathcal{O}_Q$-modules
\[
\Sing\rho =
\mathrm{V}(
\sO_\rho(-1) \otimes \rho^*\mathcal{L}^\vee \to 
\rho^*\mathcal{E}^\vee
).
\]
In local coordinates, this section gives the vector of partial derivatives of
the quadratic polynomial \(s_q\). A straightforward computation therefore shows
that this is the composition
\[
\rho^*b_q(\mathrm{eu}_\pi\rvert_Q,\phantom{-}) \colon
\sO_\rho(-1) \otimes \rho^*\mathcal{L}^\vee \to
\rho^*(\mathcal{E} \otimes \mathcal{L}^\vee) \to
\rho^*(\mathcal{E}^\vee \otimes \mathcal{L} \otimes \mathcal{L}^\vee) \cong
\rho^*\mathcal{E}^\vee
\]
where \(\mathrm{eu}_\pi \colon \sO_\pi(-1) \to \pi^*\mathcal{E}\) is the
tautological section, and
\(b_q \colon \mathcal{E}\to \mathcal{E}^\vee \otimes \mathcal{L}\) is
adjoint to the polar form of \(q\). The formation of the closed subscheme $\operatorname{Sing}\rho \subset Q$ clearly commutes with base change, and in particular for $s \in S$ we have
\[
(\Sing\rho)_s =
\Sing Q_s = Q_s \cap \PP(\rad b_{q,s})
\]
where \(b_{q,s} \colon \mathcal{E}_s \to (\mathcal{E}^\vee \otimes \mathcal{L})_s\)
is the restriction of the polar form to the residue field at \(s\), and where
the bilinear radical is the kernel of \(b_{q,s}\) as defined in
\parref{quadric-bundles-orthogonals}.

\subsectiondash{Coranks}\label{quadric-bundles-corank}
Over a field \(\kk\), the \emph{corank} of a quadric \(Q \subseteq \PP^n\)
is one more than the dimension of the locus \(\Sing Q\) of points not smooth of
dimension \(n-1\); for example, a smooth quadric is of corank \(0\), and that
defined by the zero quadratic form is of corank \(n+1\). Since the singular
locus is the intersection \(Q \cap \PP(\rad b_q)\) by
\parref{quadric-bundles-singular-locus}, setting
\(\corank b_q \coloneqq \dim_{\kk} \rad b_q\), gives inequalities
\[
\corank b_q - 1 \leq
\corank Q \leq
\corank b_q.
\]
When \(\corank Q = \corank b_q\), then the singular locus of \(Q\) is,
scheme-theoretically, the linear space on \(\rad b_q\). This is essentially because $\mathbf{P} (\operatorname{rad} b_q)$ is reduced. We shall see, however, that
when \(\corank Q < \corank b_q\), the singular locus of \(Q\) is never geometrically reduced over \(\kk\).

When \(\operatorname{char}(\kk) \neq 2\), the equation $q(x) = \frac{1}{2} b_q(x,x)$
means that $\mathbf{P}(\operatorname{rad} b_q) \subset Q$, and
so the corank of \(Q\) coincides with the corank of \(b_q\). When
\(\operatorname{char}(\kk) = 2\), a choice of basis for \(\rad b_q\) gives an
expression
\[
q \rvert_{\rad b_q} = \sum\nolimits_{i = 1}^{\corank b_q} a_i x_i^2
\]
for some \(a_i \in \kk\). If this vanishes, then \(Q\) contains all of
\(\PP(\rad b_q)\) and so \(\corank Q = \corank b_q\); otherwise, the singular
locus of \(Q\) is geometrically a double plane---which may be defined only
over an inseparable extension of \(\kk\)!---inside \(\PP(\rad b_q)\).
This gives the
first statement of:

\begin{Lemma}\label{quadrics-corank-bilinear}
Let \(Q\) be a quadric hypersurface in \(\PP^n\) over a field \(\kk\). Then
\[
\corank Q =
\begin{dcases*}
\corank b_q - 1 & if \(\operatorname{char}(\kk) = 2\) and \(q|_{\operatorname{rad}b_q} \neq 0\) and \\
\corank b_q     & otherwise.
\end{dcases*}
\]
If $n + 1 - \corank Q$ is even, then $\corank Q = \corank b_q$.
\end{Lemma}

\begin{proof}
The second part concerns only $\operatorname{char}(\kk)  = 2$, in which case
the polar form $b_q$ is alternating, so its
rank is even. Since $\operatorname{corank} Q$ is either $\operatorname{corank}
b_q$ or $\operatorname{corank} b_q - 1$, the former must hold.
\end{proof}

%\begin{proof}
%It remains to treat the cases where \(\operatorname{char}(\kk) = 2\). Then
%the polar form \(b_q\) is moreover alternating, so its rank is even.
%Since \(\rank b_q \leq \rank Q \leq \rank b_q + 1\) by
%\parref{quadric-bundles-corank}, the result follows.
%\end{proof}

% Recall that if $q : V \to k$ is a quadratic form  defined over a field $k$, the
% \emph{radical} of $q$ is the subspace of $V$ defined by
% $$
% \{v \in \mathrm{Ker}(V \xrightarrow{b_q} V^*) : q(v) = 0\}.
% $$
% If $k$ has characteristic $\neq 2$ then this is equal to the kernel of the
% bilinear form $b_q : V \to V^*$, but in general it is a codimension $\leq 1$
% subspace thereof.

% \begin{Lemma}
% \label{lemma-corankofbil}
%     Let $k$ be a field. Let $V$ be a vector space over $k$. Let $q : V \to k$ be a quadratic form on $V$. Let $Q$ be the associated quadric. Assume either:
%     \begin{enumerate}
%         \item The characteristic of $k$ is 2 and $k$ is perfect; or
%         \item The characteristic of $k$ is not 2.
%     \end{enumerate}
%     Then the corank of $Q$ is equal to the dimension of the radical of $q$. 
% \end{Lemma}

\subsectiondash{Corank stratification}\label{quadrics-corank-stratification}
A family of quadrics \(\rho \colon Q \to S\) defines a decreasing filtration of
\(S\) by closed subsets
\[
S_c \coloneqq
\set{s \in S : \corank Q_s \geq c} =
\set{s \in S : \dim \Sing Q_s \geq c-1}.
\]
When \(2\) is invertible on \(S\), the corank of a quadric is equal to that of
its polar form, so the sets \(S_c\) coincide set-theoretically with the sets
\[
S_{\mathrm{bilin},c} \coloneqq
\set{s \in S : \operatorname{corank} b_q \geq c} =
\set{s \in S : \operatorname{rank} b_q \leq n+1-c}.
\]
These carry the scheme structure given locally by the vanishing of the size
\(n+2-c\) minors of \(b_q\), and this is what is given to the \(S_c\).

When \(2\) is not necessarily invertible on \(S\), it is a subtle issue to
endow the \(S_c\) with a suitable scheme structure: see, for example,
\cite[Definition 2.8]{ABBGvB} and \cite{Tanaka:Discriminant} for some
possibilities in low-dimensions. For the purposes of this
article, we will take that which is locally pulled back from the parameter
space of the universal quadric in \(\PP^n\) over \(\mathbf{Z}\), wherein they
are equipped with their reduced closed structure. Since universal degeneracy
loci are reduced, as follows from \cite[Theorem 1]{Kutz}, this is the usual
structure away from points of characteristic \(2\). Together with the corank
inequalities of \parref{quadric-bundles-corank}, this means that, for all
integers \(c \geq 0\), there are scheme-theoretic inclusions
\[
S_{\mathrm{bilin},c+1} \subseteq
S_c \subseteq
S_{\mathrm{bilin},c}.
\]
The content of these definitions is perhaps best illustrated with an example:

\begin{Example}
\label{quadric-bundles-dual-numbers-example}
Let \(S = \Spec\kk[\epsilon]\) be the spectrum of the ring of dual numbers
over a field \(\kk\), and consider the quadric surface bundle over \(S\) given
by
\[
Q \coloneqq \mathrm{V}(x_0x_1 + \epsilon x_2^2) \subset \PP^3_S.
\]
The moduli map of \(Q \to S\) identifies \(S\) with a tangent vector from a
closed point in the universal corank \(2\) locus into the corank \(1\) locus.
This means that the corank \(2\) stratum associated with the family \(Q \to S\)
is given by the reduced closed subscheme \(S_2 = S_{\mathrm{red}}\), and \(S_1 = S\)
scheme-theoretically; importantly, \(S_2 \neq S\) scheme-theoretically. Furthermore,
note that if \(\operatorname{char}\kk = 2\), then the bilinear corank \(2\)
locus is all of \(S\) scheme-theoretically, so \(S_2 \neq S_{\mathrm{bilin},2}\)
scheme-theoretically.
\end{Example}

It is often useful to restrict a quadric bundle to pieces of the corank
stratification, so that the singular locus may be identified as the projective
bundle on the bilinear radical. In general, as in cases such as
\parref{quadric-bundles-dual-numbers-example}, some care is required. The
following gives a useful situation in which this is possible, and may be viewed
as a relative version of the discussion in
\parref{quadric-bundles-corank} and \parref{quadrics-corank-bilinear}:

\begin{Lemma}
\label{quadric-bundles-even-rank}
Let \(\rho \colon Q \to S\) be a quadric \((n-1)\)-fold bundle. Assume that
\(S = S_c\) scheme-theoretically, \(S_{c+1} = \varnothing\), and that
\(r \coloneqq n+1-c > 0\) is even. Then $\rad b_q \subset \mathcal{E}$ is a subbundle and \(\Sing\rho\) is given by the
projective bundle \(\PP(\rad b_q)\).
\end{Lemma}

\begin{proof}
Since formation of the corank stratification and singular locus commute with
base change, it suffices to treat the case \(\rho \colon Q \to S\) is the
universal corank \(c\) quadric over \(\mathbf{Z}\), so that we may take
\(S\) to be reduced. The polar form
\(b_q \colon \mathcal{E} \to \mathcal{E}^\vee \otimes \mathcal{L}\) has
constant rank \(n+1-c\) so, since \(S\) is reduced, its kernel \(\rad b_q\) is
a subbundle of \(\mathcal{E}\) of corank \(c\). Now, on the one hand,
\(\Sing \rho\) is the scheme-theoretic intersection of \(Q\) with
\(\PP(\rad b_q)\). On the other hand, since \(n+1-c\) is even,
\parref{quadrics-corank-bilinear} gives a set-theoretic inclusion
\(\PP(\rad b_q) \subseteq Q\); but since the former is
reduced, being a projective bundle over a reduced base, the inclusion holds
scheme-theoretically, from which it follows that \(\Sing \rho = \PP(\rad b_q)\).
\end{proof}

% \begin{Example}
%     Let $S = \operatorname{Spec}(k[\epsilon])$ with $k$ a field and let $Q = V(T_0T_1 + \epsilon T_2T_3) \subset \mathbf{P}^3_S$. Then the closed fiber of $Q/S$ has corank 2 but the conclusion of the Lemma doesn't hold. 
% \end{Example}

% \begin{Lemma}
% \label{quadric-bundles-even-rank-etalelocalstructure}
% Let \(\rho \colon Q \to S\) be a quadric \((n-1)\)-fold bundle. Assume that
% \(S = S_c\) scheme-theoretically, \(S_{c+1} = \varnothing\), and that
% \(r \coloneqq n+1-c > 0\) is even. Then \'etale locally on $S$, $Q/S$ is isomorphic to the quadric
% $$
% V(T_0T_1+\cdots + T_{n-c-1}T_{n-c}) \subset \mathbf{P}^n_S,
% $$
% which is defined over $\mathbf{Z}$.
% \end{Lemma}

\subsectiondash{Fano schemes}\label{quadric-bundles-fano-schemes}
A submodule \(\mathcal{W} \subseteq \mathcal{E}\) is \emph{isotropic} for the
quadratic form \(q \colon \mathcal{E} \to \mathcal{L}\) if the restriction
\(q\rvert_{\mathcal{W}}\) is the zero map. The subscheme of the Grassmannian
\(\mathbf{G}(r+1,\mathcal{E})\) of rank \(r+1\) subbundles in \(\mathcal{E}\)
parameterizing isotropic subbundles is identified with the \emph{Fano scheme}
\[
\rho_r \colon \mathbf{F}_r(Q/S) \to S
\]
of \(r\)-planes in \(\rho \colon Q \to S\), which parameterizes
\(\PP^r\)-bundles contained in \(Q\). A basic fact is that when
\(\rho \colon Q \to S\) is smooth of relative dimension \(n-1\), then the Fano
schemes are smooth over \(S\), with the last nonempty one satisfying
\[
\dim(\rho_\ell \colon \mathbf{F}_\ell(Q/S) \to S) = \binom{\ell+1}{2}
\;\;\text{and}\;\;
\operatorname{rank}_{\sO_S}\big(
\rho_{\ell,*}\sO_{\mathbf{F}_\ell(Q/S)}\big) =
\begin{dcases*}
2 & if \(n-1 = 2\ell\), and \\
1 & if \(n-1 = 2\ell+1\).
\end{dcases*}
\]
In particular, this means that when \(\rho \colon Q \to S\) is of even relative
dimension \(2\ell\), each fibre of \(\rho_\ell \colon \mathbf{F}_\ell(Q/S) \to S\)
has two geometric connected components.

\subsectiondash{}\label{quadric-bundles-quotients}
Let \(\mathcal{F} \subseteq \mathcal{E}\) be an isotropic subbundle of rank
\(r\) such that its orthogonal \(\mathcal{F}^\perp \subseteq \mathcal{E}\) is
also a subbundle. Isotropicity implies that
\(\mathcal{F} \subseteq \mathcal{F}^\perp\). The quotient
\(\mathcal{F}^\perp/\mathcal{F}\) is then a finite locally free
\(\sO_S\)-module and \(q \colon \mathcal{E} \to \mathcal{L}\) induces a
quadratic form \(q' \colon \mathcal{F}^\perp/\mathcal{F} \to \mathcal{L}\),
defined on a local section \(v'\) by \(q'(v') \coloneqq q(v)\) for any lift
\(v\) to \(\mathcal{F}^\perp \subseteq \mathcal{E}\). This is well-defined
because
\[
q(v+w) = q(v) + q(w) + b_q(v,w) = q(v)
\]
for local sections \(v\) and \(w\) of \(\mathcal{F}^\perp\) and
\(\mathcal{F}\), respectively. Let \(\rho' \colon Q' \to S\) be the
corresponding family of quadrics. Comparing functors of points shows that
families of linear spaces in \(\rho \colon Q' \to S\) naturally correspond to
linear spaces in \(\rho \colon Q \to S\) containing \(\PP\mathcal{F}\):

\begin{Lemma}
\label{lemma-fano-correspondence}
For every integer \(k \geq 0\), the Fano scheme \(\mathbf{F}_k(Q'/S)\)
embeds into \(\mathbf{F}_{k+r}(Q/S)\) as the subscheme of \((k+r)\)-planes
containing \(\PP\mathcal{F}\).
\end{Lemma}

\begin{proof}
In detail, for every $S$-scheme $T$, the map
\(\mathcal{G} \mapsto \mathcal{G}/\mathcal{F}_T\)
provides a bijection between subbundles \(\mathcal{G} \subseteq \mathcal{E}_T\)
such that \(\mathcal{F}_T \subseteq \mathcal{G} \subseteq \mathcal{F}_T^\perp\)
and subbundles of \(\mathcal{F}_T^\perp/\mathcal{F}_T\), sending
\(q\)-isotropic subbundles to \(q'\)-isotropic ones. Furthermore, taking
\(v\) to be a local section of \(\mathcal{G}\) and \(t\) a local section of
\(\mathcal{F}_T\) in the computation of \parref{quadric-bundles-quotients} shows
that an isotropic subbundle of \(\mathcal{G} \subseteq \mathcal{E}_T\)
containing \(\mathcal{F}_T\) is automatically contained in \(\mathcal{F}_T^\perp\),
so that the map above restricts to a bijection between isotropic subbundles of
\(\mathcal{E}_T\) containing \(\mathcal{F}_T\) and isotropic subbundles of
\(\mathcal{F}_T^\perp/\mathcal{F}_T\).
\end{proof}

Combining this with \parref{quadric-bundles-even-rank} relates Fano schemes
of singular quadric bundles with those of smooth ones. An important
case for us is when \(\rho \colon Q \to S\) is a quadric \(2\ell\)-fold bundle
where all fibres are of corank \(2\), so that each fibre is a cone with
vertex \(\PP^1\) over a smooth quadric of dimension \(2\ell - 2\). The
following shows that such quadric bundles also have two families of maximal
isotropic subspaces. A related construction seems to have appeared in
\cite{DK:Double-Covers}.

\begin{Lemma}\label{quadric-bundles-fano-schemes-corank-2}
Let \(\rho \colon Q \to S\) be a quadric \(2\ell\)-fold bundle such that
\(S = S_2\) scheme-theoretically and \(S_3 = \varnothing\). The Stein
factorization of \(\rho_{\ell+1} \colon \mathbf{F}_{\ell+1}(Q/S) \to S\)
provides an \'etale double cover \(\tilde{S} \to S\).
\end{Lemma}

\begin{proof}
% Let $\mathcal{F}$ be the kernel of $b_q$ as in Lemma \ref{lemma-nonredevenrank}, a rank 2 isotropic subbundle of $\mathcal{E}$. Write $\overline{Q} / S$ for the quadric bundle corresponding to the induced quadratic form on $\mathcal{E}/\mathcal{F}$. This is smooth of relative dimension $2 \ell - 2$, and by Lemma \ref{lemma-fano-correspondence}, there is an isomorphism between $\mathbf{F}_{\ell-1}(\overline{Q}/S)$ and the closed subscheme of $\mathbf{F}_{\ell+1}(Q/S)$ parametrizing isotropic subbundles containing $\mathcal{F}$. We claim this closed subscheme is equal to $\mathbf{F}_{\ell+1}(Q/S)$. If the claim is true, we are done by the discussion of  \ref{quadric-bundles-fano-schemes}.
By \parref{quadric-bundles-even-rank}, \(\rad b_q \subset \mathcal{E}\) is a
rank \(2\) isotropic subbundle, and the quadric bundle
\(\bar{\rho} \colon \widebar{Q} \to S\) defined by the induced quadratic form
on \(\mathcal{E}/\rad b_q\), as in \parref{quadric-bundles-quotients}, is
smooth of relative dimension \(2\ell-2\). Then
\parref{lemma-fano-correspondence} embeds
\(\mathbf{F}_{\ell-1}(\widebar{Q}/S)\) in \(\mathbf{F}_{\ell+1}(Q/S)\) as the
locus of \((\ell+1)\)-planes containing
\(\PP(\mathrm{rad} b_q)\). We claim this is in fact all of $\mathbf{F}_{\ell+1}(Q/S)$, whereupon we can conclude by the discussion of \parref{quadric-bundles-fano-schemes}. We note that it is a routine exercise in quadratic forms to check that this is true on the level of sets. Let $T$ be an $S$-scheme and $\mathcal{F} \in \mathbf{F}_{\ell+1}(Q/S)_T$. Then by \parref{quadric-bundles-intersect-radical} below, $\mathcal{F} \cap \operatorname{rad}(b_q)_T \subset \mathcal{F}$ is a subbundle, and $\mathcal{F}/\mathcal{F} \cap \operatorname{rad}(b_q)_T$ is an isotropic subbundle of $(\mathcal{E}/\operatorname{rad}b_q)_T$. Since $(\mathcal{E}/\operatorname{rad}b_q)_T$ is non-degenerate of rank $2 \ell$, the maximal rank of an isotropic subbundle of it is $\ell$, so the rank of $\mathcal{F} \cap \operatorname{rad}(b_q)_T$ must be at least 2. But this is a subbundle of the rank two vector bundle $\operatorname{rad}(b_q)_T$ (as they are each subbundles of $\mathcal{E}_T$), so in fact we must have $\mathcal{F} \cap \operatorname{rad}(b_q)_T = \operatorname{rad}(b_q)_T$ or $\mathcal{F} \supset \operatorname{rad}(b_q)_T$, as needed.
\end{proof}

% Here we have used the following standard lemma:

% \begin{Lemma}
% \label{lemma-freeoverartin}
%     Let $(R, \mathfrak{m})$ be a local Artin ring. Let $M$ be a finitely generated $R$-module. Then 
%     $$
%     \operatorname{length}(M) \leq \operatorname{dim}_{R/\mathfrak{m}}(M/\mathfrak{m}M) \cdot \operatorname{length}(R)
%     $$
%     with equality if and only if $M$ is free. 
% \end{Lemma}

% \begin{proof}
%     Set $d = \operatorname{dim}_{R/\mathfrak{m}}(M/\mathfrak{m}M)$ and $\ell = \operatorname{length}(R)$. By Nakayama's Lemma, there is a surjection $R^{\oplus d} \to M$. Denote by $K$ the kernel. Then
%     \[
%     d \cdot \ell =
%     \operatorname{length}(R^{\oplus d}) = \operatorname{length}(K) +
%     \operatorname{length}(M) \geq \operatorname{length}(M)
%     \]
%     with equality if and only if $K = 0$. 
% \end{proof}

Linear spaces of dimension \(\ell\) and \(\ell+1\) must intersect the singular
locus of a corank \(2\) quadric \(2\ell\)-fold. This behaves well in families,
even when the base \(S\) may be nonreduced:

\begin{Lemma}\label{quadric-bundles-intersect-radical}
Let \(\rho \colon Q \to S\) be a quadric \(2\ell\)-fold bundle associated with
a quadratic form \(q \colon \mathcal{E} \to \mathcal{L}\) such that
\(S = S_2\) scheme-theoretically and \(S_3 = \varnothing\). If
\(\mathcal{F} \subset \mathcal{E}\) is an isotropic subbundle satisfying
\[
\rank(
b_q \rvert_{\mathcal{F}} \colon
\mathcal{F} \subset
\mathcal{E} \to
\mathcal{E}^\vee \otimes \mathcal{L}) \geq \ell
\]
at every point $s \in S$. Then \(\mathcal{F} \cap \rad b_q\) is a subbundle of \(\mathcal{F}\) and $\mathcal{F}/\mathcal{F} \cap \rad b_q$ is a subbundle of $\mathcal{E}/\rad b_q$. 
\end{Lemma}

\begin{proof}
Note that
\(\ker b_q\rvert_{\mathcal{F}} = \mathcal{F} \cap \rad b_q \subset \mathcal{E}\),
so writing \(b_q(\mathcal{F})\) for the image of \(\mathcal{F}\), there is a
short exact sequence of \(\sO_S\)-modules
\[
0 \to
\mathcal{F} \cap \rad b_q \to
\mathcal{F} \to
b_q(\mathcal{F}) \to
0.
\]
We show that \(b_q(\mathcal{F})\) is a local direct summand of \(\mathcal{F}\)
of rank \(\ell\), from which the conclusion follows from the short exact sequence.
The question is thus local on \(S\), so assume for the remainder that
\(S = \Spec R\) is the spectrum of a local ring.

On the one hand, the hypothesis on \(\rank b_q\rvert_{\mathcal{F}}\)
means that some $\ell \times \ell$ minor of the matrix representing $b_q|_{\mathcal{F}}$ is invertible, so \(b_q(\mathcal{F})\) contains a subsheaf \(\mathcal{F}'\) which is a
rank \(\ell\) direct summand of \(\mathcal{E}^\vee\). On the other hand, the
modules \(b_q(\mathcal{F}) \subset b_q(\mathcal{E})\) are isomorphic to the
quotient of \(\mathcal{F} \subset \mathcal{E}\) by \(\rad b_q\).
In this way, \(q\) induces a regular quadratic form on the rank \(2\ell\)
module \(b_q(\mathcal{E})\) with respect to which \(b_q(\mathcal{F})\) is a
totally isotropic submodule; in particular, the rank \(\ell\) summand
\(\mathcal{F}'\) is totally isotropic. In a regular quadratic module $M$ of rank
\(2\ell\), an isotropic summand $N$ of rank \(\ell\) is equal to its own orthogonal and hence maximal, since if $N \subset N'$ is isotropic then $N \subset N' \subset N^\perp = N$, and so this
implies \(b_q(\mathcal{F}) = \mathcal{F}'\) is a direct summand of
\(\mathcal{F}\).
\end{proof}

\subsectiondash{Hyperbolic reduction}\label{hyperbolic-reduction-setting}
Assume now that \(\rho \colon Q \to S\) is a quadric bundle, meaning that it
is flat of relative dimension $n-1$. Let \(\mathcal{F} \subseteq \mathcal{E}\) be an isotropic subbundle
of rank \(r\). Following \cite[\S2.1]{Kuznetsov:Reduction}, \(\mathcal{F}\)
is furthermore called \emph{regular isotropic} if the map
\[
b_q(-,\phantom{-})\rvert_{\mathcal{F}} \colon
\mathcal{E} \to
\mathcal{E}^\vee \otimes \mathcal{L} \to
\mathcal{F}^\vee \otimes \mathcal{L}
\]
is surjective. A straightforward argument using the discussion of
\parref{quadric-bundles-singular-locus} shows that this is equivalent to the
property that \(\PP\mathcal{F}\) is contained in the smooth locus of
\(\rho \colon Q \to S\). The orthogonal \(\mathcal{F}^\perp\) is then a
corank \(r\) subbundle of \(\mathcal{E}\) containing \(\mathcal{F}\), so \(q\)
induces via \parref{quadric-bundles-quotients} a quadratic
form \(q' \colon \mathcal{F}^\perp/\mathcal{F} \to \mathcal{L}\). The
associated family of quadrics \(\rho' \colon Q' \to S\) is called the
\emph{hyperbolic reduction} of \(Q\) along \(\mathcal{F}\). This has a modular
interpretation by \parref{lemma-fano-correspondence} as the scheme of
\(r\)-planes along \(\rho \colon Q \to S\) containing \(\PP\mathcal{F}\).

\subsectiondash{}\label{quadrics-reduction-blowup}
Hyperbolic reduction can be realized geometrically via linear projection
centred at \(\PP\mathcal{F}\). We describe this here, though see also
\cite[Proposition 2.5]{KS:Quadrics}. Let
\(\mathcal{E}' \coloneqq \mathcal{E}/\mathcal{F}\) be the quotient bundle,
\(\widetilde{\PP}\mathcal{E} \to \PP\mathcal{E}\) the blowup along
\(\PP\mathcal{F}\), and \(\widetilde{Q}\) the strict transform of \(Q\). Then
there is a diagram
\[
\begin{tikzcd}
& \widetilde{Q} \ar[dl] \rar[hook] & \widetilde{\PP}\mathcal{E} \ar[dl,"b"'] \ar[dr,"a"] \\
Q \rar[hook]
  & \PP\mathcal{E} \ar[rr,dashed]
  && \PP\mathcal{E}'
  & \lar[hook'] \PP(\mathcal{F}^\perp/\mathcal{F}) & \lar[hook'] Q'.
\end{tikzcd}
\]
The morphism \(\widetilde{Q} \to \PP\mathcal{E}'\) is generically a
\(\PP^{r-1}\)-bundle, and restricts to a \(\PP^r\)-bundle along the locus in
\(\PP\mathcal{E}'\) parameterizing \(r\)-planes in \(Q\) which contain
\(\PP\mathcal{F}\): this is \(Q'\) by \parref{lemma-fano-correspondence}.

This situation furthermore gives natural equations for \(Q'\). To describe
this, set some notation: Let \(\pi' \colon \PP\mathcal{E}' \to S\) be the
structure morphism and form the following commutative exact diagram
\[
\begin{tikzcd}
0 \rar
& \pi'^*\mathcal{F} \rar \dar[equal]
& \mathcal{G} \rar \dar
& \sO_{\pi'}(-1) \rar \dar
& 0 \\
0 \rar
& \pi'^*\mathcal{F} \rar
& \pi'^*\mathcal{E} \rar
& \pi'^*\mathcal{E}' \rar
& 0
\end{tikzcd}
\]
where the locally free \(\sO_{\PP\mathcal{E}'}\)-module \(\mathcal{G}\) sits as
the pullback of the right hand square. Write \(E\) for the exceptional divisor
of \(b \colon \widetilde{\PP}\mathcal{E} \to \PP\mathcal{E}\). The following
are standard facts:

\begin{Lemma}\label{lemma-projection-facts}
In the setting of \parref{hyperbolic-reduction-setting}, the following are true:
\begin{enumerate}
\item\label{lemma-projection-facts-bundle}
\(a \colon \widetilde{\PP}\mathcal{E} \to \PP\mathcal{E}'\) is the
\(\PP^r\)-bundle \(\PP\mathcal{G} \to \PP\mathcal{E}'\);
\item\label{lemma-projection-facts-exceptional}
the exceptional divisor \(E\) of \(b \colon \widetilde{\PP}\mathcal{E} \to \PP\mathcal{E}\)
is the subbundle \(\PP(\pi'^*\mathcal{F}) \subset \PP\mathcal{G}\);
\item\label{lemma-projection-facts-pic}
\(b^*\sO_\pi(1) = \sO_a(1)\) and
\(b^*\sO_\pi(1) = a^*\sO_{\pi'}(1) \otimes \sO_{\widetilde\PP\mathcal{E}}(E)\); and
\item\label{lemma-projection-facts-divisor}
\(\widetilde{Q}\) an effective Cartier divisor in \(\widetilde\PP\mathcal{E}\)
defined by a section
\[
\pushQED{\qed}
\tilde{q} \in
\mathrm{H}^0(\widetilde{\PP}\mathcal{E},
  \sO_a(1) \otimes a^*(\sO_{\pi'}(1) \otimes \pi'^*\mathcal{L})).
\qedhere
\popQED
\]
\end{enumerate}
\end{Lemma}

Pushing the section \(\tilde{q}\) from
\parref{lemma-projection-facts}\ref{lemma-projection-facts-divisor} along
\(a \colon \widetilde{\PP}\mathcal{E} \to \PP\mathcal{E}'\) gives a
canonical map
\[
a_*(\tilde{q}) \colon
\sO_{\PP\mathcal{E}'} \to
\mathcal{G}^\vee \otimes \sO_{\pi'}(1) \otimes \pi'^*\mathcal{L}.
\]
Since \(\tilde{q}\) is the family of linear forms along the projective bundle
\(\PP\mathcal{G} \to \PP\mathcal{E}'\) corresponding to linear sections of
\(Q\) containing \(\PP\mathcal{F}\), the vanishing locus of \(a_*(\tilde{q})\)
is supported on \(Q'\). The following explicitly verifies that they agree
scheme-theoretically:

\begin{Lemma}\label{hyperbolic-reduction-equations}
\(Q'\) is the vanishing locus of \(a_*(\tilde{q})\) in \(\PP\mathcal{E}'\).
\end{Lemma}

\begin{proof}
Having defined \(Q'\) and \(a_*(\tilde{q})\) globally, the question is local on
\(S\). Passing to a Zariski open, we may thus assume that \(\mathcal{L}\),
\(\mathcal{E}\), and \(\mathcal{F}\) are free \(\sO_S\)-modules of ranks \(1\),
\(n+1\), and \(r\), and that there are projective coordinates
\[
\PP\mathcal{E} \cong
\set{(x_0:\cdots:x_{r-1}:y_0:\cdots:y_{n-r}) \in \PP^n_S}
\]
such that \(\PP\mathcal{F} \cong \mathrm{V}(y_0,\ldots,y_{n-r})\) and
\(Q \cong \mathrm{V}(\sum\nolimits_{i = 0}^{r-1} x_iy_i + q')\) where \(q'\)
is a quadratic form in the remaining \(n-2r+1\) variables (by the ``splitting off a hyperbolic space lemma;'' see \cite[Proof of Theorem 3.6]{Baeza1978} which applies also when $q$ is degenerate but the isotropic subbundle is regular). In this way,
\(\PP\mathcal{E}'\) may be identified with the \(\PP^{n-r}_S\) on the
\(y_i\)-coordinates. Consider now the \(\PP^r\)-bundle
 The coordinates provide a splitting
\[
\mathcal{G} \cong \pi'^*\mathcal{F} \oplus \sO_{\pi'}(-1)
\]
and local fibre coordinates \((t_0:\cdots:t_r)\) for the associated
\(\PP^r\)-bundle \(a \colon \PP\mathcal{G} \to \PP\mathcal{E}'\) such that
\(\tilde{q}\) from
\parref{lemma-projection-facts}\ref{lemma-projection-facts-divisor} is
given above a point \(y = (y_0:\cdots:y_{n-r})\) of \(\PP\mathcal{E}'\) by
\[
\tilde{q}_y = y_0t_0 + \cdots + y_{r-1} t_{r-1} + q'(y_r,\ldots,y_{n-r}) t_r.
\]
Thus \(a_*(\tilde{q})\) is the section of
\(\mathcal{G}^\vee \otimes \sO_{\pi'}(1) \cong \sO_{\pi'}(1)^{\oplus r} \oplus \sO_{\pi'}(2)\)
with components \((y_0,\ldots,y_{r-1},q')\). Since the orthogonal of
\(\mathcal{F}\) in \(\mathcal{E}\) corresponds to the linear subspace
\(\mathrm{V}(y_0,\ldots,y_{r-1})\), the result follows upon comparing with the
definition of \(Q'\) from \parref{hyperbolic-reduction-setting}.
\end{proof}

This has a crucial consequence in the special case when the regular subbundle
\(\mathcal{F}\) has rank \(r = 1\). The discussion of
\parref{hyperbolic-reduction-setting} means that the morphism
\(\widetilde{Q} \to \PP\bar{\mathcal{E}}\) is an isomorphism away from \(Q'\),
and the inverse image of \(Q'\) has codimension \(1\). In fact,
\(\widetilde{Q}\) is often the blowup of \(\PP\bar{\mathcal{E}}\) along \(Q'\).
We give a proof below, though compare with \cite[Remark 2.6]{KS:Quadrics}. It
may also be interesting to note that something of this nature holds when
\(\mathcal{F}\) is not necessarily regular: see \cite[Lemma 4.3]{CPZ}.

\begin{Lemma}
\label{lemma-blowup-relation}
Assume that \(\mathcal{F}\) has rank \(1\). If \(Q'\) is an effective Cartier
divisor in \(\PP(\mathcal{F}^\perp/\mathcal{F})\), then
\(a \colon \widetilde{Q} \to \PP\mathcal{E}'\) is isomorphic to the blowup of
\(\PP\mathcal{E}'\) along \(Q'\) and its exceptional divisor \(\Lambda\)
satisfies
\[
\sO_{\widetilde{Q}}(\Lambda) \cong
\sO_a(1) \otimes \sO_{\widetilde{Q}}(-2E) \otimes \widetilde\pi^*(\mathcal{F}^\vee \otimes \mathcal{L}).
\]
\end{Lemma}

\begin{proof}
Since \(Q'\) is an effective Cartier divisor in
\(\PP(\mathcal{F}^\perp/\mathcal{F})\), which itself is a hyperplane in
\(\PP\mathcal{E}'\), its defining equation \(a_*(\tilde{q})\) from
\parref{hyperbolic-reduction-equations} is a regular section of
\(\mathcal{G}^\vee \otimes \sO_{\pi'}(1) \otimes \pi'^*\mathcal{L}\). The
ideal sheaf \(\mathcal{I}\) of \(Q'\) in \(\PP\mathcal{E}'\) therefore
admits a Koszul resolution
\[
0 \longrightarrow
\sO_{\PP\mathcal{E}'} \stackrel{a_*(\tilde{q})}{\longrightarrow}
\mathcal{G}^\vee \otimes \sO_{\pi'}(1) \otimes \pi'^*\mathcal{L} \longrightarrow
\Delta \otimes \mathcal{I} \longrightarrow
0
\]
where
\(\Delta \coloneqq \det(\mathcal{G}^\vee \otimes \sO_{\pi'}(1) \otimes \pi'^*\mathcal{L})\).
The sequence induces a surjection from the symmetric algebra on the centre term
to the Rees algebra on the twisted ideal sheaf on the right; since line
bundle twists only change the relative hyperplane class when taking
\(\Proj\)-constructions, this embeds the blowup of \(\PP\mathcal{E}'\) along
\(Q'\) into the projective bundle \(\PP\mathcal{G}\). Furthermore, the sequence
shows that the blowup is cut out by the ideal generated in degree \(1\)
by \(a_*(q)\), and this is precisely \(\widetilde{Q}\) by \parref{lemma-projection-facts}\ref{lemma-projection-facts-divisor}.

To determine the line bundle on \(\widetilde{Q}\) associated with the
exceptional divisor \(\Lambda\), recall that this is the relative \(\sO(-1)\)
for the \(\operatorname{Proj}\)-construction: see, for example, \citeSP{02OS}.
The Koszul resolution above relates twists of relative \(\sO(1)\)'s, yielding
\[
\sO_a(1) \otimes a^*(\sO_{\pi'}(1) \otimes \pi'^*\mathcal{L})\rvert_{\widetilde{Q}}
\cong a^*\Delta\rvert_{\widetilde{Q}} \otimes \sO_{\widetilde{Q}}(-\Lambda).
\]
The diagram of \parref{quadrics-reduction-blowup} implies that
\(\Delta \cong \sO_{\pi'}(3) \otimes \pi'^*(\mathcal{F}^\vee \otimes \mathcal{L}^{\otimes 2})\).
Using additionally the fact that
\(a^*\sO_{\pi'}(1) \cong \sO_a(1) \otimes \sO_{\widetilde{\PP}\mathcal{E}}(E)\)
from \parref{lemma-projection-facts}\ref{lemma-projection-facts-exceptional}
and rearranging then gives the result.
\end{proof}

The following gives an algebraic reformulation of the hypothesis in
\parref{lemma-blowup-relation}. See, for example,
\cite[\href{https://stacks.math.columbia.edu/tag/0547}{0547} and \href{https://stacks.math.columbia.edu/tag/056L}{056L}]{stacks-project}
for the definition of weakly associated points.

\begin{Lemma}\label{lemma-cartier-divisor-criterion}
Assume that \(\mathcal{F}\) has rank \(1\). Then \(Q'\) is an effective Cartier
divisor in \(\PP(\mathcal{F}^\perp/\mathcal{F})\) if and only if \(S_{n-1}\)
does not contain any weakly associated points of \(S\).
\end{Lemma}

\begin{proof}
This is local on \(S\), so pass to a Zariski open and adopt notation as in the
proof of \parref{hyperbolic-reduction-equations} with \(r = 1\). Then
\(Q\) is cut out in \(\PP^n_S\) by a quadratic form \(q = xy + q'\) and \(Q'\)
is cut out by \(q'\) in a complementary \(\PP^{n-2}_S\). That \(Q'\) is an
effective Cartier divisor means that \(q'\) is not a zero divisor in
\(\sO_S[y_1,\ldots,y_{n-1}]\). Either by a minimal multidegree
argument or else by \cite{McCoy}, the polynomial \(q'\) is a zero divisor if
and only if its coefficients are simultaneously killed by a nonzero scalar;
this means that their vanishing locus \(S_{n-1}\) contains a weakly associated
point by \citeSP{05C3}.
\end{proof}

\section{Clifford algebras and spinor sheaves}\label{section-clifford-and-spinor}
This section develops a theory of spinor sheaves, following
\cite{Addington:Spinors, Xie:Quadrics}. The new results here are
\parref{cliff-subbundle-spinors} and \parref{cliff-subbundle-modify}, which
show how spinor bundles are related through special hyperplane sections; and
\parref{cliff-cone-spinors}, which relates spinor sheaves on quadric cones with
spinors from its base. Throughout, unless otherwise stated,
\(\rho \colon Q \to S\) is a family of quadrics associated with a quadratic
form \(q \colon \mathcal{E} \to \mathcal{L}\), and we additionally assume that
\(Q\) is an effective Cartier divisor in \(\PP\mathcal{E}\): see also
\parref{lemma-cartier-divisor-criterion}.

\subsectiondash{Clifford algebras}\label{cliff-algebras}
Following \cite{BK:Clifford}, the sheaf of \emph{generalized Clifford algebras}
associated with the quadratic form \(q \colon \mathcal{E} \to \mathcal{L}\) is
the sheaf of \(\sO_S\)-algebras with presentation
\[
\Cl(\mathcal{E},q) \coloneqq
\Big(\bigoplus\nolimits_{i = 0}^\infty \mathcal{E}^{\otimes i}\Big) \otimes
\Big(\bigoplus\nolimits_{k \in \mathbf{Z}} \mathcal{L}^{\otimes k}\Big)/
\langle (v \otimes v) \otimes 1 - 1 \otimes 1 \otimes q(v) : v \in \mathcal{E}\rangle.
\]
Placing \(\mathcal{E}^{\otimes i}\otimes \mathcal{L}^{\otimes k}\) in degree \(i + 2k\), the ideal of relations is
generated by homogeneous elements of degree \(2\), whereupon
\(\Cl(\mathcal{E},q)\) becomes a sheaf of graded \(\sO_S\)-algebras, with
graded decomposition
\[
\Cl(\mathcal{E},q) =
\bigoplus\nolimits_{d \in \mathbf{Z}}\Cl_d(\mathcal{E},q).
\]
The sheaf \(\Cl_0(\mathcal{E},q)\) is the sheaf of \emph{even Clifford
algebras}; \(\Cl_1(\mathcal{E},q)\) is the \emph{odd Clifford bimodule}; and
there are isomorphisms
\[
\Cl_d(\mathcal{E},q) \cong
\begin{dcases*}
\Cl_0(\mathcal{E},q) \otimes \mathcal{L}^{\otimes k} & if \(d = 2k\), and \\
\Cl_1(\mathcal{E},q) \otimes \mathcal{L}^{\otimes k} & if \(d = 2k+1\).
\end{dcases*}
\]
Writing the rank \(n+1\) of \(\mathcal{E}\) as \(2\ell\) or \(2\ell+1\), the
tensor algebra induces \(\sO_S\)-module filtrations
\begin{align*}
\sO_S
& = \operatorname{Fil}_0
\subset \operatorname{Fil}_2
\subset \cdots
\subset \operatorname{Fil}_{2\ell\phantom{+1}}
= \Cl_0(\mathcal{E},q)\;\text{and}\\
\mathcal{E}
& = \operatorname{Fil}_1
\subset \operatorname{Fil}_3
\subset \cdots
\subset \operatorname{Fil}_{2\ell+1}
= \Cl_1(\mathcal{E},q),
\end{align*}
with associated graded pieces
\(\operatorname{Fil}_i/\operatorname{Fil}_{i-2} \cong \wedge^i \mathcal{E} \otimes (\mathcal{L}^\vee)^{\otimes \lfloor i/2\rfloor}\).
In particular, for each integer \(d\), \(\Cl_d(\mathcal{E},q)\) is a locally
free \(\sO_S\)-module of rank \(2^n\).

\subsectiondash{Clifford ideals}\label{cliff-ideals}
Functoriality of the Clifford algebra construction, verified in
\cite[Lemma 3.3]{BK:Clifford} for instance, provides, for each subbundle
\(\mathcal{E}' \subseteq \mathcal{E}\), an inclusion of graded
\(\sO_S\)-algebras
\[
\Cl(\mathcal{E}',q') \subseteq \Cl(\mathcal{E},q),
\]
where \(q' \colon \mathcal{E}' \to \mathcal{L}\) is the restriction of \(q\).
In particular, if \(\mathcal{W} \subset \mathcal{E}\) is an isotropic subbundle
of rank \(r\), this provides an inclusion of graded \(\sO_S\)-algebras
\[
\Big(\bigoplus\nolimits_{i = 0}^r \wedge^i \mathcal{W}\Big) \otimes
\Big(\bigoplus\nolimits_{k \in \mathbf{Z}} \mathcal{L}^{\otimes k}\Big) =
\Cl(\mathcal{W},0) \subset
\Cl(\mathcal{E},q).
\]
The \emph{Clifford ideal} \(\mathcal{I}^\mathcal{W} = \mathcal{I}\) associated
with \(\mathcal{W}\) is the graded left ideal in \(\Cl(\mathcal{E},q)\)
generated by the degree \(r\) summand \(\det\mathcal{W}\). Its \(d\)-th
graded piece \(\mathcal{I}_d\) is a locally free \(\sO_S\)-module of rank
\(2^{n-r}\), as can be seen by considering the tensor algebra filtration.
Since the annihilator of \(\det\mathcal{W}\) in \(\Cl(\mathcal{E},q)\)
is the left ideal generated by \(\mathcal{W}\), writing \([k]\) for
shift-by-\(k\) in grading, there is a presentation
\[
\Cl(\mathcal{E},q)[-r-1] \otimes \mathcal{W} \otimes \det\mathcal{W} \to
\Cl(\mathcal{E},q)[-r] \otimes \det\mathcal{W} \to
\mathcal{I} \to
0
\]
of graded \(\Cl(\mathcal{E},q)\)-modules: see also
\cite[Lemma 2.5]{Xie:Quadrics}.

\subsectiondash{Spinor sheaves}\label{cliff-spinors}
Let \(\pi \colon \PP\mathcal{E} \to S\) be the projective bundle associated
with \(\mathcal{E}\) and write \(\iota \colon Q \to \PP\mathcal{E}\) for the
inclusion morphism. For each \(d \in \mathbf{Z}\), the
\emph{\(d\)-th spinor sheaf} \(\mathcal{S}_d^{\mathcal{W}} = \mathcal{S}_d\)
associated with \(\mathcal{W}\) is the \(\sO_Q\)-module characterized as by
the short exact sequence
\[
0 \to
\sO_\pi(-1) \otimes \pi^*\mathcal{I}_{d-1} \xrightarrow{\phi_d}
\pi^*\mathcal{I}_d \to
\iota_*\mathcal{S}_d \to
0,
\]
where \(\phi_d\) is the map induced by Clifford multiplication upon viewing
\(\sO_\pi(-1)\) as the tautological subbundle of
\(\pi^*\mathcal{E} \subset \Cl_1(\mathcal{E},q)\). The point here is that
\(\phi_d\) is part of a matrix factorization of \(q\), that is,
\(\phi_d \circ \phi_{d-1} = q\), where we abuse notation and write
\(\phi_{d-1}\) for its twist by \(\sO_\pi(-1)\). This implies that the
cokernel of \(\phi_d\) is supported on \(Q\), and gives injectivity of
\(\phi_d\) under our assumption that \(Q\) is an effective Cartier divisor in
\(\PP\mathcal{E}\). Moreover, \cite[Proposition 2.1]{Addington:Spinors} shows
that if \(\corank(\mathcal{W} \subset \mathcal{E}) = n + 1 - r \geq 2\), then
the restriction of \(\mathcal{S}_d\) to
\begin{itemize}
\item
\(Q \setminus (\PP\mathcal{W} \cap \Sing\rho)\) is locally free of rank
\(2^{n-r-1}\); and
\item
\(\PP\mathcal{W} \cap \Sing\rho\) coincides with \(\rho^*\mathcal{I}_d\), and
so is locally free of rank \(2^{n-r}\).
\end{itemize}
The spinor sheaves also enjoy the following properties relative to \(S\):

\begin{Lemma}
\label{lemma-spinorsheavesrelperf}
\(\mathcal{S}_d\) is \(S\)-perfect; it is furthermore \(S\)-flat if
\(\rho \colon Q \to S\) is flat.
\end{Lemma}

\begin{proof}
By definition, $\mathcal{S}_d$ is $S$-perfect if and only if its pushforward to
$\mathbf{P}\mathcal{E}$ is perfect, which follows from the defining exact
sequence. If $\rho \colon Q \to S$ is flat, then
 $Q_s \subset \mathbf{P}\mathcal{E}_s$ is an effective Cartier divisor for each
\(s \in S\), so \(\phi_d\rvert_{\PP\mathcal{E}_s}\) is injective.
Then \citeSP{046Y} shows that $\iota_*\mathcal{S}_d$ is $S$-flat, at which
point \citeSP{0FLM} implies that \(\mathcal{S}_d\) itself is \(S\)-flat.
\end{proof}

\subsectiondash{}\label{cliff-spinors-duality}
Restricting the defining presentation of \(\iota_*\mathcal{S}_d\) to \(Q\) and
observing that \(\phi_i \circ \phi_{i-1}\) is a matrix factorization of \(q\)
provides two exact complexes of \(\sO_Q\)-modules
\begin{multline*}
\cdots \longrightarrow
\sO_\rho(-2) \otimes \rho^*\mathcal{I}_{d-2} \stackrel{\phi_{d-1}}{\longrightarrow}
\sO_\rho(-1) \otimes \rho^*\mathcal{I}_{d-1} \stackrel{\phi_d}{\longrightarrow}
\rho^*\mathcal{I}_d \longrightarrow
\mathcal{S}_d \longrightarrow 0, \;\text{and} \\
0 \longrightarrow
\mathcal{S}_d \longrightarrow
\sO_\rho(1) \otimes \rho^*\mathcal{I}_{d+1} \stackrel{\phi_{d+2}}{\longrightarrow}
\sO_\rho(2) \otimes \rho^*\mathcal{I}_{d+2} \stackrel{\phi_{d+3}}{\longrightarrow}
\cdots,
\end{multline*}
see \cite[\S4]{Addington:Spinors} and \cite[p.168]{Xie:Quadrics}. In
particular, this shows that
\[
\mathcal{S}_d \cong
\image(\phi_{d+1} \colon \rho^*\mathcal{I}_d \to \sO_\rho(1) \otimes \rho^*\mathcal{I}_{d+1}).
\]
By relating the dual of these complexes with the corresponding right Clifford
ideals associated with \(\mathcal{W}\), as is done in
\cite[Lemma 2.7 and Remark 2.9]{Xie:Quadrics}, this shows that the derived
dual of a spinor sheaf is simply a twist of another spinor sheaf; this is
summarized by the following duality relations, as in \cite[Proposition
4.1]{Addington:Spinors} and \cite[Corollary 2.11]{Xie:Quadrics}:
\[
\mathcal{S}_d^\vee \cong
\begin{dcases*}
\mathcal{S}_{r-d-1}
\otimes \sO_\rho(-1)
\otimes \rho^*(\det\mathcal{W}^\vee \otimes \det\mathcal{E}^\vee \otimes \mathcal{L}^{\otimes \ell})
& if \(\operatorname{rank}\mathcal{E} = 2\ell\), and \\
\mathcal{S}_{r-d\phantom{-1}}
\otimes \sO_\rho(-1)
\otimes \rho^*(\det\mathcal{W}^\vee \otimes \det\mathcal{E}^\vee \otimes \mathcal{L}^{\otimes \ell})
& if \(\operatorname{rank}\mathcal{E} = 2\ell+1\).
\end{dcases*}
\]
Finally, there is the degree shift relation
\(\mathcal{S}_d \otimes \rho^*\mathcal{L} \cong \mathcal{S}_{d+2}\).
\medskip

The defining presentation of \(\mathcal{S}_d\) makes it easy to compute its
derived pushforward along \(\rho \colon Q \to S\). For instance,
\(R\rho_*\mathcal{S}_d \cong \mathcal{I}_d\). For later use, some vanishing
pushforwards are as follows:

\begin{Lemma}\label{cliff-spinor-cohomology}
For any \(d \in \mathbf{Z}\) and each integer \(0 \leq i \leq n-2\),
\[
R\rho_*(\mathcal{S}_d \otimes \sO_\rho(-i-1)) =
R\rho_*R\sHom_{\sO_Q}(\sO_\rho(i), \mathcal{S}_d^\vee) = 0.
\]
\end{Lemma}

\begin{proof}
Vanishing for \(\mathcal{S}_d \otimes \sO_\rho(-i-1)\) follows from
the defining presentation of \(\mathcal{S}_d\) and the fact that
\(\pi \colon \PP\mathcal{E} \to S\) is a \(\PP^n\)-bundle. Together with the
duality relations from \parref{cliff-spinors-duality}, this implies the
vanishing of the
\(R\rho_*R\sHom_{\sO_Q}(\sO_\rho(i), \mathcal{S}_d^\vee)\).
\end{proof}

To illustrate the theory and for later use, the following two paragraphs
explicitly describe spinor sheaves when \(\mathcal{E}\) has small even rank.
In short: when \(\rank\mathcal{E} = 2\), spinors are structure sheaves
of points; whereas when \(\rank\mathcal{E} = 4\), spinors are essentially ideal
sheaves of lines.

\subsectiondash{Spinors in relative dimension \texorpdfstring{\(0\)}{0}}
\label{cliff-spinor-zero-dimensional}
Let \(\rho \colon Q \to S\) be a family of quadrics of relative dimension
\(0\), and consider a closed subscheme \(i \colon \PP\mathcal{W} \to Q\) given
by an isotropic line subbundle of \(\mathcal{E}\); this is the effective
Cartier divisor in \(\PP\mathcal{E}\) defined by a section
\(s_{\PP\mathcal{W}} \colon \sO_\pi \to \sO_\pi(1) \otimes \pi^*(\mathcal{E}/\mathcal{W})\).
The equation of \(Q\) in \(\PP\mathcal{E}\) factors as
\[
s_Q \colon
\sO_\pi(-2) \stackrel{s_{\PP\mathcal{W}}}{\longrightarrow}
\sO_\pi(-1) \otimes \pi^*(\mathcal{E}/\mathcal{W}) \stackrel{s_Z}{\longrightarrow}
\pi^*\mathcal{L}
\]
where \(s_Z\) is the equation of the \emph{residual subscheme}
\(j \colon Z \to Q\) to \(\PP\mathcal{W}\) in \(Q\): this is an effective
Cartier divisor of relative degree \(1\) over \(S\) away from the corank
\(2\) locus \(S_2\), and is otherwise the entire fibre. With this notation, the
\(d\)-th spinor sheaf associated with
\(\mathcal{W}\) is:
\[
\mathcal{S}_d \cong
\begin{dcases*}
\mathrlap{i_*}\phantom{j_*}\sO_{\PP\mathcal{W}}
\otimes \rho^*(\det\mathcal{E} \otimes \mathcal{L}^{\otimes k-1}) & if \(d = 2k\), and \\
j_*\sO_{\mathrlap{Z}\phantom{\PP\mathcal{W}}}
\otimes \rho^*(\mathcal{W} \otimes \mathcal{L}^{\otimes k}) & if \(d = 2k+1\).
\end{dcases*}
\]

\begin{proof}
Since \(\mathcal{W}\) is an isotropic line subbundle, the associated Clifford
ideal \(\mathcal{I}\) satisfies \(\mathcal{I}_1 = \mathcal{W}\) and
\(\mathcal{I}_2 \cong (\mathcal{E}/\mathcal{W}) \otimes \mathcal{W} \cong \det\mathcal{E}\).
Therefore the degree \(2\) spinor sheaf \(\mathcal{S}_2\) is the cokernel of
the map
\[
\phi_2 \colon
\sO_\pi(-1) \otimes \pi^*\mathcal{W} \to
\pi^*\det\mathcal{E}
\]
induced by Clifford multiplication. It is straightforward to see, via a
local computation for instance, that this is the map \(s_{\PP\mathcal{W}}\), so
\(\mathcal{S}_2 \cong i_*\sO_{\PP\mathcal{W}} \otimes \rho^*\det\mathcal{E}\).
Similarly, \(\mathcal{S}_1\) is the cokernel of
\[
\sO_\pi(-1) \otimes \pi^*\big((\mathcal{E}/\mathcal{W}) \otimes \mathcal{W} \otimes \mathcal{L}^\vee\big) \to
\pi^*\mathcal{W},
\]
which may be identified as \(s_Z\). Therefore
\(\mathcal{S}_1 \cong j_*\sO_Z \otimes \rho^*\mathcal{W}\). The
remaining degrees then follow from the degree shift relation from
\parref{cliff-spinors-duality}.
\end{proof}

\subsectiondash{Spinors in relative dimension \texorpdfstring{\(2\)}{2}}
\label{reduction-quadric-surfaces-example}
Let \(\rho \colon Q \to S\) be a quadric surface bundle. If \(\mathcal{W}\)
is a rank \(2\) isotropic subbundle of \(\mathcal{E}\), then its dual spinor
sheaf of degree \(d\) is
\[
\mathcal{S}_d^\vee \cong \begin{dcases*}
\mathcal{I}_{\PP\mathcal{W}/Q} \otimes \rho^*(\det\mathcal{W}^\vee \otimes \mathcal{L}^{\otimes -k+1}) & if \(d = 2k\), and \\
\sHom_{\sO_Q}(\mathcal{I}_{\PP\mathcal{W}/Q},\sO_Q) \otimes \sO_\rho(-1) \otimes \rho^*(\det\mathcal{E}^\vee \otimes \mathcal{L}^{\otimes -k+1}) & if \(d = 2k+1\).
\end{dcases*}
\]

\begin{proof}
It suffices to treat the case \(d = 2k\) is even, as the odd case then follows
from the duality relations of \parref{cliff-spinors-duality}. Since the
pushforward of \(\mathcal{S}_{2k}\) along \(\rho \colon Q \to S\) is
the \(2k\)-th Clifford ideal \(\mathcal{I}_{2k}\), taking the canonical line
subbundle \(\mathcal{L}^{\otimes k-1}\) inside
\(\Cl_{2k-2}(\mathcal{E},q) \cong \Cl_0(\mathcal{E},q) \otimes \mathcal{L}^{\otimes k-1}\)
in the presentation of \(\mathcal{I}_{2k}\) from \parref{cliff-ideals},
evaluating sections along \(\rho\), and dualizing provides a canonical map
\[
\sigma \colon
\mathcal{S}_{2k}^\vee \to
\rho^*(\det\mathcal{W}^\vee \otimes \mathcal{L}^{\otimes -k+1}).
\]
It suffices to show that \(\sigma\) is injective and that its cokernel is
locally isomorphic to the structure sheaf of \(\PP\mathcal{W}\), reducing the
problem to a local one on \(S\).

Shrinking \(S\) then reduces the problem to the case where all of
\(\mathcal{E}\), \(\mathcal{L}\), and \(\mathcal{W}\) are trivial
\(\sO_S\)-modules; in particular, since \(\mathcal{L}\) is trivial, only the
parity of the degree matters in Clifford algebra considerations. Choose
global projective coordinates \((x_0:x_1:x_2:x_3)\)
on \(\PP\mathcal{E} \cong \PP^3_S\) so that \(\PP\mathcal{W} = \mathrm{V}(x_0,x_1)\) and
\[
Q = \big\{(x_0:x_1:x_2:x_3) \in \PP^3_S : x_0L_0 + x_1L_1 = 0\big\}
\;\text{where}\;
L_0, L_1 \in \Gamma(\PP^3_S, \sO_{\PP^3_S}(1)),
\]
and such that \(L_1\) does not contain the variable \(x_0\).
Writing \(e_i\) for the basis vector of \(\mathcal{E}\) dual to the coordinate
\(x_i\), the odd and even Clifford ideals associated with \(\mathcal{W}\) have
bases
\[
\mathcal{I}_- \cong
(\sO_S \cdot e_0 e_2 e_3)
\oplus
(\sO_S \cdot e_1 e_2 e_3) \;\;\text{and}\;\;
\mathcal{I}_+ \cong
(\sO_S \cdot e_2 e_3)
\oplus
(\sO_S \cdot e_0 e_1 e_2 e_3)
\]
so that \(\det\mathcal{W}\) includes as the first summand of the even Clifford
ideal. A direct computation using the presentation from \parref{cliff-spinors}
identifies the even spinor sheaf \(\mathcal{S}_+\) as the cokernel of the map
\[
\begin{pmatrix}
L_0 & L_1 \\ -x_1 & x_0
\end{pmatrix} \colon
\sO_Q(-1) \otimes \rho^*\mathcal{I}_- \longrightarrow
\rho^*\mathcal{I}_+.
\]
The map dual to \(\sigma\) is the inclusion
\(\rho^*\det\mathcal{W} \to \rho^*\mathcal{I}_+\) of the first summand followed
by the projection \(\rho^*\mathcal{I}_+ \to \mathcal{S}_+\) onto the quotient.
From this, it follows that \(\sigma^\vee\) vanishes precisely at points where
the bottom row of the above matrix vanishes, identifying its cokernel as
\(j_*\sO_{\PP\mathcal{W}}\). The kernel of \(\sigma^\vee\) may be identified
as the submodule of \(\sO_Q\) spanned by local sections of the form
\(aL_0 + bL_1\) where \(a,b \in \sO_Q\) satisfy \(-ax_1 + bx_0 = 0\). The
relation implies that the kernel is both \(x_0\)- and \(x_1\)-torsion, and
therefore must be zero. Thus there is a right exact sequence
\[
\rho^*\det\mathcal{W} \stackrel{\sigma^\vee}{\to}
\mathcal{S}_+ \to
j_*\sO_{\PP\mathcal{W}} \to
0.
\]
The resolutions in \parref{cliff-spinors-duality} imply that spinors do not
have higher \(\mathcal{E}\mathit{xt}\)-sheaves, so the dual exact sequence
takes the form
\[
0 \to
\mathcal{S}_+^\vee \stackrel{\sigma}{\to}
\rho^*\det\mathcal{W}^\vee \to
\mathcal{E}\mathit{xt}^1_{\sO_Q}(j_*\sO_{\PP\mathcal{W}}, \sO_Q) \to
0.
\]
Grothendieck duality together with transitivity of relative dualizing sheaves
along the inclusions \(\PP\mathcal{W} \to Q \to \PP\mathcal{E}\) implies that
the last term is \(j_*\sO_{\PP\mathcal{W}}\), from which the conclusion
follows.
\end{proof}

\subsectiondash{Dependence on the subbundle}\label{cliff-dependence}
The \emph{special Clifford group scheme} \(\mathbf{S\Gamma}(\mathcal{E},q)\) of
the quadratic form \(q \colon \mathcal{E} \to \mathcal{L}\) is the group scheme
over \(S\) whose \(T\)-points are
\[
\mathbf{S\Gamma}(\mathcal{E},q)(T) \coloneqq
\{
u \in \Cl_0(\mathcal{E}_T,q_T)^\times :
u \cdot \mathcal{E}_T \cdot u^{-1} \subseteq \mathcal{E}_T
\},
\]
consisting of units in the base change \(\Cl_0(\mathcal{E}_T,q_T)\) of the
\(0\)-th Clifford algebra to \(T\) such that conjugation on
\(\Cl_1(\mathcal{E}_T,q_T)\) preserves the subbundle \(\mathcal{E}_T\). This
preserves the quadratic form and fits into an exact sequence of group schemes
over \(S\)
\[
1 \to
\mathbf{G}_m \to
\mathbf{S\Gamma}(\mathcal{E},q) \to
\mathbf{SO}(\mathcal{E},q)
\]
which is furthermore exact on the right when
\(q \colon \mathcal{E} \to \mathcal{L}\) is regular: see \cite[IV.8.2.3]{Knus}.

Suppose that the conjugation action of a section \(u\) of
\(\mathbf{S\Gamma}(\mathcal{E},q)\) takes an isotropic subbundle
\(\mathcal{W} \subseteq \mathcal{E}\) to \(\mathcal{V}\). Letting the special
Clifford group act on the right,
\[
\mathcal{I}^{\mathcal{W}} \cdot u^{-1} =
\Cl(\mathcal{E},q) \cdot (u \cdot \det\mathcal{W} \cdot u^{-1}) =
\Cl(\mathcal{E},q) \cdot \det\mathcal{V} =
\mathcal{I}^{\mathcal{V}}.
\]
Thus the action of \(u\) provides an isomorphism
\(\mathcal{I}^{\mathcal{W}} \cong \mathcal{I}^{\mathcal{V}}\) of
\(\Cl(\mathcal{E},q)\)-modules, whence an isomorphism
\(\mathcal{S}^{\mathcal{W}} \cong \mathcal{S}^{\mathcal{V}}\)
of the corresponding spinor sheaves.

As an application of this principle which will be useful in
\S\parref{section-spinor-moduli}, the following shows that, over an
algebraically closed field, once the quadric is at least \(4\)-dimensional and
its singular locus is at most \(1\)-dimensional, then a given spinor sheaf is
isomorphic to the spinor sheaf associated with a linear space containing any
given smooth point:

\begin{Lemma}
\label{cliff-dependence-corank-2}
Let $Q \subset \PP V$ be a quadric of dimension \(n-1 \geq 4\) and corank
\(\leq 2\) over an algebraically closed field $\kk$. Write \(n\) as
\(2\ell + 1\) or \(2\ell+2\). Given any smooth point \(x \in Q\) and any
\(\ell\)-plane \(\PP W \subset Q\), there exists an \(\ell\)-plane
\(\PP W' \subset Q\) containing \(x\) whose associated spinor sheaf is
isomorphic to that of \(\PP W\).
\end{Lemma}

\begin{proof}
If \(x \in \PP W\), then there is nothing to show. So suppose that
\(x \notin \PP W\). Then it suffices to show that there is
$x' \in \mathbf{P}W$ which is a smooth point of $Q$ in the orbit of \(x\) under
the special Clifford group: If $g\cdot x = x'$ for \(g \in \mathbf{S\Gamma}(V,q)\),
then $\PP W' \coloneqq g^{-1} \cdot \mathbf{P}W \subset Q$ is an \(\ell\)-plane
containing $x$ whose spinor sheaf is isomorphic to that of $\mathbf{P}W$ by
\parref{cliff-dependence}. Construct \(x'\) in three steps:

First, there exists a point $x' \in \mathbf{P}W$ which is a smooth point of $Q$
and such that the line in $\mathbf{P}V$ spanned by $x$ and $x'$ is disjoint
from the singular locus of $Q$: Writing \(v \in V\) for a basis vector of
the linear space underlying \(x\), this means there exists \(v' \in W\)
not contained in \(\kk \cdot v + \rad q\); in other words, it suffices to
show that
\[
(\kk \cdot v + \rad q ) \cap W \subsetneq W.
\]
Since \(\dim_\kk(\kk \cdot v + \rad q) \leq 3\) and
\(\dim_\kk W = \ell + 1 \geq 3\), the only way equality can occur is if
\(\kk \cdot v + \rad q = W\). This is impossible since $v \not\in W$. 
Now fix any such \(x' = \PP v' \in \PP W\).

Second, there is a \(4\)-dimensional subspace $U \subset V$ such that
$q\rvert_{U}$ is non-degenerate and $v, v' \in U$. Set
\(U_1 \coloneqq \kk \cdot v \oplus \kk \cdot v' \subset V\). There are then two
cases depending on the pairing between \(v\) and \(v'\): If $b_q(v,v') = 0$,
then $U_1$ is isotropic and intersects $\rad q$ trivially, so by the theory of
hyperbolic planes, there is a subspace $U_2 \subset V$ mapped isomorphically to
$U_1^\vee$ by $b_q$. If $b_q(v,v') \neq 0$, then $U_1$ is itself a
hyperbolic plane, so we have a splitting $V = U_1 \oplus U_1^\perp$. The
orthogonal $U_1^\perp$ has dimension $n-1 \geq 4$ and carries a quadratic
form of corank $\leq 2$, so it has an isotropic vector not contained in its
radical. This isotropic vector can be completed to a second hyperbolic plane
$U_2$. In either case, $U \coloneqq U_1 \oplus U_2 \subset V$ is a hyperbolic
subspace of dimension \(4\) containing \(v\) and \(v'\), as desired.

Finally, to conclude, since $q\rvert_{U}$ is non-degenerate, there is an
orthogonal decomposition $V = U \oplus U^\perp$ and a corresponding inclusion
of special Clifford groups
\[
\mathbf{S\Gamma}(U, q\rvert_{U}) \times_{\kk}
\mathbf{S\Gamma}(U^\perp,q\rvert_{U^\perp}) \subset
\mathbf{S\Gamma}(V,q).
\]
Since $U$ is nonsingular, the morphism
\(\mathbf{S\Gamma}(U,q\rvert_U) \to \mathbf{SO}(U, q\rvert_U)\) is surjective, as in
\parref{cliff-dependence}. Since \(\dim_\kk U \geq 2\), the special orthogonal
group acts transitively on its isotropic lines; this implies that there exists
an element $g' \in \mathbf{S\Gamma}(U,q\rvert_U)$ taking
\(\kk \cdot v\) to \(\kk \cdot v'\). Then the element
\(g \coloneqq (g',1) \in \mathbf{S\Gamma}(V,q)\) has the sought-after properties.  
\end{proof}

\subsectiondash{Hyperplane sections}\label{cliff-subbundle-situation}
The next few paragraphs describe the relationship between spinor sheaves on a
quadric and a hyperplane section when the hyperplane contains the defining
subspace. Let \(\mathcal{E}' \subset \mathcal{E}\) be a corank \(1\) subbundle
and let \(q' \colon \mathcal{E}' \to \mathcal{L}\) be the restriction of \(q\)
thereon. The Clifford algebra of \(\mathcal{E}'\) is then a graded subalgebra of the
Clifford algebra of \(\mathcal{E}\), and the quotient can be identified as
follows: Writing \(\mathcal{E}'' \coloneqq \mathcal{E}/\mathcal{E}'\) for the
quotient line bundle and \([-1]\) for shift of grading, there is a short exact
sequence of graded \(\Cl(\mathcal{E}',q')\)-modules:
\[
0 \to
\Cl(\mathcal{E}',q') \to
\Cl(\mathcal{E},q) \to
\mathcal{E}'' \otimes_{\sO_S} \Cl(\mathcal{E}',q')[-1] \to
0.
\]

\begin{proof}
Clifford multiplication gives a map 
\(\mathcal{E} \otimes_{\sO_S} \Cl(\mathcal{E}', q')[-1] \to \Cl(\mathcal{E}, q)/\Cl(\mathcal{E}', q')\)
which vanishes on \(\mathcal{E}' \otimes_{\sO_S} \Cl(\mathcal{E}', q')[-1]\)
and therefore factors through a map 
\[
\mathcal{E}'' \otimes_{\sO_S} \Cl(\mathcal{E}', q')[-1] \to \Cl(\mathcal{E}, q)/\Cl(\mathcal{E}', q').
\]
The degree $d$ part of each side is a locally free $\mathcal{O}_S$-module of
rank $2^{n-1}$ and easy to see that a local basis on the left is sent
to one on the right, so the map is an isomorphism.
\end{proof}

More generally, \(\mathcal{W} \subset \mathcal{E}'\) be an isotropic subbundle,
and write \(\mathcal{I}\) and \(\mathcal{I}'\) for the associated Clifford
ideals in \(\Cl(\mathcal{E},q)\) and \(\Cl(\mathcal{E}',q')\), respectively.
Since both ideals are generated by \(\det\mathcal{W}\), \(\mathcal{I}'\) may be
identified as a sub-\(\Cl(\mathcal{E}',q')\)-module of \(\mathcal{I}\).
The argument above generalizes to yield:

\begin{Lemma}\label{cliff-subbundle-ideals}
In the setting of \parref{cliff-subbundle-situation}, inclusion induces an
exact sequence graded of \(\Cl(\mathcal{E}',q')\)-modules
\[
\pushQED{\qed}
0 \to
\mathcal{I}' \to
\mathcal{I} \to
\mathcal{E}'' \otimes_{\sO_S} \mathcal{I}'[-1] \to
0.
\qedhere
\popQED
\]
\end{Lemma}

%\begin{proof}
%Again, multiplication in the $\mathcal{O}_S$-algebra $\Cl(\mathcal{E}, q)$ determines a map $\mathcal{I}'[-1] \otimes \mathcal{E} \to \mathcal{I}/\mathcal{I}', x \otimes s \mapsto sx$, which factors through a map $\mathcal{I}'[-1] \otimes \mathcal{E}'' \to \mathcal{I}/\mathcal{I}'$. One checks locally that this is an isomorphism as in the special case above.
%
%Let \(r\) be the rank of \(\mathcal{W}\). Then the exact sequence in
%\parref{cliff-subbundle-situation} provides an exact sequence
%\[
%0 \to
%\Cl_{d-r}(\mathcal{E}',q') \otimes \det\mathcal{W} \to
%\Cl_{d-r}(\mathcal{E},q) \otimes \det\mathcal{W} \to
%\Cl_{d-r-1}(\mathcal{E}',q') \otimes \mathcal{E}'' \otimes \det\mathcal{W} \to
%0.
%\]
%By definition of Clifford ideals in \parref{cliff-ideals}, Clifford
%multiplication maps the terms of the sequence surjectively onto
%\(\mathcal{I}_d'\), \(\mathcal{I}_d\), and
%\(\mathcal{I}_{d-1}' \otimes \mathcal{E}''\), respectively. Since the maps in
%this sequence are \(\Cl_0(\mathcal{E}',q')\)-module homomorphisms, they commute
%with multiplication, and so they induce maps \(\mathcal{I}_d' \to
%\mathcal{I}_d\) and \(\mathcal{I}_d \to \mathcal{I}_{d-1}' \otimes
%\mathcal{E}''\) which fit into an exact sequence, as required.
%\end{proof}

The sequence of Clifford ideals induces a corresponding sequence of spinor
sheaves. To set notation, write \(\pi' \colon \PP\mathcal{E}' \to S\) and
\(\rho' \colon Q' \to S\) for the projective and quadric bundles associated
with \(q' \colon \mathcal{E}' \to \mathcal{L}\), and let \(\mathcal{S}\) and
\(\mathcal{S}'\) be the spinor sheaves corresponding to \(\mathcal{W}\) on
\(Q\) and \(Q'\), respectively. Then:

\begin{Lemma}\label{cliff-subbundle-spinors}
In the setting of \parref{cliff-subbundle-situation}, there is an exact
sequence
\[
0 \to
\mathcal{S}_d' \to
\mathcal{S}_d\rvert_{Q'} \to
\mathcal{E}'' \otimes \mathcal{S}_{d-1}' \to
0.
\]
\end{Lemma}

\begin{proof}
This can be deduced from \parref{cliff-subbundle-ideals} via the following
general construction: For any quasi-coherent graded
\(\Cl(\mathcal{E}',q')\)-module \(\mathcal{M}\), Clifford multiplication
as in \parref{cliff-spinors} provides a short exact sequence
\[
0 \to
\sO_{\pi'}(-1) \otimes \pi'^*\mathcal{M}_{-1} \to
\pi'^*\mathcal{M}_0 \to
\iota_*\mathcal{F} \to
0
\]
where \(\mathcal{F}\) is a quasi-coherent sheaf on $Q'$ and
\(\iota \colon Q' \to \mathbf{P}\mathcal{E}'\) is the inclusion, and the
assignment $\mathcal{M} \mapsto \mathcal{F}$ determines an exact functor from
the category of quasi-coherent graded $\Cl(\mathcal{E}', q')$-modules to the
category of quasi-coherent modules on $Q'$.
%
%
%Let \(\iota' \colon Q' \to \PP\mathcal{E}'\) be the inclusion. Restricting the
%defining short exact sequence of \(\mathcal{S}_d\) to \(\PP\mathcal{E}'\)
%yields an exact sequence of \(\pi'^*\Cl_0(\mathcal{E}',q')\)-modules
%\[
%\sO_{\pi'}(-1) \otimes \pi'^*\mathcal{I}_{d-1} \to
%\pi'^*\mathcal{I}_d \to
%\iota_*'\mathcal{S}_d\rvert_{Q'} \to 0.
%\]
%Since twofold Clifford multiplication on \(\pi'^*\mathcal{I}\) by vectors from
%\(\mathcal{E}'\) still provides a matrix factorization of \(q'\), the first map
%remains injective. Comparing with the presentation of
%\(\mathcal{S}'_d\) and combining with \parref{cliff-subbundle-ideals} then
%gives the result.
\end{proof}

Going the other direction, the spinor sheaf \(\mathcal{S}\) on
\(Q\) is related to the Clifford ideal \(\rho^*\mathcal{I}'\) via modification
by \(\mathcal{S}'\) along the closed subscheme \(j' \colon Q' \to Q\). This
relation comes from considering the quotient of the Clifford ideal
\(\rho^*\mathcal{I}_d\) by the span of the Clifford ideal \(\rho^*\mathcal{I}_d'\)
for the hyperplane section, and the twisted spinor sheaf
\(\sO_\rho(-1) \otimes \mathcal{S}_{d-1}\) identified as the image of
the Clifford multiplication map \(\phi_d\) from the first complex in
\parref{cliff-spinors-duality}:

\begin{Proposition}\label{cliff-subbundle-modify}
In the setting of \parref{cliff-subbundle-situation}, assume additionally
that \(j' \colon Q' \to Q\) is the inclusion of an effective Cartier divisor.
Then \(\rho^*\mathcal{I}'_d \cap \sO_\rho(-1) \otimes \mathcal{S}_{d-1} = 0\)
as submodules of \(\rho^*\mathcal{I}_d\) and
\[
\rho^*\mathcal{I}_d/(\rho^*\mathcal{I}'_d \oplus \sO_\rho(-1) \otimes \mathcal{S}_{d-1})
\cong \rho^*\mathcal{E}'' \otimes j'_*\mathcal{S}_{d-1}.
\]
In particular, there is a short exact sequence
\[
0 \to
\sO_\rho(-1) \otimes \mathcal{S}_{d-1} \to
\rho^*(\mathcal{E}'' \otimes \mathcal{I}_{d-1}') \to
\rho^*\mathcal{E}'' \otimes j'_*\mathcal{S}'_{d-1} \to
0.
\]
\end{Proposition}

\begin{proof}
We first show that \(\sO_\rho(-1) \otimes \mathcal{S}_{d-1}\), viewed as the
image of
\(\phi_d \colon \sO_\rho(-1) \otimes \rho^*\mathcal{I}_{d-1} \to \rho^*\mathcal{I}_d\)
as above, intersects \(\rho^*\mathcal{I}_d'\) trivially away on
\(Q \setminus Q'\). There, the inclusion \(\mathcal{E}' \subset \mathcal{E}\)
is split by the composite
\(\sO_\rho(-1) \hookrightarrow \rho^*\mathcal{E} \to \rho^*\mathcal{E}''\),
and this induces a corresponding splitting of the sequence from
\parref{cliff-subbundle-ideals}:
\[
\rho^*\mathcal{I}|_{Q \setminus Q'} \cong
\rho^*\mathcal{I}'\rvert_{Q \setminus Q'} \oplus
\sO_\rho(-1) \otimes \rho^*\mathcal{I}'[-1]\rvert_{Q \setminus Q'}.
\]
With respect to this splitting, \(\phi_d\rvert_{Q \setminus Q'}\) may be identified as the map
\[
(\varphi_1, \varphi_2) \colon
\sO_\rho(-1) \otimes \rho^*\mathcal{I}_{d-1}'\rvert_{Q \setminus Q'} \oplus
\sO_\rho(-2) \otimes \rho^*\mathcal{I}_{d-2}'\rvert_{Q \setminus Q'} \longrightarrow
\rho^*\mathcal{I}_d\rvert_{Q \setminus Q'}
\]
where \(\varphi_i\) is induced by \(i\)-fold Clifford multiplication. Now
\(\varphi_2 = 0\) since it is multiplication by the equation of \(Q\) as in
\parref{cliff-spinors}, and \(\varphi_1\) is the inclusion of a direct
complement of \(\rho^*\mathcal{I}_d'\rvert_{Q \setminus Q'}\) since any local section
contains a local section in \(\mathcal{E} \setminus \mathcal{E}'\). This
shows that \(\rho^*\mathcal{I}_d \cap \sO_\rho(-1) \otimes \mathcal{S}_{d-1}\)
is zero upon restriction to $Q \setminus Q'$. Since this is a
submodule of a locally free $\sO_Q$-module, and the complement $Q \setminus Q'$
of an effective Cartier divisor is scheme-theoretically dense in $Q$,
in fact
 $\rho^*\mathcal{I}_d \cap \mathcal{O}_\rho(-1) \otimes \mathcal{S}_{d-1}= 0$.
%For the first statement, let \(x \in Q\) correspond to an %isotropic element
%\(v\) in the stalk \(\mathcal{E}_{\rho(x)}\). In the stalk of
%\(\rho^*\mathcal{I}_d = \rho^*\Cl_{d-r}(\mathcal{E},q) \cdot \det\mathcal{W}\)
%at \(x\), identifying \(\sO_\rho(-1) \otimes \mathcal{S}_{d-1}\) as the image
%of the map \(\phi_d\) from the first complex of \parref{cliff-spinors-duality}
%shows that it may be written as the span of elements of the form
%\[
%v \cdot (v \cdot \xi + \eta) \cdot \det\mathcal{W}
%= (q(v) \cdot \xi + v \cdot \eta) \cdot \det\mathcal{W}
%= v \cdot \eta \cdot \det\mathcal{W}
%\]
%where \(\xi \in \rho^*\Cl_{d-r-1}(\mathcal{E},q)_x\) and
%\(\eta \in \rho^*\Cl_{d-r-2}(\mathcal{E},q)_x\) are elements not containing
%\(v\). All such elements are divisible by \(v\), so when
%\(x \in Q \setminus Q'\), meaning
%\(v \in (\mathcal{E}\setminus \mathcal{E}')_{\rho(x)}\), their span
%is linearly disjoint from the stalk of
%\(\rho^*\mathcal{I}_d' = \rho^*\Cl_{d-r}(\mathcal{E}',q') \cdot \det\mathcal{W}\)
%at \(x\). Since \(\rho^*\mathcal{I}_d\) is a locally free \(\sO_Q\)-module, it
%has no submodules set-theoretically supported on the effective Cartier divisor
%\(Q'\), and so this implies linear disjointness of the modules
%\(\rho^*\mathcal{I}_d'\) and \(\sO_\rho(-1) \otimes \mathcal{S}_{d-1}\).

The quotient of \(\rho^*\mathcal{I}_d\) by the span of these two submodules may
be identified as the cokernel of
\[
\psi \colon \sO_\rho(-1) \otimes \mathcal{S}_{d-1}
\to \rho^*\mathcal{I}_d
\to \rho^*\mathcal{I}_d/\rho^*\mathcal{I}_d'
\cong \rho^*(\mathcal{E}'' \otimes \mathcal{I}_{d-1}')
\]
where the rightmost identification is from \parref{cliff-subbundle-ideals}.
%The argument so far already shows that \(\coker\psi\) is supported on some
%thickening of \(Q'\); in fact,
We claim it is scheme-theoretically supported on \(Q'\). This is a
statement local on \(S\), so assume that \(\mathcal{L} \cong \sO_S\) and that
\[
\mathcal{E}
\cong \sO_S \cdot e_0 \oplus \mathcal{E}'
\cong \big(\bigoplus\nolimits_{i = 0}^{n-r-1} \sO_S \cdot e_i\big) \oplus \mathcal{W}
\cong \bigoplus\nolimits_{i = 0}^n \sO_S \cdot e_i.
\]
With these trivializations,
\(\rho^*(\mathcal{E}'' \otimes \mathcal{I}_{d-1}') \cong e_0 \cdot \rho^*\Cl_{d-r-1}(\mathcal{E}',q') \cdot \det\mathcal{W}\)
and, writing \(x_i\) for the coordinates of \(\PP\mathcal{E} \cong \PP^n\)
dual to the basis \(e_i\), the image of \(\psi\) may be identified with the
submodule
\[
\big(x_0e_0 \cdot \rho^*\Cl_{d-r-1}(\mathcal{E}',q') \cdot \det\mathcal{W}\big) +
\big(e_0 \cdot (x_1e_1 + \cdots + x_ne_n) \cdot \rho^*\Cl_{d-r-2}(\mathcal{E}',q') \cdot \det\mathcal{W}\big).
\]
The first summand shows that \(\coker\psi\) is supported on the vanishing
locus \(Q' = Q \cap \PP\mathcal{E}'\) of \(x_0\).

Writing \(\coker\psi \cong  j'_*\mathcal{F}\), it remains to show that
\(\mathcal{F} \cong \rho'^*\mathcal{E}'' \otimes \mathcal{S}_{d-1}'\). We have
$\mathcal{F} = \coker j^*\psi = \coker \alpha$ where $\alpha$ is as the
diagonal map in
%This
%can be deduced locally on \(S\) by examining the second summand above, which
%shows that the restriction of \(\psi\) to \(Q'\) factors through the map
%\(\phi_{d-1}'\). More globally, write \(\mathcal{F} \cong \coker j^*\psi\).
%Since \(\psi\) is induced from \(\phi_d\), the diagram
\[
\begin{tikzcd}
0 \rar
& \rho'^*\mathcal{I}_d' \rar
& \rho'^*\mathcal{I}_d \rar
& \rho'^*(\mathcal{E}'' \otimes \mathcal{I}_{d-1}') \rar
& 0 \\
0 \rar
& \sO_{\rho'}(-1) \otimes \rho'^*\mathcal{I}_{d-1}'  \rar \uar["\phi_d'"]
& \sO_{\rho'}(-1) \otimes \rho'^*\mathcal{I}_{d-1} \ar[ur, "\alpha"] \rar \uar["\phi_d\rvert_{Q'}"]
& \sO_{\rho'}(-1) \otimes \rho'^*(\mathcal{E}'' \otimes \mathcal{I}_{d-2}') \rar \uar["\phi_{d-1}'"]
& 0
\end{tikzcd}
\]
where the rows are exact and the squares are commutative by
\parref{cliff-subbundle-ideals}. We see that the cokernel of $\alpha$ is equal to the cokernel of $\phi'_{d-1}$ which is equal to $\rho^*\mathcal{E}'' \otimes j'\mathcal{S}_{d-1}'$.
%implies that \(\mathcal{F}\) is the cokernel
%of the rightmost vertical map, which is \(\rho'^*\mathcal{E}'' \otimes \mathcal{S}_{d-1}'\)
%as required.
\end{proof}

\subsectiondash{Cones}\label{cliff-cone-situation}
The final few statements in this section describe the behaviour of spinor
sheaves along cones over quadrics. Let \(\mathcal{K} \subset \mathcal{E}\) be a
rank \(c\) subbundle contained in the radical of \(q\), set
\(\bar{\mathcal{E}} \coloneqq \mathcal{E}/\mathcal{K}\), and let
\(\bar{q} \colon \bar{\mathcal{E}} \to \mathcal{L}\) be the induced quadratic
form. The quotient map \(\mathcal{E} \to \bar{\mathcal{E}}\) induces a
short exact sequence of graded \(\Cl(\mathcal{E},q)\)-modules
\[
0 \to
\mathcal{J} \to
\Cl(\mathcal{E},q) \to
\Cl(\bar{\mathcal{E}},\bar{q}) \to
0
\]
where the kernel is the sheaf of ideals generated by \(\mathcal{K}\).
Since \(\mathcal{K}\) lies in the radical of \(q\), it is central in
\(\Cl(\mathcal{E},q)\), and so \(\mathcal{J}\) is a two-sided ideal. The
quotient map endows each \(\Cl(\bar{\mathcal{E}},\bar{q})\)-module with the
structure of a \(\Cl(\mathcal{E},q)\)-module.

Let \(\mathcal{W} \subset \mathcal{E}\) be an isotropic subbundle of rank
\(r\) containing \(\mathcal{K}\), and let
\(\widebar{\mathcal{W}} \subset \widebar{\mathcal{E}}\) be the corresponding
isotropic bundle in the quotient. The Clifford ideals \(\mathcal{I}\) and
\(\widebar{\mathcal{I}}\) associated with \(\mathcal{W}\) and
\(\widebar{\mathcal{W}}\), respectively, are related as follows:

\begin{Lemma}\label{cliff-cone-ideals}
\(\widebar{\mathcal{I}}_d \cong \mathcal{I}_{d+c} \otimes_{\sO_S} \det\mathcal{K}^\vee\)
as \(\Cl_0(\mathcal{E},q)\)-modules
for each \(d \in \mathbf{Z}\).
\end{Lemma}

\begin{proof}
The annihilator in \(\Cl(\mathcal{E},q)\) of the generator
\(\det\widebar{\mathcal{W}}\) of \(\widebar{\mathcal{I}}\) is the preimage of
the sheaf of left ideals in \(\Cl(\bar{\mathcal{E}},\bar{q})\) generated by
\(\widebar{\mathcal{W}}\). Since
\(\widebar{\mathcal{W}} = \mathcal{W}/\mathcal{K}\), and \(\mathcal{K}\)
is central in \(\Cl(\mathcal{E},q)\), this is the sheaf of left ideals in
\(\Cl(\mathcal{E},q)\) generated by \(\mathcal{W}\). So Clifford multiplication
gives rise to a presentation
\[
\Cl_{d-r+c-1}(\mathcal{E},q) \otimes \mathcal{W} \otimes \det\widebar{\mathcal{W}} \to
\Cl_{d-r+c}(\mathcal{E},q) \otimes \det\widebar{\mathcal{W}} \to
\bar{\mathcal{I}}_d \to
0.
\]
Since
\(\det\widebar{\mathcal{W}} \cong \det\mathcal{W} \otimes \det\mathcal{K}^\vee\),
this is also the presentation for
\(\mathcal{I}_{d+c} \otimes \det\mathcal{K}^\vee\).
\end{proof}

As usual, linear projection centred along \(\PP\mathcal{K}\) results in a
commutative diagram
\[
\begin{tikzcd}[column sep=1em, row sep=1.5em]
&&& \widetilde{Q} \ar[d,"\tilde{\iota}", hook] \ar[ddlll, "b_Q"', bend right=30] \ar[ddrrr, "a_{\bar{Q}}", bend left=30] \\
&&& \widetilde{\PP}\mathcal{E} \ar[dl,"b"'] \ar[dr,"a"] \\
Q \ar[rr,"\iota", hook] \ar[rrrd, "\rho"', bend right=24] &&
\PP\mathcal{E} \ar[dr,"\pi"'] \ar[rr,dashed]
&
& \PP\bar{\mathcal{E}} \ar[dl,"\bar\pi"]
&& \bar{Q} \ar[ll,"\bar{\iota}"', hook'] \ar[llld, "\bar{\rho}", bend left=24] \\
&&& S
\end{tikzcd}
\]
where \(b \colon \widetilde{\PP}\mathcal{E}\) and
\(b_Q \colon \widetilde{Q} \to Q\) are the blowups of \(\PP\mathcal{E}\) and
\(Q\) along \(\PP\mathcal{K}\), respectively, and
\(a \colon \widetilde{\PP}\mathcal{E} \to \PP\bar{\mathcal{E}}\) is the morphism
resolving linear projection. The blowup \(\widetilde{Q}\) may be identified as
the restriction of linear projection
\(a \colon \widetilde{\PP}\mathcal{E} \to \PP\bar{\mathcal{E}}\) over the base
of the cone \(\bar{Q}\); in other words, the top right square in the diagram is
Cartesian. The spinor sheaves \(\mathcal{S}\) and \(\bar{\mathcal{S}}\)
associated with \(\mathcal{W}\) and \(\widebar{\mathcal{W}}\), respectively,
are related as follows:

\begin{Proposition}\label{cliff-cone-spinors}
\(Rb_{Q,*}La_{\bar{Q}}^*\bar{\mathcal{S}}_d \cong \mathcal{S}_{d+c} \otimes \rho^*\det\mathcal{K}^\vee\)
in \(\Dqc(Q)\) for each \(d \in \mathbf{Z}\).
\end{Proposition}

\begin{proof}
View the defining presentation of \(\bar{\mathcal{S}}_d\) on
\(\PP\bar{\mathcal{E}}\) as an exact triangle in \(\Dqc(\PP\bar{\mathcal{E}})\)
of the form
\[
\sO_{\bar{\pi}}(-1) \otimes \bar{\pi}^*\bar{\mathcal{I}}_{d-1} \longrightarrow
\bar\pi^*\bar{\mathcal{I}}_d \longrightarrow
R\bar{\iota}_*\bar{\mathcal{S}}_d \stackrel{+1}{\longrightarrow}
\]
Apply \(Rb_*La^*\) to obtain a triangle in \(\Dqc(\PP\mathcal{E})\), the first
two terms of which are easy to identify: Commutativity of the
diagram, the projection formula, and \parref{cliff-cone-ideals} together give
the first line of
\begin{align*}
Rb_*La^*(\sO_{\bar{\pi}}(-i) \otimes \bar{\pi}^*\bar{\mathcal{I}}_{d-i})
& \cong Rb_* a^*\sO_{\bar{\pi}}(-i) \otimes \pi^*(\mathcal{I}_{d+c-i} \otimes \det\mathcal{K}^\vee) \\
& \cong \sO_{\pi}(-i) \otimes \pi^*(\mathcal{I}_{d+c-i} \otimes \det\mathcal{K}^\vee)
\;\text{for \(i = 0,1\)}.
\end{align*}
For the second line, note that
\(a^*\sO_{\bar{\pi}}(-1) \cong \sO_{\widetilde{\PP}\mathcal{E}}(E) \otimes b^*\sO_\pi(-1)\)
where \(E \subset \widetilde{\PP}\mathcal{E}\) is the exceptional divisor, and
so \(Rb_*a^*\sO_{\bar{\pi}}(-i) \cong \sO_{\pi}(-i)\) for \(i = 0, 1\). As for
the third term in the new triangle, observe that
\(a \colon \widetilde{\PP}\mathcal{E} \to \PP\bar{\mathcal{E}}\) is a
projective bundle, so it is tor independent with
\(\bar{\iota} \colon \bar{Q} \to \PP\bar{\mathcal{E}}\), and thus
\(La^* \circ R\bar{\iota}_* = R\tilde{\iota} \circ La_{\bar{Q}}^*\) as functors
\(\Dqc(\bar{Q}) \to \Dqc(\widetilde{\PP}\mathcal{E})\) by
\citeSP{0E23}. Commutativity of the top left square above then shows that
\(Rb_* \circ R\tilde{\iota} = R\iota_* \circ Rb_{Q,*}\). Therefore the new
triangle is
\[
\sO_\pi(-1) \otimes \pi^*(\mathcal{I}_{d+c-1} \otimes \det\mathcal{K}^\vee) \longrightarrow
\pi^*(\mathcal{I}_{d+c} \otimes \det\mathcal{K}^\vee) \longrightarrow
R\iota_*Rb_{Q,*}La_{\bar{Q}}^*\bar{\mathcal{S}}_d \stackrel{+1}{\longrightarrow}
\]
Comparing with the defining sequence for \(\mathcal{S}_{d+c}\) now shows that
\[
R\iota_*Rb_{Q,*}L_{\bar{Q}}^*\bar{\mathcal{S}}_d \cong
R\iota_*(\mathcal{S}_{d+c} \otimes \rho^*\det\mathcal{K}^\vee) \in
\Dqc(\PP\mathcal{E}).
\]
Since the closed immersion \(\iota \colon Q \to \PP\mathcal{E}\) is, in
particular, affine, this at least shows that
\(Rb_{Q,*}L_{\bar{Q}}^*\bar{\mathcal{S}}_d\) is a complex concentrated in degree
\(0\). Using that \(\iota_*\) is fully faithful on quasi-coherent modules, see
\citeSP{01QY}, the sheaf in degree \(0\) is identified
with \(\mathcal{S}_{d+c} \otimes \rho^*\deg\mathcal{K}^\vee\),
whence the result.
\end{proof}

\section{Kuznetsov components under hyperbolic reduction}
\label{section-hyperbolic-reduction}
Returning to the setting of \parref{hyperbolic-reduction-setting}, suppose
the quadric bundle \(\rho \colon Q \to S\) admits a regular section
corresponding to a rank \(1\) subbundle \(\mathcal{F} \subset \mathcal{E}\),
and let \(\rho' \colon Q' \to S\) be the associated hyperbolic reduction. Assume
throughout that \(S_{n-1}\) does not contain any weakly associated points of
\(S\), so that \(Q'\) is an effective Cartier divisor in
\(\PP(\mathcal{F}^\perp/\mathcal{F})\) by
\parref{lemma-cartier-divisor-criterion}. The aim in the first half of this
section is to relate the Kuznetsov components of \(Q\) and \(Q'\).
Specifically, a mutations argument gives in \parref{hypred-equivalence} an
\(S\)-linear equivalence \(\Phi \colon \Ku(Q') \to \Ku(Q)\). We are then able
to explicitly identify the kernel underlying \(\Phi\) in
\parref{proposition-identify-the-kernel}, with which we show in
\parref{cliff-hyperbolic-reduction-spinors} that
\(\Phi\) sends the dual spinor sheaves of \(Q'\) to those of \(Q\).

To begin, continue with the notation from the diagram in
\parref{quadrics-reduction-blowup}. By its construction together with
\parref{lemma-blowup-relation}, the scheme
\(\widetilde{\rho} \colon \widetilde{Q} \to S\) is a blowup in two ways, and fits
into a commutative diagram
\[
\begin{tikzcd}
 E \ar[d, "b_E"'] \ar[rr,hook,"\tilde{\imath}"] &&
 \widetilde{Q} \ar[dl,"b"'] \ar[dr,"a"] &&
 \ar[ll,hook',"\tilde{\jmath}"'] \Lambda \ar[d,"a_\Lambda"] \\
\PP\mathcal{F} \rar[hook] &
Q  &&
\PP\mathcal{E}' &
\lar[hook',"j'"'] Q'
\end{tikzcd}
\]
of schemes over \(S\). Orlov's blowup formula from
\cite[Theorem 4.3]{Orlov:Formulae}, see also \cite[Theorem 1.6]{Kuznetsov:SOD},
therefore gives two semiorthogonal decompositions of
\(\Dqc(\widetilde{Q})\): On the one hand, viewing \(\widetilde{Q}\) as a
blowup of \(Q\) along \(\PP\mathcal{F}\) yields a \(S\)-linear semiorthogonal
decomposition of \(\Dqc(\widetilde{Q})\) of the form
\begin{multline}\label{equation-sod-b}
\big\langle Lb^*\Ku(Q),\,
L\widetilde{\rho}^*\Dqc(S) \otimes^L \sO_{\widetilde{Q}},\,
\ldots,\,
L\widetilde{\rho}^*\Dqc(S) \otimes^L \sO_{\widetilde{Q}}((n-2)H), \\
L\widetilde{\rho}^*\Dqc(S) \otimes^L R\tilde{\imath}_*\sO_E,\,
\ldots,\,
L\widetilde{\rho}^*\Dqc(S) \otimes^L R\tilde{\imath}_*\sO_E(-(n-3)E)
\big\rangle
\tag{\(\mathrm{B}\)}
\end{multline}
where \(\sO_{\widetilde{Q}}(H) \coloneqq b^*\sO_\rho(1)\). On the other hand,
viewing \(\widetilde{Q}\) as a blowup of \(\PP\mathcal{E}'\) along
\(Q'\) and writing \(\sO_{\widetilde{Q}}(h) \coloneqq a^*\sO_{\rho'}(1)\)
yields the \(S\)-linear semiorthogonal decomposition
\begin{multline}\label{equation-sod-a}
\big\langle
R\tilde{\jmath}_*(La_\Lambda^*\Ku(Q') \otimes^L \sO_\Lambda(\Lambda)),\,
L\widetilde{\rho}^*\Dqc(S) \otimes^L R\tilde{\jmath}_*\sO_\Lambda(\Lambda),\,
\ldots,\,
L\widetilde{\rho}^*\Dqc(S) \otimes^L R\tilde{\jmath}_*\sO_\Lambda(\Lambda+ (n-4)h), \\
L\widetilde{\rho}^*\Dqc(S) \otimes^L \sO_{\widetilde{Q}},\,
\ldots,\,
L\widetilde{\rho}^*\Dqc(S) \otimes^L \sO_{\widetilde{Q}}((n-1)h)
\big\rangle.
\tag{\(\mathrm{A}\)}
\end{multline}
The following matches the two Kuznetsov components via a series of mutations:

\begin{Proposition}\label{hypred-equivalence}
As \(S\)-linear subcategories of \(\Dqc(\widetilde{Q})\),
\[
Lb^*\Ku(Q)
=
\mathbf{L}_{L\widetilde{\rho}^*\Dqc(S) \otimes^L \sO_{\widetilde{Q}}(-E)}
\big(R\tilde{\jmath}_*(La_\Lambda^*\Ku(Q') \otimes^L \sO_\Lambda(\Lambda))\big).
\]
This determines an \(S\)-linear equivalence of categories
\(\Phi \colon \Ku(Q') \to \Ku(Q)\).
\end{Proposition}

Compare the proof below to the proof of \cite[Theorem 4.2]{Xie:Quadrics}.

\begin{proof}
Starting with the semiorthogonal decomposition \eqref{equation-sod-b}, we
perform a series of mutations, which are justified by
\parref{decompositions-mutations}, to arrive at a decomposition which may be
compared with \eqref{equation-sod-a}. The overall strategy is to carefully move
the components supported on \(E\) completely toward the left, so that they end
up adjacent to the Kuznetsov component of \(Q\). In what follows, all
functors are derived, so we suppress the \(L\) and \(R\) for left and right
derived functors. Moreover, all constructions are linear over \(S\), so to ease
the notation, we suppress the tensor over \(L\widetilde{\rho}^*\Dqc(S)\)
and simply denote by an exceptional object the corresponding \(S\)-linear
subcategory it generates. 

\textbf{Step 1.}
Starting from \(\tilde{\imath}_*\sO_E\) and moving right, for each \(k = 0,
\ldots, n-3\), move \(\tilde{\imath}_*\sO_E(-kE)\) to the left via \(n-2-k\)
right mutations. Each mutation is of the form
\[
\langle
\sO_{\widetilde{Q}}(mH - kE),\,
\tilde{\imath}_*\sO_E(-kE)
\rangle =
\langle
\tilde{\imath}_*\sO_E(-kE),\,
\mathbf{R}_{\tilde{\imath}_*\sO_E(-kE)}(\sO_{\widetilde{Q}}(mH-(k+1)E))
\rangle.
\]
To determine the right mutation, note \(\sO_{\widetilde{Q}}(H)\)
restricts trivially to \(E\) because \(\widetilde{Q} \to Q\) is a blowup along
a section, so that
\[
\rho_*b_*\sHom_{\sO_{\widetilde{Q}}}(\sO_{\widetilde{Q}}(mH-kE), \tilde{\imath}_*\sO_E(-kE))
\cong \rho_*b_*\tilde{\imath}_*\sO_E
\cong \sO_S,
\]
generated by the equation of \(E\) in \(\widetilde{Q}\). Therefore the
right mutation of \(\sO_{\widetilde{Q}}(mH-kE)\), which is the cone of
the coevaluation map above up to a twist, is
\[
\operatorname{cone}(\sO_{\widetilde{Q}} \to \tilde{\imath}_*\sO_E)
\otimes \sO_{\widetilde{Q}}(mH-kE)
\cong
\sO_{\widetilde{Q}}(mH-(k+1)E).
\]
This series of right mutations produces a semiorthogonal decomposition
\[
\langle
b^*\Ku(Q),
\mathcal{A}_0,
\ldots,
\mathcal{A}_{n-3},
\sO_{\widetilde{Q}}((n-2)h)
\rangle
\]
where
\(\mathcal{A}_k = \langle \sO_{\widetilde{Q}}(kH-kE), \tilde{\imath}_*\sO_E(-kE) \rangle\)
for each \(k = 0, \ldots, n-3\).

\textbf{Step 2.} Within each of the blocks \(\mathcal{A}_k\), left mutate
\(\tilde{\imath}_*\sO_E(-kE)\) through \(\sO_{\widetilde{Q}}(kH-kE)\). The
computation from step 1 implies that this produces decompositions
\[
\mathcal{A}_k
= \langle \sO_{\widetilde{Q}}(kH - (k+1)E), \sO_{\widetilde{Q}}(kH - kE) \rangle
= \langle
  \sO_{\widetilde{Q}}(\Lambda + (k-1)h),
  \sO_{\widetilde{Q}}(kh) \rangle,
\]
for \(k = 0, \ldots, n-3\), where the second equality arises from the relations
\[
\sO_{\widetilde{Q}}(H-E)
\cong \sO_{\widetilde{Q}}(h)
\;\;\text{and}\;\;
\sO_{\widetilde{Q}}(H-2E)
\cong \sO_{\widetilde{Q}}(\Lambda) \otimes \widetilde{\rho}^*(\mathcal{F} \otimes \mathcal{L}^\vee)
\]
from \parref{lemma-projection-facts}\ref{lemma-projection-facts-pic}
and \parref{lemma-blowup-relation}, respectively, and the fact that twisting by
\(\widetilde{\rho}^*(\mathcal{F} \otimes \mathcal{L}^\vee)\) does not affect
the \(S\)-linear subcategory an object generates. Regroup the resulting
decomposition of \(\Dqc(\widetilde{Q})\) as follows:
\[
\langle
b^*\Ku(Q),
\sO_{\widetilde{Q}}(-E),
\mathcal{B}_0,
\ldots,
\mathcal{B}_{n-4},
\sO_{\widetilde{Q}}((n-3)h),
\sO_{\widetilde{Q}}((n-2)h)
\rangle
\]
where
\(
\mathcal{B}_k \coloneqq
\langle \sO_{\widetilde{Q}}(kh),
\sO_{\widetilde{Q}}(\Lambda + kh) \rangle
\)
for \(k = 0,\ldots,n-4\).

\textbf{Step 3.} Within each of the blocks \(\mathcal{B}_k\), left mutate
\(\sO_{\widetilde{Q}}(\Lambda + kh)\) through \(\sO_{\widetilde{Q}}(kh)\).
Since \(\Lambda\) is the exceptional divisor of the blowup \(a\),
\[
\widetilde{\rho}_*\sHom_{\sO_{\widetilde{Q}}}(\sO_{\widetilde{Q}}(kh),\sO_{\widetilde{Q}}(\Lambda+kh))
\cong \widetilde{\rho}_*\sO_{\widetilde{Q}}(\Lambda)
\cong \rho'_*a_*\sO_{\widetilde{Q}}(\Lambda)
\cong \sO_S,
\]
the generator being an equation of \(\Lambda\). Therefore the left mutation of
\(\sO_{\widetilde{Q}}(\Lambda + kh)\), being the cone of the evaluation map, is
\[
\operatorname{cone}(\sO_{\widetilde{Q}} \to \sO_{\widetilde{Q}}(\Lambda)) \otimes \sO_{\widetilde{Q}}(kh) \cong
\tilde{\jmath}_*\sO_\Lambda(\Lambda+kh).
\]
This gives
\(\mathcal{B}_k = \langle \tilde{\jmath}_*\sO_\Lambda(\Lambda + kh), \sO_{\widetilde{Q}}(kh) \rangle\)
for \(k = 0, \ldots, n-4\).

\textbf{Step 4.} Since \(\sO_\Lambda(\Lambda)\) restricts to
\(\sO_{\PP^1}(-1)\) along the fibres of \(\Lambda \to Q'\),
\begin{align*}
\widetilde{\rho}_*\sHom_{\widetilde{Q}}(\sO_{\widetilde{Q}}(kh), \tilde{\jmath}_*\sO_\Lambda(\Lambda+\ell h))
& \cong \rho'_*a_*\tilde{\jmath}_*\sO_\Lambda(\Lambda + (\ell - k)h) \\
& \cong \rho'_*j'_*a_{\Lambda,*} \sO_\Lambda(\Lambda + (\ell-k)h) = 0.
\end{align*}
Therefore the \(\sO_{\widetilde{Q}}(kh)\) are totally orthogonal to the
\(\tilde{\jmath}_*\sO_\Lambda(\Lambda + \ell h)\) for all integers \(k\) and
\(\ell\), and so the decomposition can be reordered to
\[
\langle
b^*\Ku(Q),
\sO_{\widetilde{Q}}(-E),
\tilde{\jmath}_*\sO_\Lambda(\Lambda), \ldots,
\tilde{\jmath}_*\sO_\Lambda(\Lambda + (n-4)h),
\sO_{\widetilde{Q}}, \ldots,
\sO_{\widetilde{Q}}((n-2)h)
\rangle.
\]

\textbf{Step 5.} Right mutate \(b^*\Ku(Q)\) through \(\sO_{\widetilde{Q}}(-E)\)
and apply Serre duality for the morphism
\(\widetilde{\rho} \colon \widetilde{Q} \to S\) to move
\(\sO_{\widetilde{Q}}(-E)\) to the right side. Since
\[
\sO_{\widetilde{Q}}(-E) \otimes \omega_{\widetilde{\rho}}^\vee
\cong \sO_{\widetilde{Q}}((n-1)h) \otimes \widetilde{\rho}^*\mathcal{M}
\]
for some invertible \(\sO_S\)-module \(\mathcal{M}\), this yields the
\(S\)-linear semiorthgonal decomposition
\[
\langle
\mathbf{R}_{\sO_{\widetilde{Q}}(-E)}(b^*\Ku(Q)),
\tilde{\jmath}_*\sO_\Lambda(\Lambda), \ldots,
\tilde{\jmath}_*\sO_\Lambda(\Lambda + (n-4)h),
\sO_{\widetilde{Q}}, \ldots,
\sO_{\widetilde{Q}}((n-1)h)
\rangle.
\]
Comparing with the semiorthogonal decomposition \eqref{equation-sod-a} shows
that
\[
\mathbf{R}_{\sO_{\widetilde{Q}}(-E)}(b^*\Ku(Q)) =
\tilde{\jmath}_*(a_\Lambda^*\Ku(Q') \otimes \sO_\Lambda(\Lambda))
\]
as \(S\)-linear subcategories of \(\Dqc(\widetilde{Q})\). Applying left
mutation through \(\sO_{\widetilde{Q}}(-E)\) on both sides then completes
the proof.
\end{proof}

Since \(\Phi \colon \Ku(Q') \to \Ku(Q)\) is a composition of pullbacks,
pushforwards, tensor products, and mutations, it is of Fourier--Mukai type. To
determine a kernel, consider the commutative diagram
\[
\begin{tikzcd}
& \widetilde{Q} \ar[dl, "b"', bend right=24] \ar[d]
& \Lambda \lar["\tilde{\jmath}"] \dar["j"] \ar[dr,"a_\Lambda", bend left=24] \\
Q
& Q \times_S \PP\mathcal{E}' \ar[l]
& Q \times_S Q' \ar[r, "\pr_2"] \ar[l]
& Q'
\end{tikzcd}
\]
where \(j \colon \Lambda \to Q \times_S Q'\) is the closed immersion as the
point-line incidence correspondence, resulting from the geometric interpretation
of \(\rho' \colon Q' \to S\) from \parref{lemma-fano-correspondence} as
the relative scheme of lines in \(\rho \colon Q \to S\) incident with the
section \(\PP\mathcal{F}\):
\[
\Lambda = \big\{(x, [\ell]) \in Q \times_S Q' : x \in \ell, \;
\ell \subset Q\;\text{a line through the section}\; \PP\mathcal{F}\big\}.
\]
Let \(\mathcal{J}\) be the ideal sheaf of \(\Lambda\) in \(Q \times_S Q'\).
The result is as follows:

\begin{Proposition}
\label{proposition-identify-the-kernel}
\(\Phi \colon \Ku(Q') \to \Ku(Q)\) is the Fourier--Mukai functor with kernel
\[
\pr_2^*\big(\sO_{\rho'}(1) \otimes \rho'^*(\mathcal{F}^\vee \otimes \mathcal{L})\big)
\otimes \mathcal{J} \in
\Dqc(Q \times_S Q').
\]
\end{Proposition}

\begin{proof}
Functors in this proof are all derived. By construction, \(\Phi\) sends an
object \(F \in \Ku(Q')\) to \(M \in \Ku(Q)\) characterized by
\begin{align*}
Lb^*M & \cong
\operatorname{cone}\big(
\widetilde{\rho}^* \mathcal{H} \otimes
\sO_{\widetilde{Q}}(-E) \to
\tilde{\jmath}_*(a_\Lambda^*F \otimes \mathcal{N}_{\Lambda/\widetilde{Q}})\big) \\
& \cong
\operatorname{cone}\big(
\widetilde{\rho}^*\mathcal{H} \to
\tilde{\jmath}_*a_\Lambda^*(F \otimes \sO_{\rho'}(1)) \otimes \widetilde{\rho}^*(\mathcal{F}^\vee \otimes \mathcal{L})
\big) \otimes \sO_{\widetilde{Q}}(-E),
\end{align*}
where \(\mathcal{N}_{\Lambda/\widetilde{\mathcal{Q}}}\) is identified as
\(\tilde{\jmath}^*(\sO_{\widetilde{Q}}(h-E) \otimes \widetilde{\rho}^*(\mathcal{F}^\vee \otimes \mathcal{L}))\)
via \parref{lemma-projection-facts}\ref{lemma-projection-facts-pic} and
\parref{lemma-blowup-relation}, and \(\mathcal{H}\) is the object
\begin{align*}
\mathcal{H}
& \coloneqq
\widetilde{\rho}_*\sHom_{\sO_{\widetilde{Q}}}(
  \sO_{\widetilde{Q}}(-E),
  \tilde{\jmath}_*(a^*F \otimes \mathcal{N}_{\Lambda/\widetilde{Q}})) \\
& \:\cong \widetilde{\rho}_*\tilde{\jmath}_* a^*(F \otimes \sO_{\rho'}(1)) \otimes (\mathcal{F}^\vee \otimes \mathcal{L})
\cong \rho'_*(F \otimes \sO_{\rho'}(1)) \otimes (\mathcal{F}^\vee \otimes \mathcal{L})
\in \Dqc(S).
\end{align*}
Since \(b \colon \widetilde{Q} \to Q\) is a blowup along a regular section,
\(b_*\sO_{\widetilde{Q}}(E) = \sO_Q\), %so pushing forward the map
%induced by multiplication by an equation of \(E\) together with the projection
%formula gives 
so the projection formula gives an isomorphism \(M \cong b_*(b^*M \otimes \sO_{\widetilde{Q}}(E))\).
Therefore there is a canonical isomorphism
\[
\Phi(F) \cong
b_*\operatorname{cone}\big(
  \widetilde{\rho}^*\mathcal{H} \to
  \tilde{\jmath}_*a^*(F \otimes \sO_{\rho'}(1)) \otimes \widetilde{\rho}^*(\mathcal{F}^\vee \otimes \mathcal{L})
\big).
\]
Writing
\(\rho^*\mathcal{H} = \pr_{1,*}\pr_2^*(F \otimes \sO_{\rho'}(1)) \otimes \rho^*(\mathcal{F}^\vee \otimes \mathcal{L})\)
and \(\pr_1 \colon Q \times_S Q' \to Q\) for the first projection, the
relations \(a = \pr_2 \circ j\) and \(b \circ \tilde{\jmath} = \pr_1 \circ j\)
from the diagram above simplifies this to
\begin{align*}
\Phi(F)
& \cong
  \pr_{1,*}\operatorname{cone}\big(
  \pr_2^*(F \otimes \sO_{\rho'}(1)) \to
  j_*j^*\pr_2^*(F \otimes \sO_{\rho'}(1))\big) \otimes \rho^*(\mathcal{F}^\vee \otimes \mathcal{L})\\
& \cong
  \pr_{1,*}\operatorname{cone}\big(
    \pr_2^*(F \otimes \sO_{\rho'}(1)) \to
    \pr_2^*(F \otimes \sO_{\rho'}(1)) \otimes j_*\sO_\Lambda\big) \otimes \rho^*(\mathcal{F}^\vee \otimes \mathcal{L})
\end{align*}
where the map is restriction to \(\Lambda\). Thus the cone is
\(\pr_2^*(F \otimes \sO_{\rho'}(1))\) tensored with the ideal sheaf
\(\mathcal{J}\) of \(\Lambda\) in \(Q \times_S Q'\), and so
\[
\Phi(F) \cong
\pr_{1,*}(
\pr_2^*(F)
    \otimes \pr_2^*\big(\sO_{\rho'}(1) \otimes \rho'^*(\mathcal{F}^\vee \otimes \mathcal{L})\big)
    \otimes \mathcal{J}
    ).
\qedhere
\]
\end{proof}

A particularly useful feature of the equivalence
\(\Phi \colon \Ku(Q') \to \Ku(Q)\) is that it is compatible with the
construction of spinor sheaves. More precisely, let
\(\mathcal{W}' \subset \mathcal{F}^\perp/\mathcal{F}\) be an isotropic
subbundle, and let \(\mathcal{W}\) be the corresponding subbundle of in
\(\mathcal{E}\). For each integer \(d\), these define spinor sheaves
\(\mathcal{S}_d'\) and \(\mathcal{S}_d\) on \(Q'\) and \(Q\), respectively.
By \parref{cliff-spinor-cohomology}, these may be viewed as objects in
\(\Ku(Q')\) and \(\Ku(Q)\), respectively. The statement is that \(\Phi\)
sends the dual of one to the dual of the other:

\begin{Proposition}\label{cliff-hyperbolic-reduction-spinors}
\(\Phi(\mathcal{S}_d'^\vee) \cong \mathcal{S}_d^\vee\) for all \(d \in \mathbf{Z}\).
\end{Proposition}

\begin{proof}
All functors are again derived in this proof. Continuing with the notation from
the diagram preceding \parref{proposition-identify-the-kernel}, tensor the
ideal sheaf sequence for \(j \colon \Lambda \to Q \times_S Q'\) with
\(\pr_2^*\mathcal{S}_d'\) and pushforward along
\(\pr_1 \colon Q \times_S Q' \to Q\) to obtain the triangle
\[
\pr_{1,*}(
\pr_2^*\mathcal{S}_d' \otimes \mathcal{J}) \longrightarrow
\pr_{1,*}\pr_2^*\mathcal{S}_d' \longrightarrow
(\pr_1 \circ j)_*(\pr_2 \circ j)^*\mathcal{S}_d' \stackrel{+1}{\longrightarrow}
\]
in \(\Dqc(Q)\). The defining presentation for the spinor sheaf from
\parref{cliff-spinors} identifies the middle term as
\[
\pr_{1,*}\pr_2^*\mathcal{S}_d' \cong
\rho^*\rho_*'\mathcal{S}_d' \cong
\rho^*\mathcal{I}_d'
\]
where \(\mathcal{I}_d\) is the \(d\)-th Clifford ideal on \(S\) associated with
\(\mathcal{W}'\). To identify the term on the right, let \(Q''\) be the
intersection of \(Q\) with its tangent space along \(\PP\mathcal{F}\); in other
words, let \(q'' \colon \mathcal{F}^\perp \to \mathcal{L}\) be the
restriction of \(q\) to the orthogonal of \(\mathcal{F}\), and let \(Q''\)
be the associated quadric bundle in \(\PP\mathcal{F}^\perp\). Then \(Q''\)
is a cone over \(Q'\) with vertex \(\PP\mathcal{F}\), and the morphism
\(\pr_1 \circ j \colon \Lambda \to Q\) factors through a morphism
\(b_\Lambda \colon \Lambda \to Q''\) identifying it with the blowup along
the vertex. Thus there is a commutative diagram
\[
\begin{tikzcd}
\Lambda \rar["j"'] \dar["b_\Lambda"']
& Q \times_S Q' \dar["\pr_1"] \rar["\pr_2"']
& Q' \\
Q'' \rar["i''"]
& Q.
\end{tikzcd}
\]
Writing \(a_\Lambda = \pr_2 \circ j \colon \Lambda \to Q'\), the term on the
right is identified via \parref{cliff-cone-spinors} as the pushforward to
\(Q\) of a spinor sheaf on \(Q''\) associated with \(\mathcal{W}\):
\[
(\pr_1 \circ j)_*(\pr_2 \circ j)^*\mathcal{S}_d' \cong
i''_*b_{\Lambda,*}a_{\Lambda}^*\mathcal{S}_d' \cong
i''_*\mathcal{S}_{d+1}'' \otimes \rho^*\mathcal{F}^\vee.
\]
Identifying furthermore the Clifford ideal \(\mathcal{I}_d'\) for
\(\rho' \colon Q' \to S\) with the twisted Clifford ideal
\(\mathcal{I}_{d+1}'' \otimes \mathcal{F}^\vee\) for
\(\rho'' \colon Q'' \to S\) via \parref{cliff-cone-ideals} transforms the
triangle to
\[
\pr_{1,*}(\pr_2^*\mathcal{S}_d' \otimes \mathcal{J}) \longrightarrow
\rho^*\mathcal{I}_{d+1}'' \otimes \rho^*\mathcal{F}^\vee \longrightarrow
i''_*\mathcal{S}_{d+1}'' \otimes \rho^*\mathcal{F}^\vee \stackrel{+1}{\longrightarrow}
\]
where the second map is induced by the evaluation map. Comparing with
\parref{cliff-subbundle-modify} and using the fact that
\(\mathcal{E}/\mathcal{F}^\perp \cong \mathcal{F}^\vee \otimes \mathcal{L}\)
shows that
\begin{equation}\label{eq:cliff-hyperbolic-reduction-spinors.identity}
\pr_{1,*}(\pr_2^*\mathcal{S}_d' \otimes \mathcal{J}) \cong
\mathcal{S}_{d+1} \otimes \sO_\rho(-1) \otimes \rho^*\mathcal{L}^\vee.
\tag{\(\star\)}
\end{equation}

To conclude, it remains to compare the two spinor sheaves to their duals via
the relations in \parref{cliff-spinors}. Writing the rank of \(\mathcal{E}\) as
\(2k\) or \(2k+1\), and that of \(\mathcal{W}\) as \(r\), the relations give
\begin{align*}
\mathcal{S}_d' \otimes \sO_{\rho'}(-1) & \cong
\mathcal{S}_{r-d-2}'^\vee \otimes \rho'^*(\det\mathcal{W}' \otimes \det(\mathcal{F}^\perp/\mathcal{F}) \otimes \mathcal{L}^{\vee, \otimes k-1}),\;\text{and} \\
\mathcal{S}_{d+1} \otimes \sO_{\rho\phantom{'}}(-1) & \cong
\mathcal{S}_{r-d-2}^\vee \otimes \rho^*(\det\mathcal{W} \otimes \det\mathcal{E} \otimes \mathcal{L}^{\vee, \otimes k}).
\end{align*}
Since \(\det\mathcal{W}' \cong \det\mathcal{W} \otimes \mathcal{F}^\vee\)
and \(\det(\mathcal{F}^\perp/\mathcal{F}) \cong \det\mathcal{E} \otimes \mathcal{L}^\vee\),
substituting these identifications into
\eqref{eq:cliff-hyperbolic-reduction-spinors.identity}, cancelling out common
invertible factors, and multiplying through by \(\rho^*\mathcal{L}\) shows
\[
\pr_{1,*}\big(
\pr_2^*\mathcal{S}_d'^\vee \otimes
\pr_2^*\big(
  \sO_{\rho'}(1) \otimes
  \rho'^*(\mathcal{F}^\vee \otimes \mathcal{L})
 \big)
\otimes \mathcal{J}\big)
\cong
\mathcal{S}_d^\vee.
\]
The functor on the left is \(\Phi\) by
\parref{proposition-identify-the-kernel}, yielding the result.
\end{proof}

One simple application of \parref{cliff-hyperbolic-reduction-spinors} is
to show that spinor sheaves are relatively simple in certain situations.
Note that this is not always the case: see \cite[Proposition
6.2]{Addington:Spinors}.

\begin{Corollary}\label{spinor-moduli-simple}
Let \(\rho \colon Q \to S\) be a quadric \(2\ell\)-fold bundle,
\(\PP\mathcal{W} \to S\) an isotropic \(\ell\)-plane intersecting
\(\Sing\rho\) in at most \(1\) point in each fibre, and
\(\mathcal{S}_d^\vee\) the associated spinor sheaf. Then the canonical map
\[
\sO_S \to \rho_*\sHom_{\sO_Q}(\mathcal{S}_d^\vee, \mathcal{S}_d^\vee)
\]
is an isomorphism for any \(d \in \mathbf{Z}\). In other words,
\(\mathcal{S}_d^\vee\) is simple over \(S\).
\end{Corollary}

\begin{proof}
The assertion is local on \(S\). Passing to an \'etale cover, we may assume
that \(\mathcal{W}\) admits a line subbundle \(\mathcal{N}\). Let
\(\rho' \colon Q' \to S\) be the associated hyperbolic reduction, set
\(\mathcal{W}' \coloneqq \mathcal{W}/\mathcal{N}\), and write
\(\mathcal{S}_d'^\vee\) for the corresponding \(d\)-th dual spinor sheaf.
There is an \(S\)-linear equivalence \(\Phi \colon \Ku(Q') \to \Ku(Q)\)
satisfying \(\mathcal{S}^\vee_d \cong \Phi(\mathcal{S}_d'^\vee)\) by
\parref{hypred-equivalence} and \parref{cliff-hyperbolic-reduction-spinors}, so
this induces an isomorphism
\[
R\rho_*R\sHom_{\sO_Q}(\mathcal{S}^\vee_d, \mathcal{S}^\vee_d) \cong
R\rho'_*R\sHom_{\sO_{Q'}}(\mathcal{S}_d'^\vee, \mathcal{S}_d'^\vee).
\]
By induction, then, it suffices to treat the case when \(\rho \colon Q \to S\)
has relative dimension \(0\) and \(\PP\mathcal{W} \to S\) provides a section
\(w \colon S \to Q\)  of \(\rho\). In this case,
\parref{cliff-spinor-zero-dimensional} shows that the spinor sheaves
\(\mathcal{S}_d^\vee\) are, up to invertible factors, of the form
\(i_*\sO_T\) for a closed subscheme \(i \colon T \to Q\) cut out in the
\(\PP^1\)-bundle \(\pi \colon \PP\mathcal{E} \to S\) by a line bundle of
relative degree \(1\). Noting that \(\rho\) is flat and \(i\) is affine, taking
degree \(0\) cohomology sheaves then gives the result:
\[
\rho_*\sHom_{\sO_Q}(\mathcal{S}_d^\vee, \mathcal{S}_d^\vee) \cong
\rho_*\sHom_{\sO_Q}(i_*\sO_T, i_*\sO_T) \cong
\rho_*i_*\sHom_{\sO_T}(\sO_T,\sO_T) \cong
\rho_*i_*\sO_T \cong
\sO_S.
\qedhere
\]
\end{proof}

\subsectiondash{Maximal hyperbolic reduction}\label{reduction-fourier-mukai}
Let \(\rho \colon Q \to S\) be a quadric bundle of even relative dimension
\(2\ell\). Suppose that there exists a regular isotropic subbundle
\(\mathcal{F} \subset \mathcal{E}\) of rank \(\ell\); in particular, this
implies that \(S_3 = \varnothing\). Assuming furthermore that \(S_2\)
contains no weakly associated points of \(S\), the corresponding hyperbolic
reduction \(\mu \colon M \to S\) is a family of quadrics of relative dimension
\(0\), having \(1\)-dimensional fibres over points of \(S_2\). As such, we
refer to \(\mu \colon M \to S\) as a \emph{maximal hyperbolic reduction} of the
quadric bundle \(\rho \colon Q \to S\).

Identify \(M\) via \parref{lemma-fano-correspondence} as the subscheme of
the relative Fano scheme \(\mathbf{F}_\ell(Q/S)\) parameterizing
\(\ell\)-planes along \(\rho \colon Q \to S\) that contain \(\PP\mathcal{F}\),
and let \(\widetilde{\mathcal{F}}\) be the tautological isotropic subbundle of
rank \(\ell+1\) in \(\mu^*\mathcal{E}\). In this way, the quadric bundle
\[
Q \times_S M \subset
\PP\mathcal{E} \times_S M \cong
\PP(\mu^*\mathcal{E})
\]
over \(M\) via \(\pr_2 \colon Q \times_S M \to M\) carries a tautological
family \(\PP\widetilde{\mathcal{F}}\) of isotropic \(\ell\)-planes. Writing
\(\mathcal{S}_d\) for the associated degree \(d\) spinor sheaf on
\(Q \times_S M\), we have the following:

\begin{Proposition}\label{reduction-fourier-mukai-equivalence}
In the setting above, the Fourier--Mukai functor
\(\Dqc(M) \to \Dqc(Q)\) given by
\[
\Psi_d(F) \coloneqq
R\pr_{1,*}(L\pr_2^*F \otimes^L \mathcal{S}_d^\vee)
\]
factors through \(\Ku(Q)\) and defines an \(S\)-linear equivalence
\(\Psi_d \colon \Dqc(M) \to \Ku(Q)\) for each \(d \in \mathbf{Z}\).
\end{Proposition}

\begin{proof}
To ease notation, all functors in this proof are derived.
That the functor factors through \(\Ku(Q)\) follows from
\parref{cliff-spinor-cohomology}: Indeed, using
\(\rho \circ \pr_1 = \mu \circ \pr_2\), the fact that the relative tautological
bundle of \(\pr_2 \colon Q \times_S M \to M\) is
\(\sO_{\pr_2}(1) \coloneqq \pr_1^*\sO_{\rho}(1)\), and the projection formula,
\begin{align*}
\rho_*\sHom_{\sO_Q}(\sO_{\rho}(i), \Psi_d(F))
& \cong
  \rho_*\pr_{1,*}(
    \pr_2^*F \otimes
    \sHom_{\sO_{Q \times_S M}}(\sO_{\pr_2}(i), \mathcal{S}_d^\vee)) \\
& \cong
  \mu_*(F \otimes
    \pr_{2,*}\sHom_{\sO_{Q \times_S M}}(\sO_{\pr_2}(i), \mathcal{S}_d^\vee)) \\
& = 0\;\;\text{for}\; 0 \leq i \leq 2r-1.
\end{align*}
Henceforth, view \(\Psi_d\) as a functor \(\Dqc(M) \to \Ku(Q)\).
Show that this is an equivalence via induction on the integer \(\ell \geq 1\).
When \(\ell = 1\), the even degree \(d = 2k\) dual spinors
\(\mathcal{S}_{2k}^\vee\) are by \parref{reduction-quadric-surfaces-example},
up to twists by invertible objects pulled back from \(S\), the ideal sheaf
\(\mathcal{J}\) of the tautological line
\(j \colon \Lambda \hookrightarrow Q \times_S M\). In this case,
\parref{hypred-equivalence} and \parref{proposition-identify-the-kernel}
together imply that \(\Psi_{2k}\) is an equivalence. For odd \(d = 2k+1\),
\parref{reduction-quadric-surfaces-example} shows that, up to twists by
invertible objects pulled up from \(S\), \(\Psi_{2k+1}\) is of the form
\[
F \mapsto
\pr_{1,*}\big(
  \pr_2^*F \otimes
  \sHom_{\sO_{Q \times_S M}}(\mathcal{J}, \sO_{Q \times_S M}) \otimes
  \pr_1^*\sO_\rho(-1)\big).
\]
Since the ideal sheaf \(\mathcal{J}\) %is finitely presented,
is pseudo-coherent, \(\pr_2^*F\) may
be brought into the \(\sHom\)-complex; since \(\mu \colon M \to S\) is an
effective Cartier divisor of relative degree \(2\) in a \(\PP^1\)-bundle over
\(S\), it has a dualizing sheaf which is the pullback of a line bundle on 
\(S\). Therefore, the functor above is, up to invertible objects from \(S\),
isomorphic to
\[
F \mapsto
\sHom_{\sO_Q}\big(\pr_{1,*}(\pr_2^*\sHom_{\sO_M}(F, \sO_M) \otimes \mathcal{J}), \sO_Q(-1)\big).
\]
This shows that \(\Psi_{2k+1}\) is obtained from the Fourier--Mukai equivalence
of \parref{hypred-equivalence} and \parref{proposition-identify-the-kernel}
upon applying autoequivalences to the source and target, and so \(\Psi_{2k+1}\)
itself is an equivalence. This establishes the base case of the induction.

Assume \(\ell \geq 2\). By \parref{thm-locallytwisted}, which applies thanks
to \parref{lemma-thekernelisrelperf}, it suffices to show that \(\Psi_d\) is an
equivalence after replacing \(S\) by an \'etale cover.
Doing so, assume further that \(\mathcal{F}\) admits a line subbundle
\(\mathcal{N}\), and let \(\rho' \colon Q' \to S\) be the associated hyperbolic
reduction. Then \(\mu \colon M \to S\) is the maximal hyperbolic reduction of
the quadric \((2\ell-2)\)-fold bundle \(\rho' \colon Q' \to S\)
along the regular isotropic \((\ell-1)\)-bundle
\(\mathcal{F}' \coloneqq \mathcal{F}/\mathcal{N}\). The constructions of
\parref{reduction-fourier-mukai} then apply to yield, for each integer \(d\),
dual spinor sheaves
\[
\mathcal{S}_d'^\vee
\;\text{on the quadric bundle}\;
\pr_2' \colon Q' \times_S M \to M
\]
constructed from the tautological isotropic rank \(\ell\) subbundle
\(\widetilde{\mathcal{F}}'\) of \(\mu^*(\mathcal{N}^\perp/\mathcal{N})\).
The inductive hypothesis implies that the corresponding Fourier--Mukai functor
\(\Psi_d' \colon \Dqc(M) \to \Ku(Q')\) is an equivalence. Let
\(\Phi \colon \Ku(Q') \to \Ku(Q)\) be the equivalence from
\parref{hypred-equivalence}. We now show that \(\Psi_d\) and
\(\Phi \circ \Psi_d'\) are isomorphic as functors \(\Dqc(M) \to \Ku(Q)\),
from which the result follows.

On the one hand, \(\Phi \circ \Psi_d'\) is a Fourier--Mukai functor whose
kernel, in view of \parref{proposition-identify-the-kernel}, is
\[
\mathcal{K} \coloneqq
\pr_{1,3,*}\big(
\pr_{1,2}^*\mathcal{J} \otimes
\pr_2^*(\sO_{\rho'}(1) \otimes \rho'^*(\mathcal{N}^\vee \otimes \mathcal{L})
)
\otimes
\pr_{2,3}^*\mathcal{S}_d'^\vee
\big)
\in \Dqc(Q \times_S M),
\]
where \(\mathcal{J}\) is the ideal sheaf of the incidence correspondence
\(\Lambda \hookrightarrow Q \times_S Q'\) as before. On the other hand,
base change along \(\mu \colon M \to S\), and view the
quadric bundle \(\pr_2' \colon Q' \times_S M \to M\) as the hyperbolic
reduction of \(\pr_2 \colon Q \times_S M \to M\) along the regular isotropic
line subbundle \(\mu^*\mathcal{N} \subset \mu^*\mathcal{E}\). Applying
\parref{hypred-equivalence} to this situation thus gives an equivalence
\[
\Phi_M \colon
\Ku(Q' \times_S M) \to
\Ku(Q \times_S M).
\]
The tautological incidence correspondence between \(Q \times_S M\) and
\(Q' \times M\) is none other than
\[
\Lambda \times_S M =
\pr_{1,2}^{-1}(\Lambda) \hookrightarrow
Q \times_S Q' \times_S M \cong
(Q \times_S M) \times_{M} (Q' \times_S M),
\]
so its ideal sheaf is \(\pr_{1,2}^*\mathcal{J}\).
Comparing with \parref{proposition-identify-the-kernel} then shows that
\[
\mathcal{K}
\cong \Phi_M(\mathcal{S}_d'^\vee)
\in \Dqc(Q \times_S M).
\]
Since \(\mathcal{S}_d'^\vee\) is the \(d\)-th dual spinor bundle associated
with \(\widetilde{\mathcal{F}}' = \widetilde{\mathcal{F}}/\mathcal{N}\) on
\(Q' \times_S M\), and \(\Phi_M\) respects dual spinor sheaves by
\parref{cliff-hyperbolic-reduction-spinors}, it follows that
\(\mathcal{K} \cong \mathcal{S}_d^\vee\). Thus \(\Phi \circ \Psi_d'\)
and \(\Psi_d\) are both Fourier--Mukai transforms \(\Dqc(M) \to \Ku(Q)\)
with isomorphic kernels, and so the functors are isomorphic.
\end{proof}

In order to apply \parref{thm-locallytwisted}, we need to verify that the
tautological dual spinor sheaves \(\mathcal{S}_d^\vee\) are relatively
perfect relative to both factors of \(Q \times_S M\). That it is \(M\)-perfect
follows from \parref{lemma-spinorsheavesrelperf}, whereas \(Q\)-perfectness can
be established by induction using the technique of
\parref{cliff-hyperbolic-reduction-spinors}:

\begin{Lemma}
\label{lemma-thekernelisrelperf}
In the setting of \parref{reduction-fourier-mukai}, the dual spinor sheaf
\(\mathcal{S}^\vee_d\) is perfect relative to \(Q\).
\end{Lemma}

\begin{proof}
Induct on the integer \(\ell \geq 1\). The base case \(\ell = 1\) concerns
a quadric surface bundle \(\rho \colon Q \to S\) and by
\parref{reduction-quadric-surfaces-example} and
\parref{proposition-identify-the-kernel}, ignoring twists by line bundles, the
ideal sheaf \(\mathcal{J}\) of the tautological line
\(j \colon \Lambda \to Q \times_S M\). By the ideal sheaf sequence, it suffices
to verify that \(\sO_{Q \times_S M}\) and \(j_*\sO_\Lambda\) are \(Q\)-perfect.
For the former, observe that first projection factors as
\[
\pr_1 \colon
Q \times_S M \hookrightarrow
Q \times_S \PP(\mathcal{E}/\mathcal{F}) \to
Q
\]
where the first arrow is a closed immersion, which is regular by the standing
assumption on weakly associated points together with
\parref{lemma-cartier-divisor-criterion}, and the second arrow is a projective
bundle. Thus \(\pr_1\) is a composition of perfect morphisms, and thus is itself
perfect, meaning that \(\sO_{Q \times_S M}\) is \(Q\)-perfect. For
\(j_*\sO_\Lambda\), as explained above \parref{proposition-identify-the-kernel},
the map \(\Lambda \to Q\) factors as \(b \circ \tilde{\jmath}\) where
\(\tilde{\jmath}\) is the inclusion of an effective Cartier divisor and \(b\)
is a blowup in a local complete intersection closed subscheme. Therefore
$\Lambda \to Q$ is an local complete intersection morphism in the sense of
\citeSP{068E}, and so \(j_*\sO_\Lambda\) is perfect relative to $Q$ by
\citeSP{069H}. This establishes the base case.

Assume \(\ell > 1\) and that the tautological dual spinors between a quadric
\(2(\ell-1)\)-bundle and its maximal hyperbolic reduction is perfect relative
to the quadric. Let \(\rho \colon Q \to S\) be a quadric \(2\ell\)-fold bundle
with maximal hyperbolic reduction \(\mu \colon M \to S\) with respect to a
regular isotropic subbundle \(\mathcal{F} \subset \mathcal{E}\) of rank \(\ell\).
The problem is local on \(S\), so pass to an \'etale cover to assume that
\(\mathcal{F}\) admits a line subbundle \(\mathcal{N}\) so that the associated
hyperbolic reduction \(\rho' \colon Q' \to S\) shares \(\mu \colon M \to S\) as
its maximal hyperbolic reduction along \(\mathcal{F}/\mathcal{N}\). The proof
of \parref{cliff-hyperbolic-reduction-spinors} shows, using
\parref{cliff-subbundle-modify} and \parref{cliff-cone-spinors}, that the
spinors of \(Q \times_S M\) and \(Q' \times_S M\) fit into a short exact
sequence
\[
0 \to
\mathcal{S}_d \otimes \sO_{\pr_2}(-1) \otimes (\rho \circ \pr_1)^*\mathcal{L}^\vee \to
\pr_2^*\mathcal{I}'_{d-1} \to
(i'' \times \id_M)_*(b \times \id_M)_*(a \times \id_M)^*\mathcal{S}_{d-1}' \to
0
\]
where \(\mathcal{I}_{d-1}'\) is the Clifford ideal associated with the tautological
\((\ell-1)\)-plane in \(Q' \times_S M\), \(i'' \colon Q'' \to Q\) is the
inclusion of the tangent hyperplane section along \(\PP\mathcal{N}\),
\(a \colon \widetilde{Q} \to Q'\) is a \(\PP^1\)-bundle, and
\(b \colon \widetilde{Q} \to Q''\) is the blowup along the vertex
\(\PP\mathcal{N}\).

The morphism \(\pr_1 \colon Q \times_S M \to Q\) is Gorenstein, so
\parref{complexes-duality}\ref{complexes-duality.perfect} implies
that it suffices to show that \(\mathcal{S}_d\) is \(Q\)-perfect, and this
would follow from seeing that the second and third terms of this sequence are
\(Q\)-perfect. This is true for \(\pr_2^*\mathcal{I}_{d-1}'\) since it is
locally free and \(\pr_1 \colon Q \times_S M \to Q\) is perfect by the argument
in the base case. As for the term on the right,
induction and \parref{complexes-duality}\ref{complexes-duality.perfect}
together mean that \(\mathcal{S}_{d-1}'\) is \(Q'\)-perfect. Relative
perfectness is clearly preserved upon pullback along the \(\PP^1\)-bundle
\(a \times \id_M\), and also pushforward along \(b \times \id_M\) by
\parref{lemma-perfectpush}. Finally, \(i'' \colon Q'' \to Q\) is the
inclusion of a effective Cartier divisor since it is the tangent hyperplane
section along a smooth point, so the pushforward of a \(Q''\)-perfect object
yields a \(Q\)-perfect object. Together, this completes the proof.
\end{proof}

\section{Stack of spinor sheaves}\label{section-spinor-moduli}
The goal of this section is to prove a global version of
\parref{reduction-fourier-mukai-equivalence}, in which the quadric bundle may
not possess a regular isotropic subbundle of maximal dimension. This is
achieved in \parref{spinor-main-theorem}, wherein a global incarnation of the
maximal hyperbolic reduction is constructed as an open subspace of a moduli
space of spinor sheaves. Much of this section is therefore devoted to a
study of the moduli stack of spinor sheaves on quadric bundles. Throughout
this section, unless otherwise specified, \(\rho \colon Q \to S\) denotes a
quadric \(2\ell\)-fold bundle with \(\ell \geq 1\) and \(S_3 = \varnothing\).
%, such that
%\(S_2\) contains no weakly associated points of \(S\), and \(S_3 = \varnothing\). -- Hope to remove from many results, but not derived category ones.

\subsectiondash{Definition and local structure}
\label{spinor-moduli-stack}
Applying the constructions of \parref{cliff-spinors} to the tautological
rank \(\ell+1\) subbundle on \(\mathbf{F}_\ell(Q/S)\) provides, for each
\(d \in \mathbf{Z}\), a morphism
\begin{equation*}
%\label{equn-fanotocoh}
\zeta_d \colon \mathbf{F}_\ell(Q/S) \to \sCoh_{Q/S}
\end{equation*}
from the relative Fano scheme of $\ell$-planes of $\rho \colon Q \to S$ to the
stack of coherent sheaves on $Q$ over \(S\) which sends an $\ell$-plane to its
associated degree $d$ dual spinor sheaf. We aim to prove:

\begin{Proposition}
\label{prop-fanotocoh}
The morphism \(\zeta_d \colon \mathbf{F}_\ell(Q/S) \to \sCoh_{Q/S}\) is smooth.
\end{Proposition}

Since smooth morphisms are open, the image of \(\zeta_d\) is an open
substack $\widebar{\mathcal{M}}_d(Q/S) \subset \sCoh_{Q/S}$: this is the
\emph{stack of spinor sheaves} of \(\rho \colon Q \to S\) of degree \(d\). Of
primary interest is the open substack
\[
\mathcal{M}_d(Q/S) \subseteq
\widebar{\mathcal{M}}_d(Q/S)
\]
consisting of spinors locally free away from at most one point in each fibre of
\(\rho \colon Q \to S\); by \parref{cliff-spinors}, this is the image of the
restriction \(\zeta_d^\circ \colon \mathbf{F}_\ell(Q/S)^\circ \to \sCoh_{Q/S}\)
of the morphism \(\zeta_d\) to the open subscheme of $\mathbf{F}_\ell(Q/S)$
parameterizing \(\ell\)-planes which fibrewise intersect the singular locus of
\(\rho \colon Q \to S\) in at most one point. Notably, thanks to
\parref{spinor-moduli-simple}, \(\mathcal{M}_d(Q/S)\) is contained in the open
substack \(\ssCoh_{Q/S} \subset \sCoh_{Q/S}\) of simple sheaves, and so it
admits a coarse moduli space \(\mathrm{M}_d(Q/S)\)---an algebraic space
over \(S\)---over which \(\mathcal{M}_d(Q/S)\) is a $\mathbf{G}_m$-gerbe.

The local picture of \(\mathcal{M}_d(Q/S)\) can be described as follows: Suppose
that \(\rho \colon Q \to S\) admits a family of \((\ell-1)\)-planes
$L \subset Q$ contained in its smooth locus, and let
$\mu \colon M \to S$ be the associated maximal hyperbolic reduction as in
\parref{reduction-fourier-mukai}. Identify $M$ via
\parref{lemma-fano-correspondence} as the closed subscheme of
$\mathbf{F}_\ell(Q/S)$ parametrizing $\ell$-planes in $Q$
containing $L$. In fact, $M \subset \mathbf{F}_\ell(Q/S)^\circ$ since an
$\ell$-plane containing $L \subset Q \setminus \Sing\rho$ can fibrewise
intersect the singular locus of $\rho \colon Q \to S$ in at most one point. Then:

\begin{Proposition}
\label{prop-hypredtocoh}
Let \(\mu \colon M \to S\) be a maximal hyperbolic reduction of
\(\rho \colon Q \to S\) as above. Then
\[
\zeta_d\rvert_M \colon
M \subset \mathbf{F}_\ell(Q/S)^\circ \to \mathcal{M}_d(Q/S)
\]
is smooth of relative dimension $1$, and the composite
% \begin{equation}
% \label{equn-composition}
% M \subset \mathbf{F}_\ell(Q/S)^\circ \to \sCoh(Q/S)
% \end{equation}
\(M \to \mathcal{M}_d(Q/S) \to \mathrm{M}_d(Q/S)\)
is an open immersion.
\end{Proposition}

In general, \(\rho \colon Q \to S\) may not admit a maximal hyperbolic
reduction as there may not be a global family of \((\ell-1)\)-planes contained
in the smooth locus. Nonetheless, \'etale locally on \(S\), the space
$\mathrm{M}_d(Q/S)$ is covered by opens of the form $M$ as above:

\begin{Lemma}
\label{lemma-sectionetalelocally}
    Let $f \colon T \to \mathcal{M}_d(Q/S)$ be a morphism from an $S$-scheme
    $T$. Then there exists an \'etale covering $U \to T$ and a family of
    $(\ell-1)$-planes $L \subset Q_U$  contained in the smooth locus of
    $\rho_U \colon Q_U \to U$ such that the composite $U \to T \to
    \mathcal{M}_d(Q/S)$ factors through the maximal hyperbolic reduction
    $\mu_U \colon M_U \to U$ of $\rho_U \colon Q_U \to U$ with respect to $L$:
    $$
    \begin{tikzcd}
    U \ar[r] \ar[d] & M_U \ar[d, "\zeta_d\rvert_{M_U}"] \\
    T \ar[r, "f"] & \mathcal{M}_d(Q/S).
    \end{tikzcd}
    $$
\end{Lemma}

\begin{proof}
After replacing $T$ by an \'etale cover, we may assume $f$ factors through
\(\zeta_d^\circ\): Indeed, this is because, by \parref{prop-fanotocoh}, the
fibre product $T \times _{\mathcal{M}_d(Q/S)} \mathbf{F}_\ell(Q/S)^\circ$
is a smooth algebraic space over $T$ mapping surjectively onto $T$, and so it
admits a section after an \'etale covering. This means that $f$ may be taken
to be the classifying map for the spinor sheaf associated with
a family $P \subset Q_T$ of $\ell$-planes which intersects the singular locus
of \(Q_T \to T\) fibrewise in at most one point. Since most hyperplanes avoid a
given point, there is an open subscheme \(\mathbf{G}\) of the Grassmannian
\(\mathbf{G}(\ell-1,P) \to T\) of hyperplanes in \(P \to T\) which is nonempty
over every point of \(T\) and which parameterizes those \((\ell-1)\)-planes
\(L \subset P\) contained in the smooth locus of \(Q_T \to T\). Since
\(\mathbf{G}\) is smooth over \(T\), replacing \(T\) by yet another \'etale
covering provides a section, whence a family \(L \subset Q_T\) of
\((\ell-1)\)-planes contained in the smooth locus. Writing \(M_T \to T\) for
the associated maximal hyperbolic reduction, the containment
\(L \subset P\) means via \parref{lemma-fano-correspondence} that the
classifying morphism \(T \to \mathbf{F}_\ell(Q/S)^\circ\) for \(P\) factors
through \(M_T \subseteq \mathbf{F}_\ell(Q/S)^\circ\), from which the
result follows.
% Then for every geometric point $\bar{t} \to T$, the plane $P_{\bar{t}} \subset Q_{\bar{t}}$ intersects the singular locus of $Q_{\bar{t}}$ in at most one point, so there exists a codimension one linear space $L_0 \subset P_{\bar{t}}$ which is contained in the smooth locus of $Q_{\bar{t}}$. After replacing $T$ by an \'etale neighborhood of $\bar{t}$, we may assume there exists a family $L \subset P \subset Q_T$ of $(\ell-1)$-planes contained in the smooth locus of $Q_T/T$ such that $L_0 = L_{\bar{t}}$. This is the Grassmanian of $P$  is smooth at the point $L_0$. Then if $\bar{Q}_T$ denotes the hyperbolic reduction of $Q_T$ with respect to $L$, the morphism $T \to \mathbf{F}_\ell(Q/S)$ classified by $P$ factors through $\bar{Q}_T$ by construction (since $L \subset P$). Thus we have solved the problem in an \'etale neighborhood of $\bar{t}$. Doing this for every geometric point proves the lemma. 
\end{proof}

Propositions \parref{prop-fanotocoh} and \parref{prop-hypredtocoh} will be
proven over the course of the following few paragraphs. We first treat the
special case where \(\rho \colon Q \to S\) is a quadric surface bundle in
\parref{spinors-fanotocoh-surfaces} and \parref{spinors-hypredtocoh-surfaces},
respectively. Then we reduce to the case of relative dimension \(2\) using the
following, the statement of which only makes sense after
\parref{prop-fanotocoh} is established:

\begin{Lemma}
\label{lemma-reductiontosurface}
Let \(\rho \colon Q \to S\) be a quadric \(2\ell\)-bundle with \(\ell > 1\)
and let \(\rho' \colon Q' \to S\) be the family of quadrics of relative
dimension \(2\ell-2\) obtained via hyperbolic reduction of
\(\rho \colon Q \to S\) along a regular section. Then there is an isomorphism
$\alpha\colon \widebar{\mathcal{M}}_d(Q'/S) \cong \widebar{\mathcal{M}}_d(Q/S)$
fitting into a commutative square
\[
\begin{tikzcd}
    \mathbf{F}_{\ell-1}(Q'/S) \ar[r, "\zeta_d'"'] \ar[d, "\beta"']
    & \widebar{\mathcal{M}}_d(Q'/S) \ar[d, "\alpha"] \\
    \mathbf{F}_{\ell}(Q/S) \ar[r, "\zeta_d"] & \widebar{\mathcal{M}}_d(Q/S),
\end{tikzcd}
\]
where $\beta$ is the closed immersion of \parref{lemma-fano-correspondence}
and \(\zeta'_d \colon \mathbf{F}_{\ell-1}(Q'/S) \to \mathcal{M}_d(Q'/S)\) is
the tautological degree \(d\) dual spinor morphism of
\parref{spinor-moduli-stack} associated with \(\rho' \colon Q' \to S\).
\end{Lemma}

\begin{proof}
Embedding the stack of coherent sheaves as an open substack of the stack of
complexes as in \parref{complexes-stack}, view \(\zeta_d\) and \(\zeta_d'\) as
taking values in their respective stacks of complexes. Furthermore, by \parref{cliff-spinor-cohomology}, we may actually view $\zeta_d$ and $\zeta_d'$ as taking values in the open substack parametrizing complexes in the respective  residual categories. The morphism 
\(\alpha \colon \widebar{\mathcal{M}}_d(Q'/S) \to \widebar{\mathcal{M}}_d(Q/S)\)
is that induced by the open immersion from
\parref{theorem-fully-faithful-open-immersion-2} coming from the equivalence \(\Phi\) of
\parref{hypred-equivalence}; that \(\alpha\) factors through
\(\widebar{\mathcal{M}}_d(Q/S)\) and that the square commutes is because
\(\Phi\) preserves dual spinor sheaves by
\parref{cliff-hyperbolic-reduction-spinors}. Since \(\alpha\) is an open
immersion, it remains to show $\alpha$ induces a bijection on geometric
points. We must now show that a given spinor sheaf on \(Q\) is isomorphic to
one defined by a linear space containing the regular section defining \(Q'\),
and this follows from \parref{cliff-dependence-corank-2}.
\end{proof}

\subsectiondash{Stack of spinors for quadric surfaces}
\label{subsection-spinorsonsurface}
Until \parref{spinor-stack-general}, assume that $\ell = 1$, meaning that
\(\rho \colon Q \to S\) is a quadric surface bundle. The standing
hypothesis that \(S_3 = \varnothing\) is equivalent to the statement that every
geometric fiber $Q_{\bar{s}}$ is a reduced quadric surface. Begin with a
more concrete description of the morphism
\(\zeta_d \colon \mathbf{F}_1(Q/S) \to \sCoh_{Q/S}\):

According to \parref{reduction-quadric-surfaces-example}, the $d$-th spinor
sheaf corresponding to a flat family of lines is identified with the
corresponding ideal sheaf, at least up to duals and tensoring with a line
bundle. Since the derived dual of a spinor coincides with its linear dual as in
\parref{cliff-spinors-duality}, and since quadric bundles are Gorenstein, up to
automorphisms of the stack \(\ComplexesStack_{Q/S}\) as in
\parref{complexes-autoequivalences}, this means that the morphism \(\zeta_d\)
may be identified with the morphism
\begin{equation*}
%\label{equn-fanotocohsurface}
\zeta_{\mathrm{ideal}} \colon
\mathbf{F}_1(Q/S) \to \sCoh_{Q/S} \subset \ComplexesStack_{Q/S}
\end{equation*}
which on $T$-points takes a family of lines $L \subset Q_T$ to its ideal sheaf
$\mathcal{I}_{L/Q_T}$. Since both statements \parref{prop-fanotocoh} and
\parref{prop-hypredtocoh} are insensitive to modification by automorphisms of
the target, it suffices to prove the analogous statements for
\(\zeta_{\mathrm{ideal}}\). These proofs are based on the following two
computations:

\begin{Lemma}
\label{lemma-surfacecomputation}
Let $L \subset Q$ be an $S$-flat family of lines with ideal sheaf
$\mathcal{I}$. Then the object
\[
R\rho_* R\sHom_{\sO_Q}(\mathcal{I}, \mathcal{O}_Q)
\in \Dqc(S)
\]
is a rank two vector bundle in degree zero. In particular, the sheaf
$\rho_*\sHom_{\sO_Q}(\mathcal{I}, \mathcal{O}_Q)$ is a vector bundle of
rank two whose formation commutes with arbitrary base change.
\end{Lemma}

\begin{proof}
The statement is local on \(S\), so we may assume that
$Q \subset \mathbf{P}^3_S$ is defined by a section of $\mathcal{O}_{\mathbf{P}^3_S}(2)$
and $L \cong \mathbf{P}^1_S$. Grothendieck duality gives
\[
R\rho_* R\sHom_{\mathcal{O}_{Q}}(\mathcal{I}, \mathcal{O}_Q) =
R\rho_* R\sHom_{\mathcal{O}_{Q}}(\mathcal{I}(-2), \mathcal{O}_Q(-2)) =
R\sHom_{\mathcal{O}_S}(R\rho_*\mathcal{I}(-2), \mathcal{O}_S)[2].
\]
Twisting and pushing the ideal sheaf sequence for \(L\) provides a triangle
\[
R\rho_*\mathcal{I}(-2) \longrightarrow
R\rho_*\mathcal{O}_Q(-2) \longrightarrow
R\rho_* \mathcal{O}_L(-2) \stackrel{+1}{\longrightarrow}
\]
in \(\Dqc(S)\), whose long exact sequence in cohomology sheaves reduces to a short
exact sequence
\[
0 \to R^1 \rho_*\mathcal{O}_L(-2) \to R^2\rho_*\mathcal{I}_L(-2) \to R^2\rho_*\mathcal{O}_Q(-2) \to 0.
\]
Since the external terms are line bundles by standard computations, this shows
that $R\rho_*\mathcal{I}(-2)$ is a rank 2 vector bundle in degree 2, yielding
the first statement. The underived second statement then follows from the
derived one by \parref{lemma-internalhomandbasechange}.
\end{proof}

\begin{Lemma}
\label{lemma-surfacecomputationwithsection}
Let $L \subset Q$ be an $S$-flat family of lines and let
\(\sigma \colon S \to Q\) be a regular section of \(\rho \colon Q \to S\).
Let $\mathcal{I}$ and \(\mathcal{J}\) be the ideal sheaves of \(L\) and
\(\sigma(S)\) in \(Q\), respectively. Then
the object
\[
R\rho_* R\sHom_{\sO_Q}(\mathcal{I}, \mathcal{J}) \in \Dqc(S)
\]
is a line bundle in degree zero. In particular, the sheaf
$\rho_*\sHom_{\sO_Q}(\mathcal{I}, \mathcal{J})$ is a line bundle whose
formation commutes with arbitrary base change.
\end{Lemma}

\begin{proof}
The ideal sheaf sequence for \(\sigma(S) \subset Q\) gives rise to a triangle
\[
R\rho_*R\sHom_{\sO_Q}(\mathcal{I}, \mathcal{J}) \longrightarrow
R\rho_*R\sHom_{\sO_Q}(\mathcal{I}, \mathcal{O}_Q) \longrightarrow
R\rho_*R\sHom_{\sO_Q}(\mathcal{I}, \sigma_*\mathcal{O}_S) \stackrel{+1}{\longrightarrow}
\]
The middle term was identified in \parref{lemma-surfacecomputation} as a rank
\(2\) vector bundle in degree \(0\), and the term on the right is a line bundle
in degree \(0\) because $\sigma$ is a smooth
section (so $\mathcal{I}$ is a line bundle in a neighborhood of $\sigma(S)$) and
\[
R\rho_*R\sHom_{\sO_Q}(\mathcal{I},\sigma_*\mathcal{O}_S)
= R\sHom_{\sO_S}(\sigma^*\mathcal{I}, \mathcal{O}_S).
\]
To complete the proof, it suffices to show that the second arrow induces a
surjection
\[
\rho_*\sHom_{\sO_Q}(\mathcal{I},\sO_Q) \to
\sHom_{\sO_S}(\sigma^*\mathcal{I},\sO_S)
\]
of locally free $\mathcal{O}_S$-modules on degree \(0\) cohomology sheaves.
Over a geometric point $\bar{s} \to S$, this is the evaluation at
\(\sigma(\bar{s})\) map. There is always a line $M$ whose ideal sheaf is isomorphic to that of $L_{\bar{s}}$ and such that $\sigma(\bar{s}) \not\in M$, and this defines a morphism $\mathcal{I}_{\bar{s}} \hookrightarrow \mathcal{O}_{Q_{\bar{s}}}$ which is not zero after evaluation at $\sigma(\bar{s})$. If $Q_{\bar{s}}$ has corank $2$ and $L_{\bar{s}}$ is equal to the singular locus, then this is clear because $\sigma$ is a regular section. Otherwise, $L_{\bar{s}}$ intersects the singular locus in at most one point and the fact that one can find such an $M$ follows from the description of lines with the same ideal sheaf as $L_{\bar{s}}$, see \cite[Section 3]{Addington:Spinors}. 
%the source is the space of
%equations for the pencil of lines linearly equivalent to \(L_{\bar{s}}\);
%otherwise, \(L_{\bar{s}}\) is the singular line in a corank \(2\) quadric
%surface \(Q_{\bar{s}}\), in which case the source may be identified as the
%\(2\)-dimensional space of linear forms from \(\PP^3\) vanishing on
%\(L_{\bar{s}}\). Since \(\sigma(\bar{s})\) is a smooth point of
%\(Q_{\bar{s}}\), there is a function which does not vanish there. This implies
%that the displayed map is surjective.
\end{proof}

We now prove smoothness of
\(\zeta_{\mathrm{ideal}} \colon \mathbf{F}_1(Q/S) \to \sCoh_{Q/S}\) via
the infinitesimal lifting criterion:

\subsectiondash{Proof of \parref{prop-fanotocoh} when $\ell = 1$}
\label{spinors-fanotocoh-surfaces}
We may assume $S$ is Noetherian. Consider a solid commutative diagram
\[
\begin{tikzcd}
    \Spec(R/J) \ar[r] \ar[d] & \mathbf{F}_1(Q/S) \ar[d, "\zeta_{\mathrm{ideal}}"]\\
    \Spec(R)\ar[ur, dashed] \ar[r] &  \sCoh_{Q/S}
\end{tikzcd}
\]
in which $(R, \mathfrak{m}_R)$ is an Artinian local ring over $S$ with residue
field $\kappa$ and $J \subset R$ is an ideal such that
$\mathfrak{m}_R\cdot J = 0$. The task is to construct a dashed arrow making the
diagram commute.

Denote by $i \colon Q_{R/J} \to Q_R$ the inclusion. The data of the diagram
gives an $R/J$-flat family of lines $L' \subset Q_{R/J}$ with ideal sheaf
$\mathcal{I}'$, an $R$-flat coherent sheaf $\mathcal{F}$ on $Q_R$, and an
isomorphism $\alpha: i^*\mathcal{F} \cong \mathcal{I}'$. Flatness of $Q \to S$
means that the ideal sheaf of the closed subscheme $Q_{R/J} \subset Q_R$ is
isomorphic to $J \otimes_{\kappa} \mathcal{O}_{Q_\kappa}$, so there is a
distinguished triangle
\[
J \otimes_\kappa \mathcal{O}_{Q_\kappa} \longrightarrow
\mathcal{O}_{Q_R} \longrightarrow
i_*\mathcal{O}_{Q_{R/J}} \longrightarrow
J \otimes_\kappa \mathcal{O}_{Q_\kappa}[1]
\]
in the derived category of $Q_R$. The isomorphism $\alpha$ gives rise to a
morphism $\beta \colon \mathcal{F} \to i_*\mathcal{O}_{Q_{R/J}}$. The
triangle together with the vanishing of
$\Ext^1_{Q_\kappa}(\mathcal{I}, \mathcal{O}_{Q_\kappa})$
from \parref{lemma-surfacecomputation} over \(\Spec(\kappa)\) shows that
$\beta$ lifts to a morphism
$\gamma \colon \mathcal{F} \to \mathcal{O}_{Q_R}$. Applying \citeSP{046Y}, we
see that $\gamma$ is injective and has flat cokernel, which is necessarily the
structure sheaf of a family of lines in $Q_R$. This family gives the dashed
arrow of the diagram.
\qed

\medskip
Having established smoothness of \(\zeta_{\mathrm{ideal}}\), define as in
\parref{spinor-moduli-stack} open substacks
$$
\mathcal{M} \subset \widebar{\mathcal{M}} \subset \sCoh_{Q/S}
$$
as the images of \(\mathbf{F}_1(Q/S)^\circ\) and \(\mathbf{F}_1(Q/S)\). Thus
\(T\)-points of \(\widebar{\mathcal{M}}\) consist of sheaves
$\mathcal{F} \in \sCoh_{Q/S}(T)$ such that \(\mathcal{F}_{\bar{t}}\) is
isomorphic to the ideal sheaf of a line on $Q_{\bar{t}}$ for every geometric
point $\bar{t} \to T$. Similarly, \(T\)-points
of \(\mathcal{M}\) are those \(\mathcal{F}\) such that each
\(\mathcal{F}_{\bar{t}}\) is isomorphic to the ideal of a line not contained in
the singular locus of $Q_{\bar{t}}$.

Writing \(\mathcal{U}\) for the universal sheaf on \(Q \times_S \mathcal{M}\),
the following provides a useful alternative description of
\(\mathbf{F}_1(Q/S)^\circ\) in terms of maps from an ideal sheaf into
the structure sheaf:

\begin{Lemma}
\label{lemma-abeljac}
$\zeta_{\mathrm{ideal}}^\circ \colon \mathbf{F}_1(Q/S)^\circ \to \mathcal{M}$
factors through an open immersion
\(j \colon \mathbf{F}_1(Q/S)^\circ \to \mathbf{A}(\mathcal{H})\) where
\(\mathcal{H}\) is the rank \(2\) vector bundle on \(\mathcal{M}\) given by
\[
\mathcal{H} \coloneqq
\pr_{2,*}\sHom_{\sO_{Q \times _S \mathcal{M}}}(\mathcal{U}, \mathcal{O}_{Q\times _S \mathcal{M}}).
\]
%  $$
%   \begin{tikzcd}
%       \mathbf{F}_1(Q/S)^\circ \ar[r, hook, "j"] \ar[dr] & \mathbf{A}(\mathcal{H}^\vee) \ar[d] \\
%       & \mathcal{M}
%   \end{tikzcd}
%  $$
%  where $j$ is an open immersion and $\mathcal{H}$ is a rank 2 vector bundle on $\mathcal{M}$.
\end{Lemma}

\begin{proof}
%   As above we identify the morphism (\ref{equn-fanotocoh}) with the one given by the ideal $\mathcal{J}$ of the universal line on $Q \times _S \mathbf{F}_1(Q/S)$.
That \(\mathcal{H}\) is a rank \(2\) vector bundle whose formation commutes
with arbitrary base change may be verified after pullback along the smooth
covering
\(\zeta_{\mathrm{ideal}}^\circ \colon \mathbf{F}_1(Q/S)^\circ \to \mathcal{M}\),
whereupon this is the statement of \parref{lemma-surfacecomputation} for the
quadric bundle $Q \times_S \mathbf{F}_1(Q/S)^\circ \to \mathbf{F}_1(Q/S)^\circ$.
Given an \(S\)-scheme \(T\), an object of the fiber category
$\mathbf{A}(\mathcal{H})(T)$ is a pair $(\mathcal{F}, \varphi)$ where
$\mathcal{F} \in \mathcal{M}_d(Q/S)(T)$ and
$\varphi \in \Gamma(T, \sHom_{\sO_{Q_T}}(\mathcal{F}, \mathcal{O}_{Q_T})) = \Hom_{Q_T}(\mathcal{F}, \mathcal{O}_{Q_T})$,
and a morphism $(\mathcal{F}, \varphi) \to (\mathcal{G}, \psi)$ is an isomorphism
$\alpha\colon \mathcal{F} \to \mathcal{G}$ such that the diagram
\[
\begin{tikzcd}
    \mathcal{F} \ar[r, "\varphi"'] \ar[d, "\alpha"'] & \mathcal{O}_{Q_T} \ar[d, "="] \\
    \mathcal{G} \ar[r, "\psi"] &\mathcal{O}_{Q_T}
\end{tikzcd}
\]
commutes. Consider the open substack
$\mathcal{V} \subset \mathbf{A}(\mathcal{H})$ parametrizing pairs
$(\mathcal{F}, \varphi)$ with $\varphi$ injective; that this is an open
substack can be seen using \citeSP{046Y}. The stack $\mathcal{V}$ is in fact an
algebraic space: Since $\mathcal{M}_d(Q/S)$ is a $\mathbf{G}_m$-gerbe by
\parref{spinor-moduli-simple}, \'etale locally on $T$, any two morphisms
$(\mathcal{F}, \varphi) \to (\mathcal{G}, \psi)$ differ by some
$u \in \sO(T)^\times$. But the only way both $\alpha$ and $u\alpha$ can make
the square above commute when \(\varphi\) and \(\psi\) are injective is if
$u = 1$. Thus $\mathcal{V}$ is the algebraic space whose $T$-points are the set
of subobjects
$\mathcal{F} \subset \sO_{Q_T}$ with $\mathcal{F} \in \mathcal{M}_d(Q/S)$,
and this is precisely $\mathbf{F}_1(Q/S)^\circ$.
\end{proof}

Suppose now that \(\sigma \colon S \to Q\) is a regular section of
\(\rho \colon Q \to S\). Let \(\mu \colon M \to S\) be the corresponding
maximal hyperbolic reduction and, as usual, identify \(M\) via
\parref{lemma-fano-correspondence} as the closed subscheme of
\(\mathbf{F}_1(Q/S)^\circ\) parameterizing lines through \(\sigma\). This
puts us in the setting of:

\subsectiondash{Proof of \parref{prop-hypredtocoh} when $\ell = 1$}
\label{spinors-hypredtocoh-surfaces}
To show that \(\zeta_{\mathrm{ideal}}\rvert_M \colon M \to \mathcal{M}\)
is smooth of relative dimension \(1\), we claim that under the identification
from \parref{lemma-abeljac} of \(\mathbf{F}_1(Q/S)^\circ\) as an open subscheme
of the affine \(2\)-space bundle \(\mathbf{A}(\mathcal{H}) \to \mathcal{M}\),
\(M\) is an open subscheme of the subbundle on the line subbundle
\[
\mathcal{N} \coloneqq
\pr_{2,*}\sHom_{\sO_{Q \times_S \mathcal{M}}}(\mathcal{U},\mathcal{J}_{\mathcal{M}})
\subseteq
\pr_{2,*}\sHom_{\sO_{Q \times_S \mathcal{M}}}(\mathcal{U}, \sO_{Q \times_S \mathcal{M}}) =
\mathcal{H}
\]
where $\mathcal{J}_{\mathcal{M}}$ is the pullback of the ideal sheaf
$\mathcal{J} \subset \mathcal{O}_Q$ of $\sigma(S)$ to
$Q \times_S \mathcal{M}$; that \(\mathcal{N}\) is a line bundle on
$\mathcal{M}$ whose formation commutes with arbitrary base change follows from
\parref{lemma-surfacecomputationwithsection}. Since the structure map
\(\mathbf{A}(\mathcal{N}) \to \mathcal{M}\) is smooth of relative dimension
\(1\), this would give the result. With the notation of \parref{lemma-abeljac},
observe that a $T$-point $(\mathcal{F}, \varphi)$ of
$\mathbf{A}(\mathcal{H}) \supseteq \mathbf{F}_1(Q/S)^\circ$ is in the locally
closed subscheme $M$ if and only if $\varphi$ is injective and the composition
\[
\mathcal{F} \xrightarrow{\varphi} \mathcal{O}_{Q_T} \to \sigma_{T, *}\mathcal{O}_T
\]
is zero. Equivalently, this means that $\varphi$ is injective and factors
through the ideal sheaf $\mathcal{J}_T \subset \mathcal{O}_{Q_T}$
of $\sigma_T(T) \subset Q_T$. The latter condition defines the subbundle
$\mathbf{A}(\mathcal{N}) \subset \mathbf{A}(\mathcal{H})$ and the former
condition, as before, is an open condition on $\varphi$, proving
the claim.

For the remaining statement of \parref{prop-hypredtocoh}, observe that
\(\mathbf{G}_m\) acts on the line bundle \(\mathcal{N}\) with weight \(-1\),
and this means that \(\mathcal{M}\) is, in fact, a trivial
\(\mathbf{G}_m\)-gerbe: see \parref{example-trivialityofgerbe}. Furthermore,
its \(\mathbf{G}_m\)-rigidification may be identified as the complement
of the zero section of \(\mathbf{A}(\mathcal{N})\). The same is true for the
\(\mathbf{G}_m\)-gerbe given by the open substack
\(\zeta_{\mathrm{ideal}}(M) \subseteq \mathcal{M}\). The previous paragraph
now shows that the complement of the zero section of
\(\mathbf{A}(\mathcal{N}) \times_{\mathcal{M}} \zeta_{\mathrm{ideal}}(M)\)
is the maximal hyperbolic reduction \(M\)---the space of \(\varphi\) is
\(1\)-dimensional, so \(\varphi\) is injective if and only it is
nonzero!---from which the result follows.
\qed

% \subsectiondash{Remark}\label{spinor-ideals-on-reducible-quadrics}
% The stacks \(\zeta_{\mathrm{ideal}}(M)\) and \(\mathcal{M}\) differ above
% points of \(S_2\) because for a reducible quadric surface \(Q = P_1 \cup P_2\),
% if \(\mathcal{J}\) is the ideal sheaf of a smooth point \(x \in P_1\) of \(Q\)
% and \(\mathcal{I}\) is the ideal sheaf of a line \(\ell \subset P_2\), then
% although \(\dim_\kk \Hom_Q(\mathcal{I},\mathcal{J}) = 1\),
% no member \(\varphi \colon \mathcal{I} \to \mathcal{J}\) is injective.

\subsectiondash{Stack of spinors in general}
\label{spinor-stack-general}
Having established the surface cases, we now prove \parref{prop-fanotocoh} and
\parref{prop-hypredtocoh} for quadric \(2\ell\)-fold bundles \(\rho \colon Q
\to S\) with \(\ell > 1\). Smoothness of \(\zeta_d\) is established via
induction on dimension together with, once again, the infinitesimal lifting
criterion:

\begin{proof}[Proof of \parref{prop-fanotocoh}]
View \(\zeta_d\) as a morphism
$\mathbf{F}_{\ell}(Q/S) \to \ComplexesStack_{\Ku(Q)/S}$, as is possible by
\parref{cliff-spinor-cohomology}. Proceed by induction on $\ell \geq 1$. The
base case $\ell = 1$ handled in \parref{spinors-fanotocoh-surfaces}. So let
\(\ell > 1\) and assume that \parref{prop-fanotocoh} holds for all quadric
bundles of relative dimension $2 (\ell - 1)$.

We may assume $S$ is Noetherian. Consider a solid commutative diagram
\begin{equation}
\label{equn-fanotocohsqr}
\tag{\(\star\star\)}
\begin{tikzcd}
  \Spec(R/J) \ar[r, "t"] \ar[d] & \mathbf{F}_\ell(Q/S) \ar[d, "\zeta_d"]\\
  \Spec(R)\ar[ur, dashed, "v"] \ar[r, "u"] &  \ComplexesStack_{\Ku(Q)/S}
\end{tikzcd}
\end{equation}
in which $(R, \mathfrak{m}_R)$ is an Artinian local ring over $S$ with
algebraically closed residue field $\kappa$ and $J \subset R$ is an ideal such
that $\mathfrak{m}_R\cdot J = 0$. The task is to construct a dashed arrow
\(v\) making the diagram commute.

The morphism $t$ corresponds to a flat family $L \subset Q_{R/J}$ of
$\ell$-planes. The standing assumption that $S_3 = \varnothing$ means that the
singular locus of $Q_\kappa$ has dimension $\leq 1$. Since $L_\kappa$ has
dimension $\ell > 1$, there exists a point $x \in L_\kappa(\kappa)$ not
contained in the singular locus of $Q_\kappa$. Then since $L$ is smooth over
$R/J$, there exists a section $\bar{\sigma}$ of $L \to \Spec(R/J)$ whose
restriction to $\Spec(\kappa)$ is $x$. Moreover, since $\rho \colon Q \to S$ is
smooth at $x$, there is a section $\sigma$ of $\rho_R \colon Q_R \to \Spec(R)$
extending $\bar{\sigma}$. Then $\sigma$ is a regular section of $Q_R$ over
\(R\), the associated hyperbolic reduction $Q'$ is a quadric $2(\ell -1)$-fold
bundle over $R$, and \(L\) induces a flat family \(L' \subset Q'_{R/J}\) of
\((\ell-1)\)-planes. This provides the morphism \(t'\) in the commutative
diagram
\[
\begin{tikzcd}
\Spec(R/J) \ar[r, "t'"'] \ar[dr, "t", bend right=20] &
\mathbf{F}_{\ell-1}(Q'/R) \ar[d, "\beta"] \ar[r, "\zeta_d'"'] &
\ComplexesStack_{\Ku(Q')/R} \ar[d, "\alpha"] \\
&
\mathbf{F}_\ell(Q_R/R) \ar[r, "\zeta_d"] &
\ComplexesStack_{\Ku(Q_R)/R}
\end{tikzcd}
\]
where \(\zeta_d\) and \(\zeta_d'\) are the morphisms from
\parref{spinor-moduli-stack}; $\beta$ is the closed immersion coming from the
identification from \parref{lemma-fano-correspondence} of $(\ell-1)$-planes on
$Q'$ with $\ell$-planes on $Q_R$ going through $\sigma$; $\alpha$ is the
open immersion \parref{theorem-fully-faithful-open-immersion-2} coming from the
equivalence of \parref{hypred-equivalence}; and that the square commutes
follows from \parref{cliff-hyperbolic-reduction-spinors}.

The two commutative diagrams together show that
\(u \colon \Spec(R) \to \ComplexesStack_{\Ku(Q_R)/R}\) restricted to the closed
subscheme \(\Spec(R/I)\) may be written as
\[
u\rvert_{\Spec(R/I)} =
\zeta_d \circ t =
\alpha \circ \zeta_d' \circ t'.
\]
Since \(\alpha\) is an open immersion and \(\Spec(R)\) and \(\Spec(R/I)\)
have the same support, \(u\) uniquely factors through \(\alpha\), providing a
morphism \(u' \colon \Spec(R) \to \ComplexesStack_{\Ku(Q')/R}\) whose restriction
to \(\Spec(R/J)\) is the composition \(\zeta_d' \circ t'\); in other words,
\(t'\) and \(u'\) together fit into a solid diagram as in
\eqref{equn-fanotocohsqr} for \(Q'\). By induction, \(\zeta_d'\) is smooth,
and so there exists a morphism
\(v' \colon \Spec(R) \to \mathbf{F}_{\ell-1}(Q'/R)\) filling in the
corresponding dashed arrow. It is now straightforward to see that
\(v \coloneqq \beta \circ v' \colon \Spec(R) \to \mathbf{F}_\ell(Q/S)\) fits
as a dashed arrow making the square \eqref{equn-fanotocohsqr}
commute.
\end{proof}

Suppose now that \(\rho \colon Q \to S\) admits a family \(L \subset Q\) of
\((\ell-1)\)-planes contained in its smooth locus, and let
\(\mu \colon M \to S\) be the associated maximal hyperbolic reduction.

\begin{proof}[Proof of \parref{prop-hypredtocoh}]
The statements may be verified \'etale locally on \(S\), so passing to a
cover, we may additionally assume that \(L\) contains a flat family
\(L' \subset L\) of \((\ell-2)\)-planes. Hyperbolic reduction along \(L'\)
yields a quadric surface bundle \(\rho' \colon Q' \to S\). The
\(\ell\)-plane \(L\) induces a regular section of \(\rho'\) and, by
comparing functors of points using \parref{lemma-fano-correspondence},
the assoicated maximal hyperbolic reduction may be canonically identified with
\(\mu \colon M \to S\). Now, iteratively applying
\parref{lemma-reductiontosurface} shows that
\(\zeta_d \rvert_M = \alpha \circ \zeta_d' \rvert_M\) where \(\alpha\) is
the isomorphism \(\mathcal{M}_d(Q'/S) \cong \mathcal{M}_d(Q/S)\) induced
by the equivalence \parref{hypred-equivalence}. In this way, the result is
reduced to the surface case, which was established in
\parref{spinors-hypredtocoh-surfaces}.
% We know now that $\mathcal{M}$ is a $\mathbf{G}_m$-gerbe over an algebraic space $M$ and we have to show the composition $\bar{Q} \to \mathcal{M} \to M$ is an open immersion. This can be checked \'etale locally on $S$, so we may assume there is a flat family $L' \subset L$ of $(\ell-2)$-planes contained in $L$. Then let $Q'/S$ be the hyperbolic reduction of $Q$ with respect to $L'$, a flat quadric surface bundle. Then $Q'/S$ itself has a smooth section corresponding to $L$, and the hyperbolic reduction with respect to this section is $\bar{Q}$, as can be seen with the functor of points description of the hyperbolic reduction, see Lemma \ref{lemma-fano-correspondence}. Then as in the proof of Proposition \ref{prop-fanotocoh}, we have a commutative diagram
% $$
% \begin{tikzcd}
% \bar{Q'} \ar[d, "="] \ar[r] &\mathcal{M}_d(Q'/S) \ar[d, hook, "g"] \\
% \bar{Q} \ar[r] & \mathcal{M}_d(Q/S)
% \end{tikzcd}
% $$
% where $g$ is the open immersion coming from the equivalence of categories between the Kuznetsov component of $Q'/S$ and $Q/S$. By the surface case, we know the top horizontal arrow induces an open immersion to the $\mathbf{G}_m$-rigidification of $\mathcal{M}_d(Q'/S)$ and we deduce that the same holds for the bottom horizontal arrow. 
\end{proof}

\subsectiondash{Coarse moduli space}\label{spinor-coarse-moduli}
As already mentioned following \parref{prop-fanotocoh}, \(\mathcal{M}_d(Q/S)\)
consists of simple sheaves by \parref{spinor-moduli-simple} and is therefore
a $\mathbf{G}_m$-gerbe over an algebraic space $\mathrm{M}_d(Q/S)$ which we
will refer to as the \emph{coarse moduli space of spinors} on
\(\rho \colon Q \to S\). By \parref{prop-hypredtocoh} and
\parref{lemma-sectionetalelocally}, \'etale locally on the source and target,
$\mathrm{M}_d(Q/S) \to S$ is identified with a maximal hyperbolic reduction of
$\rho \colon Q \to S$. In fact, away from $S_2$, the morphism
$\mathrm{M}_d(Q/S) \to S$ is isomorphic to a maximal hyperbolic reduction
\'etale locally on the target. Over geometric points of \(S_2\) on the other hand, a
maximal hyperbolic reduction provides one of the two connected components:

\begin{Lemma}\label{spinors-coarse-moduli-fibres}
The fibre of \(\mathrm{M}_d(Q/S) \to S\) over a geometric point
\(\bar{s} \to S\) is isomorphic to
\[
\mathrm{M}_d(Q/S)_{\bar{s}} \cong
\begin{dcases*}
\bar{s} \sqcup \bar{s} & if \(\bar{s} \to S \setminus S_1\), \\
\bar{s}[\epsilon] & if \(\bar{s} \to S_1 \setminus S_2\), and \\
\PP^1_{\bar{s}} \sqcup \PP^1_{\bar{s}} & if \(\bar{s} \to S_2\).
\end{dcases*}
\]
If \(\mu \colon M \to S\) is a maximal hyperbolic reduction of
\(\rho \colon Q \to S\), then the morphism
\(\zeta_d \rvert_M \colon M \to \mathrm{M}_d(Q/S)\) is an isomorphism over
\(S \setminus S_2\).
\end{Lemma}

\begin{proof}
The statement is local on \(S\), so assume that \(\rho \colon Q \to S\) admits
a maximal hyperbolic reduction \(\mu \colon M \to S\). By
\parref{prop-hypredtocoh}, the morphism \(\zeta_d\rvert_M\) is an open
immersion. So for the statements away from \(S_2\), it remains to see it is
surjective on geometric points. On the one hand, the geometric fibres of
$\mu \colon M \to S$ have cardinality \(2\) over points
of $S \setminus S_1$, and \(1\) over points of $S_1 \setminus S_2$.
%On the other hand, the morphism $\mathrm{M}_d(Q/S) \to S$ is quasi-finite away from \(S_2\) by \parref{prop-hypredtocoh},
On the other hand, the fibres of \(\mathrm{M}_d(Q/S) \to S\) away from \(S_2\)
are \(0\)-dimensional by \parref{prop-hypredtocoh}, and its geometric fibers
have cardinality at most 2 over points of
$S \setminus S_1$ and at most 1 over $S_1 \setminus S_2$ since they are
surjected on respectively by the Fano scheme of a smooth quadric $2\ell$-fold
and a corank 1 quadric $2\ell$-fold. Together, this implies that
\(\zeta_d\rvert_M\) must be surjective on geometric points away from \(S_2\),
as desired.

To analyze geometric fibres over \(S_2\), apply
\parref{lemma-reductiontosurface} to reduce this to the case of a quadric
surface bundle, at which point the problem is to show that if \(Q\) is a
reducible, reduced quadric surface over an algebraically closed field \(\kk\),
then $\mathrm{M}_d(Q/\kk)$ is a disjoint union of two projective lines. Since
\(\mathrm{M}_d(Q/\kk)\) admits a surjection from
\[
\mathbf{F}_1(Q/\kk)^\circ \cong
\PP^2_{\kk} \setminus \{\mathrm{pt}\} \sqcup
\PP^2_{\kk} \setminus \{\mathrm{pt}\}
\]
by construction, it too has at most two connected components. Recall from
\parref{subsection-spinorsonsurface} that this surjection may be identified
with the morphism \(\zeta_{\mathrm{ideal}}^\circ\) which takes a line to its
ideal sheaf. This implies that the images of the two connected components of
\(\mathbf{F}_1(Q/\kk)^\circ\) are disjoint: The ideal sheaf of a line on one
irreducible component of $Q$ cannot be isomorphic to the ideal sheaf of a line
on the other. Since the images are open by \parref{prop-fanotocoh},
$\mathrm{M}_d(Q/\kk)$ must have exactly two connected components and they consist
of ideal sheaves of lines on the respective irreducible components of $Q$.

Finally, let $M$ be the hyperbolic reduction of $Q$ with respect to a smooth
point on an irreducible component $W \subset Q$. It now suffices to show
that the open immersion
\[
\zeta_{\mathrm{ideal}}\rvert_M \colon
M \cong \mathbf{P}^1_{\kk} \to \mathrm{M}_d(Q/\kk)
\]
of \parref{prop-hypredtocoh} surjects onto the connected component
parametrizing ideal sheaves of lines on $W$. This follows from the fact that
the ideal sheaves of two lines on the reducible quadric $Q$ which are not equal
to the singular locus of $Q$ are isomorphic if and only if they lie on the same
irreducible component and have the same intersection point with the singular
locus.
See \cite[Section 3]{Addington:Spinors}.
\end{proof}

% \begin{Lemma}\label{spinor-away-from-S2}
% Let \(\mu \colon M \to S\) be a maximal hyperbolic reduction of
% \(\rho \colon Q \to S\). If \(S_2 = \varnothing\), then the morphism
% $\zeta_d\rvert_M \colon M \to \mathrm{M}_d(Q/S)$ is an isomorphism.
% \end{Lemma}

% If $S_2 \neq \varnothing$ this is not true. In fact fibers of
% $\mu \colon M \to S$ over a geometric point of $\bar{s} \to S_2$ is isomorphic
% to $\mathbf{P}^1_{\bar{s}}$, whereas:

% \begin{Lemma}
% \label{lemma-fiberoverS2}
% $\mathrm{M}_d(Q/S)_{\bar{s}} \cong \PP^1_{\bar{s}} \sqcup \PP^1_{\bar{s}}$
% for any geometric point \(\bar{s} \to S_2\).
% \end{Lemma}

% In summary, the geometric fibers of $\mathrm{M}_d(Q/S) \to S$ are isomorphic to

These arguments show that, if \(\rho \colon Q \to S\) admits a maximal
hyperbolic reduction \(\mu \colon M \to S\), then the base change of the
morphism $M \to \mathrm{M}_d(Q/S)$ of \parref{prop-hypredtocoh} to a geometric
point of $S_2$ is the inclusion of one of two connected components of
$\PP^1 \sqcup \PP^1$; moreover, this means that \(\mathrm{M}_d(Q/S)\) is in
general not separated over $S$ if $S_2 \neq \varnothing$ as points over
$S \setminus S_2$ can specialize to points of both connected components of a
geometric fiber over a point of $S_2$. We aim to construct an open subspace of
$\mathrm{M}_d(Q/S)$ which is \'etale locally on the base isomorphic to a
maximal hyperbolic reduction. Toward this, observe that the restriction of
\(\mathrm{M}_d(Q/S) \to S\) to \(S_2\) factors through the canonical double
cover \(\tilde{S}_2 \to S_2\) of the corank \(2\) locus provided by
\parref{quadric-bundles-fano-schemes-corank-2}, separating the two projective
lines in the fibre. More precisely:

%The discussion of \parref{spinor-moduli-open-immersion-surfaces} and
%\parref{spinor-moduli-open-proof}, together with the dimensional reduction
%statement \parref{spinor-moduli-reduction}, implies that \(\mathrm{M}_d(Q/S)
%\to S\) is locally isomorphic to a maximal hyperbolic reduction of \(Q \to S\)
%away from \(S_2\). In fact, combined further with
%\parref{spinor-moduli-injectivity} gives the following statement:

%\begin{Lemma}\label{spinor-coarse-moduli-reduction}
%Let \(\mu \colon M \to S\) be a maximal hyperbolic reduction of
%\(\rho \colon Q \to S\). Then there is a canonical open immersion
%\(M \to \mathrm{M}_d(Q/S)\).
%\end{Lemma}

%\begin{proof}
%Consider the composite of the closed immersion
%\(M \to \mathbf{F}_\ell(Q/S)^\circ\) from \parref{lemma-fano-correspondence},
%the morphism
%\(\zeta_d \colon \mathbf{F}_\ell(Q/S)^\circ \to \mathcal{M}_d(Q/S)\) from
%\parref{spinor-moduli-stack}, and finally the coarse moduli space map
%\(\mathcal{M}_d(Q/S) \to \mathrm{M}_d(Q/S)\). The maps
%\(M \to \mathcal{M}_d(Q/S)\) and \(\mathcal{M}_d(Q/S) \to \mathrm{M}_d(Q/S)\)
%are universally injective and smooth of relative dimensions \(1\) and
%\(-1\), respectively, by \parref{spinor-moduli-open-immersion-surfaces} and
%\parref{spinor-moduli-injectivity} for the former, and since the latter
%is a \(\mathbf{G}_m\)-gerbe. Thus \(M \to \mathrm{M}_d(Q/S)\)
%is universally injective and \'etale, so it is an open immersion.
%\end{proof}

\begin{Proposition}\label{spinor-moduli-corank-2-stein}
Assume that \(\rho \colon Q \to S\) satisfies $S = S_2$ scheme-theoretically
and $S_3 = \varnothing$. Then the structure morphism
\(\mathrm{M}_d(Q/S) \to S\) factors through the canonical \'etale double cover
\(\tilde{S} \to S\) constructed in
\parref{quadric-bundles-fano-schemes-corank-2}. Furthermore,
\(\mathrm{M}_d(Q/S) \to \tilde{S}\) has geometrically connected fibres.
\end{Proposition}

One way to prove this might go as follows: Generalizing \cite[Lemma 6.1 and
Proposition 6.2]{Addington:Spinors},
the spinor sheaf \(\mathcal{S}_d\) associated with an isotropic family of
\(\ell\)-planes in \(Q \to S\), fibrewise intersecting the singular locus at a
single point, has a unique simple Clifford submodule isomorphic to the spinor
associated with the corresponding family of \((\ell+1)\)-planes. Let \(\bar{Q} \to S\)
be the smooth quadric bundle of relative dimension \(2\ell-2\) obtained by
projecting away from the radical. An analogue of \parref{cliff-cone-spinors}
shows that this determines a spinor bundle on \(\bar{Q}\), thereby producing
a morphism \(\mathcal{M}_d(Q/S) \to \mathcal{M}_d(\bar{Q}/S)\). Passing to
coarse moduli spaces and comparing with
\parref{quadric-bundles-fano-schemes-corank-2} would then give the result.

Rather than pursuing this route, we prove \parref{spinor-moduli-corank-2-stein}
directly by observing that \(\tilde{S}_2\) parameterizes geometric connected
components of the fibres \(\mathbf{F}_\ell(Q/S)^\circ \to S\) over \(S_2\),
and that the morphism \(\mathbf{F}_\ell(Q/S)^\circ \to \mathrm{M}_d(Q/S)\)
induced by \(\zeta_d\) reflects connected components as it has geometrically
connected fibres. We begin with a series of general results, cumulating in a
rigidity lemma \parref{lemma-rigidity2}, from which we may directly deduce
\parref{spinor-moduli-corank-2-stein}.

\begin{Lemma}
\label{lemma-connectedness}
Let $f \colon X \to S$ be a continuous map of topological spaces. Assume that
$f$ is open, $S$ is connected, and the fibers $f^{-1}(s) \subseteq X$ are
connected. Then $X$ is connected.
\end{Lemma}

\begin{proof}
Since $S \neq \varnothing$ and each fiber of \(f\) is non-empty,
$X \neq \varnothing$. Suppose $U \subseteq X$ is a non-empty open and closed
subset. Then $V \coloneqq f(U) \subseteq S$ is open and non-empty, and for
$s \in S$ we have
\begin{equation*}
%\label{equn-connectedness}
s \in V \iff
f^{-1}(s) \cap U \neq \varnothing \iff
f^{-1}(s) \subseteq U \iff
s \not \in f(X \setminus U)
\end{equation*}
since $f^{-1}(s)$ is connected. Thus $V = S \setminus f(X \setminus U)$ is
closed and open, so $V = S$. Applying
%(\ref{equn-connectedness})
the equivalences again shows that $f^{-1}(s) \subseteq U$ for all $s \in S$,
hence $U = X$. 
\end{proof}

\begin{Lemma}
\label{lemma-samemorph}
Let \(f, g \colon X \to Y\) be morphisms of algebraic spaces over an algebraic
space \(S\). Assume that \(X\) is connected, \(Y\) is separated and
\'etale over \(S\), and that there is \(x \in |X|\) such that
\(f(x) = g(x) \in |Y|\). Then \(f = g\).
\end{Lemma}

\begin{proof}
The hypotheses on \(Y\) mean that the diagonal
$\Delta_Y \colon Y \to Y \times _S Y$ is both an open and a
closed immersion. Thus $U \coloneqq (f, g)^{-1}(\Delta_Y) \subseteq X$ is
an open and closed subspace containing the point \(x\), and so \(U = X\) by
connectedness. This implies \(f = g\).
%Then $|U| \subset |X|$ is open and closed and contains the point $x$, hence $U = X$. But this implies $f = g$.
\end{proof}

\begin{Lemma}
\label{lemma-rigsection}
Let $p\colon X \to S$ and $q \colon Y \to S$ be morphisms of algebraic spaces.
Assume  $p$ is flat, finitely presented, and has geometrically connected
fibers. Assume $q$ is separated and \'etale. Then any morphism $f \colon X \to
Y$ over $S$ is of the form $\eta \circ p$ for some unique section $\eta$ of
$q$.
\end{Lemma}

\begin{proof}
Uniqueness of $\eta$ follows from the fact that $X \to S$ is flat, surjective,
and finitely presented since $Y$ is a sheaf for the fppf topology. For
existence, it suffices to show that the two morphisms
$f \circ \pr_i \colon X \times _S X \to Y$ are equal. For
this, one reduces to the case when $S$ is the spectrum of a local ring, hence
connected. Then $X \times _S X$ is connected by
\parref{lemma-connectedness}: $X \times _S X \to S$ is open since it is
flat and finitely presented, and has connected fibers since the product
of geometrically connected spaces of finite type over a field is geometrically
connected. Now, for any $x \in |X|$, there is a point
$(x,x) \in |X \times _S X|$ and $f \circ \pr_1$ and $f \circ \pr_2$ agree on
this point, so by \parref{lemma-samemorph}, the two morphisms are equal.
\end{proof}

\begin{Lemma}
\label{lemma-rigidity2}
Let $S$ be an algebraic space. Let $f \colon P \to X$ be a morphism of algebraic spaces
over $S$ which is flat, finitely presented, and has geometrically connected
fibers. Let $Y \to S$ be a separated, \'etale morphism of algebraic spaces.
Then any $S$-morphism $h \colon P \to Y$ factors through uniquely through $X$. 
\end{Lemma}

\begin{proof}
We must show that there is a unique \(S\)-morphism \(g \colon X \to Y\) such
that \(h = g \circ f\). Consider the \(X\)-morphism
\((f,h) \colon P \to X \times_S Y\). Then the data of \(g\) is equivalent
to a section \(s \colon X \to X \times_S Y\) of \(\pr_1\) satisfying
\((f,h) = s \circ f\). Replacing \(S\) by \(X\), \(Y\) by \(X \times_S Y\),
and \(X\) by \(P\) then reduces this to \parref{lemma-rigsection}.
\end{proof}

\begin{proof}[Proof of \parref{spinor-moduli-corank-2-stein}]
Consider the Fano incidence correspondence
\[
\mathbf{F}_{\ell,\ell+1}(Q/S)^\circ \coloneqq
\big\{
([P_\ell], [P_{\ell+1}]) \in
\mathbf{F}_\ell(Q/S)^\circ \times_S \mathbf{F}_{\ell+1}(Q/S) :
P_\ell \subset P_{\ell+1}
\big\}
\]
parameterizing flags of \(\ell\)- and \((\ell+1)\)-planes in the corank \(2\)
quadric \(2\ell\)-fold bundle \(\rho \colon Q \to S\).
First projection maps this isomorphically onto \(\mathbf{F}_\ell(Q/S)^\circ\):
Indeed, as in the argument of \parref{quadric-bundles-fano-schemes-corank-2},
the \((\ell+1)\)-plane \(P_{\ell+1}\) contains \(\Sing\rho\), so it may
be uniquely determined as the span of \(\Sing\rho\) and the \(\ell\)-plane
\(P_\ell\)---note that \(P_\ell\) intersects the singular locus in exactly one
point in each fibre, and that this works in families follows from
\parref{quadric-bundles-intersect-radical}. Identifying
\(\mathbf{F}_\ell(Q/S)^\circ\) with the incidence correspondence provides a
morphism \(\mathbf{F}_\ell(Q/S)^\circ \to \mathbf{F}_{\ell+1}(Q/S)\), and hence
a commutative diagram of solid arrows:
\[
\begin{tikzcd}
\mathbf{F}_\ell(Q/S)^\circ \rar["\zeta_d^\circ"'] \dar
& \mathrm{M}_d(Q/S) \dar \ar[dl,dashed]\\
\tilde{S} \rar
& S\punct{.}
\end{tikzcd}
\]

We seek to show that there is a morphism \(\mathrm{M}_d(Q/S) \to \tilde{S}\)
filling in the dashed arrow and making the diagram commute. This would
follow from \parref{lemma-rigidity2} upon verifying that \(\zeta_d^\circ\) has
geometrically connected fibres, and this reduces to the situation where \(S\)
is the spectrum of an algebraically closed field \(\kk\). The arguments of
\parref{spinors-coarse-moduli-fibres} show that points of
\(\mathrm{M}_d(Q/\kk)\) correspond to a choice of connected component of
\(\mathbf{F}_\ell(Q)^\circ\) and a point \(x \in \Sing Q\). Furthermore,
the fibre of \(\zeta_d^\circ\) over consists of \(\ell\)-planes in the chosen 
connected component which intersects \(\Sing Q\) precisely at \(x\). Thus it
remains to show that
\[
\mathbf{F}_\ell(Q;x)^\circ \coloneqq
\{[P] \in \mathbf{F}_\ell(Q) : P \cap \Sing Q = \{x\}\}
\]
has just two connected components. But projection from the singular locus
identifies this with \(\mathbf{F}_{\ell-1}(\bar{Q})\), where \(\bar{Q}\) is
the smooth quadric \((2\ell-2)\)-fold at the base of the cone over \(\Sing Q\),
and the latter has two connected components. % as it is a homogeneous space for an orthogonal group.
\end{proof}

\subsectiondash{Choosing a family of spinors}\label{spinor-coarse-moduli-section}
The discussion of \parref{spinors-coarse-moduli-fibres} shows that in order to
construct a space locally isomorphic to a maximal hyperbolic reduction of
\(\rho \colon Q \to S\) out of \(\mathrm{M}_d(Q/S)\), we must be able to
consistently choose one of the two connected components in the fibres over the
corank \(2\) locus \(S_2\). This is possible precisely when the double cover
\(\tilde{S}_2 \to S_2\) splits.

Consider this case, and suppose \(\sigma \colon S_2 \to \tilde{S}_2\) is a
section of the double cover \(\tilde{S}_2 \to S_2\) from
\parref{quadric-bundles-fano-schemes-corank-2}. Pulling back the morphism
\(\mathrm{M}_d(Q/S)\rvert_{S_2} \to \tilde{S}_2\)
from \parref{spinor-moduli-corank-2-stein} along \(\sigma\) produces a
closed subspace of \(\mathrm{M}_d(Q/S)\) whose open complement
\[
M_\sigma \coloneqq
\mathrm{M}_d(Q/S) \setminus
(S_2 \times_{\sigma, \tilde{S}_2} \mathrm{M}_d(Q/S)\rvert_{S_2})
\]
we shall see is separated over \(S\). Note that here $\widetilde{S}_2 = \sigma (S_2) \coprod T$ where $T$  is the image of the complementary section, so the fiber product above is simply the preimage of the closed and open set $\sigma(S_2) \subset \tilde{S}_2$.

\subsectiondash{Maximal hyperbolic reduction and sections}
\label{spinor-coarse-moduli-section-reduction}
One situation in which \(\tilde{S}_2 \to S_2\) splits is when
\(\rho \colon Q \to S\) carries a regular isotropic \((\ell-1)\)-bundle
\(\PP\mathcal{F}\). Let \(\mu \colon M \to S\) be the associated maximal
hyperbolic reduction, as in \parref{reduction-fourier-mukai}. Consider the
closed subscheme
\[
Z \coloneqq
\{[P] \in \mathbf{F}_{\ell+1}(Q/S) \rvert_{S_2} :
\PP\mathcal{F}\rvert_{S_2} \subset P\}
\]
parameterizing isotropic \((\ell+1)\)-planes over \(S_2\) containing the
restriction of \(\PP\mathcal{F}\). Then \parref{lemma-fano-correspondence}
identifies \(Z\) as the scheme of lines of \(M \to S\) restricted to \(S_2\):
this implies that the map \(Z \to S_2\) is an isomorphism. Post-composing the
inclusion of \(Z\) with the map
\(\mathbf{F}_{\ell+1}(Q/S)\rvert_{S_2} \to \tilde{S}_2\) from
\parref{quadric-bundles-fano-schemes-corank-2} arising from the Stein
factorization provides a section
\(S_2 \xrightarrow{\sim} Z \to \tilde{S}_2\). Call the complementary section
\(\sigma \colon S_2 \to \tilde{S}_2\).

This situation thus provides two distinct ways of obtaining open subspaces of
\(\mathrm{M}_d(Q/S)\): On the one hand, the maximal hyperbolic
reduction \(\mu \colon M \to S\) canonically embeds as an open via
\parref{prop-hypredtocoh} and \parref{spinors-coarse-moduli-fibres}. On the
other hand, the section \(\sigma \colon S_2 \to \tilde{S}_2\) provides an open
subspace \(M_\sigma\) via \parref{spinor-coarse-moduli-section}. The two
methods produce the same open subspace:

\begin{Lemma}
\label{spinor-moduli-compatible}
The open immersion \(M \to \mathrm{M}_d(Q/S)\) provides an identification
\(M \cong M_\sigma\).
\end{Lemma}

\begin{proof}
The two opens agree away from the corank $2$ locus by
\parref{spinors-coarse-moduli-fibres}, so we may assume
\(\rho \colon Q \to S\) is everywhere of corank \(2\). In this situation,
\(\tilde{S} \cong \sigma(S) \sqcup Z\) decomposes as a disjoint union
of the two sections over \(S\), and the open subspace corresponding to the
section \(\sigma\) is taken in \parref{spinor-coarse-moduli-section} to be
\[
M_\sigma \coloneqq \mathrm{M}_d(Q/S) \times_{\tilde{S}} Z.
\]

As explained in the proof of \parref{spinor-moduli-corank-2-stein}, the
structure map \(\mathbf{F}_\ell(Q/S)^\circ \to S\) factors through
\(\tilde{S}\), meaning it has two connected components over \(S\). The
embedding of \parref{lemma-fano-correspondence} identifies \(M\) as
the closed subscheme
\[
M \cong \{[P] \in \mathbf{F}_\ell(Q/S)^\circ : \PP\mathcal{F} \subset P\}
\]
and so, comparing with the definition of \(Z \subset \mathbf{F}_{\ell+1}(Q/S)\)
from \parref{spinor-coarse-moduli-section-reduction}, the composite map
\(M \to \tilde{S}\) factors through \(Z\). This implies the result.
\end{proof}

Write $\mathcal{M}_\sigma$ for the restriction of the $\mathbf{G}_m$-gerbe
$\mathcal{M}_d(Q/S) \to \mathrm{M}_d(Q/S)$ to the open $M_\sigma \subset \mathrm{M}_d(Q/S)$.

\begin{Corollary}
\label{corollary-kernelscorrespond}
    There is an equivalence $M \times B\mathbf{G}_m \cong \mathcal{M}_\sigma$ such that, under the associated equivalence
    $$
    (M \times B \mathbf{G}_m) \times _S Q \cong \mathcal{M}_\sigma \times _S Q,
    $$
    the universal degree $d$ dual spinor sheaf $\tilde{\mathcal{S}}_d^\vee$ on the right hand side corresponds to the degree $d$ dual spinor sheaf $\mathcal{S}_d^\vee$ associated to the tautological family of $\ell$-planes in $M \times _S Q$ placed in $\mathbf{G}_m$-weight one. 
\end{Corollary}

\begin{proof}
By its definition, the open immersion $M \to M_d(Q/S)$ with image
$M_\sigma$ factors through the morphism $M \to \mathcal{M}_d(Q/S)$ for
which the pullback of $\tilde{\mathcal{S}}_d^\vee$ along $M \times _S Q \to
\mathcal{M}_d(Q/S) \times _S Q$ is $\mathcal{S}_d^\vee$. By general
properties of gerbes, $M \to \mathcal{M}_d(Q/S)$ factors as $M \to M \times
B\mathbf{G}_m \to \mathcal{M}_d(Q/S)$ where the first morphism is the
tautological section and the second induces the identity on autormorphism
groups, see \citeSP{06QG}. Thus there is a commutative diagram 
 % https://q.uiver.app/#q=WzAsNCxbMSwwLCJcXG1hdGhjYWx7TX1fXFxzaWdtYSJdLFswLDEsIk0iXSxbMCwwLCJNIFxcdGltZXMgQiBcXG1hdGhiZntHfV9tIl0sWzEsMSwiTV9cXHNpZ21hIl0sWzEsMF0sWzAsM10sWzEsMywiXFxjb25nIl0sWzIsMCwiXFxjb25nIl0sWzEsMl1d
\[\begin{tikzcd}
	{M \times B \mathbf{G}_m} & {\mathcal{M}_\sigma} \\
	M & {M_\sigma}
	\arrow["\cong", from=1-1, to=1-2]
	\arrow[from=1-2, to=2-2]
	\arrow[from=2-1, to=1-1]
	\arrow[from=2-1, to=1-2]
	\arrow["\cong", from=2-1, to=2-2]
\end{tikzcd}\]
The sheaf $\tilde{\mathcal{S}}_d^\vee$ on $\mathcal{M}_\sigma$ has $\mathbf{G}_m$-weight one, and so it corresponds to an object of $\mathbf{G}_m$-weight one under $M \times B \mathbf{G}_m \cong \mathcal{M}_\sigma \times _S Q$ since this equivalence induces the identity on bands. Since the pullback of the corresponding object under the tautological section $M \to M \times B \mathbf{G}_m$ is $\mathcal{S}_d^\vee$, the result follows.
\end{proof}

\subsectiondash{Main result}
\label{spinor-coarse-moduli-section-mainresult}
We now arrive at the main result of this section, which geometrically
identifies the Kuznetsov component of a quadric bundle of even relative
dimension in a rather general setting. The description is explicit in the
sense that the algebraic space \(M_\sigma\) below is, by
\parref{spinor-moduli-compatible}, \'etale locally on $S$ isomorphic to a
maximal hyperbolic reduction of $Q \to S$. In
particular, $M_\sigma \to S$ is proper, finite locally free over $S \setminus
S_1$, and a $\mathbf{P}^1$-bundle over $S_2$, with geometric fibres
isomorphic to two points over $S \setminus S_1$,
$\Spec(\kk[\epsilon])$ over $S_1 \setminus S_2$, and $\mathbf{P}^1$
over $S_2$.

\begin{Theorem}\label{spinor-main-theorem}
Let \(\rho \colon Q \to S\) be a quadric bundle of relative dimension \(2\ell\).
Assume that:
\begin{enumerate}
\item\label{spinor-main-theorem.S3}
\(S_3 = \varnothing\);
\item\label{spinor-main-theorem.associated-points}
\(S_2\) contains no weakly associated points of \(S\); and
\item\label{spinor-main-theorem.S2-cover}
the \'etale double cover \(\tilde{S}_2 \to S_2\) splits.
% \item \(\operatorname{depth}\sO_{S,s} \geq 2\) for every \(s \in S_2\); and
\end{enumerate}
Then each section \(\sigma \colon S_2 \to \tilde{S}_2\) provides an open
subspace \(M_\sigma \subset \mathrm{M}_d(Q/S)\), a Brauer class
\(\beta \in \Br(M_\sigma)\), and an \(S\)-linear equivalence
\[
\Phi_\sigma \colon \Dqc(M_\sigma, \beta) \to \Ku(Q).
\]
\end{Theorem}

\begin{proof}
Fix a section \(\sigma \colon S_2 \to \tilde{S}_2\) and let \(M_\sigma\) be
the open subspace of \(\mathrm{M}_d(Q/S)\) constructed in
\parref{spinor-coarse-moduli-section}. Restricting the coarse moduli map
\(\mathcal{M}_d(Q/S) \to \mathrm{M}_d(Q/S)\) yields a
\(\mathbf{G}_m\)-gerbe \(\mathcal{M}_\sigma \to M_\sigma\). Let
\(\tilde{\mathcal{S}}_d^\vee\) be the universal \(d\)-th dual spinor sheaf
on \(Q \times_S \mathcal{M}_\sigma\), and consider the associated
Fourier--Mukai functor
\(\Phi \colon \Dqc(\mathcal{M}_\sigma) \to \Dqc(Q)\). As recalled
in \parref{subsection-gmgerbe},
%\cite[Theorem 5.4]{BS:BS},
the derived category of \(\mathcal{M}_\sigma\) decomposes into homogeneous
components
\[
\Dqc(\mathcal{M}_\sigma) \cong
\prod\nolimits_{k \in \mathbf{Z}} \mathrm{D}_{\mathrm{qc},k}(\mathcal{M}_\sigma) \cong
\prod\nolimits_{k \in \mathbf{Z}} \Dqc(M_\sigma,\beta^k)
\]
where \(\beta \in \Br(M_\sigma)\) is the Brauer class associated with the
\(\mathbf{G}_m\)-gerbe \(\mathcal{M}_\sigma \to M_\sigma\).

We claim that the restriction
\(\Phi_{-1} \colon \Dqc(M_\sigma, \beta^{-1}) \to \Dqc(Q)\) of
\(\Phi\) to the weight \(-1\) component induces an equivalence onto the subcategory $\operatorname{Ku}(Q)$. By
\parref{thm-locallytwisted}, it suffices to check this upon passing to an \'etale
cover of \(S\). Doing so, we may assume that \(\mathcal{M}_\sigma\) is identified with $M \times B \mathbf{G}_m$
with $M$ a maximal hyperbolic reduction of \(\rho \colon Q \to S\)
as in \parref{spinor-coarse-moduli-section-reduction} and
\parref{spinor-moduli-compatible}. By \parref{corollary-kernelscorrespond},
under this equivalence, the kernel $\tilde{S}_d^\vee$ corresponds to  the
\(d\)-th dual spinor sheaf \(\mathcal{S}_d^\vee\) on \(Q \times_S M_\sigma\)
given $\mathbf{G}_m$-weight one. It therefore follows from
\parref{reduction-fourier-mukai-equivalence} that this object induces the
sought-after equivalence
\[
\operatorname{D}_{\mathrm{qc},-1}(\mathcal{M}_\sigma) \to \operatorname{Ku}(Q).
\qedhere
\]
\end{proof}

\subsectiondash{On the hypotheses}\label{spinors-hypotheses}
Theorem \parref{spinor-main-theorem} says that, along with the usual
assumptions on the quadric bundle \(\rho \colon Q \to S\), once it is possible
to consistently choose one of the two families of spinor sheaves over the
corank \(2\) locus---precisely, if the canonical \'etale cover
\(\tilde{S}_2 \to S_2\) splits---then there is a Fourier--Mukai equivalence
\(\Phi_\sigma \colon \Dqc(M_\sigma,\beta) \to \Ku(Q)\). This is, of course,
always possible \'etale locally on \(S\), wherein \(M_\sigma\) may be
identified with a suitable maximal hyperbolic reduction of
\(\rho \colon Q \to S\) by \parref{spinors-coarse-moduli-fibres} and
\parref{spinor-moduli-compatible}; moreover, the equivalence \(\Phi_\sigma\)
is locally isomorphic to the Fourier--Mukai equivalence defined by the kernel
featured in \parref{reduction-fourier-mukai-equivalence}. The hypothesis
\parref{spinor-main-theorem}\ref{spinor-main-theorem.S2-cover} turns out to
be essentially necessary for a result of this form: in the setting of the
Theorem, if for some \'etale cover of \(S\), the local maximal hyperbolic
reductions and kernels glue and descend to an algebraic space \(M \to S\) and
equivalence \(\Dqc(M,\beta) \to \Ku(Q)\) for some \(\beta \in \Br(M)\), then
the double cover \(\tilde{S}_2 \to S\) is necessarily split:

%Suppose there is an \'etale cover $\{U_i \to S\}_i$ such that $Q_{U_i}/U_i$ admits a family of $(\ell-1)$-planes contained in its smooth locus. Denote $\bar{Q}_{U_i}/U_i$ the hyperbolic reduction with respect to one. There is an object $K_i \in D_{qc}(\bar{Q}_{U_i} \times_{U_i} Q_{U_i})$ which, under the assumption that $S_2 \subset S$ contains no weakly associated points, induces a Fourier--Mukai equivalence $D_{qc}(\bar{Q}_{U_i}) \to Ku(Q_{U_i}/U_i)$, see Proposition \ref{reduction-fourier-mukai-equivalence}. It is natural to ask whether the schemes $\bar{Q}_{U_i}$ and kernels $K_i$ can be descended to an $S$-scheme $\bar{Q}$ and object $K \in D_{qc}(\bar{Q} \times _S Q)$. Indeed, this seems to be considered by Xie in \cite{Xie:Quadrics} in the text following Conjecture 1. Using the results of this section, we can see that the double cover $\widetilde{S}_2 \to S_2$ must split in order to perform this glueing, even if we allow ourselves to glue in the category of algebraic spaces and twisted complexes. It is not too difficult to give examples of $Q/S$ such that the double cover $\widetilde{S}_2 \to S$ does not split. 

\begin{Proposition}\label{spinors-converse}
Let \(\rho \colon Q \to S\) be a quadric \(2\ell\)-fold bundle with
\(S_3 = \varnothing\) and \(S_2\) containing no weakly associated points.
Suppose there exists
\begin{enumerate}
\item\label{spinors-converse.global}
an algebraic space $Y \to S$, a $\mathbf{G}_m$-gerbe $\mathcal{Y} \to Y$, and
an object $K \in \mathrm{D}_{\mathrm{qc}, 1}(\mathcal{Y} \times _S Q)$;
\item\label{spinors-converse.local}
an \'etale cover $\{U_i \to S\}_i$, maximal hyperbolic
reductions $\mu_i \colon M_i \to U_i$ of \(\rho_{U_i} \colon Q_{U_i} \to U_i\), and objects $K_i \in \Dqc(M_i \times_{U_i} Q_{U_i})$
as in \parref{reduction-fourier-mukai-equivalence}; and
\item\label{spinors-converse.glue}
equivalences $\mathcal{Y}_{U_i} \cong M_i \times B\mathbf{G}_m$ over $U_i$
under which $K_{U_i}$ corresponds to $K_i$ with $\mathbf{G}_m$-weight one.
\end{enumerate}
Then the \'etale double cover $\tilde{S}_2 \to S_2$ admits a section.
\end{Proposition}

\begin{proof}
The object \(K \in \mathrm{D}_{\mathrm{qc},1}(\mathcal{Y} \times_S Q)\) induces
an \(S\)-morphism \(\mathcal{Y} \to \mathcal{M}_d(Q/S)\) since this is true
locally by \ref{spinors-converse.local} and \ref{spinors-converse.glue}. This
induces an \(S\)-morphism \(Y \to \mathrm{M}_d(Q/S)\) of
\(\mathbf{G}_m\)-rigidifications. Over \(S_2\), post-composing this with the
morphism from \parref{spinor-moduli-corank-2-stein} provides a morphism
\(f \colon Y \times_S S_2 \to \tilde{S}_2\). Since each \(M_i \to U_i\) is a 
\(\PP^1\)-bundle, \(Y \times_S S_2 \to S_2\) is a Brauer--Severi space of relative
dimension \(1\) and, in particular, has geometrically connected fibres. Thus
the rigidity lemma \parref{lemma-rigsection} applies with \(f\) to show that
\(\tilde{S}_2 \to S\) must admit a section.
%   We may assume after base change to $S_2$ that $S_2 = S$ scheme theoretically (and $S_3 = \emptyset$ by our standing assumption). In this case, the hyperbolic reductions $\bar{Q}_{U_i} /U_i$ are $\mathbf{P}^1$-bundles, so we see that $Y/S$ is a Brauer--Severi space of relative dimension 1. The object $K \in D_{qc}(\mathcal{Y} \times _S Q)$ classifies a morphism $\mathcal{Y} \to \mathcal{M}_d(Q/S)$ and therefore we get also a morphism of $\mathbf{G}_m$-rigidifications $Y \to M_d(Q/S)$. Postcomposing with the morphism of Proposition \ref{spinor-moduli-corank-2-stein}, we obtain an $S$-morphism $Y \to \widetilde{S}_2$. By the Rigidity Lemma \ref{lemma-rigsection}, there must be a section of $\widetilde{S}_2 \to S$, as needed. 
\end{proof}

\begin{Example}\label{spinors-no-section}
The \'etale double cover \(\tilde{S}_2 \to S_2\) often will not split when
\(\dim S_2 \geq 1\), suggesting that the Kuznetsov component of quadric bundle
over fourfolds, for instance, may not be twisted geometric. When \(\dim S_2 = 0\),
the cover also might not split for arithmetic reasons: let
\(S \coloneqq \mathbf{A}^1_{\mathbf{R}}\) be the affine line with coordinate
\(t\), and consider the quadric surface bundle over \(S\) given by
\[
Q \coloneqq \{(x_0:x_1:x_2:x_3) \in \PP^3_S : tx_0x_1 + x_2^2 + x_3^2 = 0\}.
\]
Then \(S_2 \subset S\) is the closed subscheme \(t = 0\), and the canonical
double covering \(\tilde{S}_2 \to S_2\) from
\parref{quadric-bundles-fano-schemes-corank-2} is the Fano scheme of planes in
the real quadric surface \(x_2^2 + x_3^2 = 0\). It is straightforward to see
that this is isomorphic to the non-split \'etale cover
\(\Spec\mathbf{C} \to \Spec\mathbf{R}\).
\qed
\end{Example}

% \begin{Lemma}\label{spinor-moduli-equivalences}
% Let \(\rho \colon Q \to S\) be a quadric bundle. The endofunctors of \(\sCoh_{Q/S}\)
% \[
% - \otimes \rho^*\mathcal{L}^\vee
% \quad\text{and}\quad
% \mathcal{F} \mapsto \ker(\rho^*\rho_*\mathcal{F} \to \mathcal{F}) \otimes \sO_{Q/S}(1)
% \]
% induce equivalences
% \(\iota_{\ell,m} \colon \mathcal{M}_{\ell,m} \to \mathcal{M}_{\ell,m-2}\) and
% \(\tau_{\ell,m} \colon \mathcal{M}_{\ell,m} \to \mathcal{M}_{\ell,m-1}\) for
% all integers \(\ell \geq 0\) and \(m \in \mathbf{Z}\). The two equivalences
% \[
% \tau_{\ell,m-1} \circ \tau_{\ell,m} \colon \mathcal{M}_{\ell,m} \to \mathcal{M}_{\ell,m-2}
% \quad\text{and}\quad
% \iota_{\ell,m} \colon \mathcal{M}_{\ell,m} \to \mathcal{M}_{\ell,m-2}
% \]
% are naturally isomorphic.
% \end{Lemma}

% \begin{proof}
% The equivalence \(\iota_{\ell,m}\) arises from the degree shift relation
% \(\mathcal{S}_m \otimes \rho^*\mathcal{L}^\vee \cong \mathcal{S}_{m-2}\)
% recorded in \S\ref{cliff-spinors}, and \(\tau_{\ell,m}\) arises from the short
% exact sequence
% \[
% 0 \to
% \mathcal{S}_{m-1} \otimes \sO_{Q/S}(-1) \to
% \rho^*\mathcal{I}_m \to
% \mathcal{S}_m \to
% 0,
% \]
% as in the proof of Proposition \ref{cliff-subbundle-modify}, together with the
% observation that \(\rho_*\mathcal{S}_m \cong \mathcal{I}_m\). That
% \(\tau_{\ell,m-1} \circ \tau_{\ell,m}\) and \(\iota_{\ell,m}\) are naturally
% isomorphic is now clear.\todo{Is there really nothing to explain?}
% \end{proof}

\section{Complements and applications}\label{section-applications}
In this section, we sketch applications of our results to the derived
categories of intersections of quadrics and cubic fourfolds containing a plane
over a base field \(\kk\) of arbitrary characteristic \(p \geq 0\). We phrase
the results in terms of their bounded derived categories \(\Dbcoh\) of coherent
sheaves; that our results imply corresponding results in this setting may be
deduced by adapting the arguments of \cite[Theorem 6.2]{BS:Descent}. Both
examples are classical and well-studied, but our results allow us to give
uniform statements and proofs for general \(\kk\), some of which are new even
for \(p = 0\).

\subsectiondash{Discriminant algebra and cover}
\label{applications-discriminant-algebra}
Following \cite[IV.4.8]{Knus}, associated with any quadric bundle
\(\rho \colon Q \to S\) is a graded locally free commutative \(\sO_S\)-algebra
\(\sDisc(\mathcal{E},q)\) called its \emph{discriminant algebra}. This comes
with a canonical \(\sO_S\)-algebra morphism
\[
\psi \colon
\sDisc(\mathcal{E},q) \to
\Cl(\mathcal{E},q).
\]
For the universal quadric in \(\PP^n_{\mathbf{Z}}\), this is constructed as the
commutant of the Clifford algebra relative to its even subalgebra. In general,
\(\sDisc(\mathcal{E},q)\) is pulled back from the universal case, using the
fact that the pair \((\sDisc(\mathcal{E},q), \psi)\) is functorial for
isometries of quadratic forms over \(S\): see \cite[IV.4.4.8.10]{Knus}.

The discriminant algebra behaves differently depending on the parity of
\(\rank_{\sO_S}(\mathcal{E})\). We focus on the case \(\mathcal{E}\) has even
rank \(2\ell+2\), wherein \cite[IV.4.4.8.5]{Knus} implies that the \(0\)-th
graded component is a locally free \(\sO_S\)-module of rank \(2\) fitting into
a short exact sequence
\[
0 \to
\sO_S \to
\sDisc_0(\mathcal{E},q) \to
\det(\mathcal{E}) \otimes \mathcal{L}^{\vee, \otimes \ell+1} \to
0.
\]
The morphism
\(\psi_0 \colon \sDisc_0(\mathcal{E},q) \to \Cl_0(\mathcal{E},q)\)
factors through the centre \(Z(\Cl_0(\mathcal{E},q))\) of the \(0\)-th Clifford
algebra: compare with \cite[IV.4.4.8.1]{Knus} and
\cite[Theorem 3.7(ii)]{BK:Clifford}. Furthermore, when \(S\) is regular and
\((\mathcal{E},q)\) is generically nonsingular, then \(\psi_0\) provides
an isomorphism \(\sDisc_0(\mathcal{E},q) \cong Z(\Cl_0(\mathcal{E},q))\): see
\cite[IV.4.4.8.3]{Knus}. In any case, the spectrum
\[
\Disc(Q/S) \coloneqq
\Spec\sDisc_0(\mathcal{E},q)
\]
shall be called the \emph{discriminant cover} of \(S\) associated with \(\rho
\colon Q \to S\).

\subsectiondash{}\label{applications-discriminant-algebra-split}
An explicit presentation of \(\sDisc_0(\mathcal{E},q)\) is available locally
on \(S\). When the short exact sequence in
\parref{applications-discriminant-algebra} splits, there is a generator \(z\)
of \(\sDisc_0(\mathcal{E},q)\) satisfying
\[
\sDisc_0(\mathcal{E},q) \cong \sO_S[z]/(z^2 - \beta z + \gamma)
\]
for some sections
\(\beta \colon \sO_S \to \det(\mathcal{E}^\vee) \otimes \mathcal{L}^{\otimes \ell+1}\)
and
\(\gamma \colon \sO_S \to \det(\mathcal{E}^\vee)^{\otimes 2} \otimes \mathcal{L}^{\otimes 2\ell+2}\).
These sections may be explicitly expressed when
\(\mathcal{E} \cong \bigoplus\nolimits_{i = 0}^{2\ell+1} \mathcal{E}_i\)
furthermore splits into a sum of line bundles: Decompose the section
\(s_q \colon \sO_S \to \Sym^2(\mathcal{E}^\vee) \otimes \mathcal{L}\)
corresponding to the equation of \(Q\) in \(\PP\mathcal{E}\) from
\parref{quadric-bundles} as
\[
s_q =
\sum\nolimits_{i = 0}^n \xi_i +
\sum\nolimits_{0 \leq i < j \leq n} \xi_{i,j}
\]
where
\(\xi_i \colon \sO_S \to \mathcal{E}_i^{\vee,\otimes 2} \otimes \mathcal{L}\)
and
\(\xi_{i,j} \colon \sO_S \to \mathcal{E}_i^\vee \otimes \mathcal{E}_j^\vee \otimes \mathcal{L}\).
Let \(T\) be the set of fix-point free order \(2\) permutations on the set
\(\{0,1,\ldots,2\ell+1\}\) and, for each element
\(\tau = (t_0,t_1) \cdots (t_{2\ell}, t_{2\ell+1})\) written as a product of
disjoint transposition with \(t_{2i} < t_{2i+1}\), write
\(\xi_\tau \coloneqq \xi_{t_0,t_1} \cdots \xi_{t_{2\ell}, t_{2\ell+1}}\). With
this notation, \cite[IV.4.4.8.5]{Knus} gives
\[
\beta = \sum\nolimits_{\tau \in T}
\operatorname{sign}(\tau) \xi_\tau
\quad\text{and}\quad
\gamma = \Delta(\xi_0,\ldots,\xi_{2\ell,2\ell+1})
\]
where \(\Delta\) is a universal integral polynomial characterized by the
equality \(\beta^2 - 4\Delta = (-1)^\ell\det(b_q)\); observe that \(\beta\)
is the Pfaffian of the antisymmetric matrix determined by the \(\xi_{i,j}\).
For example, \(\beta = -\xi_{0,1}\) and
\(\Delta = \xi_0\xi_1\) when \(\ell = 1\), and
\(\beta = \xi_{0,1} \xi_{2,3} - \xi_{0,2} \xi_{1,3} + \xi_{0,3} \xi_{1,2}\) and
\begin{align*}
\Delta & =
 -4\xi_0\xi_1\xi_2\xi_3
+ \xi_0\xi_1\xi_{2,3}^2
+ \xi_0\xi_2\xi_{1,3}^2
+ \xi_0\xi_3\xi_{1,2}^2
+ \xi_1\xi_2\xi_{0,3}^2
+ \xi_1\xi_3\xi_{0,2}^2
+ \xi_2\xi_3\xi_{0,1}^2 \\
& \phantom{=4}- \xi_0\xi_{1,2}\xi_{1,3}\xi_{2,3}
- \xi_1\xi_{0,2}\xi_{2,3}\xi_{2,3}
- \xi_2\xi_{0,1}\xi_{0,3}\xi_{1,3}
- \xi_3\xi_{0,1}\xi_{0,2}\xi_{1,2}
+ \xi_{0,1}\xi_{0,3}\xi_{1,2}\xi_{2,3}
\end{align*}
when \(\ell = 2\). When \(2\) is invertible on \(S\), the coordinate change
\(z \mapsto z + \frac{1}{2}\beta\) identifies \(\Disc(Q/S)\) with the usual
cyclic double cover of \(S\) branched along the discriminant divisor of
\(\rho \colon Q \to S\).
\medskip

The discriminant cover is intimately related to the moduli of spinor sheaves as
in \S\parref{section-spinor-moduli}. The following two statements generalize
\cite[Proposition B.5]{ABB}:

\begin{Lemma}\label{applications-discriminant-cover}
Let \(\rho \colon Q \to S\) be a quadric \(2\ell\)-fold bundle with
\(S_3 = \varnothing\). There are canonical \(S\)-morphisms
\(\mathcal{M}_d(Q/S) \to \mathrm{M}_d(Q/S) \to \mathbf{Disc}(Q/S)\).
\end{Lemma}

\begin{proof}
Construct an \(S\)-morphism \(\mathcal{M}_d(Q/S) \to \mathbf{Disc}(Q/S)\) as
follows: Let \(\widetilde{\mu} \colon \mathcal{M}_d(Q/S) \to S\) be the
structure morphism and \(\mathcal{S}_d^\vee\) the universal spinor
on \(Q \times_S \mathcal{M}_d(Q/S)\). We claim that the morphism \(\psi_0\)
provides \(\mathcal{S}_d^\vee\) with a
\(\pr_2^*\widetilde{\mu}^*\sDisc_0(\mathcal{E},q)\)-module structure. For this,
it suffices to show that left multiplication by the discriminant preserves the
image of the Clifford multiplication map \(\phi_d\) appearing in the defining
presentation of the spinor sheaf in \parref{cliff-spinors}. This follows from
the universal formula for image \(z_0 \coloneqq \psi_0(z)\) in
\(Z(\Cl_0(\mathcal{E},q))\) of the nontrivial generator of the discriminant
algebra together with the formulae in \cite[IV.4.8.5 and IV.2.3.2]{Knus},
which imply that for any local section \(x\) of \(\mathcal{E}\),
\(z_0 \cdot x = x \cdot (b-z_0)\) for some local section \(b\) of
\(\Cl_0(\mathcal{E},q)\). This module structure now provides a morphism of
sheaves of algebras
\[
\pr_2^*\widetilde{\mu}^*\sDisc_0(\mathcal{E},q) \to
\sHom_{\sO_{Q \times_S \mathcal{M}_d(Q/S)}}(\mathcal{S}_d^\vee, \mathcal{S}_d^\vee).
\]
Since spinors are simple here by \parref{spinor-moduli-simple},
the target is the structure sheaf. Adjunction for \(\pr_2\) provides a map
\(\widetilde{\mu}^*\sDisc_0(\mathcal{E},q) \to \sO_{\mathcal{M}_d(Q/S)}\) of
algebras and thus an \(S\)-morphism
\(\mathcal{M}_d(Q/S) \to \Disc(Q/S)\). The universal property of the coarse
moduli space map then provides \(\mathrm{M}_d(Q/S) \to \Disc(Q/S)\).
\end{proof}

The coarse moduli space of spinor sheaves matches with the discriminant cover
whenever the quadric bundle \(\rho \colon Q \to S\) has at worst corank \(1\)
fibres:

\begin{Lemma}\label{applications-discriminant-cover-simple-degen}
If \(S_2 = \varnothing\), then the map
\(\mathrm{M}_d(Q/S) \to \mathbf{Disc}(Q/S)\) is an isomorphism.
\end{Lemma}

\begin{proof}
Having constructed a canonical \(S\)-morphism
\(\mathrm{M}_d(Q/S) \to \Disc(Q/S)\) in
\parref{applications-discriminant-cover}, whether it is an isomorphism is a
question local on \(S\). Replacing \(S\) by an fppf cover, we may assume that
\[
Q =
\Big\{(T_0:T_1:\cdots:T_{2\ell+1}) \in \PP^{2\ell+1}_S : 
\gamma_0 T_0^2 - \beta T_0T_1 + \gamma_1 T_1^2 + \sum\nolimits_{i = 1}^\ell T_{2i}T_{2i+1} = 0\Big\}
\]
for some \(\beta, \gamma_0, \gamma_1 \in \mathrm{H}^0(S,\sO_S)\). Maximal
hyperbolic reduction along
\(\PP W \coloneqq \mathrm{V}(T_3,T_5,\ldots,T_{2\ell+1})\) may be identified with
the quadric bundle
\[
M =
\{(T_0:T_1) \in \PP^1_S : \gamma_0 T_0^2 - \beta T_0T_1 + \gamma_1 T_1^2 = 0\}.
\]
Since \(S_2 = \varnothing\), \(\mathrm{M}_d(Q/S) \cong M\) by
\parref{spinors-coarse-moduli-fibres}, and so it remains to show that the
resulting morphism \(M \to \Disc(Q/S)\) is an isomorphism.

For each \(0 \leq i \leq 2\ell+1\), let \(e_i\) be the \(i\)-th basis vector
of \(\mathcal{E} \cong \sO_S^{\oplus 2\ell+2}\) dual to the coordinate \(T_i\).
By \cite[IV.4.4.8.5]{Knus}, the generator \(z\) of \(\sDisc_0(\mathcal{E},q)\)
maps, up to a sign, to
\[
z \mapsto
e_0e_1 - \beta\Big(\sum\nolimits_{i = 1}^\ell e_{2i}e_{2i+1}\Big) + \Xi \in 
\Cl_0(\mathcal{E},q)
\]
where every term in \(\Xi\) involves \(\geq 4\) basis vectors and contains
\(e_{2i}\) for some \(1 \leq i \leq \ell\). By its construction in
\parref{applications-discriminant-cover}, the map \(M \to \Disc(Q/S)\)
is determined by the action of \(z\) on the universal sheaf
\(\mathcal{S}_d^\vee\) on \(M \times_S Q\) corresponding to the tautological
\(\ell\)-plane
\[
\PP\mathcal{W} \coloneqq
\big\{((T_0:T_1), (T_0:T_1:T_2:0:\cdots:T_{2\ell}:0)) \in M \times_S Q :
\gamma_0 T_0^2 - \beta T_0T_1 + \gamma_1 T_1^2 = 0\big\}.
\]
Since the image of \(z\) lies in the centre of \(\Cl_0(\mathcal{E},q)\), its
action on \(\mathcal{S}_d^\vee\) is determined by its action on the generator
\(\det\mathcal{W}\) of the corresponding Clifford ideal, as in
\parref{cliff-ideals}. Then
\begin{align*}
z \cdot \det\mathcal{W}
= z \cdot (T_0e_0 + T_1e_1) \wedge \det W
& = e_0e_1 \cdot (T_0e_0 + T_1e_1) \wedge \det W \\
& = \big((\gamma_1 T_1 - \beta T_0) e_0 - \gamma_0 T_0 e_1\big) \wedge \det W
\end{align*}
Since the equation of \(M\) implies
\(-\gamma_0 T_0^2 = (\gamma_1 T_1 - \beta T_0)T_1\), the action of \(z\) on
\(\det\mathcal{W}\) may be locally expressed on the standard open cover
\(\mathrm{D}_+(T_1) \cup \mathrm{D}_+(T_0)\) of \(M \subset \PP^1_S\) as
\[
z \cdot \det\mathcal{W} =
\begin{dcases*}
-\gamma_0 T_0T_1^{-1} \det\mathcal{W} & on \(\mathrm{D}_+(T_1)\), and \\
(\gamma_1 T_1 + \beta T_0)T_0^{-1}T_1 \det\mathcal{W} & on \(\mathrm{D}_+(T_0)\).
\end{dcases*}
\]
By \parref{applications-discriminant-global-sections}, these scalars locally
represent the generator of \(\mu_*\sO_M\). Therefore the map
\(\sDisc_0(\mathcal{E},q) \to \mu_*\sO_M\) between locally free rank \(2\)
\(\sO_S\)-algebras matches the generators, so it is an isomorphism.
\end{proof}

The following gives an explicit description of the relative coordinate algebra
of a family of quadrics of relative dimension \(0\). Note it holds even when
there may be some corank \(2\) fibres.

\begin{Lemma}\label{applications-discriminant-global-sections}
Let \(\mu \colon M \to S\) be cut out in \(\PP^1_S\) by a regular section
\(\gamma_0 T_0^2 - \beta T_0T_1 + \gamma_1 T_1^2\). Then
\[
\mu_*\sO_M \cong \sO_S[z]/(z^2 - \beta z + \gamma_0\gamma_1)
\]
as a free rank \(2\) \(\sO_S\)-algebra. A global section
\(\tilde{z} \colon \sO_M \to \sO_M\) representing \(z\) may be locally given by
\[
\tilde{z}\rvert_{\mathrm{D}_+(T_i)} =
\begin{dcases*}
\gamma_0 T_0T_1^{-1} & on \(\mathrm{D}_+(T_1)\), and \\
(\beta T_0 - \gamma_1 T_1)T_1T_0^{-1} & on \(\mathrm{D}_+(T_0)\).
\end{dcases*}
\]
\end{Lemma}

\begin{proof}
Compute \(\mu_*\sO_M\) using the \v{C}ech complex for the standard open cover
\(\PP^1_S = \mathrm{D}_+(T_1) \cup \mathrm{D}_+(T_0)\):
\[
\sO_S[s]/(\gamma_0 s^2 - \beta s + \gamma_1) \times
\sO_S[t]/(\gamma_0 - \beta t + \gamma_1 t^2) \to
\sO_S[s,s^{-1}]/(\gamma_0 s^2 - \beta s + \gamma_1)
\]
where \(s \coloneqq T_0T_1^{-1}\), \(t \coloneqq T_1T_0^{-1}\), and
the differential is \((f(s),g(t)) \mapsto f(s) - g(s^{-1})\). The element
\(\tilde{z} \coloneqq (\gamma_0 s, \beta - \gamma_1 t)\) lies in the kernel,
so defines a section \(z \colon \sO_S \to \mu_*\sO_M\). A direct computation
shows that \(\tilde{z}^2 - \beta \tilde{z} + \gamma_0\gamma_1 = 0\), and so
\(z\) satisfies the same relation.

It remains to see that \(1\) and \(z\) freely generate \(\mu_*\sO_M\) as an
\(\sO_S\)-module. By the short exact sequence
\[
0 \to \sO_S \to \mu_*\sO_M \to \mathrm{R}^1\pi_*\sO_{\PP^1_S}(-2) \to 0,
\]
we must show that \(z\) maps to a generator of
\(\mathrm{R}^1\pi_*\sO_{\PP^1_S}(-2)\). The boundary map is computed by
viewing \(\tilde{z} = (\gamma_0 s, \beta - \gamma_1 t)\) as a local section of
\(\sO_{\PP^1_S}\), applying the \v{C}ech differential to obtain the element
\[
\gamma_0 s - \beta + \gamma_1 t =
(\gamma_0 T_0^2 - \beta T_0T_1 + \gamma_1 T_1^2) \cdot \frac{1}{T_0T_1},
\]
and dividing out the equation \(\gamma_0 T_0^2 - \beta T_0T_1 + \gamma_1 T_1^2\)
of \(M\) in \(\PP^1_S\). This means that \(z\) maps to the generator
\(1/T_0T_1\) of \(\mathrm{R}^1\pi_*\sO_{\PP^1_S}\), as desired.
\end{proof}

\subsectiondash{Intersection of quadrics}\label{applications-intersections}
As a first application of our results, consider a complete intersection
\(X \subset \PP^{2\ell+1}\) of \(m\) even-dimensional quadrics over an
arbitrary field \(\kk\). Write \(\rho \colon Q \to \PP^{m-1}\) for the
corresponding family of quadrics over the associated linear system. Whenever
\(n+2 \geq 2m\), Kuznetsov gives in \cite[Theorem 5.5]{Kuznetsov:Quadrics} a
semiorthogonal decomposition
\[
\Dbcoh(X) =
\langle
\Ku(Q),
\sO_X(1),
\sO_X(2),
\ldots,
\sO_X(n+2-2m)
\rangle
\]
where \(\Ku(Q)\) is the Kuznetsov component of the quadric bundle
\(\rho \colon Q \to \PP^{m-1}\) as in the Introduction. Combined with
\parref{spinor-main-theorem}, this shows that the Kuznetsov component of the
intersection of at most \(4\) quadrics is twisted geometric:

\begin{Proposition}\label{applications-intersections-result}
Let \(X \subset \PP^{2\ell+1}\) be a complete intersection of
\(m \leq \ell + 1\) quadrics with corresponding quadric \(2\ell\)-fold bundle
\(\rho \colon Q \to \PP^{m-1}\). Assume that \(X\) and \(Q\) are smooth over
\(\kk\).
\begin{enumerate}
\item\label{applications-intersections-result.curve}
If \(m = 2\), then there is a semiorthogonal decomposition
\[
\Dbcoh(X) =
\langle
\Dbcoh(C,\alpha),
\sO_X(1),
\sO_X(2),
\ldots,
\sO_X(2\ell-2)
\rangle
\]
where \(C\) is a smooth curve with a finite flat map \(C \to \PP^1\) of
degree \(2\) and \(\alpha \in  \Br(C)\).
\end{enumerate}
Assume henceforth that \(\kk\) is algebraically closed and
\(\rho \colon Q \to \PP^{m-1}\) has finitely many corank \(2\) fibres.
\begin{enumerate}
\setcounter{enumi}{1}
\item\label{applications-intersections-result.surface}
If \(m = 3\), then there is a semiorthogonal decomposition
\[
\Dbcoh(X) =
\langle
\Dbcoh(S,\alpha),
\sO_X(1),
\sO_X(2),
\ldots,
\sO_X(2\ell-4)
\rangle
\]
where \(S \to \PP^2\) is a proper generically finite morphism of degree
\(2\) from a smooth projective surface and \(\alpha \in \operatorname{Br}(S)\).
\item\label{applications-intersections-result.threefold}
If \(m = 4\) and \(\rho \colon Q \to \PP^3\) has no corank \(3\) fibres, then
there is a semiorthogonal decomposition
\[
\Dbcoh(X) =
\langle
\Dbcoh(M,\alpha),
\sO_X(1),
\sO_X(2),
\ldots,
\sO_X(2\ell-6)
\rangle
\]
where \(M \to \PP^3\) is a proper generically finite morphism of degree \(2\)
from a smooth proper algebraic space and \(\alpha \in \operatorname{Br}(M)\).
\end{enumerate}
\end{Proposition}

\begin{proof}
Smoothness of \(X\) implies that the corank \(\geq m\) locus of
\(\rho \colon Q \to \PP^{m-1}\) is empty. Combined with the assumptions, it is
straightforward to see that the hypotheses of \parref{spinor-main-theorem} are
satisfied, whereupon the decompositions in each case follow from the discussion
of \parref{applications-intersections}.

Regarding the geometric properties: Smoothness of \(Q\) implies smoothness of
each of \(C\), \(S\), and \(M\). When \(m = 1\), \(C \to \PP^1\) is flat
because \(\rho \colon Q \to \PP^1\) has no corank \(\geq 2\) fibres
or simply because \(C\) and $\mathbf{P}^1$ are smooth curves, \(C\)
is a scheme by \citeSP{0ADD}.
%, and the Brauer class \(\alpha\) comes from
%the base field \(\kk\) follows from Tsen's theorem. 
Finally, when \(m = 2\),
\(S\) is a scheme by a result of Artin in \cite[Th\'eor\`eme
4.7]{Artin:Algebraic-Spaces}.
\end{proof}

\subsectiondash{Remarks}\label{applications-intersections-remarks}
A few remarks regarding this are in order:
\begin{itemize}
\item
Bertini's theorem shows that the geometric hypotheses in
\parref{applications-intersections-result} hold for general
\(\rho \colon Q \to \PP^{m-1}\).
\item When \(\operatorname{char}\kk \neq 2\), smoothness of \(Q\) implies that
the general fibre of \(\rho \colon Q \to \PP^{m-1}\) is smooth. In particular,
each cover of \(\PP^{m-1}\) is generically \'etale. However, when
\(\operatorname{char}\kk = 2\), it is possible that every fibre of
\(\rho \colon Q \to \PP^{m-1}\) has corank \(\geq 1\), in which case the covers
of \(\PP^{m-1}\) would be generically purely inseparable. Examples of this
nature do occur.
\item
\parref{applications-intersections-result}\ref{applications-intersections-result.surface}
improves on \cite[Proposition 6.5]{Xie:Quadrics} and verifies expectations remarked there.
\item
\parref{applications-intersections-result}\ref{applications-intersections-result.threefold}
generalizes \cite{Addington:Intersections} from the complex numbers to any
algebraically closed field.
\item
The algebraic space \(M\) appearing in
\parref{applications-intersections-result}\ref{applications-intersections-result.threefold}
may sometimes be non-schematic, and this can sometimes be verified explicitly:
see \cite[Proposition 3.12]{CPZ} for an example.
\item
The methods here may be adapted to study families of complete intersections of
quadrics, as considered already in \cite[\S5]{ABB}.
\end{itemize}

\subsectiondash{Cubic fourfolds}\label{applications-cubic-fourfolds}
As a second application of our results, consider a smooth cubic fourfold
\(X \subset \PP^5\) over an algebraically closed field \(\kk\) of arbitrary
characteristic. As mentioned in the Introduction, the derived category of
\(X\) has a well-known semiorthogonal decomposition
\[
\Dbcoh(X) =
\langle
\mathcal{A}_X, \sO_X, \sO_X(1), \sO_X(2)
\rangle
\]
where the residual component \(\mathcal{A}_X\) looks like the derived category
of a K3 surface

At least when \(\operatorname{char}\kk \neq 2\), a
particularly simple and well-studied case in which \(\mathcal{A}_X\) is
equivalent to the twisted derived category of a K3 surface is when \(X\)
contains a plane. The following completes the picture for arbitrary
base field characteristics:

\begin{Theorem}\label{applications-cubic-fourfolds-decomposition}
Let \(X \subset \PP^5\) be a smooth cubic fourfold over an algebraically
closed field \(\kk\). Assume that \(X\) contains a plane \(P\). Then
there is a semiorthogonal decomposition
\[
\Dbcoh(X) =
\langle \Dbcoh(S,\alpha), \sO_X, \sO_X(1), \sO_X(2) \rangle
\]
where \(S\) is a K3 surface of degree \(2\) and \(\alpha \in \mathrm{Br}(S)\).
\end{Theorem}

\begin{proof}
Fix a plane \(P \subset X\). Linear projection centred at \(P\) is resolved
on the blowup \(b \colon \widetilde{X} \to X\) along \(P\). The resulting
morphism \(a \colon \widetilde{X} \to \PP^2\) is a quadric surface bundle.
The Fano scheme \(\mathbf{F}_2(X)\) is a finite set of reduced points by
an unpublished result of Debarre in \cite[Lemma 3]{Debarre:Lines}---see also
Starr's treatment in \cite[Appendix]{BH:Density}---and so
\(a \colon \widetilde{X} \to \PP^2\) has only finitely many corank \(2\) fibres
and no corank \(3\) fibres. Since \(\kk\) is algebraically closed, the \'etale
cover over the corank \(2\) locus from
\parref{quadric-bundles-fano-schemes-corank-2} has a section. The hypotheses of
\parref{spinor-main-theorem} thus are satisfied, and so there is a
semiorthogonal decomposition
\[
\Dbcoh(\widetilde{X}) =
\langle
\Dbcoh(S,\alpha),
La^*\Dbcoh(\PP^2) \otimes^L \sO_a,
La^*\Dbcoh(\PP^2) \otimes^L \sO_a(1)
\rangle
\]
where \(S \to \PP^2\) is an open in the moduli space of spinor sheaves along
\(a \colon \widetilde{X} \to \PP^2\) and \(\alpha \in \mathrm{Br}(S)\). The
mutation argument in \cite[pp. 228--230]{Kuznetsov:Cubic} remains valid here
and yields an equivalence
\[
\Dbcoh(S,\alpha) \simeq \mathcal{A}_X.
\]
Finally, remark that \(S\) is a K3 surface: it is smooth since
\(\widetilde{X}\) is; \(\omega_S \cong \sO_S\) since \(\mathcal{A}_X\) is a K3
category; and \(\mathrm{H}^1(S,\sO_S) = 0\) thanks to Noether's formula
\(12\chi(S,\sO_S) = c_2(S)^2\) and the fact that the Hochschild homology of
\(\mathcal{A}_X\) recovers the Betti numbers of a K3 surface.
\end{proof}

The final few paragraphs give explicit models for the K3 surface \(S\). Without
further assumptions on \(X\), the discriminant cover
\(S_0 \coloneqq \Disc(\widetilde{X}/\PP^2)\) of
\(a \colon \widetilde{X} \to \PP^2\) from
\parref{applications-discriminant-algebra} provides perhaps the most convenient
model for \(S\). Note that \parref{applications-discriminant-cover} and
\parref{applications-discriminant-cover-simple-degen} give a canonical
birational morphism \(S \to S_0\) over \(\PP^2\)
which is an isomorphism away from from the corank \(2\) locus of
\(a \colon \widetilde{X} \to \PP^2\).

\subsectiondash{}\label{applications-cubic-fourfolds-discriminant}
To give an equation for \(S_0\), choose projective coordinates
\((x_0:x_1:x_2:y_0:y_1:y_2)\) for \(\PP^5\) such that the plane in \(X\) is
given by \(P = \mathrm{V}(y_0,y_1,y_2)\), and identify \((y_0:y_1:y_2)\) as the
coordinates on the \(\PP^2\) at the base of linear projection
\(a \colon \widetilde{X} \to \PP^2\). The equation of \(X\) takes the form
\[
f =
g(y_0,y_1,y_2) +
\sum\nolimits_{i = 0}^2 q_i(y_0,y_1,y_2) x_i +
\sum\nolimits_{0 \leq i \leq j \leq 2} \ell_{i,j}(y_0,y_1,y_2) x_ix_j
\]
where \(g \in \mathrm{H}^0(\PP^2,\sO_{\PP^2}(3))\),
\(q_i \in \mathrm{H}^0(\PP^2,\sO_{\PP^2}(2))\), and
\(\ell_{i,j} \in \mathrm{H}^0(\PP^2,\sO_{\PP^2}(1))\). The quadratic form
defining \(\widetilde{X}\) may be defined by the non-symmetric bilinear form
\[
\begin{pmatrix}
\ell_{0,0} & \ell_{0,1} & \ell_{0,2} & q_0 \\
           & \ell_{1,1} & \ell_{1,2} & q_1 \\
           &            & \ell_{2,2} & q_2 \\
           &            &            & g
\end{pmatrix} \colon
\sO_{\PP^2}^{\oplus 3} \oplus \sO_{\PP^2}(-1) \to
\sO_{\PP^2}(1)^{\oplus 3} \oplus \sO_{\PP^2}(2).
\]
By \parref{applications-discriminant-algebra-split}, \(S_0 \to \PP^2\) is the
flat double cover given by
\(z^2 - (\ell_{0,1}q_2 - \ell_{0,2}q_1 + \ell_{1,2} q_0) z + \gamma = 0\)
where
\begin{multline*}
\gamma \coloneqq
( \ell_{0,0}\ell_{1,2}^2
+ \ell_{1,1}\ell_{0,2}^2
+ \ell_{2,2}\ell_{0,1}^2
- \ell_{0,1}\ell_{0,2}\ell_{1,2}
-4\ell_{0,0}\ell_{1,1}\ell_{2,2})g
+ \ell_{0,1}\ell_{1,2}q_0q_2 \\
+ \ell_{1,1}\ell_{2,2}q_0^2
+ \ell_{0,0}\ell_{2,2}q_1^2
+ \ell_{0,0}\ell_{1,1}q_2^2
- \ell_{0,0}\ell_{1,2}q_1q_2
- \ell_{1,1}\ell_{0,2}q_0q_2
- \ell_{2,2}\ell_{0,1}q_0q_1.
\end{multline*}

\subsectiondash{}\label{applications-cubic-fourfolds-two-planes}
One particularly pleasant situation in which global equations for \(S\) can
be given is when \(X\) contains a second plane \(P'\) disjoint from \(P\).
Choose coordinates as in \parref{applications-cubic-fourfolds-discriminant}
such that \(P = \mathrm{V}(y_0,y_1,y_2)\) and, furthermore, \(P' = \mathrm{V}(x_0,x_1,x_2)\).
The equation of \(X\) may then be written as the sum
\[
f =
f_{1,2}(x_0,x_1,x_2;y_0,y_1,y_2) +
f_{2,1}(x_0,x_1,x_2;y_0,y_1,y_2)
\]
where \(f_{1,2}\) and \(f_{2,1}\) are bihomogeneous of bidegrees \((1,2)\) and
\((2,1)\), respectively. Viewed as functions on \(\PP^2 \times \PP^2\), these
cut out \(S\): see
\cite[\S1.2]{Hassett:Survey}.

\subsectiondash{}\label{applications-cubic-fourfolds-fermat}
As an explicit example, consider the cubic fourfold
\[
X \coloneqq
\big\{(x_0:x_1:x_2:y_0:y_1:y_2) \in \PP^5 :
x_0^2 y_0 + x_1^2 y_1 + x_2^2 y_2 + x_0 y_0^2 + x_1 y_1^2 + x_2 y_2^2 = 0\big\}.
\]
By \parref{applications-cubic-fourfolds-discriminant}, the discriminant cover
\(S_0 \to \PP^2\) is given by \(z^2 - y_0y_1y_2(y_0^3 + y_1^3 + y_2^3) = 0\).
In particular, \(S_0\) is singular above points of intersection amongst the
various components of the divisor appearing. By
\parref{applications-cubic-fourfolds-two-planes}, \(S_0\) is resolved by
the K3 surface \(S\) may be given as
\[
S \coloneqq
\big\{((x_0:x_1:x_2),(y_0:y_1:y_2)) \in \PP^2 \times \PP^2 :
x_0^2 y_0 + x_1^2 y_1 + x_2^2 y_2 =
x_0 y_0^2 + x_1 y_1^2 + x_2 y_2^2 = 0
\big\}.
\]
When \(\kk\) has characteristic \(2\), both
\cite[Theorem 1.1(vi) and (viii)]{DK:K3} show that \(S\) is the supersingular
K3 surface of Artin invariant \(1\). Since \(X\) here is a smooth \(q\)-bic
fourfold with \(q = 2\) in the sense of \cite{fano-schemes}, this is isomorphic
to the Fermat cubic fourfold. This completes the proof of Theorem
\parref{intro-theorem-fermat-cubic}.
\qed

\bibliographystyle{amsalpha}
\bibliography{main}
\end{document}